\documentclass[preprint,authoryear]{elsarticle}

\usepackage{manfnt}

\usepackage{url}
\usepackage{xparse}
\usepackage{microtype}
\usepackage{ebproof}
\usepackage{upgreek}
\usepackage{stmaryrd}
\usepackage{etoolbox}
\usepackage{ifluatex}

\usepackage{amsmath}
\usepackage{amsfonts}
\usepackage{amssymb}
\usepackage[hang,flushmargin]{footmisc}

\usepackage{xcolor}

\definecolor{Matterhorn}{RGB}{77,77,77}
\definecolor{RegalBlue}{RGB}{3,69,117}
\definecolor{RedDevil}{RGB}{134,1,17}

\ifluatex
\microtypesetup{
  activate={true,nocompatibility},
  final,
  tracking=false,
  factor=1100,
  stretch=10,
  shrink=10
}
\else
\microtypesetup{
  activate={true,nocompatibility},
  final,
  tracking=false,
  kerning=true,
  spacing=true,
  factor=1100,
  stretch=10,
  shrink=10
}
\fi

\usepackage{hyperref}
\hypersetup{
  breaklinks=true,
  colorlinks=true,
  linkcolor=RedDevil,
  urlcolor=RegalBlue,
  citecolor=RedDevil,
  bookmarksopen=true,
  bookmarksnumbered=true
}

\usepackage{silence}
\usepackage[shortlabels]{enumitem}

\setlist{itemsep = 0pt}
\setlist[1]{labelindent=\parindent}
\setlist[enumerate,1]{label = \arabic*),ref = \arabic*}
\setlist[enumerate,2]{label = \emph{\alph*}),ref = \theenumi.\emph{\alph*}}
\setlist[enumerate,3]{label = \roman*),ref = \theenumii.\roman*}
\setlist[description]{%
  font={\normalfont\itshape\bfseries}
}

\usepackage{amsthm}
\usepackage[nameinlink,capitalize,noabbrev]{cleveref}

\usepackage{caption, subcaption}

\usepackage[utf8]{inputenc}

\usepackage{tikz, tikz-cd}
\usepackage{mathtools} 
\usepackage{thmtools, thm-restate} 
\usepackage{bbm} 

\Crefname{ex}{Example}{Examples}
\Crefname{defn}{Definition}{Definitions}

\crefformat{equation}{#2Equation~#1#3}

\usepackage{stackengine, old-arrows}
\usepackage{scalerel}

\usepackage[all]{nowidow}
\usepackage[lastparline]{impnattypo}

\allowdisplaybreaks

\usepackage{./quiver}

\microtypecontext{spacing=nonfrench}

\usetikzlibrary{decorations.pathreplacing} 

\graphicspath{graphics/}

\usepackage{rotating}
\usepackage{adjustbox}

\DeclareMathOperator{\Aa}{\mathcal{A}}
\DeclareMathOperator{\Ba}{\mathcal{B}}
\DeclareMathOperator{\Ca}{\mathcal{C}}
\DeclareMathOperator{\Da}{\mathcal{D}}
\DeclareMathOperator{\Ea}{\mathcal{E}}
\DeclareMathOperator{\Fa}{\mathcal{F}}
\DeclareMathOperator{\Ga}{\mathcal{G}}

\DeclareMathOperator{\Ma}{\mathcal{M}}

\DeclareMathOperator{\Sa}{\mathcal{S}}

\DeclareMathOperator{\Xa}{\mathcal{X}}

\DeclareMathOperator{\Ab}{\mathbb{A}}
\DeclareMathOperator{\Bb}{\mathbb{B}}
\DeclareMathOperator{\Cb}{\mathbb{C}}

\DeclareMathOperator{\Hb}{\mathbb{H}}
\DeclareMathOperator{\Kb}{\mathbb{K}}

\DeclareMathOperator{\Mb}{\mathbb{M}}

\DeclareMathOperator{\Xb}{\mathbb{X}}
\DeclareMathOperator{\Yb}{\mathbb{Y}}

\DeclareMathOperator{\Hc}{\mathfrak{H}}
\DeclareMathOperator{\Kc}{\mathfrak{K}}

\DeclareMathOperator{\id}{\mathrm{id}}

\DeclareMathOperator{\Fun}{\mathrm{Fun}}

\DeclareMathOperator{\Type}{\textbf{Type}}

\DeclareMathOperator{\Zero}{\textbf{0}}
\DeclareMathOperator{\One}{\textbf{1}}
\DeclareMathOperator{\fst}{\texttt{fst}}
\DeclareMathOperator{\snd}{\texttt{snd}}

\DeclareMathOperator{\inl}{\texttt{inl}}
\DeclareMathOperator{\inr}{\texttt{inr}}

\DeclareMathOperator{\Bool}{\textbf{Bool}}

\newcommand{\letin}[3]{\texttt{\textbf{let}}\ {#1}:={#2}\ \texttt{\textbf{in}}\ {#3}}

\newcommand{\dom}{\partial_0}
\newcommand{\cod}{\partial_1}
\newcommand{\biglens}[2]{
    \begin{pmatrix}{\vphantom{f_f^f}#1} \\ {\vphantom{f_f^f}#2} \end{pmatrix}
}
\newcommand{\littlelens}[2]{
    \begin{psmallmatrix}{\vphantom{f}#1} \\ {\vphantom{f}#2} \end{psmallmatrix}
}
\newcommand{\lens}[2]{
  \relax\if@display
    \biglens{#1}{#2}
  \else
    \littlelens{#1}{#2}
  \fi
}

\newcommand{\dsum}[1]{({#1}) \times }       

\newcommand{\xto}[1]{\xrightarrow{#1}}

\newcommand{\pto}{\,\cdot\kern-.1em{\to}\,}
\newcommand{\pxto}[1]{\,\cdot\kern-.1em{\xto{#1}}\,}

\makeatletter
\providecommand*{\xmapstofill@}{%
  \arrowfill@{\mapstochar\relbar}\relbar\rightarrow
}
\providecommand*{\xmapsto}[2][]{%
  \ext@arrow 0395\xmapstofill@{#1}{#2}%
}

\makeatletter
\def\slashedarrowfill@#1#2#3#4#5{%
  $\m@th\thickmuskip0mu\medmuskip\thickmuskip\thinmuskip\thickmuskip
  \relax#5#1\mkern-7mu%
  \cleaders\hbox{$#5\mkern-2mu#2\mkern-2mu$}\hfill
  \mathclap{#3}\mathclap{#2}%
  \cleaders\hbox{$#5\mkern-2mu#2\mkern-2mu$}\hfill
  \mkern-7mu#4$%
}
\def\rightslashedarrowfill@{%
  \slashedarrowfill@\relbar\relbar\mapstochar\rightarrow}
\newcommand\xslashedrightarrow[2][]{%
  \ext@arrow 0055{\rightslashedarrowfill@}{#1}{#2}}
\makeatother

\newcommand{\topro}{\xslashedrightarrow{}}

\tikzset{ vert/.style={anchor=south, rotate=90, inner sep=.5mm} }

\newcommand{\internal}[1]{\ulcorner{#1}\urcorner}

\makeatletter
\let\aliasctr@orig@addtoreset\@addtoreset
\renewcommand*\@addtoreset[2]{\aliasctr@orig@addtoreset{#1}{#2}%
  \expandafter\xdef\csname theH#1\endcsname{%
    \expandafter\noexpand\csname theH#2\endcsname.\noexpand\the\noexpand\value{#1}}}
\makeatother

\newtheorem{thm}{Theorem}[section]

\theoremstyle{definition}
\newtheorem{defn}[thm]{Definition}

\newtheorem{ex}[thm]{Example}

\newtheorem{rmk}[thm]{Remark}
\newtheorem{notation}[thm]{Notation}

\newtheorem{lem}[thm]{Lemma}

\newtheorem{cor}[thm]{Corollary}
\newtheorem{prop}[thm]{Proposition}

\numberwithin{equation}{subsection}

\newcommand{\makeitwide}{\displayindent0pt\displaywidth\textwidth}

\usepackage{scrextend}

\AddToHook{env/equation/before}{%
  \makeitwide%
  \begin{addmargin*}[-3cm]{-2.25cm}
  \setlength{\belowdisplayshortskip}{10pt}%
}
\AddToHook{env/equation/after}{%
  \end{addmargin*}
}

\newenvironment{eqalign}{\begin{equation}\begin{aligned}}{\end{aligned}\end{equation}}
\newenvironment{eqalign*}{\begin{equation*}\begin{aligned}}{\end{aligned}\end{equation*}}

\tikzcdset{
  row sep/normal={
    6ex
  },
  column sep/normal={
    8ex
  }
}

\newcommand{\Overset}[2]{%
  \mathop{#2}\limits^{\vbox to -.1ex{%
  \kern -1.8ex\hbox{$#1$}\vss}}%
}
\newcommand{\Underset}[2]{%
  \mathop{#2}\limits_{\vbox to .1ex{%
  \kern -.6ex\hbox{$#1$}\vss}}%
}

\newcommand{\then}{\mathbin{\fatsemi}}

\mathchardef\dash="2D

\newcommand{\from}{\leftarrow}

\newcommand{\longto}{\longrightarrow}

\newcommand{\into}{\hookrightarrow}

\newcommand{\monoto}{\rightarrowtail}
\newcommand{\monofrom}{\leftarrowtail}

\newcommand{\epito}{\twoheadrightarrow}
\newcommand{\epifrom}{\twoheadleftarrow}

\newcommand{\equalto}{=\mathrel{\mkern-3mu}=}
\newcommand{\nequalto}[1]{\overset{#1}{=\mathrel{\mkern-3mu}=}}

\newcommand{\twoto}{\Rightarrow}
\newcommand{\twofrom}{\Leftarrow}

\newcommand{\threeto}{\Rrightarrow}

\newcommand{\narrow}[2]{\overset{#1}{#2}}
\newcommand{\nto}[1]{\xrightarrow{#1}}
\newcommand{\nlongto}[1]{\xrightarrow{\;#1\;}}

\newcommand{\nepito}[1]{\narrow{#1}{\epito}}

\newcommand{\nfrom}[1]{\xleftarrow{#1}}

\newcommand{\nepifrom}[1]{\narrow{#1}{\epifrom}}

\newcommand{\nisoto}[1]{\xrightarrow[#1]{\sim}}

\newcommand{\isoto}{\nisoto{}}
\newcommand{\isolongto}{\overset{\sim}\longrightarrow}

\newcommand{\looseto}{{\stackMath\mathrel{\stackinset{c}{-0.15ex}{c}{0.15ex}{\shortmid}{\longto}}}}

\newcommand{\nlooseto}[1]{\narrow{#1}{\looseto}}

\newcommand{\N}{\mathbb{N}}

\newcommand{\cat}[1]{\mathbf{#1}}

\newcommand{\dbl}[1]{\mathbb{#1}}
\newcommand{\dblcat}[1]{\cat{\dbl #1}}
\newcommand{\trpl}[1]{\mathfrak{#1}}
\newcommand{\trplcat}[1]{\cat{\trpl #1}}

\newcommand{\iso}{\cong}
\newcommand{\equi}{\simeq}

\newcommand{\adj}{\dashv}

\newcommand{\Cat}{\dblcat{Cat}}

\newcommand{\Set}{\cat{Set}}

\newcommand{\Mon}{\dblcat{Mon}}

\newcommand{\Gray}{\cat{Gray}}

\newcommand{\MonCat}{{\dblcat{Mon}(\Cat)}}
\newcommand{\BrMonCat}{{\dblcat{Br}\MonCat}}
\newcommand{\SymMonCat}{{\dblcat{Sym}\MonCat}}
\newcommand{\Cart}{\dblcat{Cart}}
\newcommand{\CartCat}{{\Cart(\Cat)}}

\newcommand{\re}{^\mathrm{re}}
\newcommand{\co}{^\mathrm{co}}
\newcommand{\op}{^\mathrm{op}}
\newcommand{\opre}{^\mathrm{reop}}
\newcommand{\core}{^\mathrm{reco}}
\newcommand{\coop}{^\mathrm{coop}}

\newcommand{\rev}{^\mathrm{rev}}
\newcommand{\lop}{^\mathrm{lop}}

\newcommand{\ltop}{^\mathrm{ltop}}

\newcommand{\Fib}{\dblcat{Fib}}
\newcommand{\Span}{\dblcat{Span}}
\newcommand{\Cospan}{\dblcat{Cospan}}
\newcommand{\TriSpan}{\trplcat{Span}}

\newcommand{\FibSpan}{{\cat{f}\TriSpan}}
\newcommand{\OpfibSpan}{{\cat{o}\TriSpan}}

\newcommand{\PsCat}{{\trplcat{Ps}\trplcat{Cat}}}

\newcommand{\Ctx}{\mathfrak{Ctx}}
\newcommand{\Cnt}{\mathfrak{Cnt}}
\newcommand{\DispSpan}{{\cat{d}\TriSpan}}

\newcommand{\Comnd}{\dblcat{Comnd}}

\newcommand{\Para}{\dblcat{Para}}
\newcommand{\Copara}{\dblcat{Copara}}
\newcommand{\para}{\Para}
\newcommand{\copara}{\Copara}

\newcommand{\Kl}{\dblcat{Kl}}
\newcommand{\Alg}{\dblcat{Alg}}

\newcommand{\lax}{l}
\newcommand{\colax}{c}
\newcommand{\pseudo}{p}
\newcommand{\strict}{s}

\newcommand{\DblCat}{\dblcat{Dbl}\Cat}

\newcommand{\acted}{\Ca}
\newcommand{\actor}{\Ma}
\newcommand{\combine}{\mathbin{\otimes}}
\newcommand{\combineunit}{I}
\newcommand{\action}{\mathbin{\odot}}

\newcommand{\fibcolaxaction}[1][]{{\acted#1 \nepifrom{p#1} \actor#1 \nto{\action#1} \acted#1}}
\newcommand{\opfiblaxaction}[1][]{{\acted#1 \nfrom{\action#1} \actor#1 \nepito{q#1} \acted#1}}
\newcommand{\ctxtad}[1][]{{\acted#1 \nto{p#1} \actor#1 \nto{\action#1} \acted#1}}

\newcommand{\lineator}{\ell}
\newcommand{\colineator}{o}
\newcommand{\naturator}{\nu}

\newcommand{\mapunitor}{\eta}
\newcommand{\mapmultiplicator}{\mu}

\newcommand{\counitor}{\varepsilon}
\newcommand{\coassociator}{\delta}
\newcommand{\unitor}{\eta}
\newcommand{\associator}{\mu}

\newcommand{\Comonad}{D}
\newcommand{\Monad}{T}

\newcommand{\monoidal}{\mathbin{\boxtimes}}
\newcommand{\monoidalunit}{1}

\newcommand{\leftunitlaw}{\lambda}
\newcommand{\rightunitlaw}{\rho}
\newcommand{\associativitylaw}{\alpha}

\newcommand{\isofibto}{\to}
\newcommand{\fibto}{\epito}

\newcommand{\spancomp}{\mathbin{\rotatebox[origin=c]{90}{$\ltimes$}}}

\tikzcdset{
  isofib/.code={
    \pgfsetarrowsend{tikzcd to}
  }
}

\newcommand{\consdownarrows}{\begin{smallmatrix}\downarrow\\[-0.2ex]\downarrow\end{smallmatrix}}

\newcommand{\KL}{\mathsf{KL}}
\newcommand{\EM}{\mathsf{EM}}
\newcommand{\wreath}{\wr}

\newcommand{\fun}{{\mathsf{fun}}}
\newcommand{\opfun}{{\mathsf{opfun}}}

\newcommand{\unit}{\eta}
\newcommand{\counit}{\epsilon}

\newcommand{\monad}{s}
\newcommand{\unitmnd}{\upsilon}
\newcommand{\multmnd}{\mu}

\newcommand{\pull}{\mathrm{pull}}
\newcommand{\push}{\mathrm{push}}
\newcommand{\lift}{\mathrm{lift}}
\newcommand{\counitpull}{\lambda}
\newcommand{\unitpull}{\eta}

\newcommand{\unitarrow}{\mathrm{\iota}}
\newcommand{\multarrow}{\mathrm{\gamma}}

\newcommand{\intertwiner}{w}

\newcommand{\Cosmos}{\Kb}
\newcommand{\Display}{\Da}
\newcommand{\Paradise}{(\Cosmos,\Display)}
\newcommand{\DualParadise}{(\Cosmos\co,\Display)}

\providecommand{\comma}{}
\renewcommand{\comma}{\downarrow}

\newcommand{\generic}{\chi}

\newcommand{\Bicat}{\trplcat{Bicat}}
\newcommand{\TwoCat}{\Bicat^\strict_\strict}

\newcommand{\InclusionTrifun}{(-,\, (-)^\downarrow)}

\newcommand{\colaxity}[1]{\overline{#1}}

\newcommand{\simple}{s}
\newcommand{\Simple}[1]{#1 \ltimes #1}

\newcommand{\Fam}{\mathbf{Fam}}

\newcommand{\Cxd}{\trplcat{Cxd}}
\newcommand{\Cnd}{\trplcat{Cnd}}

\newcommand{\extch}{\mathbin{\&}}

\newcommand{\lcomp}{\then}
\newcommand{\lid}{1}

\newcommand{\AdTrpl}{{\trplcat{Ad}\trplcat{Trpl}}}

\newcommand{\forwto}{\rightarrowtail}
\newcommand{\backto}{\twoheadrightarrow}
\newcommand{\forwdown}{{\rotatebox[origin=c]{-90}{$\rightarrowtail$}}}
\newcommand{\backdown}{{\rotatebox[origin=c]{-90}{$\twoheadrightarrow$}}}

\newcommand{\FunWreaths}{\trplcat{Wrt}}

\newcommand{\Doc}{\trplcat{Doc}}
\newcommand{\Dom}{\dblcat{Dom}}
\newcommand{\Par}{\dblcat{Par}}

\newcommand{\tow}[1]{#1 \cdot}
\newcommand{\cartesianator}{\kappa}
\newcommand{\combineunitcart}{\cartesianator^\combineunit}
\newcommand{\combinecart}{\cartesianator^{\combine}}
\newcommand{\barcombineunitcart}{\bar{\cartesianator}^\combineunit}
\newcommand{\barcombinecart}{\bar{\cartesianator}^{\combine}}

\newcommand{\State}{\mathsf{State}}

\newcommand{\FinProb}{\cat{FinProb}}
\newcommand{\law}{\mathrm{law}}

\newcommand{\MonadIO}{\texttt{IO}}
\newcommand{\ComonadOI}{\texttt{OI}}
\newcommand{\BoolIf}{\texttt{If}}
\newcommand{\Unit}{\One}
\newcommand{\True}{\texttt{true}}
\newcommand{\False}{\texttt{false}}
\newcommand{\supp}{\texttt{supp}}
\newcommand{\andbool}{\texttt{and}}
\newcommand{\ok}{\texttt{ok}}
\newcommand{\splitt}{\texttt{split}}
\newcommand{\seq}{\texttt{seq}}

\hypersetup{bookmarksdepth=3}

\journal{Journal of Pure and Applied Algebra}

\begin{document}

\begin{frontmatter}

\title{\texorpdfstring{Contextful Arrows via Wreath Products\\[1.5ex]
  {\normalsize\normalfont\itshape Dedicated to Ross Street on the occasion of his 80th birthday}}{Contextful Arrows via Wreath Products}}

\author[strath,mc]{Matteo Capucci}
\ead{matteo.capucci@gmail.com}
\affiliation[strath]{organization={University of Strathclyde},
            addressline={Department of Computer and Information Sciences},
            city={Glasgow},
            country={UK}}
\affiliation[mc]{organization={Independent Researcher},
            city={Modena},
            country={IT}}

\author[djm]{David Jaz Myers}
\ead{davidjaz@topos.institute}
\affiliation[djm]{organization={Topos Research UK},
            city={Oxford},
            country={UK}}

\begin{abstract}
	We introduce contextads and the $\Ctx$ construction, unifying various structures and constructions in category theory dealing with context and contextful arrows---comonads and their Kleisli construction, actegories and their Para construction, adequate triples and their Span construction.

	Contextads are defined in terms of Lack--Street wreaths, suitably categorified for pseudomonads in a tricategory of spans in a 2-category with display maps.
	The associated wreath product provides the $\Ctx$ construction, and by its universal property we conclude trifunctoriality.
	This abstract approach lets us work \emph{up to structure}, and thus swiftly prove that, under very mild assumptions, a contextad equipped colaxly with a 2-algebraic structure produces a similarly structured double category of contextful arrows.

	We also explore the role contextads might play \emph{qua} dependently graded comonads in organizing contextful computation in functional programming.
	We show that many side-effects monads can be dually captured by dependently graded comonads, and gesture towards a general result on the `transposability' of parametric right adjoint monads to dependently graded comonads.
	We will use this theorem to deduce that duoidal actions of monoidal categories have monoidal $\para$-constructions.
\end{abstract}

\end{frontmatter}

\section{Introduction}
\label{sec:intro}

There are a number of situations across mathematics where we want to consider a map ${f : A \to B}$ as depending not only on its domain $A$ but also on some extra, perhaps implicit context.
The most basic and common case of this phenomenon is \emph{parameterization}, where we consider a function ${f : A \times P \to B}$ of two variables as ``really'' a function ${A \to B}$ which nevertheless depends on some external \emph{parameters} in $P$.
Likewise, mathematicians working on the semantics of functional programming languages have long observed that context-sensitive computations of type $A \to B$ are well represented by (co)Kleisli morphisms ${f : \Comonad A \to B}$, where $\Comonad$ is a comonad abstracting appropriate elements of the context.
Finally, in dependent type theories context is a foundational notion, and it is commonplace to ponder morphisms $f : (a:A) \times P(a) \to B$ which have access to additional context which depend on the input.

In this paper, we will see all of these examples as part of a single universal construction in higher category theory: the wreath product of pseudomonads in tricategories of spans.
Wreaths and their products have been introduced by Street and Lack in \cite{lack_formal_2002}.

While one can form parameterized maps (the $\para$ construction) given an action of a monoidal category (of parameter spaces on a category of interface spaces), and while one can form the Kleisli construction given a comonad, we will see our construction as taking a \emph{contextad} $\action$ on a category (which might also have been called a \emph{colax fibred action} or \emph{dependently graded comonad} on that category) and producing a double category $\Ctx(\action)$ of contextful arrows.
We call this \emph{the $\Ctx$ construction}.
This level of abstraction has a number of benefits.

\begin{enumerate}
	\item Simply, it is precisely the right level of generality, which we got to by following the math: contextads on a category $\acted$ are exactly wreaths around the double category of commuting squares in $\acted$, considered as a pseudomonad in spans, and the $\Ctx$ construction is the wreath product.

	\item There are new examples of contextads which are not just actegories or graded comonads.
	The $\Ctx$ constructions of these examples include double categories of partial maps, spans, and even Kleisli categories of polynomial monads.
	We encourage readers interested in dependently graded comonads to skip forward to the largely self-contained \cref{sec:ctx.as.dep.graded.comonad}.

	\item Since our construction is entirely abstract, it works just as well in 2-categories other than $\Cat$.
	We will exploit this fact to show that when the contextad is suitably structured, its $\Ctx$ construction inherits this structure. This proves a variety of folklore theorems concerning the $\para$ construction; specifically, the $\para$ construction of a braided (resp. symmetric) monoidal action is a (symmetric) monoidal double category \cref{cor:para.monoidal} (resp. \cref{cor:para.symmetric.monoidal}).
\end{enumerate}

We thus deem fitting to dedicate this work to Ross Street.
He more than anyone showed how useful and insightful pursuing the `formal' nature of a categorical concept can be.
Category theory \emph{Street style} has deeply shaped our own taste and approach, and we believe this comes across especially clearly in this work.
Moreover, nearly every categorical notion we use in the following pages bears Street's mark, chiefly the formal theory of monads \cite{streetFormalTheoryMonads1972} and wreaths \cite{lack_formal_2002} and the characterization of fibrations as algebras for arrow monads in span \cite{street1980fibrations}.

\medskip
To explain our generalization and how the wreath product appears, we will begin at two starting points: parameterized maps in \cref{sec:para.intro}, and Kleisli maps for a comonad in \cref{sec:kleisli.intro}.

\subsection{Parameterized maps and the Para construction} \label{sec:para.intro}
In a parameterized function $f : A \times P \to B$, we think of the parameters $P$ as playing a qualitatively different role to that of the input domain $A$; for example, their values may be set by a different method.
At a very high level, one may understand modern machine learning as a process for finding a parameter $p \in P$ so $f(-, p) : A \to B$ closely matches a target function; here the parameters are set by, say, gradient descent, while the input $a \in A$ would be set by a user of the model.
This difference in roles between parameter and input expresses itself in the way that parameterized maps are composed: given $f : A \times P \to B$ and $g : B \times Q \to C$, we may define a composite $g \hat{\circ} f : A \times (P \times Q) \to C$ by $(g \hat{\circ} f)(a, (p, q)) := g(f(a, p), q)$---that is, we feed forward the input, but we preserve both parameters.

This kind of composition of parameterized maps was, to our knowledge, first described categorically by Hermida and Tennent in \cite{hermida_monoidal_2012}, where the data of a symmetric monoidal functor is used to build a bicategory (which is then truncated to a category) with `extra maps'.

Only much later, sparked by the work of Fong, Spivak, and Tuyeras on a functorial description of backpropagation \cite{fong-spivak-tuyeras-backprop-as-functor}, the wider applied category theory community got notice of the interesting abstract construction which became known as $\para$.
In \emph{ibid.}, they define two categories---one of parameterized \emph{functions} and one of parameterized \emph{learners} and show that backpropagation can be defined as a functor from the first to the latter.
This use of parameterized lenses in understanding gradient-descent based learning is continued by Gavranovi{\'c} \emph{et al} in \cite{gavranovic_compositional_2019, cruttwell_categorical_2022}.
Parameterized lenses do not only appear as learners, but also as \emph{controlled systems} $\lens{A^-}{A^+} \otimes  \lens{TS}{S}\leftrightarrows \lens{B^-}{B^+}$ in the setting of categorical cybernetics \cite{capucci_towards_2022}, where the parameters $\lens{TS}{S}$ play the role of the \emph{controller} whose state controls a \emph{plant} or process $\lens{A^-}{A^+} \leftrightarrows \lens{B^-}{B^+}$.
In this case, we see that the parameters need not be an object of the same category as the inputs (here we take a single state space $S$ as parameter, while the input is an upstream-downstream pair of spaces), so long as we can form the combined input-parameter space $\lens{A^-}{A^+} \otimes \lens{TS}{S}$.

This suggests that rather than taking $\para$ as a single category, we should see it as a \emph{construction} which takes the action $\action : \acted \times \actor \to \acted$ of a monoidal category $(\actor, \combine, \combineunit)$---an \emph{actegory}---and produces a new category $\para(\action)$ whose objects are the same as $\acted$ but whose morphisms $A \to B$ are parameterized maps $f : A \action P \to B$ with $P \in \actor$.
In fact this connects back to \cite{hermida_monoidal_2012}, since by \cite[Corollary~5.5.12]{capucci2022actegories} a symmetric monoidal functor $\Ma \to \Ca$ is equivalent to a (symmetric monoidal) action of $\Ma$ on $\Ca$.

Since we compose not only the parameterized functions but also the parameter spaces themselves ($P$ and $Q$ are combined into $P \combine Q$), composition is not strictly unital or associative.
To end up with a category, we might therefore quotient by the relation that $(P, f) \sim (P', f')$ when $\alpha : P \cong P'$ is an isomorphism for which $f' \circ (\id_A \action \alpha) = f : A \action P \to B$; this is the approach taken in \cite[Definition~1]{gavranovic_compositional_2019}.
But the relation $f' \circ (\alpha \otimes \id_A) = f$ shows up quite often, and not only for invertible $\alpha : P \to P'$; this equality expresses that $f'$ is a \emph{reparameterization} of $f$.
It is worth keeping track of the reparameterizations as morphisms between parameterized maps.
We may instead see $\para(\action)$ as a bicategory whose 2-cells are such reparameterizations, following \cite[Definition~2]{capucci_towards_2022}.

In this paper, however, we will see that $\para(\action)$ is most naturally thought of as a \emph{double category} with tight morphisms those of $\acted$ and loose arrows the parameterized maps.

\begin{defn}
	A (right) \textbf{actegory} consists of a monoidal category $(\actor, \combine, \combineunit)$ together with an action $\action : \acted \times \actor \to \acted$ with natural isomorphisms
	 \begin{equation}
	 	\coassociator : C \action (M \combine N) \isoto (C \action M) \action N
	 	\quad\text{and}\quad
	 	\counitor : C \action \combineunit \isoto C
	 \end{equation}
	 satisfying coherences reminiscent of those defining a monoidal category.%
	\footnote{See \cite[Definition~3.1.1]{capucci2022actegories}, though note the change of direction of $\coassociator$ and $\counitor$.}

	Equivalently, $\action : \acted \times \actor \to \acted$ defines an actegory when its transpose $\action : \actor \to \Fun(\acted, \acted)$ has the structure of a strong monoidal functor $(\actor, \combine, \combineunit) \to (\Fun(\acted, \acted), \circ, \id)$: the compositor is $\coassociator$ and the unitor is $\counitor$.
\end{defn}

\begin{defn}
\label{defn:para.actegory}
	Given an actegory $\action : \acted \times \actor \to \acted$, we define its \textbf{Para construction} $\para(\action)$ to be the double category with:
	\begin{enumerate}
		\item Tight category the category $\acted$.
		\item Loose arrows $(M, f) : A \topro B$ pairs $M \in \Ma$ and $f : A \action M \to B$ in $\Ca$.
		      Composition of loose arrows $(M, f) : A \topro B$ and $(N, g) : B \topro C$ is given by $((M \combine N), g (M \action f) \action \coassociator)$:
		      \begin{equation}
			  \label{eqn:para.comp.intro}
			      A \action (M \combine N) \xto{\coassociator} (A \action M) \action N \xto{f \action N} B \action N \xto{g} C.
		      \end{equation}
		      The loose identity is $(\combineunit, \counitor) : A \topro A$.
		\item A square as on the left consists of a $\varphi : M \to M'$ such that the square on the right commutes:
		      \begin{equation}
			      \label{eqn:para.square}
			      \begin{tikzcd}
				      {A} & {B} \\
				      {A'} & {B'}
				      \arrow["{h}"', from=1-1, to=2-1]
				      \arrow["{k}", from=1-2, to=2-2]
				      \arrow[""{name=0, anchor=center, inner sep=0}, "{(M, f)}", "\shortmid"{marking}, from=1-1, to=1-2]
				      \arrow[""{name=1, anchor=center, inner sep=0}, "{(M', f')}"', "\shortmid"{marking}, from=2-1, to=2-2]
				      \arrow["\varphi", shorten <=4pt, shorten >=4pt, Rightarrow, from=0, to=1]
			      \end{tikzcd}
			      \quad\quad\quad
			      \begin{tikzcd}
				      {A \action M} & {B} \\
				      {A' \action M'} & {B'}
				      \arrow["{f}", from=1-1, to=1-2]
				      \arrow["{f'}"', from=2-1, to=2-2]
				      \arrow["{h \action \varphi}"', from=1-1, to=2-1]
				      \arrow["{k}", from=1-2, to=2-2]
			      \end{tikzcd}
		      \end{equation}
		      Tight composition of squares is given by composition in $\actor$, while loose composition of squares is given by $\combine$.
		\item The unitors $\leftunitlaw$ and $\rightunitlaw$ and associator $\associativitylaw$ of the monoidal category $\combine$ give unitors and associators for loose composition, and the interchange law in $\actor$ gives the interchange law in $\para(\action)$.
	\end{enumerate}
\end{defn}

The main difference between \cref{defn:para.actegory} and the $\para$ bicategory in \cite[§2.2]{hermida_monoidal_2012} or Definition 2 of \cite{capucci_towards_2022} is that reparameterizations in our double category may also modify domain and codomain, by means of the tight maps in the double category.
Moreover, compared to the latter, we use 2-cells in the reverse direction (agreeing with Hermida and Tennent).

For us, the primary benefit of considering $\para(\action)$ as a double category is that double categories are equivalently pseudomonads in the tricategory of spans of categories---just as categories are monads in the bicategory of spans of sets.
The central result of this paper is that the Para construction arises as a very natural construction for pseudomonads---the \emph{wreath product}.

But before we get there, let's notice one important feature of \cref{defn:para.actegory}: there was no reason to assume that $\coassociator$ and $\counitor$ were invertible.
The definition works just as well if $\coassociator$ and $\counitor$ give $\action : \acted \times \actor \to \acted$ the structure of a \emph{colax action}, or equivalently if they give the transpose $\action : \actor \to \Fun(\acted, \acted)$ the structure of a colax monoidal functor.
The latter notion---a colax monoidal functor into a functor category---is known as a \emph{graded comonad}, and the Para construction is exactly its \emph{Kleisli (double) category}.

\subsection{Comonads and the Kleisli construction}\label{sec:kleisli.intro}

The use of monads to give semantics for programs with side-effects goes back to Moggi's seminal work on models of computation in Kleisli categories for monads \cite{moggi_notions_1991}.
In the monadic picture, a side effect is produced as a result of a computation; an effectful computation is therefore typed as $f : A \to \Monad B$ for a monad $\Monad$.
But Kieburtz argues \cite{kieburtz_codata_1999} that not all side-effects should be seen in this way.
Some side effects arise from the \emph{context} that a computation is performed in, even if the result of that computation is some data fully captured by the type system.
For this reason, Kieburtz argues that computations whose side-effects depend on their context---in particular I/O---should be typed as $f : \Comonad A \to B$---with $\Comonad$ a comonad, consuming both the input and the context to produce a (contextless) element.
In \cite{uustalu_comonadic_2008}, Uustalu and Vene extend this interpretation to more general contextful computations.

Recall that a comonad $\Comonad : \acted \to \acted$ is an endofunctor equipped with two natural transformations, $\counitor : \Comonad A \to A$ (which we think of as decontextualizing a value) and $\coassociator : \Comonad A \to \Comonad \Comonad A$ (which we think of as duplicating the context), that together endow $\Comonad$ with the structure of a comonoid in $\Fun(\acted, \acted)$.
Thinking of $\Comonad A$ as the type of $A$-values in a context determined by $\Comonad$, a contextful arrow is a Kleisli morphism $f : \Comonad A \to B$.
Composition of Kleisli morphisms involves duplicating the context and passing it forward: the composite of $f : \Comonad A \to B$ with $g : \Comonad B \to C$ is given by
\begin{equation}
\label{eqn:comonad.kleisli.comp}
	\Comonad A \xto{\coassociator} \Comonad \Comonad A \xto{\Comonad f} \Comonad B \xto{g} C.
\end{equation}
The identity for Kleisli composition is the counit $\counitor : \Comonad A \to A$.
The Kleisli category for a comonad is usually considered as just that: a category.
However, the interplay between ``pure'' and ``context-sensitive'' functions is a central theme of the use of comonads in programming; for this reason, it makes sense to consider instead the Kleisli \emph{double category} whose tight morphisms are the ``pure'' morphisms of $\acted$, whose loose arrows are the ``context-sensitive'' Kleisli morphisms, and where there is a unique square as on the left whenever the square on the right commutes:

\begin{equation}
\label{eqn:kleisli.square}
	\begin{tikzcd}[sep=scriptsize]
		{A} & {B} \\
		{A'} & {B'}
		\arrow["{h}"', from=1-1, to=2-1]
		\arrow["{k}", from=1-2, to=2-2]
		\arrow[""{name=0, anchor=center, inner sep=0}, "{f}", "\shortmid"{marking}, from=1-1, to=1-2]
		\arrow[""{name=1, anchor=center, inner sep=0}, "{f'}"', "\shortmid"{marking}, from=2-1, to=2-2]
	\end{tikzcd}
	\quad\quad\quad
	\begin{tikzcd}[sep=scriptsize]
		{\Comonad A} & {B} \\
		{\Comonad A'} & {B'}
		\arrow["{f}", from=1-1, to=1-2]
		\arrow["{f'}"', from=2-1, to=2-2]
		\arrow["{\Comonad h}"', from=1-1, to=2-1]
		\arrow["{k}", from=1-2, to=2-2]
	\end{tikzcd}
\end{equation}

Note the similarity of coKleisli composition of \cref{eqn:comonad.kleisli.comp} to the composition of parameterized maps of \cref{eqn:para.comp.intro}:
\begin{equation}
	A \action (M \combine N) \xto{\coassociator} (A \action M) \action N \xto{f \action N} B \action N \xto{g} C.
\end{equation}

We can bring these two notions of composition closer together.
Recall that a comonoid in a monoidal category is equivalently a colax monoidal functor from the terminal monoidal category.
A comonad $\Comonad : \acted \to \acted$ is therefore equivalently a colax monoidal functor $\Comonad : 1 \to \Fun(\acted, \acted)$, which in the terminology of \cref{sec:para.intro} is a \emph{colax action} of the terminal monoidal category.
As we remarked above, \cref{defn:para.actegory} works just as well for colax actions of monoidal categories as it does for strong ones.
In the literature, a general colax action $\action : \actor \to \Fun(\acted, \acted)$ is known as a \emph{graded comonad}.
Graded monads (lax actions) were first introduced in program semantics under the name ``parametric effect monads'' by Katsumata in \cite{katsumata:graded.monads}; graded comonads were used for studying ``coeffects'' in \cite{GKO:effects.grading}.
Indeed, we might instead call the Para construction the Kleisli construction for
\emph{graded comonads}.%
\footnote{Although note that this is not the same construction as Fujii--Katsumata--Melli\`es Kleisli construction (originally from \cite[Definition~6]{FKM:graded.monads} for graded \emph{monads}, and spelled out in \cite[Definition~4.65]{fujii_2-categorical_2019} for graded comonads too).}

Because we have two names for the same construction, we will invent a third: we will refer to both the $\para$ and Kleisli constructions as instances of the more general $\Ctx$ construction which we come to now.

\subsection[Ctx as wreath product]{$\Ctx$ as a wreath product}\label{sec:wreath.product.intro}

In this section, we will see how the $\Ctx$ construction---which subsumes the $\para$ construction of a colax action and the Kleisli construction of a comonad---may be understood as a wreath product of pseudomonads in spans of categories.

A category $\Ca$ may be described as a monad in the bicategory of spans of sets; the underlying span itself is $\Ca_0 \xleftarrow{\partial_0} \Ca_1 \xrightarrow{\partial_1} \Ca_0$ has the set $\Ca_0$ of objects of $\Ca$ as its feet and the set $\Ca_1$ of arrows of $\Ca$ as its apex, with the legs assigning each arrow to its source and target respectively. This story categorifies neatly: a (pseudo-)double category $\Ca$ is a \emph{pseudo}monad $\Ca_0 \xleftarrow{\partial_0} \Ca_1 \xrightarrow{\partial_1} \Ca_0$ in the \emph{tri}category of spans of categories, where now $\Ca_0$ is the category of objects and tight morphisms and $\Ca_1$ is the category of loose arrows and squares of $\Ca$.

Thinking of double categories as pseudomonads in the tricategory of spans of categories, we are led to ask what $\Ctx(\action)$ looks like as a span of categories. If $\action : \acted \times \actor \to \acted$ is a colax action (or graded comonad), then the category of objects and tight morphisms of $\Ctx(\action)$ is $\acted$ again. Looking at the definition of square in $\Ctx(\action)$---\cref{eqn:para.square} (or \cref{eqn:kleisli.square})---we can see that the underlying span of $\Ctx(\action)$ arises as a pullback or composite of spans:

\begin{equation}
	\begin{tikzcd}[sep=small]
		&& {\Ctx(\action)_1} \\
		& {\acted \times \actor} && {\acted^{\downarrow}} \\
		\acted && \acted && \acted
		\arrow[from=1-3, to=2-2]
		\arrow[from=1-3, to=2-4]
		\arrow["\lrcorner"{anchor=center, pos=0.125, rotate=-45}, draw=none, from=1-3, to=3-3]
		\arrow["{\pi_{\acted}}"', from=2-2, to=3-1]
		\arrow["\action", from=2-2, to=3-3]
		\arrow["{\partial_0}"', from=2-4, to=3-3]
		\arrow["{\partial_1}", from=2-4, to=3-5]
	\end{tikzcd}
\end{equation}

On the right, we have the underlying span $\acted \xleftarrow{\partial_0} \acted^{\downarrow} \xrightarrow{\partial_1} \acted$ of the double category of commuting squares in $\acted$, whose apex is the category of arrows in $\acted$ and whose legs take the domain and codomain respectively.
Therefore, the pullback $\Ctx(\action)_1$ has as objects the triples $(A, M, f : A \action M \to B)$ and morphisms the triples of maps $(h, \varphi, k) : (A, M, f : A \action M \to B) \to (A', M', f' : A' \action M' \to B')$ so that the pair $(h \action \varphi, k)$ is a morphism in $\acted^{\downarrow}$; that is, for which the following square commutes:
\begin{equation}
	\begin{tikzcd}[sep=scriptsize]
		{A \action M} & {B} \\
		{A' \action M'} & {B'}
		\arrow["{f}", from=1-1, to=1-2]
		\arrow["{f'}"', from=2-1, to=2-2]
		\arrow["{h \action \varphi}"', from=1-1, to=2-1]
		\arrow["{k}", from=1-2, to=2-2]
	\end{tikzcd}
\end{equation}
This is exactly \cref{eqn:para.square}. The legs of the span $\acted \leftarrow \Ctx(\action)_1 \rightarrow \acted$ project out $A$ and $B$;
therefore, this composite span is the span underlying the double category $\Ctx(\action)$ when considered as a pseudomonad.
We just need to put the correct pseudomonad structure on it.

\subsubsection{Seeking a distributive law.}

Now, the span $\acted \xleftarrow{\partial_0} \acted^{\downarrow} \xrightarrow{\partial_1} \acted$ underlies a pseudomonad in $\Span(\Cat)$---the double category of commuting squares in $\acted$.
Suppose for the moment that $\action : \acted \times \actor \to \acted$ were not a colax action but a \emph{strict} action of $\actor$ on $\acted$ with both $\counitor$ and $\coassociator$ being identity morphisms.
In this case, we could endow the span $\acted \xleftarrow{\pi} \acted \times \actor \xto{\action} \acted$ with the structure of a pseudomonad corresponding to the \emph{Cayley} or \emph{action double category} (a categorification of the Cayley graph or action category of a monoid action):

\begin{equation}
\label{eqn:colax.action.spans}
	\begin{tikzcd}[row sep=scriptsize]
		\Ca & \Ca & \Ca \\[2ex]
		\Ca & {\Ca \times \Ma} & \Ca
		\arrow[from=1-2, to=1-1, equals]
		\arrow[from=1-2, to=1-3, equals]
		\arrow[from=1-1, to=2-1, equals]
		\arrow[""{name=0, anchor=center, inner sep=0}, from=1-3, to=2-3, equals]
		\arrow["\pi_1", from=2-2, to=2-1]
		\arrow["\action"', from=2-2, to=2-3]
		\arrow[""{name=1, anchor=center, inner sep=0}, "{\Ca \times \combineunit}"', from=1-2, to=2-2]
		\arrow["\counitor", shorten <=9pt, shorten >=9pt, Rightarrow, from=1, to=0]
	\end{tikzcd}\quad\quad
	\begin{tikzcd}[ampersand replacement=\&, sep=scriptsize]
		\Ca \& {\Ca \times \Ma} \& \Ca \& {\Ca \times \Ma} \& \Ca \\[-1ex]
		\&\& {\Ca \times \Ma \times \Ma} \\[2ex]
		\Ca \&\& {\Ca \times \Ma} \&\& \Ca
		\arrow["{{\pi_{12}}}", from=2-3, to=1-2]
		\arrow["{{\action \times \Ma}}"', from=2-3, to=1-4]
		\arrow["\action", from=1-2, to=1-3]
		\arrow["{\pi_1}"', from=1-4, to=1-3]
		\arrow["\lrcorner"{anchor=center, pos=0.125, rotate=135}, draw=none, from=2-3, to=1-3]
		\arrow["{\pi_1}"', from=1-2, to=1-1]
		\arrow["\action", from=1-4, to=1-5]
		\arrow["\action"', from=3-3, to=3-5]
		\arrow["{\pi_1}", from=3-3, to=3-1]
		\arrow[Rightarrow, no head, from=1-1, to=3-1]
		\arrow[""{name=0, anchor=center, inner sep=0}, Rightarrow, no head, from=1-5, to=3-5]
		\arrow[""{name=1, anchor=center, inner sep=0}, "{{\Ca \times \combine}}"', from=2-3, to=3-3]
		\arrow["\coassociator", shorten <=38pt, shorten >=38pt, Rightarrow, from=1, to=0]
	\end{tikzcd}
\end{equation}

Since we are, for now, assuming that $\counitor$ and $\coassociator$ are identities, the above diagrams are commuting diagrams and therefore express morphisms of spans.
Together with the unit and associativity isomorphisms of the monoidal structure of $\actor$, these give a $\acted \xleftarrow{\pi} \acted \times \actor \xto{\action} \acted$ the structure of a pseudomonad.
Therefore, the span $\acted \leftarrow \Ctx(\action)_1 \rightarrow \acted$ underlying $\Ctx(\action)$ arises as a composite of spans underlying pseudomonads;
since we want this span to have a pseudomonad structure itself given by parametric composition, we might search for a distributive law of $\acted \xleftarrow{\pi} \acted \times \actor \xto{\action} \acted$ over $\acted \xleftarrow{\partial_0} \acted^{\downarrow} \xrightarrow{\partial_1} \acted$.

Such a distributive law would be given by the functor
\begin{equation}
\label{eqn:distributive.law}
	\Lambda : \acted^{\downarrow}  {{}_{\partial_1}\times_{\pi}} (\acted \times \actor)\to (\acted \times \actor) {{}_{\action}\times_{\partial_0}} \acted^{\downarrow}
\end{equation}

defined by
\begin{equation}
	\Lambda(f : A \to B,\, (B, M)) := ((A, M), f \action \id_{N} : A \action M \to B \action M).
\end{equation}
The usual formula for computing the composite multiplication induced by a distributive law then gives us the map
\begin{eqalign}
\label{eqn:distributive.law.comp}
	((A, M),\ f : A \action M \to B,\ (B, N),\ g : B \action N \to C) & \mapsto ((A, M),\ (A\action M, N),\ f \action N, g) \\
	                                                                                   & \mapsto ((A, M \combine N),\ g (f \action N)).
\end{eqalign}
where $g (f \action N) : A \action (M \combine N) = (A \action M) \action N \to C$ has the correct domain because we assumed that the action was strict.
This distributive law $\Lambda$ is not chosen arbitrarily;
it emerges from the fact that the first projection $\acted \xleftarrow{\pi_1} \acted \times \actor$ is a cartesian fibration.
Explicitly, $\Lambda(f, (B, M)) = (f^*(B, M), \lift_{\pi_1}(f))$ takes
the cartesian lift of the map $f$ provided by the fibration structure on the
first projection $\pi_1$.

\subsubsection{Accounting for colaxity: left-fibrant spans}

So long as the action is strict, the distributive law $\Lambda$ gives us the formula \cref{eqn:distributive.law.comp} for parametric composition.
But in general, actions will not be strict; they will be counital and coassociative only up to coherent isomorphism or, as we have seen above, only up to coherent non-invertible morphisms---hence strong or colax.
If 2-cells in \cref{eqn:colax.action.spans} are not
identities, the span $\acted \xleftarrow{\pi} \acted \times \actor \xto{\action} \acted$ does not even underlie a pseudomonad---at least, not in the
tricategory of spans of categories with strictly commuting morphisms of spans.
To express the algebraic structure of a colax action as a pseudomonad, we will
therefore need to change the tricategory we work in.

Indeed, $\acted \xleftarrow{\pi} \acted \times \actor \xto{\action} \acted$ will have a pseudomonad structure in a tricategory of spans of categories where a morphism of spans commutes strictly on the left but only colaxly on the right to match with the diagrams \cref{eqn:colax.action.spans}.
That is, we we need a tricategory where a 2-cell is a morphism of spans like so:
\begin{equation}
\label{eqn:fibspan.2cell}
	\begin{tikzcd}[row sep=scriptsize]
		\Ca & \Ea & \Da \\
		\\
		\Ca & \Ea' & \Da
		\arrow["p"', from=1-2, to=1-1]
		\arrow["f", from=1-2, to=1-3]
		\arrow["{p'}", from=3-2, to=3-1]
		\arrow["{f'}"', from=3-2, to=3-3]
		\arrow[""{name=0, anchor=center, inner sep=0}, "k"', from=1-2, to=3-2]
		\arrow[from=1-1, to=3-1, equals]
		\arrow[""{name=1, anchor=center, inner sep=0}, from=1-3, to=3-3,equals]
		\arrow["\varphi", shorten <=6pt, shorten >=6pt, Rightarrow, from=0, to=1]
	\end{tikzcd}
\end{equation}

Spans between $\Ca$ and $\Da$, the morphisms as above and the evident transformations between them (see \cref{eqn:transformation.of.right.colax.maps.of.spans}) do form a 2-category $\Span^{\Rightarrow}(\Ca, \Da)$ between two fixed categories $\Ca$ and $\Da$; but composition of spans does not extend into a 2-functor $\Span^{\Rightarrow}(\Ca, \Da) \times \Span^{\Rightarrow}(\Da, \Ba) \to \Span^{\Rightarrow}(\Ca, \Ba)$ because we can no longer apply the universal property of the pullback to the non-commuting squares on the right.

However, when the left legs of the spans involved are cartesian (Grothendieck)
fibrations (and the span maps are cartesian with respect to these), we can then
use the cartesian lift to define a composition.
We take this approach in \cref{lem:fspan.composition} to define a tricategory
$\FibSpan^{\Rightarrow}$ of spans whose left leg is a fibration and with 2-cells
as in \cref{eqn:fibspan.2cell}.
We call these \emph{left-fibrant spans}.

With the composition defined in \cref{lem:fspan.composition}, colax actions become examples of pseudomonads in this tricategory $\FibSpan^{\Rightarrow}$.
But there is really no reason to restrict ourselves to pseudomonads in $\FibSpan^{\Rightarrow}$ whose left leg is a product projection.
As mentioned above, the purported distributive law (\ref{eqn:distributive.law}) we are aiming at only really used the fact that the product projection $\pi_C : \acted \times \actor \to \acted$ was a cartesian fibration.

In \cref{defn:colax.fibred.action}, we define a \emph{contextad} to be a pseudomonad $\fibcolaxaction$ in $\FibSpan^{\Rightarrow}$; we can think of these either as \textbf{\emph{fibred} colax actions} (where the category of objects which can act on $C \in \acted$ depends on $C$) or \textbf{\emph{dependently} graded comonads} (where the category of grades can depend on the object we will apply the graded comonad to).
Explicitly, if $P \in \actor$ is in the fiber of $p$ over $A\in \acted$ (that is, $p(P) = A$), then we can act on $A$ by $P$ to give $A \action P \in \acted$.
We have units $\combineunit_A \in \acted$ in each fiber, and given $P$ over $A$ and $B$ over $A \action P$, we can form $A \combine B$ over $A$.

We give a by-hand definition of the $\Ctx$ construction for contextads $\fibcolaxaction$ in \cref{defn:ctx.dbl.cat}, where now a contextful morphism is $f : A \action P \to B$ where, as before, $A$ and $B$ are in $\acted$, while $P \in \actor$ is in the fiber over $A$.
Composition now requires us to pull back using the (non-trivial!) fibration structure of $p$; if $f : A \action P \to B$ and $g : B \action Q \to C$, then we define their composite to be
\begin{equation}
	A \action (P \combine f^*Q) \xto{\coassociator} (A \action P) \action
	f^*Q \xto{f \action Q} B \action Q \xto{g} C.
\end{equation}

Examples of these more general contextads abound: if $\acted$ is a category with pullbacks, then $\acted \xleftarrow{\cod} \acted^{\downarrow} \xrightarrow{\dom} \acted$ (note that this is flipped around from the double category of commuting squares) is a contextad, and $\Ctx(\dom) = \Span(\acted)$ is the double category of spans in $\acted$.
Thinking type-theoretically, we could see this example as a dependent version of the right action of $\acted$ on itself by a cartesian product $(A, P) \mapsto A \times P$, where we see a map in $\acted$ as a dependent type $x : A \,\,\vdash\,\, P(x)$ and its domain as the dependent sum $\Sigma(x : A) \times P(x)$.
See \cref{sec:ctx.as.colax.fibred.action} for more details and for more examples of contextads which act like colax actions.
See also \cref{sec:ctx.as.dep.graded.comonad} for examples of contextads which act more like dependently graded comonads.

\subsubsection{Losing the distributive law, discovering the wreath.}

For any contextad $\fibcolaxaction$, we can define a map
\begin{equation}
	\Lambda : 	\acted^{\downarrow}  {{}_{\partial_1}\times_{p}} \actor \to \actor {{}_{\action}\times_{\partial_0}} \acted^{\downarrow}
\end{equation}
given by
\begin{equation}
	\left(\begin{tikzcd}[cramped, row sep=scriptsize] A
			\arrow[swap]{d}{\beta}\\[1.75ex] p(P)\end{tikzcd}, P \in \actor\right) \longmapsto
	\left(\beta^*P \in \actor,\begin{tikzcd}[cramped, row sep=scriptsize]
		A \action \beta^* P \arrow{d}{\action(\lift_p(\beta))}\\[1.75ex] p(P)
		\action P\end{tikzcd}\right)
\end{equation}
However, this will not turn out to be a distributive law over the arrow pseudomonad in $\FibSpan^{\Rightarrow}$.
It seems that by incorporating colaxity, we have lost the distributive law.

Here we are saved by gaining a deeper understanding of the tricategory $\FibSpan^{\Rightarrow}$. There is a straightforward isomorphism between the 2-categories of diagrams as on the left below with those on the right below given by the universal property of the arrow category:
\begin{equation}
	\begin{tikzcd}
		\Ba & \Ea & \Ca && \Ba & \Ea & \Ca \\
		\Ba & \Ea' & \Ca && \Ba & {\Ea' {}_{f} \times_{\dom} \Ca^{\downarrow}} & \Ca
		\arrow["p"', from=2-2, to=2-1]
		\arrow["f", from=2-2, to=2-3]
		\arrow["{\cod}", from=2-6, to=2-7]
		\arrow["{p\dom}"', from=2-6, to=2-5]
		\arrow[from=1-2, to=1-1]
		\arrow[from=1-2, to=1-3]
		\arrow[from=1-6, to=1-5]
		\arrow[from=1-6, to=1-7]
		\arrow[""{name=0, anchor=center, inner sep=0}, "{\internal{\varphi}}"', from=1-6, to=2-6]
		\arrow[""{name=1, anchor=center, inner sep=0}, "{a}"', from=1-2, to=2-2]
		\arrow[from=1-1, to=2-1,equals]
		\arrow[""{name=2, anchor=center, inner sep=0}, from=1-3, to=2-3,equals]
		\arrow[from=1-5, to=2-5,equals]
		\arrow[""{name=3, anchor=center, inner sep=0}, from=1-7, to=2-7,equals]
		\arrow["\varphi", shorten <=6pt, shorten >=6pt, Rightarrow, from=1, to=2]
	\end{tikzcd}
\end{equation}

A map $\Ea \to \Ea' \spancomp \Ca^{\downarrow}$ as above right is a map in the Kleisli bicategory of the pseudomonad on $\Span(\Ba, \Ca)$ given by the right action of the arrow pseudomonad $\Ca \xleftarrow{\dom} \Ca^{\downarrow} \xrightarrow{\cod} \Ca$---an observation which we record as \cref{lem:lax.maps.kleisli}.
Furthermore, the condition that the left leg $p$ of the above span be a fibration is equivalent to it being an algebra for the left action of the arrow pseudomonad $\Ba \xleftarrow{\dom} \Ba^{\downarrow} \xrightarrow{\cod} \Ba$ on $\Span(\Ba, \Ca)$ giving us the following equivalence (\cref{thm:main.iso}):
\begin{equation}
	\InclusionTrifun_{\Ba,\Ca}:\FibSpan^{\Rightarrow}(\Ba, \Ca) \isolongto \Alg(\Ba^{\downarrow} \spancomp -,\, \Kl(- \spancomp \Ca^\downarrow,\, \Span(\Ba, \Ca))).
\end{equation}

For those who have read Lack and Street's \emph{Formal Theory of Monads II} \cite{lack_formal_2002}, the expression on the right will be familiar: these are the hom (bi)categories of the free cocompletion of $\Span$ under Kleisli objects for (pseudo)monads.
Of course, Lack and Street actually work with monads and the cocompletion of 2-categories under their Kleisli objects; here we are working with tricategories and need to cocomplete them under Kleisli objects for \emph{pseudo}monads.
Luckily for us, the job of pursuing such a categorification was undertaken by Miranda in his thesis \cite{miranda_topics_2024}.
The free Kleisli cocompletion $\KL(\Span)$ of the tricategory $\Span$ of spans of categories has as objects the pseudomonads in $\Span$ (that is, the double categories) with hom-bicategories given by algebras in Kleisli bicategories as above; see \cref{sec:trikleisli.comp} for our discussion.

In total, we discover a tri-fully faithful inclusion (\cref{thm:fSpan.tricat})
\begin{equation}
	\FibSpan^{\Rightarrow} \hookrightarrow \KL(\Span)
\end{equation}
of the tricategory of left-fibrant spans and colax span maps into the free Kleisli cocompletion of the tricategory of spans of categories, given by sending a category $\Ca$ to its arrow pseudomonad in $\KL(\Span)$.

As a corollary, a contextad is therefore equivalently a pseudomonad in $\KL(\Span)$---that is, an object of $\KL(\KL(\Span))$---on an arrow pseudomonad, which is precisely what Lack and Street refer to as a \emph{wreath} around the arrow pseudomonad.
In \cref{sec:ctx.as.wreath.product}, we will see that the $\Ctx$ construction emerges as the \emph{wreath product} $\wreath : \KL(\KL(\Span)) \to \KL(\Span)$ induced by the universal property of the free cocompletion under Kleisli objects.
This gives us not only an abstract look at the $\Ctx$ construction, but also lets us deduce its trifunctoriality in \cref{thm:ctx.const.on.objects} and \cref{defn:ctx.construction}.

\subsection[Structures on Ctx]{Structures on $\Ctx$}\label{sec:structures.intro}

It is a well known theorem that a colax monoidal comonad $\Comonad : \acted \to \acted$ on a monoidal category $(\acted, \monoidalunit, \monoidal)$ has a monoidal Kleisli category.
\footnote{
	It is more common to find this theorem expressed dually: a lax monoidal monad has a monoidal Kleisli category. We do not know where exactly to cite this from.
	Lax monoidal monads are also known as \emph{commutative} monads \cite{kock:commutative.monad}.
}
It is similarly folklore that the $\para$ construction of a braided action of a braided monoidal category on a monoidal category itself carries a monoidal structure.
Likewise, double categories of spans in categories with all finite limits are
\emph{cartesian} double categories.
We can recover these theorems and more as instances of the $\Ctx$ construction
applied in 2-categories other than the 2-category $\Cat$ of categories.

We have seen that the $\Ctx$ construction---which subsumes the $\para$ construction for colax actions and the Kleisli construction for (graded) comonads---arises as the wreath product of wreaths around arrow pseudomonads in the tricategory of spans of categories.
But there is really nothing particular about the role of the 2-category $\Cat$ here: from any 2-category $\Cosmos$ with pullbacks, we can form the tricategory $\Span(\Cosmos)$ of spans in it, and if $\Cosmos$ has arrow objects, then we can run the above story entirely in the 2-category $\Cosmos$.
Our strategy for proving the theorems we just recalled will be to choose the appropriate 2-category $\Cosmos$ to perform the $\Ctx$ construction in.

But (strict) pullbacks are rare creatures in the world of 2-categories.
It is very often the case that only some pullbacks exist.
We therefore define (\cref{defn:paradise}) a \emph{paradise} $\Paradise$ to be a 2-category $\Cosmos$ with arrow objects $\Ca^{\downarrow}$ equipped with a class $\Display$ of \emph{display maps} (\cref{defn:display.map.2.cat}) which are stable under pullback and containing the source map $\dom : \Ca^{\downarrow} \to \Ca$.

We then work with the tricategory $\DispSpan\Paradise$ of spans whose left leg is a display map in $\Paradise$.
By the assumption that $\dom : \Ca^{\downarrow} \to \Ca$ is a display map, the span $\Ca \xleftarrow{\dom} \Ca^{\downarrow} \xrightarrow{\cod} \Ca$ becomes a pseudomonad in $\DispSpan\Paradise$ and we can define the $\Ctx$ construction in $\Paradise$ to be the wreath product of wreaths around the arrow pseudomonad.
The only extra assumption now is that the left leg $p$ of a contextad $\fibcolaxaction$ is not only a fibration but also a display map.

Taking the paradise $(\Cat, \{\mathsf{all}\})$ of $\Cat$ equipped with the class of all maps, we recover the story for categories we described above.
But with this extra freedom, we can now put 2-algebraic structure on our categories and see how it transfers to the $\Ctx$ construction.

Suppose that $\Paradise$ is a paradise (such as $(\Cat, \{\mathsf{all}\})$) and that $T : \Cosmos \to \Cosmos$ is a 2-monad encoding some sort of 2-algebraic structure (such as the free monoidal category 2-monad on $\Cat$).
In \cref{thm:colax.paradise}, we show that the 2-category $\Alg(T)_c$ of (strict) $T$-algebras and \emph{colax} $T$-morphisms is a paradise when equipped with the display map class $\{\mathsf{sdnf}\}$ of \emph{strict displayed normal fibrations}---that is, strict $T$-algebra morphisms whose underlying map in $\Cosmos$ is a display map there, and which are fibrations in $\Alg(T)_c$ for which lifts of identities are identities.
We call a contextad in $\Alg(T)_c$ \emph{colaxly $T$-structured} (\cref{defn:colaxly.T.structured.colax.fibred.action}); for $T$ the free monoidal category $2$-monad on $\Cat$, the colaxly $T$-structured contextads include the braided colax actions of monoidal categories.

In \cref{thm:para.is.colaxly.structured}, we show that $\Ctx(\action)$ of any colaxly $T$-structured contextad $\fibcolaxaction$ is a \emph{colaxly $T$-structured pseudocategory} (\cref{defn:colaxly.T.structured.double.category})---that is, a pseudocategory $\Ca_0 \xleftarrow{s} \Ca_1 \xrightarrow{t} \Ca_0$ in $\Cosmos$ whose objects of objects $\Ca_0$ and morphisms $\Ca_1$ are strict $T$-algebras, whose source and target maps $s$ and $t$ are strict $T$-morphisms,
and whose identity and composition maps are colax $T$-morphisms.
In the case
that $T$ is the free monoidal category 2-monad on $\Cat$, a colaxly
$T$-structured double category is a \emph{lax} monoidal double category---that
is, a pseudomonoid in the 2-category of double categories and \emph{lax}
functors.
The change in variance here is a general feature of 2-categorical
algebra known as \emph{symmetry of internalization} \cite{arkor_2024}; simply put, we may consider colax $T$-structured pseudocategories as $T$-structures internal to pseudocategories and lax morphisms.

We conclude by deriving the appropriate algebraic structure on all instances of $\Ctx$.

\subsection{Outline}
The rest of the work roughly follows the order we presented its content in this introduction.
In \cref{sec:preliminaries} we introduce some 2-categorical notions, such as the aforementioned paradises and the associated tricategories of spans, which serve as the backdrop for the remaining sections.
In \cref{sec:contextad.intro} left-fibrant spans enter the picture and we give a direct definition of contextads on categories, as well as their $\Ctx$ construction.
We tour some examples in \cref{sec:examples.of.ctx}, with the aim of building up the intuition about how contextads generalize actegories (in \cref{sec:ctx.as.colax.fibred.action}) and graded comonads (in \cref{sec:ctx.as.dep.graded.comonad}).
In \cref{sec:fibs.span} we develop the most technical part of our story, showing that $\FibSpan$ embeds in the Kleisli tricompletion of $\DispSpan$.
This is where we show contextads are wreaths and $\Ctx$ is their wreath product, as well as expound 1- to 3-morphisms of contextads and the trifunctoriality of $\Ctx$.
In \cref{sec:structured.para} we develop the theory of colaxly structured contextads and their $\Ctx$ constructions in order to understand how algebraic structure on $\Ctx$ arises.

We have also included appendices further developing formal aspecs of the theory. In \cref{sec:duality} we unravel the dual story of \emph{contentads} and describe the dualities at play between them.
Finally, in \cref{sec:docs.of.wreaths} we move beyond contextads on categories by considering contextads on arbitrary double categories, a generalization which suggests that we should understand general wreaths in $\DispSpan$ as notions of contextuality more broadly.

\subsection{Acknowledgements}
We thank Adrian Miranda for his invaluable help in recognizing the need of a proper theory of tricompletions, in providing it, and in assisting us whilst navigating the rich world of tricategory theory.
We thank David Spivak, David Corfield and Nathanael Arkor for helpful comments and corrections on an earlier draft of this work.
During the preparation of this work, the first-named author has been supported by EPSRC and by Advanced Research + Invention Agency grant MSAI-PRO-01; while the second-named author has been supported by Tamkeen under the NYU Abu Dhabi Research Institute grant CG008 and by Advanced Research + Invention Agency grant MSAI-PRO-14.

\subsection{Notational convention}
We denote 2-categories with blackboard bold capital letters $\Kb, \Hb, \ldots$ or initials $\Cat$,$\dblcat{Lex},\ldots$, and their objects with calligraphics letters $\Aa,\Ba, \ldots$.
Their 1-cells are usually denoted by lowercase Latin letters $f,g,\ldots$ and 2-cells by Greek letters $\alpha,\beta,\ldots$.
Tricategories and trifunctors are in fraktur letters $\Kc, \Hc, \ldots$ or begin with fraktur letters $\TriSpan, \PsCat, \ldots$.

We write composition in algebraic order, i.e.~$fg$ means `$f$ after $g$'.
We don't distinguish between composition and whiskering, so that $f\alpha g$ means whiskering of $\alpha$ with $f$ on the right and $g$ on the left.
For vertical composition of 2-cells, we write $\alpha \cdot \beta$, again in algebraic order.

Since we deal with tricategories, we emply three different dualities: $(-)\op$, $(-)\co$, $(-)\re$, which reverse, respectively, 1-, 2- and 3-cells.

\section{Laying the ground: display map 2-categories and paradises}
\label{sec:preliminaries}
We aim to formulate the $\Ctx$ construction as the wreath product in a tricategory\footnotemark~of spans of category-like objects, taking a wreath around the arrow pseudomonad $\Ca \leftarrow \Ca^{\downarrow} \to \Ca$ and producing a new pseudomonad of spans on $\Ca$.
\footnotetext{We recall a tricategory is a fully weak tridimensional category, see \cite{gurski_coherence_2013} for a full definition---we will not require that much detail.}
We can therefore perform this construction in any 2-category in which we can construct a tricategory of spans, and which has powers $\Ca^{\downarrow}$ by the walking arrow $\downarrow$.

Another way, followed by Hoffnung \cite{hoffnung_spans_2013} is to define a tricategory of spans for 2-category with isocomma objects (aka `weak pullbacks').%
\footnote{If the reader is unfamiliar with the jungle of higher notions of limit, we recommend \cite{bird_flexible_1989,kelly_elementary_1989} for a recap of the basic definitions, with the former focusing on the difference between limits and bilimits which we care about here.}
However, a pseudomonad in this tricategory of spans, taken in the 2-category of categories, would not be a double category in the usual sense.
Rather, a double category is a pseudomonad in the tricategory of spans of categories composing by \emph{strict (2-)pullback}; while the isocomma version
would allow us to compose a triple
\begin{tikzcd}[ampersand replacement=\&, column sep=scriptsize, row sep=small, cramped]
	\cdot \& \cdot \\
	\& \cdot \& \cdot
	\arrow["\shortmid"{marking}, from=1-1, to=1-2]
	\arrow["\,\wr", from=1-2, to=2-2]
	\arrow["\shortmid"{marking}, from=2-2, to=2-3]
\end{tikzcd}
into a single proarrow.
This is not so terrible, but also not quite what one normally thinks of as composable loose arrows.
We want to get exactly a double category out.
Therefore, we will use strict (2-)pullbacks to compose our spans.

\subsection{Paradises}

While the 2-category of categories has all strict pullbacks, it is quite common for general 2-categories to only admit some strict pullbacks%
\footnote{For instance, the inclusions $1 \nto{0} {\downarrow\!\!\wr} \nfrom{1} 1$ of the endpoints of the walking isomorphism is a cospan in $\MonCat$ that doesn't admit a strict pullback.}%
---pullbacks of a particular class of special arrows.
For this reason, we will take as our general setting a \emph{display map 2-category}, or a 2-category equipped with a class of arrows which are stable under pullback.

\begin{defn}
\label{defn:display.map.2.cat}
	A \textbf{display map 2-category} is a 2-category $\Cosmos$ equipped with a class $\Display$ of arrows (called \textbf{display maps} or \textbf{displays}), so that for each $d : A \to B$ in $\Display$ and $f : X \to B$, there is a strict 2-pullback
	\begin{equation}
		\begin{tikzcd}[sep=5ex]
			P & A \\
			X & B
			\arrow[from=1-1, to=1-2]
			\arrow["{f^*d}"', from=1-1, to=2-1]
			\arrow["\lrcorner"{anchor=center, pos=0.125}, draw=none, from=1-1, to=2-2]
			\arrow["d", from=1-2, to=2-2]
			\arrow["f"', from=2-1, to=2-2]
		\end{tikzcd}
	\end{equation}
	in $\Cosmos$, and $f^*d \in \Display$.
	In short, display maps are closed under pullback.
\end{defn}

We take a moment to remind the reader of the definition of arrow objects.
\begin{defn}
	Let $\Cosmos$ be a 2-category.
	If $\Ca \in \Cosmos$, an \textbf{arrow} object
	$\Ca^{\downarrow}$ is an object equipped with two maps $\dom,\, \cod :
		\Ca^{\downarrow} \to \Ca$ and a 2-cell $\generic_{\Ca} : \dom \Rightarrow \cod$ (called \textbf{generic arrow} of $\Ca$) so that whiskering with $\generic_{\Ca}$ induces an isomorphism of categories
	\begin{equation*}
		\Cosmos(\Xa, \Ca^{\downarrow}) \nlongto{\sim} \Cosmos(\Xa, \Ca)^{\downarrow}
	\end{equation*}
	for any $\Xa \in \Cosmos$.
\end{defn}

Of the display map 2-categories, we will focus on those for which display maps contain all isomorphisms and are closed under composition, and that admit arrow objects $\Ca^{\downarrow}$ whose domain map $\dom : \Ca^{\downarrow} \to \Ca$ is a display map.
For lack of a better name, we will call these 2-categories \emph{paradises}.

\begin{defn}
\label{defn:paradise}
	A \textbf{paradise} is a display map 2-category $\Paradise$ where:
	\begin{enumerate}
		\item $\Display$ is closed under composition and contains all isomorphisms,
		\item $\Cosmos$ has all arrow objects $\Ca^{\downarrow}$, and the domain map $\dom : \Ca^{\downarrow} \to \Ca$ is in $\Display$.
	\end{enumerate}
	A \textbf{paradisiacal 2-functor} $F:\Paradise \to (\Cosmos', \Display')$ between paradises is a 2-functor that preserves the chosen display maps as well as the chosen limits.
\end{defn}

\begin{ex}
\label{ex:cat.as.paradise}
	Any 2-category with pullbacks and arrow objects can be made into a paradise by taking every map to be a display map.
	In particular, the category of categories, with all functors considered to be
	display maps, is a paradise $(\Cat, \{\mathsf{all}\})$.
\end{ex}

A large class of examples of paradises is given by choosing display maps to be certain kinds of isofibrations.
These 2-categories often enjoy a richer structure, that of \emph{2-cosmoi}.
Morally, a cosmos is a `complete 2-category' but with completeness interpreted as to (in the words of Riehl and Verity \cite[Chapter 6]{riehl_elements_2022}) ``take advantage of the strictness that is available in so many examples to simplify proofs''.

The definition of cosmos hinges on choosing a class of well-behaved representable isofibrations, whose definition we recall:

\begin{defn}
\label{defn:isofib}
	A functor $f : \Aa \to \Ba$ is an \textbf{isofibration of categories} when for any diagram on the left below with natural isomorphism $\varphi : fa \iso b$, there is a lift $\hat{\varphi} : a \iso \hat{b}$ of $\varphi$ through $f$ so that $f\hat{\varphi} = \varphi$ as on the right below:

	\begin{equation*}
		\begin{tikzcd}[ampersand replacement=\&]
			\Xa \&\& \Aa \\
			\&\& \Ba
			\arrow["a", from=1-1, to=1-3]
			\arrow["f", from=1-3, to=2-3]
			\arrow[""{name=0, anchor=center, inner sep=0}, "b"', from=1-1, to=2-3]
			\arrow["\varphi", shift right, shorten <=7pt, shorten >=7pt, Rightarrow, from=0, to=1-3]
		\end{tikzcd}
		\quad =\ %
		\begin{tikzcd}[ampersand replacement=\&]
			\Xa \&\& \Aa \\
			\&\& \Ba
			\arrow["b"', from=1-1, to=2-3]
			\arrow["f", from=1-3, to=2-3]
			\arrow[""{name=0, anchor=center, inner sep=0}, "a", shift left, from=1-1, to=1-3]
			\arrow[""{name=1, anchor=center, inner sep=0}, "{\hat{b}}"'{pos=0.6}, curve={height=15pt}, dashed, from=1-1, to=1-3]
			\arrow["{\hat{\varphi}}"', shorten <=2pt, shorten >=2pt, Rightarrow, dashed, from=1, to=0]
		\end{tikzcd}
	\end{equation*}

	In an arbitrary 2-category $\Cosmos$, a 1-cell $f:\Aa \to \Ba$ is a \textbf{representable isofibration} when composition with $f$ induces an isofibration of categories $\Cosmos(\Xa, \Aa) \isofibto \Cosmos(\Xa, \Ba)$ for every $\Xa \in \Cosmos$.
\end{defn}

Thus, a bit less informally now, we can say a \emph{cosmos} is a 2-category with suitable limits and a choice of representable isofibrations which interact well with such limits.

Bourke and Lack \cite{bourke_cosmoi_2023} adapt the $\infty$-categorical definition of \cite{riehl_elements_2022} to give a 2-dimensional version.
We further adapt their definition to our purposes, which are less ambitious than those cosmoi were conceived for.
In particular, we won't need infinitary limits, so we state the following definition.

\begin{defn}
\label{defn:fin.2-cosmos}
	A \textbf{finitary $2$-cosmos} is a $2$-category $\Cosmos$ equipped with a class of morphisms called \textbf{isofibrations} so that
	\begin{enumerate}
		\item Every isofibration is a representable isofibration.
		\item All pullbacks of isofibrations exist and are isofibrations, all isomorphisms are isofibrations, and composites of isofibrations are isofibrations.
		\item $\Cosmos$ has finite products, and all product projections and terminal morphisms are isofibrations.
		\item For any $\Aa : \Cosmos$, there is an arrow object $\Aa^{\downarrow}$ with an isofibration $(\dom, \cod) : \Aa^{\downarrow} \isofibto \Aa \times \Aa$ and a 2-cell $\generic_{\Aa} : \dom \Rightarrow \cod$ (called \textbf{generic arrow} of $\Aa$) so that whiskering with $\generic_{\Aa}$ induces an isomorphism of categories
		\begin{equation*}
			\Cosmos(\Xa, \Aa^{\downarrow}) \nlongto{\sim} \Cosmos(\Xa, \Aa)^{\downarrow}
		\end{equation*}
		for any $\Xa \in \Cosmos$.
	\end{enumerate}
	A \textbf{cosmological $2$-functor} $F:\Cosmos \to \Cosmos'$ between finitary 2-cosmoi is a 2-functor that preserves the chosen isofibrations as well as the chosen limits.
\end{defn}

From now on, when we refer to 2-cosmoi we mean the finitary version, unless otherwise specificed.

Evidently, 2-cosmoi are paradises when equipped with their class of isofibrations:

\begin{lem}
	If $\Cosmos$ is a finitary cosmos, then taking $\Display$ to be the class of isofibrations endows $\Cosmos$ with the structure of a paradise.
\end{lem}
\begin{proof}
	By assumption, the class of isofibrations is closed under pullback, contains all isomorphisms, and is closed under composition.
	Futhermore, we have assumed that $\Cosmos$ has arrow objects $\Aa^{\downarrow}$, and since product projections and the map $(\dom, \cod) : \Aa^{\downarrow} \to \Aa \times \Aa$ are assumed to be isofibrations, the map $\dom : \Aa^{\downarrow} \to \Aa$ is an isofibration.
\end{proof}

\begin{rmk}[Cosmoi as \emph{flexibly complete} 2-categories]\label{rmk:cosmoi.flexibly.complete}
	We'd like now to reassure the unfamiliar reader that 2-cosmoi are in fact commonplace, easy to obtain, and more fun to use than finitely 2-complete 2-categories.
	Additionally, having a reference such as \cite{riehl_elements_2022} (and, to a lesser extent, \cite{bourke_cosmoi_2023}), makes it easier to fish for results about them---we do that extensively in this work.%
	\footnote{We do so with the understanding that the facts we invoke about Riehl--Verity or Bourke--Lack cosmoi only make use of the finitary fragment of their axioms.}

	The fact 2-cosmoi are commonplace and natural is a consequence of their close relationship with \emph{flexible limits} \cite{bird_flexible_1989}, i.e.~those constructed from products, inserters, equifiers and splittings of idempotents.
	Intuitively, flexible limits are (weighted) limits whose construction doesn't require equality on objects.
	Indeed e.g.~equifiers (a limit that equalizes two parallel 2-cells) are flexible but equalizers (a limit that equalizes two parallel 1-cells), and thus pullbacks, are not.

	Theorem 4.4 of \cite{bourke_cosmoi_2023} shows that a 2-category admitting all flexible limits may be canonically equipped with the structure of a 2-cosmos by taking the chosen isofibrations to be the \emph{normal isofibrations}.%
	\footnote{
		An isofibration is \emph{normal} when it lifts identities to identities.
		See Section 3.2 of \cite{bourke_cosmoi_2023} or \cref{defn:normal.fib}.
	}
	In fact every 2-cosmos in which idempotents split and retracts of isofibrations are isofibrations (i.e.~\emph{Cauchy-complete 2-cosmos}) arise in this way.
	This fact substantiates the claim finitary 2-cosmoi are a flexible (pun intended) notion of `2-complete' 2-category.
\end{rmk}

Categories with flexible limits are not only easy to come by, but are stable under various constructions.
Chiefly, if $\Cosmos$ is a $2$-category with flexible limits and $T : \Cosmos \to \Cosmos$ is a \emph{flexible 2-monad} in the sense of \cite{blackwell_two-dimensional_1989}, then the 2-category $\Alg(T, \Cosmos)$ of strict algebras and pseudomorphisms for $T$ also has flexible limits.
As described by Lack in \cite[p.~62]{lack_companion_2010}, any 2-monad which can be presented without using equations between 1-cells will be flexible, including the 2-monads for monoidal categories, symmetric monoidal categories, categories with finite limits and other completions and cocompletions, and many other ``non-evil'' structured category.

Since cosmoi are paradises, we can now describe a number of new examples of paradises:

\begin{ex}
\label{ex:cat}
	Of course, the 2-category of categories, functors and natural transformations, $\Cat$ is a 2-cosmos.
	There the isofibrations are exactly the isofibrations of categories of \cref{defn:isofib}.
	Note that this is a different paradise than \cref{ex:cat.as.paradise}, since here we are only taking the normal isofibrations to be the display maps.
	In $\DispSpan(\cat{Cat},\{\mathsf{normal\ isofibrations}\})$, the pseudomonads are the double categories (in the ordinary sense) whose source map is a normal isofibration.
	This extra condition is rather mild; it is satisfied by any double category which admits companions, for example.%
	\footnote{
		To be a bit more precise, we really need a functorial choice of companion so that the companion associated to an identity is the identity.
		But having companions is already much stronger than having source map be a normal isofibration.
	}
\end{ex}

\begin{ex}
\label{ex:moncat}
	The 2-category $\MonCat$ of monoidal categories, strong monoidal functors,
	and monoidal natural transformations has all flexible limits.
	Thus normal isofibrations equip $\MonCat$ with the structure of a 2-cosmos.
	Similarly for braided and symmetric monoidal categories.
	The arrow object associated to a monoidal category $(\Ca, 1, \otimes)$ is given by the obvious monoidal structure on the category of arrows of the underlying category $\Ca^\downarrow$.
\end{ex}

\begin{ex}
	The 2-category $\Cart(\Cat)$ of cartesian monoidal categories, product-preserving functors, and monoidal natural transformations; and arrow objects are again the same as those in $\Cat$.
	Isofibrations are the normal representable ones.
\end{ex}

\begin{ex}
\label{ex:higher.mons}
	In general, if $\Cosmos$ is a 2-cosmos then $\Mon(\Cosmos)$ and $\Cart(\Cosmos)$ are again 2-cosmoi.
	In this way we can conclude e.g.~that $\BrMonCat = \Mon(\MonCat)$ and $\SymMonCat = \Mon(\BrMonCat)$ are 2-cosmoi.
\end{ex}

\begin{ex}
	The 2-category $\DblCat$ of (weak) double categories, strong double functors and tight natural transformations is \emph{not} the category of algebras for a flexible 2-monad on the 2-category $\Cat^{\downdownarrows}$ of graphs in categories, at least not in an obvious way (indeed, the structure of double category uses equality on objects when prescribing that composable arrows should share an end).%
	\footnote{In a previous draft of this work, we overlooked this fact and stated otherwise.}
	We can nonetheless make it into a 2-cosmos by picking as isofibrations the representable normal isofibrations.
	These are evidently closed under isomorphism, composition, products, and pullback.
	This latter one is not as easy to show, so we prove it in \cref{prop:pb.of.rni.dbl.cats}.
	Finally, in this 2-category, the arrow category associated to a double category $\dblcat C$ is the double category $\dblcat C^\downarrow$ of tight maps of $\dblcat C$ (objects are tight maps, loose arrows in $\dblcat C^\downarrow$ are squares from $\dblcat C$, tight maps in $\dblcat C^\downarrow$ are commutative squares of tight maps in $\dblcat C$, and squares in $\dblcat C^\downarrow$ are commutative squares of squares in $\dblcat C$).
	The joint projection $(\dom, \cod)$ is then evidently a representable normal isofibration.
\end{ex}


\begin{ex}
\label{ex:cosmoi.duality}
	The definition of 2-cosmos is self-dual.
	In fact, on the one hand, representable isofibrations are a self-dual concept, that only refers to invertible 2-cells, so that a class of isofibrations in $\Cosmos$ is still valid for $\Cosmos\co$.
	On the other hand, the limits available in a 2-cosmos are also self-dual; specifically, it is evident that $\Aa^\uparrow \equiv \Aa^\downarrow$.
	Therefore $\Cosmos$ is a 2-cosmos if and only if $\Cosmos\co$ is.
\end{ex}

We will meet more examples of paradises in \cref{sec:structured.para}, and dualize the notion in \cref{sec:duality}.

\subsection{Left-displayed spans in a Paradise}

Given a paradise $\Paradise$, we may define a tricategory $\DispSpan\Paradise$---or, more narrowly, a 2-category-enriched bicategory---of spans in $\Cosmos$ whose left leg is in $\Display$.
We call these \emph{left-displayed spans}.
Fixing a paradise $\Paradise$, we will just write $\DispSpan$ for $\DispSpan\Paradise$, for short.

We will, in fact, make use of a few different morphisms between spans.
This leads us to the following definition.
Since we will need plenty of variants, we start by defining a general 2-category of spans between two fixed objects:

\begin{defn}
\label{defn:locally.span}
	Fix a paradise $\Paradise$.
	The 2-category $\DispSpan\Paradise^\twoto(\Ca, \Da)$ of \textbf{left-displayed spans and right-colax maps} between $\Ca,\Da \in \Cosmos$ is defined by the following data:
	\begin{enumerate}
		\item Objects are \textbf{left-displayed spans}, i.e.~spans whose left leg is a display map in $\Cosmos$:
		      \begin{equation}
			      \begin{tikzcd}[row sep=scriptsize]
				      & \Ea \\
				      \Ca && {\Da,}
				      \arrow["p"', from=1-2, to=2-1]
				      \arrow["f", from=1-2, to=2-3]
			      \end{tikzcd}
		      \end{equation}
		\item $1$-cells are \textbf{right-colax} map of spans, i.e.~diagrams
		      \begin{equation}
			      \begin{tikzcd}[row sep=scriptsize]
				      \Ca & \Ea & \Da \\
				      \\
				      \Ca & \Ea' & \Da
				      \arrow["p"', from=1-2, to=1-1]
				      \arrow["f", from=1-2, to=1-3]
				      \arrow["{p'}", from=3-2, to=3-1]
				      \arrow["{f'}"', from=3-2, to=3-3]
				      \arrow[""{name=0, anchor=center, inner sep=0}, "k"', from=1-2, to=3-2]
				      \arrow[from=1-1, to=3-1, equals]
				      \arrow[""{name=1, anchor=center, inner sep=0}, from=1-3, to=3-3,equals]
				      \arrow["\varphi", shorten <=6pt, shorten >=6pt, Rightarrow, from=0, to=1]
			      \end{tikzcd}
		      \end{equation}
		      for which $p' k = p$ and $\varphi : f'k \Rightarrow f$,
		\item $2$-cells are diagrams of the form:
		      \begin{equation}
		\label{eqn:transformation.of.right.colax.maps.of.spans}
			      \begin{tikzcd}[column sep=scriptsize]
				      \Ca && \Ea && \Da \\
				      \\
				      \Ca && \Ea' && \Da
				      \arrow["p"', from=1-3, to=1-1]
				      \arrow["f", from=1-3, to=1-5]
				      \arrow["{{p'}}", from=3-3, to=3-1]
				      \arrow["{{f'}}"', from=3-3, to=3-5]
				      \arrow[Rightarrow, no head, from=1-1, to=3-1]
				      \arrow[""{name=0, anchor=center, inner sep=0}, Rightarrow, no head, from=1-5, to=3-5]
				      \arrow[""{name=1, anchor=center, inner sep=0}, "h"'{pos=0.2}, curve={height=18pt}, from=1-3, to=3-3]
				      \arrow["\psi"{pos=0.7}, curve={height=-12pt}, shorten <=9pt, shorten >=9pt, Rightarrow, from=1, to=0]
				      \arrow[""{name=2, anchor=center, inner sep=0}, "k"{pos=0.2}, curve={height=-18pt}, from=1-3, to=3-3]
				      \arrow["\alpha"', curve={height=6pt}, shorten <=4pt, Rightarrow, from=1, to=2]
				      \arrow["\varphi"', curve={height=6pt}, shorten <=5pt, shorten >=10pt, Rightarrow, from=2, to=0]
			      \end{tikzcd}
		      \end{equation}
		      That is, 2-cells $\alpha : k \Rightarrow h$ so that $p' \alpha = p$ and $\varphi f'\alpha = \psi$.
	\end{enumerate}

	We define the 2-category $\DispSpan\Paradise(\Ca, \Da)$ to be the full sub-2-category spanned by those morphisms whose colaxator $\varphi$ is an \emph{identity}---that is, for which $f'k = f$.
\end{defn}

\begin{thm}
\label{thm:dspan.is.tricat}
	The strict pullback
	\begin{equation}
	\label{eqn:comp.of.dspans}
		\begin{tikzcd}[sep=small]
			&& {\Fa \spancomp \Ea} \\
			& \Fa && \Ea \\
			{\Ba } && \Ca && \Da
			\arrow[from=1-3, to=2-2]
			\arrow[from=1-3, to=2-4]
			\arrow[curve={height=24pt}, dashed, from=1-3, to=3-1]
			\arrow["\lrcorner"{anchor=center, pos=0.125, rotate=-45}, draw=none, from=1-3, to=3-3]
			\arrow[curve={height=-24pt}, dashed, from=1-3, to=3-5]
			\arrow[from=2-2, to=3-1]
			\arrow[from=2-2, to=3-3]
			\arrow[from=2-4, to=3-3]
			\arrow[from=2-4, to=3-5]
		\end{tikzcd}
	\end{equation}
	gives a 2-functor $\spancomp : \DispSpan(\Ba, \Ca) \times \DispSpan(\Ca,\Da) \to \DispSpan(\Ba,\Da)$ which endows $\DispSpan$ with the structure of a 2-category enriched bicategory (in the sense of \cite[\S~3.1]{garner_enriched_2016})---which is, in particular, a slightly strict tricategory.%
	\footnote{A \emph{tricategory} is a fully weak 3-dimensional category, whose definition can be found e.g.~in \cite{gurski_coherence_2013}.}.
\end{thm}
\begin{proof}[Proof sketch]
	Since we are using strict pullbacks, the construction works exactly as it does in the 1-categorical case, taking a 1-category with pullbacks and producing a bicategory of spans.
\end{proof}

\begin{notation}
\label{not:spans}
	By synecdoche, we often denote a span $\Ca \from \Ea \to \Da$ by its apex $\Ea$ alone.
	We also write composition of spans in diagrammatic order, using the symbol $\spancomp$.
	Thus the composition in \eqref{eqn:comp.of.dspans} would be denoted as $\Fa \spancomp \Ea$.
\end{notation}

\begin{rmk}
	Since $\DispSpan$ is not only a tricategory but a 2-category enriched bicategory, it has an underlying bicategory given by forgetting the 3-cells.
	This is the sense in which $\DispSpan$ is a stricter sort of tricategory---so long as we don't need the 3-cells, we can treat is as the ordinary bicategory of spans given by taking pullback in the underlying $1$-category of $\Kb$.
\end{rmk}

\begin{rmk}
	Note that the 2-categories $\DispSpan^{\Rightarrow}(\Ca, \Da)$ \emph{do not} assemble into a tricategory with pullback of spans as composition.
	In order to get a good composition law for right-colax maps of spans, we will need for the left leg to be a fibration so that we may lift the laxator.
	The upcoming \cref{lem:fspan.composition} gives the precise construction of composition of left-fibrant spans.
\end{rmk}

\begin{rmk}
	Note that the tricategory $\DispSpan(\cat{Cat}, \{\mathsf{all}\})$ constructed from the paradise of categories and all functors (\cref{ex:cat.as.paradise}) is the usual tricategory of spans of categories composing by strict pullback.
	Its underlying bicategory is the usual bicategory of spans of categories which one gets by considering $\cat{Cat}$ as a category with pullbacks.
	Therefore, pseudomonads in $\DispSpan(\cat{Cat}, \{\mathsf{all}\})$ are precisely double categories in the usual sense.
\end{rmk}

\section{Left-fibrant spans and contextads}
\label{sec:contextad.intro}

In this section, we will introduce\footnote{We will not fully define it here. For the full definition including structure equivalences and coherences, we will appeal to \cref{thm:fSpan.tricat}.} the tricategory $\FibSpan\Paradise^\twoto$ of \emph{left-fibrant spans} ${\Ca \epifrom \Ea \to \Da}$, whose left leg is a \emph{cloven cartesian displayed fibration}. We will describe the composition of left-fibrant spans, and give our first definition of a \emph{contextad}. We will finish with examples of contextads which are not colax actions of monoidal categories.

We will begin by presenting the notion of cloven cartesian displayed fibration in a paradise.

\subsection{Fibrations in a paradise}
In this section, fix a paradise $\Paradise$.
We will work with \emph{cloven cartesian display fibrations} (from now on, just fibrations), that is, display maps $p:\Ea \to \Ba$ explicitly equipped with a cleavage of (strong) cartesian lifts.
First, let's review the notions of comma and slice objects in a 2-category.


\begin{defn}
\label{defn:comma}
	Let $f : \Aa \to \Ca$ and $g : \Ba \to \Ca$ be maps in $\Cosmos$.
	The \textbf{comma category} $f \comma g$ is the universal cone making the diagram below lax commutative:
	\begin{equation}
		\begin{tikzcd}[ampersand replacement=\&]
			{f \comma g} \& \Ba \\
			\Aa \& \Ca
			\arrow[""{name=0, anchor=center, inner sep=0}, "f"', from=2-1, to=2-2]
			\arrow[""{name=1, anchor=center, inner sep=0}, "g", from=1-2, to=2-2]
			\arrow["{\dom}"', from=1-1, to=2-1]
			\arrow["{\cod}", from=1-1, to=1-2]
			\arrow["{\generic_{f,g}}", shift left=2, shorten <=4pt, shorten >=4pt, Rightarrow, from=0, to=1]
		\end{tikzcd}
	\end{equation}
	The universal property of the comma is such that each square like below on the left factors as one like below on the right:
	\begin{equation}
		\begin{tikzcd}[ampersand replacement=\&]
			\Xa \& \Ba \\
			\Aa \& \Ca
			\arrow[""{name=0, anchor=center, inner sep=0}, "f"', from=2-1, to=2-2]
			\arrow[""{name=1, anchor=center, inner sep=0}, "g", from=1-2, to=2-2]
			\arrow["a"', from=1-1, to=2-1]
			\arrow["b", from=1-1, to=1-2]
			\arrow["\varphi", shift left=2, shorten <=4pt, shorten >=4pt, Rightarrow, from=0, to=1]
		\end{tikzcd}
		\quad = \quad
		\begin{tikzcd}[ampersand replacement=\&]
			\Xa \&[-3.5ex]\\[-1.25ex]
			\& {f \comma g} \& \Ba \\
			\& \Aa \& \Ca
			\arrow[""{name=0, anchor=center, inner sep=0}, "f"', from=3-2, to=3-3]
			\arrow[""{name=1, anchor=center, inner sep=0}, "g", from=2-3, to=3-3]
			\arrow["a"', curve={height=6pt}, from=1-1, to=3-2]
			\arrow["b", curve={height=-6pt}, from=1-1, to=2-3]
			\arrow[from=2-2, to=3-2, "\dom"']
			\arrow[from=2-2, to=2-3, "\cod"]
			\arrow["{\internal{\varphi}}"{description}, from=1-1, to=2-2]
			\arrow["{\generic_{f,g}}", shift left=2, shorten <=4pt, shorten >=4pt, Rightarrow, from=0, to=1]
		\end{tikzcd}
	\end{equation}
	The 2-cell $\generic_{f,g} : f\dom \twoto g\cod$ is called the \emph{generic 2-cell} $f \twoto g$.
\end{defn}

\begin{defn}
\label{defn:slicable}
	A map $p : \Ea \to \Ba$ in a 2-category $\Cosmos$ is \textbf{sliceable} if the comma object $\Ba \comma p$ exists in $\Cosmos$.
\end{defn}

\begin{lem}
	Every display map $p : \Ea \to \Ba$ in a paradise $\Paradise$ is sliceable.
\end{lem}
\begin{proof}
	We may construct the slice $\Ba \comma p$ as the following pullback by
	appealing to the universal property of the arrow object:
	\begin{equation}
		\begin{tikzcd}
			{\Ba \comma p} & \Ea \\
			{\Ba^{\downarrow}} & \Ba
			\arrow[from=1-1, to=1-2]
			\arrow[from=1-1, to=2-1]
			\arrow["\lrcorner"{anchor=center, pos=0.125}, draw=none, from=1-1, to=2-2]
			\arrow["p", from=1-2, to=2-2]
			\arrow["\cod"', from=2-1, to=2-2]
		\end{tikzcd}
	\end{equation}
	This pullback exists since $p$ is a display map.
\end{proof}

\begin{defn}
\label{defn:fib}
	Suppose that $\Cosmos$ is a 2-category.
	A sliceable map $p$ is a \textbf{fibration}, denoted $p:\Ea \fibto \Ba$, if either of the following equivalent property-like structures is present:
	\begin{enumerate}
		\item The map $\internal{\id_p} : \Ea \to \Ba \comma p$, induced by applying $p$ to identity arrows, has a fibred right adjoint:
		      \begin{equation}
			      \begin{tikzcd}[ampersand replacement=\&]
				      {\Ba \comma p} \arrow[loop left, ""{name=2, anchor=center, inner sep=0}, distance=8ex, start anchor={[yshift=-2ex]west}, end anchor={[yshift=2ex]west}] \&\& \Ea \arrow[loop right, ""{name=3, anchor=center, inner sep=0}, distance=8ex, start anchor={[yshift=2ex]east}, end anchor={[yshift=-2ex]east}] \\
				      \& \Ba
				      \arrow["p", from=1-3, to=2-2]
				      \arrow["{{\dom}}"', from=1-1, to=2-2]
				      \arrow[""{name=0, anchor=center, inner sep=0}, "{\internal{\id_p}}"', curve={height=12pt}, from=1-3, to=1-1]
				      \arrow[""{name=1, anchor=center, inner sep=0}, "{{\pull}}"', curve={height=12pt}, dashed, from=1-1, to=1-3]
				      \arrow["{{\bot}}"{description}, draw=none, from=0, to=1]
				      \arrow[Rightarrow, shift right=0.2, from=2, to=1-1, shorten >= 3pt, shorten <= 5pt, "\counitpull"]
				      \arrow[Rightarrow, shift right=0.2, to=3, from=1-3, shorten >= 3pt, shorten <= 5pt, "\unitpull"]
			      \end{tikzcd}
		      \end{equation}
		      We refer to $\pull$ as the \textbf{pullback} functor of the fibration.
		\item (Supposing $\Ea$ admits an arrow object $\Ea^{\downarrow}$.) The map $\hat{p} : \Ea^{\downarrow} \to \Ba \comma p$, induced by the action of $p$ on arrows, has a right adjoint right inverse:
		      \begin{equation}
			      \begin{tikzcd}[ampersand replacement=\&]
				      {\Ba \comma p} \&\& {\Ea^\downarrow}
				      \arrow[""{name=0, anchor=center, inner sep=0}, "{\hat p}"', curve={height=12pt}, from=1-3, to=1-1]
				      \arrow[""{name=0p, anchor=center, inner sep=0}, phantom, from=1-3, to=1-1, start anchor=center, end anchor=center, curve={height=12pt}]
				      \arrow[""{name=1, anchor=center, inner sep=0}, "\lift"', curve={height=12pt}, dashed, from=1-1, to=1-3]
				      \arrow[""{name=1p, anchor=center, inner sep=0}, phantom, from=1-1, to=1-3, start anchor=center, end anchor=center, curve={height=12pt}]
				      \arrow["\bot"{description}, draw=none, from=0p, to=1p]
			      \end{tikzcd}
		      \end{equation}
		      We refer to $\lift$ as the \textbf{cartesian lift} functor of the fibration.
	\end{enumerate}

	If $\Cosmos$ underlies a paradise $\Paradise$, then we will furthermore ask
	that a fibration $p$ is a display map.
\end{defn}

\begin{rmk}
\label{rmk:fib.adj}
	The fact $\pull$ is a fibred right adjoint means that $p\pull = \dom$ ($\pull$ is a morphism in $\Cosmos/\Ba$) and the unit and counit of the adjunction are vertical, i.e.~$\dom \counitpull = \id_{\dom}$ and $p \unitpull = \id_p$.
\end{rmk}

\begin{rmk}
\label{rmk:fib.struct}
	Let's briefly unpack these two presentation of cartesian fibrations, as well as their equivalence, since it will be useful later.
	If we have $\pull : \Ba \comma p \to \Ea$, given $b \nto{\beta} p(e) : \Ba \comma p$, $\pull(\beta)$ is the domain of the cartesian lift of $\beta$, and thus can also be denoted as $\beta^*e$.
	The unit $\unitpull : \id_{\Ea} \twoto \pull\internal{\id_p}$ expresses the fact that $\id_{p(e)}^*e \iso e$.
	The counit $\counitpull : \internal{\id_p}\pull \twoto \id_{\Ba \comma p}$ encodes lifting, since its component at $\beta$ is given by
	\begin{equation}
		\counitpull_\beta :=
		\begin{tikzcd}[ampersand replacement=\&, sep=scriptsize]
			{p(\pull(\beta))} \&[5ex] b \\
			{p(\pull(\beta))} \& {p(e)}
			\arrow[Rightarrow, no head, from=1-1, to=2-1]
			\arrow["\beta", from=1-2, to=2-2]
			\arrow["\id_b", Rightarrow, no head, from=1-1, to=1-2]
			\arrow["{p(\lift(\beta))}"', from=2-1, to=2-2]
		\end{tikzcd}
		\ = \ (\id_b, \lift(\beta) : \beta^*e \to e).
	\end{equation}
	Notice the fact the above diagram commutes encodes the fact $p(\lift(\beta)) = \beta$, i.e.~that $\lift(\beta)$ is in fact a lift.
	Indeed, as suggested by the notation we can define $\lift : \Ba \comma p \to \Ea^{\downarrow}$ by applying the universal property of $\Ea^{\downarrow}$ to the natural transformation $\cod \counitpull : \pull \Rightarrow \cod$, noting that $\cod \internal{\id_p} = \id_{\Ba \comma p}$.
	On the other hand, $\lift : \Ba \comma p \to \Ea^\downarrow$ directly gives the whole cartesian lift, so that $\lift(\beta) = \beta^*e \to e$ as before and we may recover $\pull$ as $\dom \lift$.
	The unit of $\hat p \adj \lift$ is the comparison map between any other $\gamma:e \to e'$ in $\Ea$ and the lift of $p(\gamma)$, exhibiting the `vertical' part of $\gamma$:
	\begin{equation}
		\begin{tikzcd}[ampersand replacement=\&,row sep=scriptsize]
			{e'} \\[2ex]
			{p(\gamma)^*e} \& e \\
			{p(e')} \& {p(e)}
			\arrow["\gamma", from=1-1, to=2-2]
			\arrow["{p(\gamma)}", from=3-1, to=3-2]
			\arrow["{\lift(p(\gamma))}"', from=2-1, to=2-2]
			\arrow["{\exists!\,\unit_\gamma}"', dashed, from=1-1, to=2-1]
		\end{tikzcd}
	\end{equation}
	The fact the counit of $\hat p \adj \lift$ is an identity reflects that a lift is strictly over the morphism it's lifting, i.e.~$p(\lift(\beta)) = \beta$.
	For $\internal{\id_p} \adj \pull$, this corresponds to verticality of $\counitpull$.
\end{rmk}

\begin{rmk}
	The lifting property we just described can also be visualized as follows:
	\begin{equation}
		\begin{tikzcd}[ampersand replacement=\&]
			\Xa \&\& \Ea \\
			\&\& \Ba
			\arrow["e", from=1-1, to=1-3]
			\arrow["p", two heads, from=1-3, to=2-3]
			\arrow[""{name=0, anchor=center, inner sep=0}, "b"', from=1-1, to=2-3]
			\arrow["\beta", shift right, shorten <=7pt, shorten >=7pt, Rightarrow, from=0, to=1-3]
		\end{tikzcd}
		\quad =\ %
		\begin{tikzcd}[ampersand replacement=\&]
			\Xa \&\& \Ea \\
			\&\& \Ba
			\arrow["b"', from=1-1, to=2-3]
			\arrow["p", two heads, from=1-3, to=2-3]
			\arrow[""{name=0, anchor=center, inner sep=0}, "e", shift left, from=1-1, to=1-3]
			\arrow[""{name=1, anchor=center, inner sep=0}, "{\pull(\beta)}"'{pos=0.6}, curve={height=16pt}, dashed, from=1-1, to=1-3]
			\arrow["{\lift(\beta)}"', shorten <=2pt, shorten >=2pt, Rightarrow, dashed, from=1, to=0]
		\end{tikzcd}
	\end{equation}
\end{rmk}

\begin{defn}
\label{defn:cartesian.functor}
	Let $p_1 :\Ea_1 \fibto \Ba$ and $p_2 : \Ea_2 \fibto \Ba$ be fibrations.
	A \textbf{cartesian functor} is a morphism $f : \Ea_1 \to \Ea_2$ so that the following triangle commutes strictly
	\begin{equation}
		\begin{tikzcd}[ampersand replacement=\&]
			{\Ea_1} \&\& {\Ea_2} \\
			\& \Ba
			\arrow["{p_1}"', two heads, from=1-1, to=2-2]
			\arrow["{p_2}", two heads, from=1-3, to=2-2]
			\arrow["f", from=1-1, to=1-3]
		\end{tikzcd}
	\end{equation}
	and the following mate---which we call \textbf{cartesianator}---is an isomorphism:
	\begin{equation}
	\label{eqn:cartesian.fun}
		\begin{tikzcd}[ampersand replacement=\&]
			{\Ba \comma p_1} \& {\Ba \comma p_2} \\
			{\Ea_1} \& \Ea_2
			\arrow["{{\hat f}}", from=1-1, to=1-2]
			\arrow["{\pull_1}"', from=1-1, to=2-1]
			\arrow["{\pull_2}", from=1-2, to=2-2]
			\arrow["{\cartesianator^f}"{pos=0.6}, shift right=.5, shorten <=13pt, shorten >=10pt, Rightarrow, from=2-1, to=1-2]
			\arrow["f"', from=2-1, to=2-2]
		\end{tikzcd}
		\qquad := \quad
		\begin{tikzcd}[ampersand replacement=\&]
			{\Ba \comma p_1} \&[-4ex] {\Ba \comma p_1} \& {\Ba \comma p_2} \\
			\& \Ea_1 \& \Ea_2 \&[-4ex] \Ea_2
			\arrow["f"', from=2-2, to=2-3]
			\arrow["{\hat f}", from=1-2, to=1-3]
			\arrow["{\internal{\id_{p_1}}}"', from=2-2, to=1-2]
			\arrow["{\internal{\id_{p_2}}}", from=2-3, to=1-3]
			\arrow["\counitpull_1"{pos=0.6}, shift left=5, shorten <=9pt, shorten >=2pt, Rightarrow, from=2-2, to=1-2]
			\arrow["\unitpull_2"'{pos=0.4}, shift right=5, shorten <=2pt, shorten >=9pt, Rightarrow, from=2-3, to=1-3]
			\arrow["\pull_1"', curve={height=6pt}, from=1-1, to=2-2]
			\arrow["\pull_2", curve={height=-6pt}, from=1-3, to=2-4]
			\arrow[Rightarrow, no head, from=2-3, to=2-4]
			\arrow[Rightarrow, no head, from=1-1, to=1-2]
		\end{tikzcd}
	\end{equation}
	where $\hat f$ is the functor induced by $f$ between the evident comma squares.
\end{defn}

\begin{lem}
\label{lem:fib.compose}
\label{lem:pb.fibs}
	Let $p : \Ea \epito \Ba$ be a cloven cartesian fibration.
	Then:
	\begin{enumerate}
		\item if $f: \Aa \to \Ba$ is a 1-cell, the pullback $f^*p : f^*\Ea \epito \Aa$ is a cloven cartesian fibration,
		\item if $q : \Ba \epito \Ca$ is a cloven cartesian fibration, $qp:\Ea \epito \Ca$ is a cloven cartesian fibration.
	\end{enumerate}
\end{lem}
\begin{proof}
	See Lemma 5.2.3 and Proposition 5.2.4 of \cite{riehl_elements_2022}.
\end{proof}

\subsection[The tricategory of left-fibrant spans]{The tricategory of left-fibrant spans}

We are now ready to present enough of the tricategory $\FibSpan^{\twoto}$ of left-fibrant spans in order to define a contextad. In this section, we will only present the composition law for $\FibSpan^{\twoto}$; for the full definition including coherences, we will wait for \cref{thm:fSpan.tricat} in the upcoming \cref{sec:fibs.span} where we will transport the coherences from a more abstractly defined tricategory. The definition that we present here will be sufficient to give the definition of a contextad in \cref{defn:colax.fibred.action}.

\begin{defn}
\label{defn:left.fibrant.right.colax.spans}
	Fix a paradise $\Paradise$.
	For objects $\Ca$ and $\Da$ in $\Cosmos$, the 2-category $\FibSpan\Paradise^\twoto(\Ca, \Da)$ of \textbf{left-fibrant spans and right-colax maps} between $\Ca$ and $\Da$ is the full sub-2-category of $\DispSpan\Paradise^\twoto(\Ca, \Da)$ spanned by the left-displayed spans $\Ca \xleftarrow{p} \Ea \xrightarrow{f} \Da$ whose left map $p$ is a fibration in $\Cosmos$.
\end{defn}

Now show that composition of spans via pullback extends to a 2-functor on the hom 2-categories $\FibSpan\Paradise^\twoto(\Ca, \Da)$.

\begin{lem}
\label{lem:fspan.composition}
	For every $\Ba, \Ca, \Da \in \Kb$, there is a 2-functor
	\begin{equation}
		\spancomp : \FibSpan^\twoto(\Ba, \Ca) \times \FibSpan^\twoto(\Ca, \Da) \longto \FibSpan^\twoto(\Ba, \Da)
	\end{equation}
	given by composition of spans and acting on morphisms by
	\begin{equation}
		(k, \varphi), (h,\psi) \mapsto ((k p_1', h \pull_{q'}(\varphi p_1')), \psi \lift_{q'}(\varphi p_1'))
	\end{equation}
	where notation refers to the diagram \eqref{eqn:fibspan-comp} below.
\end{lem}

\begin{notation}
	We denote the composite of $(k, \varphi)$ and $(h,\psi)$ as $(k \spancomp_\varphi h, \varphi \spancomp_\varphi \psi)$.
\end{notation}

\begin{proof}
	On objects, the 2-functor is defined by strict pullback.
	We need to verify this definition extends to the 1- and 2-cells.
	This is non-trivial since the 1-cells are cartesian for the fibrations and right-lax.
	The first aspect can be liquidated quickly since, by \cref{lem:pb.fibs}, pullbacks preserve cartesian fibrations and cartesian functors, and furthermore cartesian fibrations compose.
	Regarding the second aspect, we need to explain how the laxator appearing on the right of the first composee is lifted to become a laxator for the composite.

	Thus consider the solid diagram below:
	\begin{equation}
	\label{eqn:fibspan-comp}
		\begin{tikzcd}[ampersand replacement=\&]
			\&\&\&\& {\Ea' \spancomp \Fa'} \\
			\\
			\&\& {\Ea'} \&\& {\Ea \spancomp \Fa} \&\& {\Fa'} \\
			\&\& \Ea \&\&\&\& \Fa \\
			\Ba \&\&\&\& \Ca \&\&\&\& \Da
			\arrow["p",two heads, from=4-3, to=5-1]
			\arrow["f"', from=4-3, to=5-5]
			\arrow["q", two heads, from=4-7, to=5-5]
			\arrow["g"',from=4-7, to=5-9]
			\arrow["{p_1}"', two heads, from=3-5, to=4-3]
			\arrow["{p_2}", from=3-5, to=4-7]
			\arrow["\lrcorner"{anchor=center, pos=0.125, rotate=-45}, draw=none, from=3-5, to=5-5]
			\arrow["p'"',two heads, from=3-3, to=5-1]
			\arrow[""{name=0, anchor=center, inner sep=0}, "{f'}", from=3-3, to=5-5]
			\arrow[""{name=0p, anchor=center, inner sep=0}, phantom, from=3-3, to=5-5, start anchor=center, end anchor=center]
			\arrow["k"', from=3-3, to=4-3]
			\arrow["g'", from=3-7, to=5-9]
			\arrow[""{name=1p, anchor=center, inner sep=0}, phantom, from=3-7, to=5-9, start anchor=center, end anchor=center]
			\arrow["{h}", from=3-7, to=4-7]
			\arrow["{q'}"', two heads, from=3-7, to=5-5]
			\arrow["{p_1'}"', two heads, from=1-5, to=3-3]
			\arrow[""{name=2, anchor=center, inner sep=0}, "{p_2'}", from=1-5, to=3-7]
			\arrow[""{name=3, anchor=center, inner sep=0}, dashed,"{\pull_{q'}(\varphi p_1')}"{description, pos=0.7}, shift right=2, curve={height=26pt}, shorten <=8pt, dashed, from=1-5, to=3-7]
			\arrow["{\exists!\,k \spancomp_\varphi h}"'{pos=0.7}, shorten <=6pt, dashed, from=1-5, to=3-5]
			\arrow["\lrcorner"{anchor=center, pos=0.012, rotate=-45}, draw=none, from=1-5, to=5-5]
			\arrow["\varphi", shorten >=6pt, Rightarrow, from=4-3, to=0p, shift right=0.25]
			\arrow["{\lift_{q'}(\varphi p_1')}"{rotate=-35, pos=0.25}, dashed, shorten <=3pt, shorten >=3pt, Rightarrow, from=3, to=2]
			\arrow["\psi", shorten >=5pt, Rightarrow, from=4-7, to=1p]
		\end{tikzcd}
	\end{equation}
	Observe that $\varphi p_1' : fkp_1' \twoto f'p_1' = q'p_2'$.
	This means the components of $\varphi p_1'$ can be lifted using the fibrational structure of $q'$.
	Concretely, that gives us the arrow $\pull_{q'}(\varphi p_1') :\Ea' \spancomp \Fa' \to \Fa'$ as well as $\lift_{q'}(\varphi p_1') : \pull_{q'}(\varphi p_1') \twoto p_2'$:
	\begin{equation}
		\begin{tikzcd}[ampersand replacement=\&,sep=scriptsize]
			{\Ea' \spancomp \Fa'} \&\& {\Fa'} \\[2ex]
			\\
			\&\& \Ca
			\arrow["{p_2'}", shift left, from=1-1, to=1-3]
			\arrow["{q'}", two heads, from=1-3, to=3-3]
			\arrow[""{name=0, anchor=center, inner sep=0}, "{fkp_1'}"', curve={height=10pt}, from=1-1, to=3-3]
			\arrow["{\varphi p_1'}", shorten <=6pt, shorten >=6pt, Rightarrow, from=0, to=1-3]
		\end{tikzcd}
		\quad = \ %
		\begin{tikzcd}[ampersand replacement=\&,sep=scriptsize]
			{\Ea' \spancomp \Fa'} \&[3.5ex]\&[3.5ex] {\Fa'} \\[2ex]
			\\
			\&\& \Ca
			\arrow[""{name=0, anchor=center, inner sep=0}, "{p_2'}", shift left=1.5, from=1-1, to=1-3]
			\arrow[""{name=0p, anchor=center, inner sep=0}, phantom, from=1-1, to=1-3, start anchor=center, end anchor=center, shift left]
			\arrow["{q'}", two heads, from=1-3, to=3-3]
			\arrow["{fkp_1'}"', curve={height=10pt}, from=1-1, to=3-3]
			\arrow[""{name=1, anchor=center, inner sep=0}, "{\pull_{q'}(\varphi p_1')}"'{pos=0.55}, curve={height=18pt}, from=1-1, to=1-3, dashed]
			\arrow[""{name=1p, anchor=center, inner sep=0}, phantom, from=1-1, to=1-3, start anchor=center, end anchor=center, curve={height=12pt}]
			\arrow["{\lift_{q'}(\varphi p_1')}", shift right=3, Rightarrow, from=1p, to=0p]
		\end{tikzcd}
	\end{equation}
	The square $fkp_1' = q'\pull_{q'}(\varphi p_1') = qk'\pull_{q'}(\varphi p_1')$ witnesses a cone insisting on the cospan $\Ea \nto{f} \Ca \nepifrom{q} \Fa$, and thus induces a unique comparison map $k \spancomp_\varphi h := (k p_1', h \pull_{q'}(\varphi p_1'))$ such that $p_2 (k \spancomp_\varphi h) = h \pull_{q'}(\varphi p_1')$.
	All in all, this gives the desired filling 2-cell for the right hand side of the composite map of spans, defined as $\psi \lift_{q'}(\varphi p_1')$.

	Reasoning in the same way shows how to compose 2-cells of left-fibrant spans.
	If $\alpha : (k, \varphi) \twoto (k', \varphi)'$ and $\beta : (h, \psi) \twoto (h', \psi')$, we would have the horizontal composite $\alpha \spancomp_\alpha \beta$ being the 2-cell between $k \spancomp_\varphi h$ and $k' \spancomp_{\varphi'} h'$ given by $(\alpha p_1',\, \beta h'\pull_{q'} f(\alpha p_1'))$.
	We thereby conclude that $\then_{\Ba,\Ca,\Da}$ is well-defined.
\end{proof}

\begin{rmk}
	Using elements, given $(E',F') \in \Ea' \spancomp \Fa'$, the composite 1-cell defined above is given on apexes as
	\begin{equation}
		(k \spancomp_\varphi h)(E', F') = \left(k(E'),\ h(\varphi_{E'}^* F')\right)
	\end{equation}
	and on the filler by
	\begin{equation}
		(\varphi \spancomp_\varphi \psi)_{(E',F')} := g(h(\varphi^*_{E'}F')) \xto{g(\lift_{q'}\varphi_{E'})} g(h(F')) \xto{\psi_{F'}} g'(F')
	\end{equation}
	where the star denotes pull according to $q'$.
	Likewise, if $\alpha$ and $\beta$ are 2-cells as above, we get
	\begin{equation}
		(\alpha \spancomp_\alpha \beta)(E', F') := \begin{pmatrix}
			k(E') \xto{\alpha_{E'}} k'(E')\\
			h(\varphi_{E'}^* F') \xto{\beta_{\varphi_{E'}^*F'}} h'(\varphi_{E'}^* F') \xto{h'(\lift_{q'}f(\alpha_{E'}))} h'({\varphi'}_{E'}^* F')
		\end{pmatrix}
	\end{equation}
\end{rmk}

\begin{rmk}
	We could verify directly that this composition is suitably unital and
	associative, but we will instead appeal to \cref{thm:fSpan.tricat} to
	transport the tricategory structure of the free Kleisli cocompletion of
	$\DispSpan\Paradise$. This argument is carried out in \cref{sec:fibs.span},
	specifically \cref{thm:fSpan.tricat}.
\end{rmk}

\subsection[Contextads as pseudomonads in left-fibrant spans]{Contextads as pseudomonads in $\FibSpan^\twoto$}

Now that we have a tricategory $\FibSpan^\twoto$ constructed from a paradise, we may give our definition of a \emph{contextad} on an object of a paradise.

\begin{defn}
\label{defn:colax.fibred.action}
	Let $\Paradise$ be a paradise.
	A \textbf{contextad} on an object $\acted$ in $\Cosmos$ is a pseudomonad in the tricategory $\FibSpan\Paradise^\twoto$ on $\acted$.
	This consists of a span $\fibcolaxaction$ where $p$ is a displayed fibration (that is, a fibration in $\Display$), together with the following data:
	\begin{enumerate}
		\item unit and multiplication 2-cells
		\begin{eqalign}
			\begin{tikzcd}[ampersand replacement=\&]
				\acted \& \acted \& \acted \\
				\acted \& \actor \& \acted
				\arrow[Rightarrow, no head, from=1-1, to=2-1]
				\arrow[Rightarrow, no head, from=1-2, to=1-1]
				\arrow[""{name=0, anchor=center, inner sep=0}, Rightarrow, no head, from=1-2, to=1-3]
				\arrow[""{name=0p, anchor=center, inner sep=0}, phantom, from=1-2, to=1-3, start anchor=center, end anchor=center]
				\arrow["\combineunit"', from=1-2, to=2-2]
				\arrow[Rightarrow, no head, from=1-3, to=2-3]
				\arrow["p", from=2-2, to=2-1]
				\arrow[""{name=1, anchor=center, inner sep=0}, "\action"', from=2-2, to=2-3]
				\arrow[""{name=1p, anchor=center, inner sep=0}, phantom, from=2-2, to=2-3, start anchor=center, end anchor=center]
				\arrow["\counitor"', shorten <=4pt, shorten >=4pt, Rightarrow, from=1p, to=0p]
			\end{tikzcd}
			&\qquad\qquad
			\begin{gathered}
				\combineunit_A \in \actor_A\\
				\counitor : A \action \combineunit \to A
			\end{gathered}
			\\
			\begin{tikzcd}[ampersand replacement=\&,column sep=scriptsize,row sep=small]
				\acted \& \actor \& \acted \& \actor \& \acted \\
				\&\& {\actor \spancomp \actor} \\[2ex] {\acted} \&\& \actor \& {} \& \acted
				\arrow[Rightarrow, no head, from=1-1, to=3-1]
				\arrow["p"', from=1-2, to=1-1]
				\arrow["\action", from=1-2, to=1-3]
				\arrow["p"', from=1-4, to=1-3]
				\arrow["\action", from=1-4, to=1-5]
				\arrow[Rightarrow, no head, from=1-5, to=3-5]
				\arrow[from=2-3, to=1-2]
				\arrow["\lrcorner"{anchor=center, pos=0.125, rotate=135}, draw=none, from=2-3, to=1-3]
				\arrow[from=2-3, to=1-4]
				\arrow["\combine"', from=2-3, to=3-3]
				\arrow["p", from=3-3, to=3-1]
				\arrow["\action"', from=3-3, to=3-5]
				\arrow["\coassociator"', shorten <=6pt, shorten >=6pt, Rightarrow, from=3-4, to=1-4]
			\end{tikzcd}
			&\qquad\qquad
			\begin{gathered}
				\lens{P}{A} \combine \lens{Q}{A \action P} \in \actor_A\\
				\coassociator : A \action (P \combine Q) \to (A \action P) \action Q
			\end{gathered}
		\end{eqalign}
		where $\combineunit$ and $\combine$ are both cartesian with respect to $p$, i.e.~there are natural $p$-vertical isomorphisms
		\begin{equation}
		\label{eq:combine-unit-cart}
			\combineunitcart : \combineunit \iso f^* \combineunit, \qquad \combinecart : f^* P \combine (f \action P)^*Q \iso f^* (P \combine Q).
		\end{equation}
		\item invertible 3-cells:
		      \begin{eqalign}
			      \label{eqn:fib.colax.action.unitors.associator}
			      \begin{tikzcd}
				      {\acted \spancomp \actor} &[1ex] {\actor \spancomp \actor} &[1ex] {\actor \spancomp \acted} \\[2.5ex]
				      & \actor
				      \arrow["{\combineunit \spancomp_\counitor \actor}", from=1-1, to=1-2]
				      \arrow[""{name=0, anchor=center, inner sep=0}, from=1-1, to=2-2]
				      \arrow["\combine"{description}, from=1-2, to=2-2]
				      \arrow["{\actor \spancomp \combineunit}"', from=1-3, to=1-2]
				      \arrow[""{name=1, anchor=center, inner sep=0}, from=1-3, to=2-2]
				      \arrow["\leftunitlaw"', shorten <=5pt, Rightarrow, from=0, to=1-2]
				      \arrow["\rightunitlaw"', shorten >=5pt, Rightarrow, from=1-2, to=1]
			      \end{tikzcd}
				  &\qquad\qquad
				  \begin{gathered}
					\leftunitlaw : \lens{P}{A} \isoto \lens{\combineunit}{A} \combine \lens{\counitor^* P}{A \action \combineunit}\\
					\rightunitlaw : \lens{P}{A} \combine \lens{\combineunit}{A \action P} \isoto \lens{P}{A}
				  \end{gathered}
			      \\
			      \begin{tikzcd}[sep=scriptsize]
				      & {\actor \spancomp \actor \spancomp \actor} \\
				      {\actor \spancomp (\actor \spancomp \actor)} && {(\actor \spancomp \actor) \spancomp \actor} \\[2ex]
				      {\actor \spancomp \actor} && {\actor \spancomp \actor} \\
				      & \actor
				      \arrow[from=1-2, to=2-1]
				      \arrow[from=1-2, to=2-3]
				      \arrow[""{name=0, anchor=center, inner sep=0}, "{\actor \,\spancomp\, \combine}"', from=2-1, to=3-1]
				      \arrow[""{name=1, anchor=center, inner sep=0}, "{\combine \,\spancomp_\coassociator\, \actor}", from=2-3, to=3-3]
				      \arrow["\combine"', from=3-1, to=4-2]
				      \arrow["\combine", from=3-3, to=4-2]
				      \arrow["\associativitylaw", shorten <=50pt, shorten >=50pt, Rightarrow, from=0, to=1]
			      \end{tikzcd}
				  &\qquad\qquad
				  \begin{gathered}
					\associativitylaw : P \combine (Q \combine R) \isoto (P \combine Q) \combine \coassociator^* R
				  \end{gathered}
		      \end{eqalign}
		      As 3-cells, these must satisfy the following equations relating them to $\counitor$ and $\coassociator$:
		      \begin{eqalign}
				\begin{tikzcd}[ampersand replacement=\&,sep=scriptsize]
					{\acted \spancomp\actor} \&[3ex]\&[3ex] \\
					\\
					{\actor \spancomp\actor} \&\& \acted \\
					\\
					\actor
					\arrow["{\combineunit \spancomp_\counitor \actor}"{description}, from=1-1, to=3-1]
					\arrow[""{name=0, anchor=center, inner sep=0}, "{\action p_{\actor}}", from=1-1, to=3-3]
					\arrow[""{name=0p, anchor=center, inner sep=0}, phantom, from=1-1, to=3-3, start anchor=center, end anchor=center]
					\arrow[""{name=1, anchor=center, inner sep=0}, "{p_{\actor}}"', shift right=5, curve={height=30pt}, from=1-1, to=5-1]
					\arrow[""{name=1p, anchor=center, inner sep=0}, phantom, from=1-1, to=5-1, start anchor=center, end anchor=center, shift right=5, curve={height=30pt}]
					\arrow[""{name=2, anchor=center, inner sep=0}, "{\action\action}"{description}, from=3-1, to=3-3]
					\arrow[""{name=2p, anchor=center, inner sep=0}, phantom, from=3-1, to=3-3, start anchor=center, end anchor=center]
					\arrow[""{name=2p, anchor=center, inner sep=0}, phantom, from=3-1, to=3-3, start anchor=center, end anchor=center]
					\arrow["\combine"{description}, from=3-1, to=5-1]
					\arrow[""{name=3, anchor=center, inner sep=0}, "\action"', from=5-1, to=3-3]
					\arrow[""{name=3p, anchor=center, inner sep=0}, phantom, from=5-1, to=3-3, start anchor=center, end anchor=center]
					\arrow["\leftunitlaw", shorten <=0pt, Rightarrow, from=1p, to=3-1]
					\arrow["{\action \lift_p(\counitor p_{\acted})}", shorten <=4pt, shorten >=4pt, Rightarrow, from=2p, to=0p]
					\arrow["\coassociator", shorten <=4pt, shorten >=4pt, Rightarrow, from=3p, to=2p]
				\end{tikzcd}
				\quad &= \quad
				\begin{tikzcd}[ampersand replacement=\&, sep=scriptsize]
					{\acted\spancomp\actor} \&[3ex]\&[3ex] \\
					\\
					\&\& \acted \\
					\\
					\actor
					\arrow["{{\action p_{\actor}}}", from=1-1, to=3-3]
					\arrow["{{p_{\actor}}}"', from=1-1, to=5-1]
					\arrow["\action"', from=5-1, to=3-3]
				\end{tikzcd}
				\\
				\begin{tikzcd}[ampersand replacement=\&,sep=scriptsize]
					{\actor \spancomp\acted} \&[3ex]\&[3ex] \\
					\\
					{\actor \spancomp\actor} \&\& \acted \\
					\\
					\actor
					\arrow["{{\actor \spancomp \combineunit}}"{description}, from=1-1, to=3-1]
					\arrow[""{name=0, anchor=center, inner sep=0}, "{{\action p_{\actor}}}", from=1-1, to=3-3]
					\arrow[""{name=0p, anchor=center, inner sep=0}, phantom, from=1-1, to=3-3, start anchor=center, end anchor=center]
					\arrow[""{name=1, anchor=center, inner sep=0}, "{\action\action}"{description}, from=3-1, to=3-3]
					\arrow[""{name=1p, anchor=center, inner sep=0}, phantom, from=3-1, to=3-3, start anchor=center, end anchor=center]
					\arrow[""{name=1p, anchor=center, inner sep=0}, phantom, from=3-1, to=3-3, start anchor=center, end anchor=center]
					\arrow["\combine"{description}, from=3-1, to=5-1]
					\arrow[""{name=2, anchor=center, inner sep=0}, "\action"', from=5-1, to=3-3]
					\arrow[""{name=2p, anchor=center, inner sep=0}, phantom, from=5-1, to=3-3, start anchor=center, end anchor=center]
					\arrow["{{\counitor p_{\actor}}}", shorten <=4pt, shorten >=4pt, Rightarrow, from=1p, to=0p]
					\arrow["\coassociator", shorten <=4pt, shorten >=4pt, Rightarrow, from=2p, to=1p]
				\end{tikzcd}
				\quad &= \quad
				\begin{tikzcd}[ampersand replacement=\&, sep=scriptsize]
					{\acted\spancomp\actor} \&[3ex]\&[3ex] \\
					\\
					{\actor\spancomp\actor} \&\& \acted \\
					\\
					\actor
					\arrow["{{\actor \spancomp \combineunit}}"{description}, from=1-1, to=3-1]
					\arrow[""{name=0, anchor=center, inner sep=0}, "{{\action p_{\actor}}}", from=1-1, to=3-3]
					\arrow[""{name=0p, anchor=center, inner sep=0}, phantom, from=1-1, to=3-3, start anchor=center, end anchor=center]
					\arrow["\combine"{description}, from=3-1, to=5-1]
					\arrow[""{name=1, anchor=center, inner sep=0}, "\action"', from=5-1, to=3-3]
					\arrow[""{name=1p, anchor=center, inner sep=0}, phantom, from=5-1, to=3-3, start anchor=center, end anchor=center]
					\arrow["\action\rightunitlaw", shorten <=9pt, shorten >=9pt, Rightarrow, from=1p, to=0p]
				\end{tikzcd}
				\\
				\begin{tikzcd}[ampersand replacement=\&, sep=scriptsize]
					\& {\actor \spancomp \actor \spancomp \actor} \&[3ex]\&[3ex] \\
					\\
					{\actor \spancomp \actor} \& {\actor\spancomp\actor} \&\& \acted \\
					\\
					\& \actor
					\arrow["{{\actor \spancomp \combine}}"', from=1-2, to=3-1]
					\arrow["{{\combine \spancomp_\coassociator \actor}}"{description}, from=1-2, to=3-2]
					\arrow[""{name=0, anchor=center, inner sep=0}, "{\action\action\action}", from=1-2, to=3-4]
					\arrow[""{name=0p, anchor=center, inner sep=0}, phantom, from=1-2, to=3-4, start anchor=center, end anchor=center]
					\arrow["\associativitylaw", shorten <=8pt, shorten >=8pt, Rightarrow, from=3-1, to=3-2]
					\arrow["\combine"', from=3-1, to=5-2]
					\arrow[""{name=1, anchor=center, inner sep=0}, "{\action\action}"{description}, from=3-2, to=3-4]
					\arrow[""{name=1p, anchor=center, inner sep=0}, phantom, from=3-2, to=3-4, start anchor=center, end anchor=center]
					\arrow[""{name=1p, anchor=center, inner sep=0}, phantom, from=3-2, to=3-4, start anchor=center, end anchor=center]
					\arrow["\combine"{description}, from=3-2, to=5-2]
					\arrow[""{name=2, anchor=center, inner sep=0}, "{\action}"', from=5-2, to=3-4]
					\arrow[""{name=2p, anchor=center, inner sep=0}, phantom, from=5-2, to=3-4, start anchor=center, end anchor=center]
					\arrow["{\action\lift_p(\coassociator p_{\hat{3}})}", shorten <=4pt, shorten >=4pt, Rightarrow, from=1p, to=0p]
					\arrow["\coassociator", shorten <=4pt, shorten >=4pt, Rightarrow, from=2p, to=1p]
				\end{tikzcd}
				\quad &= \quad
				\begin{tikzcd}[ampersand replacement=\&,sep=scriptsize]
					{\actor \spancomp \actor \spancomp \actor} \&[3ex]\&[3ex]\\
					\\
					{\actor \spancomp \actor} \&\& \acted \\
					\\
					\actor
					\arrow["{{\actor \spancomp \combine}}"{description}, from=1-1, to=3-1]
					\arrow[""{name=0, anchor=center, inner sep=0}, "{\action\action\action}", from=1-1, to=3-3]
					\arrow[""{name=0p, anchor=center, inner sep=0}, phantom, from=1-1, to=3-3, start anchor=center, end anchor=center]
					\arrow[""{name=1, anchor=center, inner sep=0}, "{\action\action}"{description}, from=3-1, to=3-3]
					\arrow[""{name=1p, anchor=center, inner sep=0}, phantom, from=3-1, to=3-3, start anchor=center, end anchor=center]
					\arrow[""{name=1p, anchor=center, inner sep=0}, phantom, from=3-1, to=3-3, start anchor=center, end anchor=center]
					\arrow["\combine"{description}, from=3-1, to=5-1]
					\arrow[""{name=2, anchor=center, inner sep=0}, "\action"', from=5-1, to=3-3]
					\arrow[""{name=2p, anchor=center, inner sep=0}, phantom, from=5-1, to=3-3, start anchor=center, end anchor=center]
					\arrow["{{\coassociator p_{\hat{1}}}}", shorten <=4pt, shorten >=4pt, Rightarrow, from=1p, to=0p]
					\arrow["\coassociator", shorten <=4pt, shorten >=4pt, Rightarrow, from=2p, to=1p]
				\end{tikzcd}
		      \end{eqalign}
			  which on elements correspond to:
			  \begin{equation}
				\begin{tikzcd}[ampersand replacement=\&]
					{A \action (\combineunit \combine \counitor^* P)} \&\& {(A \action \combineunit) \action \counitor^* P} \\
					\& {A \action P}
					\arrow["\coassociator", from=1-1, to=1-3]
					\arrow["{A \action \leftunitlaw^{-1}}"', from=1-1, to=2-2]
					\arrow["{\counitor \action P}", from=1-3, to=2-2]
				\end{tikzcd}
			  \end{equation}
			  \begin{equation}
				\label{eqn:ctx.rightunitlaw.3cell.cond}
				\begin{tikzcd}[ampersand replacement=\&]
					{A \action (P \combine \combineunit)} \&\& {(A \action P) \action \combineunit} \\
					\& {A \action P}
					\arrow["\coassociator", from=1-1, to=1-3]
					\arrow["{A \action \rightunitlaw}"', from=1-1, to=2-2]
					\arrow["\counitor", from=1-3, to=2-2]
				\end{tikzcd}
			  \end{equation}
			  \begin{equation}
				\label{eqn:ctx.associativitylaw.3cell.cond}
				\begin{tikzcd}[ampersand replacement=\&]
					{A \action (P \combine (Q \combine R))} \&[-6.5ex]\&[-6.5ex] {A \action ((P \combine Q) \combine \coassociator^* R)} \\
					{(A \action P) \action (Q \combine R)} \&\& {(A \action (P \combine Q)) \action \coassociator^* R} \\
					\& {((A \action P) \action Q) \action R}
					\arrow["{A \action \associativitylaw}", from=1-1, to=1-3]
					\arrow["\coassociator"', from=1-1, to=2-1]
					\arrow["\coassociator", from=1-3, to=2-3]
					\arrow["\coassociator"', from=2-1, to=3-2]
					\arrow["{\coassociator \action R}", from=2-3, to=3-2]
				\end{tikzcd}
			  \end{equation}
		\item Furthermore, the 3-cells have to satisfy coherences reminiscent of those of a monoidal category, namely unitality and associativity as per \cite{lack_coherent_2000} (though note the associator there goes the opposite way):
		\begin{equation}
			\begin{tikzcd}[ampersand replacement=\&,sep=scriptsize]
				\&[2ex]\&[-3ex] {\actor \spancomp \actor} \&[-2ex]\\
				{\actor \spancomp \actor} \& {\actor^{\spancomp 3}} \&\& \actor \\
				\&\& {\actor \spancomp \actor}
				\arrow["\combine", from=1-3, to=2-4]
				\arrow["\associativitylaw"', shorten <=6pt, shorten >=6pt, Rightarrow, from=3-3, to=1-3]
				\arrow["{\actor \spancomp \combineunit \spancomp_\counitor \actor}", from=2-1, to=2-2]
				\arrow["{\combine \spancomp_\coassociator \actor}", from=2-2, to=1-3]
				\arrow["{\actor \spancomp \combine}"', from=2-2, to=3-3]
				\arrow["\combine"', from=3-3, to=2-4]
			\end{tikzcd}
			\quad = \quad
			\begin{tikzcd}[ampersand replacement=\&,sep=scriptsize]
				\&[-2ex] {\actor^{\spancomp 3}} \&[-2ex] \&[-2ex]\\
				{\actor \spancomp \actor} \&\& {\actor \spancomp \actor} \& \actor \\
				\& {\actor^{\spancomp 3}}
				\arrow["{\combine \spancomp_\coassociator \actor}", from=1-2, to=2-3]
				\arrow["{\actor \spancomp \combineunit \spancomp_\counitor \actor}", from=2-1, to=1-2]
				\arrow[""{name=0, anchor=center, inner sep=0}, Rightarrow, no head, from=2-1, to=2-3]
				\arrow[""{name=0p, anchor=center, inner sep=0}, phantom, from=2-1, to=2-3, start anchor=center, end anchor=center]
				\arrow[""{name=0p, anchor=center, inner sep=0}, phantom, from=2-1, to=2-3, start anchor=center, end anchor=center]
				\arrow["{\actor \spancomp \combineunit \spancomp_\counitor \actor}"', from=2-1, to=3-2]
				\arrow["\combine", from=2-3, to=2-4]
				\arrow["{\actor \spancomp \combine}"', from=3-2, to=2-3]
				\arrow["{\rightunitlaw \spancomp_\rightunitlaw \actor}"', shorten >=3pt, Rightarrow, from=1-2, to=0p]
				\arrow["{\actor \spancomp \leftunitlaw}"', shorten <=3pt, Rightarrow, from=0p, to=3-2]
			\end{tikzcd}
		\end{equation}
		\begin{equation}
			\begin{tikzcd}[ampersand replacement=\&]
				{\actor^{\spancomp 4}} \&[2ex] {\actor^{\spancomp 3}} \&[2ex] \\
				\& {\actor^{\spancomp 3}} \& {\actor \spancomp \actor} \\
				{\actor^{\spancomp 3}} \\
				\& {\actor \spancomp \actor} \& \actor
				\arrow[""{name=0, anchor=center, inner sep=0}, "{\combine \spancomp_\coassociator \actor \spancomp \actor}", from=1-1, to=1-2]
				\arrow[""{name=1, anchor=center, inner sep=0}, "{\actor \spancomp \combine \spancomp_\coassociator \actor}"{description}, from=1-1, to=2-2]
				\arrow[""{name=1p, anchor=center, inner sep=0}, phantom, from=1-1, to=2-2, start anchor=center, end anchor=center]
				\arrow["{\actor \spancomp \actor \spancomp \combine}"{description}, from=1-1, to=3-1]
				\arrow["{\combine \spancomp_\coassociator \actor}"{description}, from=1-2, to=2-3]
				\arrow[""{name=2, anchor=center, inner sep=0}, "{\combine \spancomp_\coassociator \actor}"{description}, from=2-2, to=2-3]
				\arrow[""{name=2p, anchor=center, inner sep=0}, phantom, from=2-2, to=2-3, start anchor=center, end anchor=center]
				\arrow["{\actor \spancomp \combine}"{description}, from=2-2, to=4-2]
				\arrow["\combine"{description}, from=2-3, to=4-3]
				\arrow[""{name=3, anchor=center, inner sep=0}, "{\actor \spancomp \combine}"{description}, from=3-1, to=4-2]
				\arrow[""{name=3p, anchor=center, inner sep=0}, phantom, from=3-1, to=4-2, start anchor=center, end anchor=center]
				\arrow[""{name=4, anchor=center, inner sep=0}, "\combine"{description}, from=4-2, to=4-3]
				\arrow[""{name=4p, anchor=center, inner sep=0}, phantom, from=4-2, to=4-3, start anchor=center, end anchor=center]
				\arrow["{\associativitylaw \spancomp_\associativitylaw \actor}"{description}, shorten <=13pt, shorten >=13pt, Rightarrow, from=2, to=0]
				\arrow["{\actor \spancomp \associativitylaw}"{description}, shorten <=9pt, shorten >=9pt, Rightarrow, from=3p, to=1p]
				\arrow["\associativitylaw"{description}, shorten <=9pt, shorten >=9pt, Rightarrow, from=4p, to=2p]
			\end{tikzcd}
			\quad = \quad
			\begin{tikzcd}[ampersand replacement=\&]
				{\actor^{\spancomp 4}} \& {\actor^{\spancomp 3}} \\
				\&\& {\actor \spancomp \actor} \\
				{\actor^{\spancomp 3}} \& {\actor \spancomp \actor} \\
				\& {\actor \spancomp \actor} \& \actor
				\arrow[""{name=0, anchor=center, inner sep=0}, "{\combine \spancomp_\coassociator \actor \spancomp \actor}", from=1-1, to=1-2]
				\arrow[""{name=0p, anchor=center, inner sep=0}, phantom, from=1-1, to=1-2, start anchor=center, end anchor=center]
				\arrow["{\actor \spancomp \actor \spancomp \combine}"{description}, from=1-1, to=3-1]
				\arrow[""{name=1, anchor=center, inner sep=0}, "{\combine \spancomp_\coassociator \actor}"{description}, from=1-2, to=2-3]
				\arrow[""{name=1p, anchor=center, inner sep=0}, phantom, from=1-2, to=2-3, start anchor=center, end anchor=center]
				\arrow["{\actor \spancomp \combine}"{description}, from=1-2, to=3-2]
				\arrow["\combine"{description}, from=2-3, to=4-3]
				\arrow[""{name=2, anchor=center, inner sep=0}, "{\combine \spancomp_\coassociator \actor}"{description}, from=3-1, to=3-2]
				\arrow[""{name=2p, anchor=center, inner sep=0}, phantom, from=3-1, to=3-2, start anchor=center, end anchor=center]
				\arrow[""{name=2p, anchor=center, inner sep=0}, phantom, from=3-1, to=3-2, start anchor=center, end anchor=center]
				\arrow["{\actor \spancomp \combine}"{description}, from=3-1, to=4-2]
				\arrow[""{name=3, anchor=center, inner sep=0}, "\combine"{description}, from=3-2, to=4-3]
				\arrow[""{name=3p, anchor=center, inner sep=0}, phantom, from=3-2, to=4-3, start anchor=center, end anchor=center]
				\arrow[""{name=4, anchor=center, inner sep=0}, "\combine"{description}, from=4-2, to=4-3]
				\arrow[""{name=4p, anchor=center, inner sep=0}, phantom, from=4-2, to=4-3, start anchor=center, end anchor=center]
				\arrow["{\combine \spancomp \combinecart}"{description}, shorten <=9pt, shorten >=9pt, Rightarrow, from=2p, to=0p]
				\arrow["\associativitylaw"{description}, shorten <=9pt, shorten >=9pt, Rightarrow, from=3p, to=1p]
				\arrow["\associativitylaw"{description}, shorten <=14pt, shorten >=14pt, Rightarrow, from=4p, to=2p]
			\end{tikzcd}
		\end{equation}
		which on elements correspond to
		\begin{equation}
			\begin{tikzcd}[ampersand replacement=\&]
				{(P \combine \combineunit) \combine (A \action \rightunitlaw)^* Q} \& {P \combine Q} \\
				{(P \combine \combineunit) \combine \coassociator^* \counitor^* Q} \& {P \combine (\combineunit \combine \counitor^* Q)}
				\arrow["{\rightunitlaw \,\combine\, \lift_p (A \action \rightunitlaw)}", from=1-1, to=1-2]
				\arrow["{\id \combine \leftunitlaw}", from=1-2, to=2-2]
				\arrow["{\eqref{eqn:ctx.rightunitlaw.3cell.cond}}", Rightarrow, no head, from=2-1, to=1-1]
				\arrow["\associativitylaw", from=2-2, to=2-1]
			\end{tikzcd}
		\end{equation}
		\begin{equation}
			\begin{tikzcd}[ampersand replacement=\&]
				{P \combine (Q \combine (R \combine S))} \& {P \combine ((Q \combine R) \combine \coassociator^*S)} \\
				{(P \combine Q) \combine \coassociator^*(RS)} \& {(P \combine (Q \combine R)) \combine \coassociator^* \coassociator^* S} \\
				\& {(P \combine (Q \combine R)) \combine ((A \action \associativitylaw)^* \coassociator^* (\coassociator \action R)^* S)} \\
				{(P \combine Q) \combine ((\coassociator^*R) \combine ((\coassociator \action R)^*S))} \& {((P \combine Q) \combine (\coassociator^* R)) \combine (\coassociator^*(\coassociator \action R)^*S)}
				\arrow["{P \combine \associativitylaw}", from=1-1, to=1-2]
				\arrow["\associativitylaw"', from=1-1, to=2-1]
				\arrow["\associativitylaw", from=1-2, to=2-2]
				\arrow["{(P \combine Q) \combine \combinecart}"', from=2-1, to=4-1]
				\arrow["{\eqref{eqn:ctx.associativitylaw.3cell.cond}}", Rightarrow, no head, from=2-2, to=3-2]
				\arrow["{\associativitylaw \,\combine\, \lift_p (A \action \associativitylaw)}", from=3-2, to=4-2]
				\arrow["\associativitylaw"', from=4-1, to=4-2]
			\end{tikzcd}
		\end{equation}
	\end{enumerate}
\end{defn}

\begin{notation}
	Oftentimes we abbreviate the whole data of a contextad $\fibcolaxaction$ as either $\action$ or $\actor$ depending on what is more recognizable.
\end{notation}

With this definition in hand, we can define the double category $\Ctx(\action)$ of contextful arrows associated to a contextad.

\begin{defn}
\label{defn:ctx.dbl.cat}
	Given a contextad $\fibcolaxaction$ on a category $\acted$, its associated \textbf{double category of contextful arrows} $\Ctx(\action)$ is given as follows.
	\begin{enumerate}
		\item Its objects and tight 1-cells are those of $\acted$,
		\item Its loose 1-cells $A \looseto B$ are \textbf{contextful arrows}, thus pairs of a context $P \in \actor_A$ and a 1-cell $f:A \action P \to B$ in $\acted$.
		      Units in this direction are given by the pairs $(\combineunit, \counitor)$, and composition by (see also \eqref{eqn:para.comp}):
		      \begin{eqalign}
		      	\left(\lens{P}{A},\ A \action P \nto{f} B\right) \ &\lcomp\ \left(\lens{Q}{B},\ B \action Q \nto{g} C\right)\\
		      	&= \left(\lens{P \combine f^*Q}{A},\ A \action (P \combine f^*Q) \nto{\coassociator} (A \action P) \action f^*Q \nto{f \action Q} B \action Q \nto{g} C\right).
		      \end{eqalign}
		      This operation is associative and unital up to the invertible 2-cells defined below.
		\item Its squares are arrangements
		      \begin{equation}
			      \begin{tikzcd}[ampersand replacement=\&]
				      A \& B \\
				      {A'} \& {B'}
				      \arrow[""{name=0, anchor=center, inner sep=0}, "{(P,f)}", "\shortmid"{marking}, from=1-1, to=1-2]
				      \arrow[""{name=0p, anchor=center, inner sep=0}, phantom, from=1-1, to=1-2, start anchor=center, end anchor=center]
				      \arrow["h"', from=1-1, to=2-1]
				      \arrow["k", from=1-2, to=2-2]
				      \arrow[""{name=1, anchor=center, inner sep=0}, "{(P',f')}"', "\shortmid"{marking}, from=2-1, to=2-2]
				      \arrow[""{name=1p, anchor=center, inner sep=0}, phantom, from=2-1, to=2-2, start anchor=center, end anchor=center]
				      \arrow["{\varphi}"', shorten <=4pt, shorten >=4pt, Rightarrow, from=0p, to=1p]
			      \end{tikzcd}
		      \end{equation}
		      where $\varphi : P \to P'$ is a map in $\actor$ over $h$ making the following diagram commute:
		      \begin{equation}
			      \begin{tikzcd}[ampersand replacement=\&]
				      {A \action P} \& B \\
				      {A' \action P'} \& {B'}
				      \arrow[""{name=0, anchor=center, inner sep=0}, "f", from=1-1, to=1-2]
				      \arrow[""{name=0p, anchor=center, inner sep=0}, phantom, from=1-1, to=1-2, start anchor=center, end anchor=center]
				      \arrow["{h \action \varphi}"', from=1-1, to=2-1]
				      \arrow["k", from=1-2, to=2-2]
				      \arrow[""{name=1, anchor=center, inner sep=0}, "{f'}"', from=2-1, to=2-2]
				      \arrow[""{name=1p, anchor=center, inner sep=0}, phantom, from=2-1, to=2-2, start anchor=center, end anchor=center]
			      \end{tikzcd}
		      \end{equation}
		\item Tight composition of squares is given by pasting vertically:
		      \begin{equation}
			      \begin{tikzcd}[ampersand replacement=\&]
				      A \& B \\
				      {A'} \& {B'} \\
				      {A''} \& {B''}
				      \arrow[""{name=0, anchor=center, inner sep=0}, "{(P,f)}", "\shortmid"{marking}, from=1-1, to=1-2]
				      \arrow[""{name=0p, anchor=center, inner sep=0}, phantom, from=1-1, to=1-2, start anchor=center, end anchor=center]
				      \arrow["h"', from=1-1, to=2-1]
				      \arrow["k", from=1-2, to=2-2]
				      \arrow[""{name=1, anchor=center, inner sep=0}, "\shortmid"{marking}, from=2-1, to=2-2]
				      \arrow[""{name=1p, anchor=center, inner sep=0}, phantom, from=2-1, to=2-2, start anchor=center, end anchor=center]
				      \arrow[""{name=1p, anchor=center, inner sep=0}, phantom, from=2-1, to=2-2, start anchor=center, end anchor=center]
				      \arrow["{h'}"', from=2-1, to=3-1]
				      \arrow["{k'}", from=2-2, to=3-2]
				      \arrow[""{name=2, anchor=center, inner sep=0}, "{(P'',f'')}"', "\shortmid"{marking}, from=3-1, to=3-2]
				      \arrow[""{name=2p, anchor=center, inner sep=0}, phantom, from=3-1, to=3-2, start anchor=center, end anchor=center]
				      \arrow["{\varphi}"', shorten <=4pt, shorten >=4pt, Rightarrow, from=0p, to=1p]
				      \arrow["{\chi}"', shorten <=4pt, shorten >=4pt, Rightarrow, from=1p, to=2p]
			      \end{tikzcd}
			      \quad = \quad
			      \begin{tikzcd}[ampersand replacement=\&]
				      A \&\& B \\
				      {A''} \&\& {B''}
				      \arrow[""{name=0, anchor=center, inner sep=0}, "{(P,f)}", "\shortmid"{marking}, from=1-1, to=1-3]
				      \arrow[""{name=0p, anchor=center, inner sep=0}, phantom, from=1-1, to=1-3, start anchor=center, end anchor=center]
				      \arrow["{h'h}"', from=1-1, to=2-1]
				      \arrow["{k'k}", from=1-3, to=2-3]
				      \arrow[""{name=1, anchor=center, inner sep=0}, "{(P'',f'')}"', "\shortmid"{marking}, from=2-1, to=2-3]
				      \arrow[""{name=1p, anchor=center, inner sep=0}, phantom, from=2-1, to=2-3, start anchor=center, end anchor=center]
				      \arrow["{\chi\varphi}"', shorten <=4pt, shorten >=4pt, Rightarrow, from=0p, to=1p]
			      \end{tikzcd}
		      \end{equation}
		\item Loose composition of squares is given by parametric composition. Specifically:
		      \begin{equation}
			      \begin{tikzcd}[ampersand replacement=\&]
				      A \& B \& C \& A \&\& C \\
				      {A'} \& {B'} \& {C'} \& {A'} \&\& {C'}
				      \arrow[""{name=0, anchor=center, inner sep=0}, "{(P,f)}", "\shortmid"{marking}, from=1-1, to=1-2]
				      \arrow[""{name=0p, anchor=center, inner sep=0}, phantom, from=1-1, to=1-2, start anchor=center, end anchor=center]
				      \arrow["h"', from=1-1, to=2-1]
				      \arrow[""{name=1, anchor=center, inner sep=0}, "{(Q,g)}", "\shortmid"{marking}, from=1-2, to=1-3]
				      \arrow[""{name=1p, anchor=center, inner sep=0}, phantom, from=1-2, to=1-3, start anchor=center, end anchor=center]
				      \arrow["k"{description}, from=1-2, to=2-2]
				      \arrow[""{name=2, anchor=center, inner sep=0}, "\ell", from=1-3, to=2-3]
				      \arrow[""{name=3, anchor=center, inner sep=0}, "{(P,f) \lcomp (Q,g)}", "\shortmid"{marking}, from=1-4, to=1-6]
				      \arrow[""{name=3p, anchor=center, inner sep=0}, phantom, from=1-4, to=1-6, start anchor=center, end anchor=center]
				      \arrow[""{name=4, anchor=center, inner sep=0}, "h"', from=1-4, to=2-4]
				      \arrow["\ell", from=1-6, to=2-6]
				      \arrow[""{name=5, anchor=center, inner sep=0}, "{(P',f')}"', "\shortmid"{marking}, from=2-1, to=2-2]
				      \arrow[""{name=5p, anchor=center, inner sep=0}, phantom, from=2-1, to=2-2, start anchor=center, end anchor=center]
				      \arrow[""{name=6, anchor=center, inner sep=0}, "{(Q',g')}"', "\shortmid"{marking}, from=2-2, to=2-3]
				      \arrow[""{name=6p, anchor=center, inner sep=0}, phantom, from=2-2, to=2-3, start anchor=center, end anchor=center]
				      \arrow[""{name=7, anchor=center, inner sep=0}, "{(P',f') \lcomp (Q',g')}"', "\shortmid"{marking}, from=2-4, to=2-6]
				      \arrow[""{name=7p, anchor=center, inner sep=0}, phantom, from=2-4, to=2-6, start anchor=center, end anchor=center]
				      \arrow["{\varphi}"', shorten <=4pt, shorten >=4pt, Rightarrow, from=0p, to=5p]
				      \arrow["{\psi}"', shorten <=4pt, shorten >=4pt, Rightarrow, from=1p, to=6p]
				      \arrow["{=}"{marking, allow upside down}, draw=none, from=2, to=4]
				      \arrow["{\varphi \combine (f,f')^*\psi}"', shorten <=4pt, shorten >=4pt, Rightarrow, from=3p, to=7p]
			      \end{tikzcd}
		      \end{equation}
		      where the right hand side square denotes the following pasting:
		      \begin{equation}
			      \begin{tikzcd}[ampersand replacement=\&]
				      {A \action (P \combine f^* Q)} \& {(A \action P) \action f^* Q} \& {B \action Q} \& C \\
				      {A' \action (P' \combine {f'}^* Q')} \& {(A' \action P') \action {f'}^* Q'} \& {B' \action Q'} \& {C'}
				      \arrow["\coassociator", from=1-1, to=1-2]
				      \arrow["{h \action (\varphi \combine (f,f')^*\psi)}"', from=1-1, to=2-1]
				      \arrow["{g \action Q}", from=1-2, to=1-3]
				      \arrow["{(h \action \varphi) \action (f,f')^* \psi}"{description}, from=1-2, to=2-2]
				      \arrow["g", from=1-3, to=1-4]
				      \arrow["{k \action \psi}"{description}, from=1-3, to=2-3]
				      \arrow["\ell", from=1-4, to=2-4]
				      \arrow["\coassociator"', from=2-1, to=2-2]
				      \arrow["{g' \action Q'}"', from=2-2, to=2-3]
				      \arrow["{g'}"', from=2-3, to=2-4]
			      \end{tikzcd}
		      \end{equation}
		      Identity squares correspond to naturality squares for $\counitor$:
		      \begin{equation}
			      \begin{tikzcd}[ampersand replacement=\&]
				      A \& A \\
				      {A'} \& {A'}
				      \arrow[""{name=0, anchor=center, inner sep=0}, "{(\combineunit, \counitor)}", "\shortmid"{marking}, Rightarrow, no head, from=1-1, to=1-2]
				      \arrow[""{name=0p, anchor=center, inner sep=0}, phantom, from=1-1, to=1-2, start anchor=center, end anchor=center]
				      \arrow["h"', from=1-1, to=2-1]
				      \arrow["h", from=1-2, to=2-2]
				      \arrow[""{name=1, anchor=center, inner sep=0}, "{(\combineunit, \counitor)}"', "\shortmid"{marking}, Rightarrow, no head, from=2-1, to=2-2]
				      \arrow[""{name=1p, anchor=center, inner sep=0}, phantom, from=2-1, to=2-2, start anchor=center, end anchor=center]
				      \arrow["{\combineunit}"', shorten <=4pt, shorten >=4pt, Rightarrow, from=0p, to=1p]
			      \end{tikzcd}
		      \end{equation}
		\item Unitors for the composition of contextful arrows are directly inherited by those of the pseudomonad structure on $\fibcolaxaction$:
		      \begin{equation}
			      \begin{tikzcd}[ampersand replacement=\&]
				      A \& A \& B \\
				      A \&\& B
				      \arrow["{(\combineunit, \counitor)}", "\shortmid"{marking}, Rightarrow, no head, from=1-1, to=1-2]
				      \arrow[Rightarrow, no head, from=1-1, to=2-1]
				      \arrow["{(P,f)}", "\shortmid"{marking}, from=1-2, to=1-3]
				      \arrow[Rightarrow, no head, from=1-3, to=2-3]
				      \arrow[""{name=0, anchor=center, inner sep=0}, "{(P,f)}"', "\shortmid"{marking}, from=2-1, to=2-3]
				      \arrow[""{name=0p, anchor=center, inner sep=0}, phantom, from=2-1, to=2-3, start anchor=center, end anchor=center]
				      \arrow["{\leftunitlaw}"', "\wr", shorten >=3pt, Rightarrow, from=1-2, to=0p]
			      \end{tikzcd}
			      \qquad\qquad
			      \begin{tikzcd}[ampersand replacement=\&]
				      A \& B \& B \\
				      A \&\& B
				      \arrow["{(P,f)}", "\shortmid"{marking}, from=1-1, to=1-2]
				      \arrow[Rightarrow, no head, from=1-1, to=2-1]
				      \arrow["{(\combineunit, \counitor)}", "\shortmid"{marking}, Rightarrow, no head, from=1-2, to=1-3]
				      \arrow[Rightarrow, no head, from=1-3, to=2-3]
				      \arrow[""{name=0, anchor=center, inner sep=0}, "{(P,f)}"', "\shortmid"{marking}, from=2-1, to=2-3]
				      \arrow[""{name=0p, anchor=center, inner sep=0}, phantom, from=2-1, to=2-3, start anchor=center, end anchor=center]
				      \arrow["{\rightunitlaw\combineunitcart}"', "\wr", shorten >=3pt, Rightarrow, from=1-2, to=0p]
			      \end{tikzcd}
		      \end{equation}
		      The associator, being given by 2-cells such as
		      \begin{equation}
			      \begin{tikzcd}[ampersand replacement=\&]
				      A \&\& C \& D \\
				      A \& B \& {} \& D
				      \arrow["{(P,f) \then (Q,g)}", "\shortmid"{marking}, from=1-1, to=1-3]
				      \arrow[Rightarrow, no head, from=1-1, to=2-1]
				      \arrow["{(R,h)}", "\shortmid"{marking}, from=1-3, to=1-4]
				      \arrow["{\associativitylaw}"', "\wr", shorten <=1pt, shorten >=2pt, shift right=5, Rightarrow, from=1-3, to=2-3]
				      \arrow[Rightarrow, no head, from=1-4, to=2-4]
				      \arrow["{(P,f)}"', "\shortmid"{marking}, from=2-1, to=2-2]
				      \arrow["{(Q,g) \then (R,h)}"', "\shortmid"{marking}, from=2-2, to=2-4]
			      \end{tikzcd}
		      \end{equation}
		      is defined by the following pasting:
			  \begin{equation}
				\adjustbox{scale=.8,valign=M,center}{%
				\begin{sideways}
					\begin{tikzcd}[ampersand replacement=\&, row sep=10ex]
						{A \action ((P \combine f^* Q) \combine (g(f \action Q)\coassociator)^* R)} \& {(A \action (P \combine f^* Q)) \action (g(f \action Q)\coassociator)^* R} \&[3ex]\&[-1ex] {C \action R} \&[-2ex] D \\
						\& {((A \action P) \action f^* Q) \action (g(f \action Q))^* R} \& {(B \action Q) \action g^* R} \& {C \action R} \& D \\
						{A \action (P \combine (f^* Q \combine (g(f \action Q))^* R))} \& {(A \action P) \action (f^* Q \combine (g(f \action Q))^* R)} \\
						{A \action (P \combine f^*(Q \combine g^*R))} \& {(A \action P) \action f^*(Q \combine g^*R)} \& {B \action (Q \combine g^*R)} \&\& D
						\arrow["\coassociator", from=1-1, to=1-2]
						\arrow["{g(f \action Q)\coassociator \action R}", from=1-2, to=1-4]
						\arrow["h", from=1-4, to=1-5]
						\arrow[""{name=0, anchor=center, inner sep=0}, "{A \action \associativitylaw}"', from=1-1, to=3-1]
						\arrow[""{name=1, anchor=center, inner sep=0}, "{A \action (P \combine \combinecart)}"', from=3-1, to=4-1]
						\arrow["\coassociator"', from=4-1, to=4-2]
						\arrow["{f \action (Q \combine g^*R)}"', from=4-2, to=4-3]
						\arrow[""{name=2, anchor=center, inner sep=0}, "{(A \action P) \action \combinecart}", from=3-2, to=4-2]
						\arrow["\coassociator"{description}, from=3-1, to=3-2]
						\arrow["{\coassociator}"', from=3-2, to=2-2]
						\arrow["{\coassociator \action (g(f \action Q))^* R}", from=1-2, to=2-2]
						\arrow["{(f \action Q) \action R}", from=2-2, to=2-3]
						\arrow["{g \action R}", from=2-3, to=2-4]
						\arrow[Rightarrow, no head, from=1-4, to=2-4]
						\arrow[Rightarrow, no head, from=2-5, to=4-5]
						\arrow[Rightarrow, no head, from=1-5, to=2-5]
						\arrow["h", from=2-4, to=2-5]
						\arrow[""{name=3, anchor=center, inner sep=0}, "\coassociator"{description}, from=4-3, to=2-3]
						\arrow["{h(g \action R)\coassociator}"', from=4-3, to=4-5]
						\arrow["{\text{pentagon for $\associativitylaw$}}"{marking, allow upside down}, draw=none, from=0, to=2-2]
						\arrow["{\text{naturality of $\coassociator$}}"{marking, allow upside down}, draw=none, from=1, to=2]
						\arrow["{\begin{matrix}\text{coherence}\\\text{for $\combinecart$}\end{matrix}}"{marking, pos=0.45, allow upside down}, draw=none, from=3-2, to=3]
					\end{tikzcd}
				\end{sideways}}
			\end{equation}
	\end{enumerate}
\end{defn}

\begin{rmk}
	The fact composition of squares is well-defined, interchange is satisfied, and that $\associativitylaw$, $\rightunitlaw$ and $\leftunitlaw$ are coherent is a consequence of \cref{thm:ctx.const.on.objects}, which we prove through the high-brow construction of such double categories as wreath products.
	Still, one can check directly that composition of squares is well-defined by naturality of $\coassociator$ and functoriality of $\action$, strict interchange holds by functoriality of $\combine$ while coherence of horizontal composition is a consequence of coherence of the pseudomonad structure defining the contextad $\fibcolaxaction$.
\end{rmk}

\section{Examples of contextads}
\label{sec:examples.of.ctx}

The notion of contextad simulatenously generalizes that of a (colax) action of a monoidal category and that of a graded comonad. In both cases, the generalization concerns the relation of the acting category or the grades with the underlying category acted on; in a contextad, the objects which act are \emph{fibred} over the objected being acted on. This gives us a notion of \emph{fibred action} (generalizing an action of a monoidal category) or a \emph{dependently-graded comonad} (generalizing a graded comonad).
In this section, we will develop these two intuitive understandings of the notion of contextad through examples. In \cref{sec:ctx.as.colax.fibred.action}, we will consider contextads which may be thought of as fibred actions. In \cref{sec:ctx.as.dep.graded.comonad}, we will consider contextads which may be thought of as dependently-graded comonads, and relate the idea to notions of context-senstive computation.

Readers more interested in the further development of the theory may happily skip forward to \cref{sec:fibs.span}.

\subsection{Contextads as colax fibred actions}
\label{sec:ctx.as.colax.fibred.action}

In this section, we will see examples of contextads which generalize the monoidal actions by allowing the category of objects which can act on a given object $C \in \acted$ to vary with the object $C$. We could call these \emph{colax fibred actions} to emphasize that we are thinking of them as actions.

\begin{ex}
\label{ex:actegory}
	Every right action of a monoidal category $\actor$ on a category $\acted$ is a contextad $\acted \nepifrom{\pi_{\acted}} \acted \times \actor \nto{\action} \acted$ whose source fibration is a trivial fibration.
	Both the monoidal structure of $\actor$ and the action itself participate in defining the contextad structure;
	the former as $\combineunit$ and $\combine$ and the latter as the contextad action $\action$ as in \cref{eqn:colax.action.spans}.
	Specifically, the structure morphisms of the pseudomonad are inherited from the analogous morphisms in $\actor$, whereas the counitor $\counitor$ and coassociator $\coassociator$ are straight up taken from $\action$.
	Indeed, colax actions of monoidal categories are precisely contextads whose left leg is a trivial fibration (product projection) and whose structure morphisms $\combineunit : \acted \to \acted \times \actor$ and $(\acted \times \actor) \spancomp (\acted \times \actor) \to \acted \times \actor$ act only in the second component.

	When $\action$ comes from the right action of a monoidal category, the associated double category $\Ctx(\action)$ is precisely the Para construction we recalled in \cref{defn:para.actegory}.
	As anticipated there, compared to the classical Para construction appearing in \cite{capucci_towards_2022} or \cite{hermida_monoidal_2012}, $\Ctx(\action)$ is a double category instead of a bicategory.
	Furthermore, even focusing on the loose bicategory of $\Ctx(\action)$ we adopt 2-cells oriented in the opposite direction compared to \cite{capucci_towards_2022} but in agreement with \cite{hermida_monoidal_2012}.
\end{ex}

The following examples are all, in some sense, a generalization of the action $(C, M) \mapsto C \times M$ of a category with finite products on itself by cartesian product.

\begin{ex}
\label{ex:comp.cat}
	Let $p:\Ea \epito \acted$ be a fibration.
	It is said $p$ is a \emph{comprehension category with unit} when it admits two further right adjoints $p \adj \top \adj -.-$.
	These can be used to give semantics to dependent type theories, as in \cite{jacobs_comprehension_1993}: $p$ displays dependent types (objects of $\Ea$) over their contexts of definition (objects of $\acted$), $\top$ picks out the unit type over a given context, and $-.-$ extends a context with a new variable of the type it is applied to.
	Comprehension of $\lens{X}{\Gamma}$ is a new context $\Gamma . X$ together with a \emph{display map} $d_X : \Gamma . X \to \Gamma$ which drops the newly introduced variable.

	Furthermore, assume $p$ has \emph{strong dependent sums} \cite[Definition~9.3.5]{jacobs_categorical_1999} meaning (1) every such $d_X$ admits a cocartesian lift $\sum_X$, so that one can pushforward a dependent type $Y$ over $\Gamma . X$ to a dependent type $\sum_X Y$ over $\Gamma$ alone, (2) cartesian and cocartesian lifts along these maps satisfy Beck--Chevalley: given a diagram as below left (which is necessarily a pullback), the 2-cell below right (obtained by mating) is invertible, and (3) composition of display maps is again display.
	\begin{equation}
	\label{eqn:beck.chevalley}
		\begin{tikzcd}[ampersand replacement=\&,row sep=scriptsize]
			{\Gamma.f^*Y} \& {\Delta.Y} \\
			\Gamma \& \Delta
			\arrow["{{f.\lift f}}", from=1-1, to=1-2]
			\arrow["{{d_{f^*Y}}}"', from=1-1, to=2-1]
			\arrow["\lrcorner"{anchor=center, pos=0.125}, draw=none, from=1-1, to=2-2]
			\arrow["{{d_Y}}", from=1-2, to=2-2]
			\arrow["f"', from=2-1, to=2-2]
		\end{tikzcd}
		\qquad\qquad
		\begin{tikzcd}[ampersand replacement=\&,sep=scriptsize]
			{\Ea_{\Gamma.f^*Y}} \&\& {\Ea_{\Delta.Y}} \\
			{\Ea_\Gamma} \&\& {\Ea_\Delta}
			\arrow["{\sum_{f^*Y}}"', from=1-1, to=2-1]
			\arrow[shorten <=9pt, shorten >=9pt, Rightarrow, from=1-1, to=2-3]
			\arrow["{(f.\lift f)^*}"', from=1-3, to=1-1]
			\arrow["{\sum_Y}", from=1-3, to=2-3]
			\arrow["{f^*}", from=2-3, to=2-1]
		\end{tikzcd}
	\end{equation}

	One can build a fibred action $\acted \nepifrom{p} \Ea \nto{-.-} \acted$, where $\combineunit$ is given by $\top$ and $\combine$ is given by sums:
	\begin{equation}
	\label{eqn:comb.comp}
		\lens{X}{\Gamma} \combine \lens{Y}{\Gamma . X} := \lens{\sum_X Y}{\Gamma}.
	\end{equation}
	The fact this defines a cartesian functor $\Ea \spancomp \Ea \to \Ea$ corresponds precisely to the Beck--Chevalley condition spelled out above.
	Then $\counitor$ is the natural isomorphism $\Gamma . \top \isoto \Gamma$, while $\coassociator$ is the canonical isomorphism $\Gamma . \sum_X Y \isoto (\Gamma . X) . Y$ we get by assuming composition of display maps is display.
	Similarly, we get $\associativitylaw:\sum_X \sum_Y Z \isoto \sum_{\sum_X Y} \coassociator^* Z$ and $\leftunitlaw : \sum_\top \counitor^* Y \isoto Y$, $\rightunitlaw : \sum_X \top \isoto X$.

	A concrete class of examples of comprehension category are fibrations of subobjects, such as $\Set^\subseteq \to \Set$ \cite[p.43]{jacobs_categorical_1999}.
	Specifically, units are given by the isomorphism class of the identity map, and comprehension of a subobject $P \subseteq A$ is defined by sending every representative mono to its image (which is determined functorially), which we denote as $\{P\}$, and its inclusion $\{P\} \monoto A$.
	Cocartesian lift is then given by composing representative monos with such an inclusion.
	Thus the fibred action associated to a subobject fibration has, as its combination operation
	$\lens{P}{A} \combine \lens{Q}{\{P\}} = \lens{P \cap Q}{A}$,
	where we are denoting by $Q$ also the composite $Q \subseteq \{P\} \subseteq A$.
	Notice that since $Q \subseteq \{P\}$, $P\cap Q=Q$, which corresponds to a direct translation of \eqref{eqn:comb.comp}.
\end{ex}

\begin{ex}
\label{ex:display.map.cat}
	An important special instance of the above example is when $p$ is the fibration of display maps of a \emph{strong display map category}.
	Recall that such a category $\acted$ is equipped with a replete\footnote{i.e.~all isomorphisms are in $\Da$.} subcategory $\Da$ of `display maps' which is a pullback ideal, meaning maps in $\Da$ admit pullback along arbitrary maps in $\acted$ and are stable under this operation.
	This property guarantees that the subcategory $\acted^{\downarrow_{\Da}}$ of the arrow category of $\acted$ spanned by those arrows which belong to $\Da$ is fibred over $\acted$ by codomain projection.
	The cartesian maps in $\cod : \acted^{\downarrow_{\Display}} \to \acted$ are the pullback squares:
	\begin{equation}
		\begin{tikzcd}[ampersand replacement=\&,sep=scriptsize]
			\& {\acted^{\downarrow_{\Display}}} \\
			\acted \&\& \acted
			\arrow["\cod"', from=1-2, to=2-1]
			\arrow["\dom", from=1-2, to=2-3]
		\end{tikzcd}
	\end{equation}

	Such a fibration is a full-blown comprehension category with unit and strong sums.
	The unit $\top$ sends an object to its identity map, while comprehension is given by projecting the domain of display maps.
	In fact, every display map $d:E \to B$ in $\acted$ `displays itself' as a natural arrow from its comprehension $E$ to its projection $B$.
	Cocartesian lifts along such maps are given by composition, and Beck--Chevalley coincides with the pasting property of pullback squares.

	Not all comprehension categories arise from strong display maps categories, for instance the subobject fibration $\Ca^\subseteq \to \Ca$ is merely equivalent to the fibration of monomorphisms, which is a fibration of strong display maps.
	This situation, however, is fully general: since the maps $d_X : \Gamma .X \to \Gamma$ of a strong comprehension category are closed under composition and pullback, they form (after closing with isomorphisms) a strong display map category structure on the base category whose associated fibration is equivalent to the starting one.

	Note that the class of (maps isomorphic to) product projections forms a strong display map class on any category $\acted$ with finite products; whose associated fibration is/is isomorphic to the \emph{simple fibration} ${\simple : \Simple{\acted} \to \acted}$ \cite[Definition~1.3.1]{jacobs_categorical_1999}.
	The resulting contextad $\acted \nepifrom{\simple} \Simple{\acted} \nto{\times} \acted$ is, on objects, just like the contextad associated to the action of $\acted$ on itself by cartesian product (i.e.~a trivial colax fibred action as per \cref{ex:actegory}), but since maps in $\Simple{\acted}$ are not just pairs of maps, it differs on those:
	\begin{eqalign}
	\label{eq:simp.fib}
		\begin{tikzcd}[ampersand replacement=\&,column sep=scriptsize,row sep=small]
			{\Simple{\acted}} \& \acted \\[-2ex]
			\begin{array}{c} \begin{matrix} 				h^\flat : A \times A' \to B'\\ 				h : A \to B 			\end{matrix} \end{array} \& {(h\pi_A, h^\flat) : A \times A' \to B \times B'}
			\arrow["\times", from=1-1, to=1-2]
			\arrow[shorten <=0pt, shorten >=6pt, maps to, from=2-1, to=2-2]
		\end{tikzcd}
	\end{eqalign}
\end{ex}

\begin{ex}
\label{ex:comp.cat.para}
	Applying $\Ctx$ to \cref{ex:comp.cat} we obtain double category of contextful arrows $(X,f):\Gamma \looseto \Delta$ where the context extension $X$, or `type over $\Gamma$', lives over the domain $\Gamma$, and where $f:\Gamma.X \to \Delta$.

	In the particular instance of subobject fibrations $\acted^\subseteq \to \acted$, we can identity such contextful arrows with \emph{partial maps} in $\acted$ since an arrow $A \looseto B$ is exactly given by the data of a domain of definition $P \subseteq A$ and a map $f:\{P\} \to B$.
	Composition $(P,f) \lcomp (Q,g)$ is given by pulling back $Q$ along $f$ to get a common domain of definiton $P \cap f^*Q = f^*Q$ on which the composite of $f$ and $g$ can be defined.
	This law is precisely what composition of contextful arrows does:
	\begin{equation}
		\{P \cap f^*Q\} \nequalto{\delta} \{f^*Q\} \nto{\{\lift f\}} \{Q\} \nto{g} C.
	\end{equation}

	When $p:\actor \fibto \acted$ is a fibration of display maps (\cref{ex:display.map.cat}), the associated notion of contextful arrows is actually left-displayed spans:
	\begin{equation}
		\begin{tikzcd}[ampersand replacement=\&,sep=small]
			\& {\Gamma.X} \\
			\Gamma \&\& \Delta
			\arrow["{d_X}"', from=1-2, to=2-1]
			\arrow["f", from=1-2, to=2-3]
		\end{tikzcd}
	\end{equation}
	Their composition is given by pull-push, and it does indeed coincide with the composition law of contextful arrows since $f^*\Delta.Y \iso \Gamma .X . f^*Y$ as maps over $\Gamma$ (since squares like \eqref{eqn:beck.chevalley} right are always pullbacks):
	\begin{equation}
		\begin{tikzcd}[ampersand replacement=\&,sep=small]
			\&\& {f^*\Delta.Y} \\
			\& {\Gamma.X} \&\& {\Delta.Y} \\
			\Gamma \&\& \Delta \&\& \Theta
			\arrow["{\pi_X}"', from=1-3, to=2-2]
			\arrow["{\pi_Y \equiv f.\lift f}", from=1-3, to=2-4]
			\arrow["\lrcorner"{anchor=center, pos=0.125, rotate=-45}, draw=none, from=1-3, to=3-3]
			\arrow["{d_X}"', from=2-2, to=3-1]
			\arrow["f", from=2-2, to=3-3]
			\arrow["{d_Y}"', from=2-4, to=3-3]
			\arrow["g", from=2-4, to=3-5]
		\end{tikzcd}
		\quad
		\Gamma . {\textstyle\sum_X} f^*Y \nto{\coassociator} \Gamma . X. f^*Y \nto{f . \lift f} \Delta . Y \nto{g} \Theta.
	\end{equation}
	Thus $\Ctx(\acted \nepifrom{\cod} \acted^{\downarrow_{\Da}} \nto{\dom} \acted) \iso {\Span(\acted, \Da)}$, where the latter is the double category of left-displayed spans (similar to \cref{defn:locally.span} but one dimension down):
	\begin{equation}
		\begin{tikzcd}[ampersand replacement=\&, sep=scriptsize]
			\Gamma \& {\Gamma.X} \& \Delta \\[2ex]
			\Theta \& {\Theta.Y} \& \Xi
			\arrow["i"', from=1-1, to=2-1]
			\arrow["{d_X}"', from=1-2, to=1-1]
			\arrow["f", from=1-2, to=1-3]
			\arrow["{i.\alpha}", from=1-2, to=2-2]
			\arrow["j", from=1-3, to=2-3]
			\arrow["{d_Y}", from=2-2, to=2-1]
			\arrow["g"', from=2-2, to=2-3]
		\end{tikzcd}
		\quad\leftrightsquigarrow\quad
		\begin{tikzcd}[ampersand replacement=\&, sep=scriptsize]
			\Gamma \&[2ex] \Delta \\[2ex]
			\Theta \& \Xi
			\arrow[""{name=0, anchor=center, inner sep=0}, "{(X,f)}", "\shortmid"{marking}, from=1-1, to=1-2]
			\arrow["i"', from=1-1, to=2-1]
			\arrow["j", from=1-2, to=2-2]
			\arrow[""{name=1, anchor=center, inner sep=0}, "{(Y,g)}"', "\shortmid"{marking}, from=2-1, to=2-2]
			\arrow["\alpha", shorten <=4pt, shorten >=4pt, Rightarrow, from=0, to=1]
		\end{tikzcd}
	\end{equation}
\end{ex}

\begin{ex}
\label{ex:fib.moncat}
	Consider a fibration $p:\actor \fibto \acted$.
	If $\acted \nepifrom{p} \actor \nto{p} \acted$ is a strong fibred action, then $p$ has split fibred monoidal products (in the terminology of \cite[Definition~1.8.1]{jacobs_categorical_1999}).
	In fact $\combineunit$ and $\combine$ for such a fibred action correspond to fiberwise monoidal operations, and since they are cartesian functors it means that reindexing preserves strictly the monoidal structure.
	Examples of these structures are ubiquitous.
	Amongst them: the fibrations obtained by indexed monoids of various kinds, the codomain fibration of a finitely complete category (which is fiberwise cartesian), the fibration of families $\Fam(\Ca) \to \Set$ for a monoidal category $\Ca$ (with $\lens{(X_i)_{i \in I}}{I}\combine_I \lens{(Y_i)_{i \in I}}{I} = \lens{(X_i \combine Y_i)_{i \in I}}{I}$).

	Then let $\acted \nepifrom{p} \actor \nepito{p} \acted$ be the fibred action arising from a fiberwise monoidal fibration as above.
	Contextful arrows for such an action are maps $f:A \to B$ \emph{decorated by} an object $P \in p^{-1}(A)$.
	Composition combines decorations, and identities are decorated by the identity decoration:
	\begin{equation}
		1_A \equiv (\combineunit \in p^{-1}(A)), \qquad (P,f) \lcomp (Q,g) \equiv (P \combine f^*Q \in p^{-1}(A),\ gf).
	\end{equation}
	Hence this construction is useful to keep track of data attached to (domains of) maps.

	For instance, here's a proof-of-concept using this construction to decorate computations, seen as morphisms in a category, with their running time.
	For $\Set$, consider the discrete fibration $\Fam(\N) \to \Set$, where $\N = (\N, \leq, 0, +)$ is considered a monoidal poset with addition.
	The double category $\Ctx(\Fam(\N))$ has as loose arrows functions $f:I \to J$ decorated by a function $s:I \to \N$, which we can use to encode a notion of `running time' for each input $i \in I$.
	Then if $(t,g):J \to K$ is another such map, $(s,f) \lcomp (t,g)$ is decorated by $\lambda (i \in I).s(i) + g(f(i))$, which is indeed the function describing the running time of $g$ after $f$.
\end{ex}

\subsection{Contextads as dependently graded comonads}
\label{sec:ctx.as.dep.graded.comonad}

Since the seminal work of Moggi \cite{moggi_notions_1991} and Kieburtz \cite{kieburtz_codata_1999}, it has become a sort of truism that Kleisli morphisms for monads can be used to type programs with side-effects, and Kleisli morphisms for comonads can be used to type programs which depend on
their evaluation context.
This leads to a sense that ``contexts'' and ``effects'' are dual categorical notions.%
\footnote{
	This point of view appears to begin with Uustalu and Vene \cite{uustalu_comonadic_2008} who give a deep investigation into the use of comonads for handling contextful computations and dataflow programming.
	However, they frame contextful computations as dual to effectful computations (which they treat monadically).
	In this section, we posit, following Kieburtz, that most effectful programming can be seen as a subset of contextful programming---namely, as those contextful programs which change their context.
}
But these slogans do not accurately reflect the claims of Moggi and Kieburtz in their respective seminal works, and the notions of context and effect are not really dual at all.
Side-effects are effects that a program has on its evaluation context (the ``side'' in ``side-effect'').

Moggi uses a monad $T$ as a ``notion of computation'', and gives side effects as one example of a notion of computation (specifically the state monad $X \mapsto (X \times S)^S$ in Example~1 of \cite{moggi_notions_1991}).
This way of thinking about monadic values $TA$ as computations which suffice to produce an $A$ underlies the way we model input-output with an $\MonadIO$ monad: a Kleisli map $f : A \to \MonadIO\, B$ is interpreted as program which produces, on input $a : A$, an IO-computation $f(a)$ which suffices to produce an element of $B$.
In other words, Moggi's monadic view sees effectful programs as programs which produce computations that, when evaluated, would have side-effects in addition to producing their return value.
This point of view is not quite the same as the common way of thinking which styles functions $f : A \to \MonadIO\, B$ as a function which produces an element of $B$ from an element of $A$ \emph{together with some side-effect} that we aren't allowed to simply get rid of.

In \cite{kieburtz_codata_1999}, Kieburtz points out that side-effects in programming often depend on the context in which the program evaluation occurs.
Kieburtz suggests typing IO operations instead using an IO \emph{comonad} $\ComonadOI$ as a Kleisli map $\ComonadOI\, A \to B$.
This has the natural interpretation of a program which takes in a value of $A$ occuring within an input-output context and produces a pure value of $B$.
In the course of its being used to produce the value of $B$, the context may be affected, and these are the side effects of the program.%
\footnote{In \cite{GKO:effects.grading}, the authors make the distinction between \emph{producer} effects which change the execution context and should be handled monadically and \emph{consumer} effects which make demands on the context and should be handled comonadically.}

Kieburtz does not argue that all notions of side effect may be captured by comonads; indeed, except for the reader monad $X \mapsto X^R$ whose Kleisli category is isomorphic to the Kleisli category of the comonad $X \mapsto X \times R$, most of Moggi's monads from \cite[Example~1]{moggi_notions_1991} do not have equivalent presentations using comonads.
But the Kleisli categories of these monads \emph{are} isomorphic to the categories $\Ctx(\action)$ of contextful morphisms produced from a contextad.

In this section, we will interpret contextads as \emph{dependently graded comonads} and their $\Ctx$ construction as their Kleisli category. We will see that all of Moggi's example monads (excepting the continuation monad $X \mapsto R^{(R^X)}$) have Kleisli categories which can equivalently be described as $\Ctx(\action)$ for judiciously chosen contextads $\action$.

We will begin by unfolding \cref{defn:colax.fibred.action} in terms of indexed categories to more closely resemble the usual notion of graded comonad (see, e.g. \cite{GKO:effects.grading}).
We will then show that that any monad $T$ whose underlying functor is polynomial (i.e. a \emph{container} \cite{AAG:containers}) has a Kleisli category equivalent to the category of contextful arrows associated to a contextad defined in terms of the functor $T$.
Simply put, we may transpose a Kleisli morphism for a polynomial monad as on the left below into a pair of a context map $c$ and a ``coKleisli'' map $f$ contextualized by $c$ as on the right below:
\begin{equation}
	\bigg(X \to \dsum{s : S} Y^{P(s)} \bigg) \quad\simeq\quad \begin{cases*}
		c : X \to S \\
		f : \dsum{x : X} P(c(x)) \to Y
	\end{cases*}
\end{equation}
This includes all of the above examples and any monad whose underlying action on types is given by a \emph{strictly positive type} in one variable \cite[Definition~2.8]{AAG:containers}, which includes many algebraic effect handlers (presented using free monads on strictly positive functors as in \cite{Wu-Schrijvers:algebraic.effect.handlers}, which are polynomial by \cite[Theorem~4.5]{gambino.kock:polynomial.monads}).
We expound this idea in \cref{sec:poly.monad.dep.graded.comonad} but do not prove a theorem yet: we leave that for future work.

\paragraph{Dependently graded comonads}
It is worth noting first that every (graded) comonad is already a contextad.

\begin{ex}
\label{ex:comonad}
\label{ex:graded.comonad}
	Every comonad $\Comonad$ on a category $\acted$ corresponds to a contextad $\acted \equalto \acted \nto{\Comonad} \acted$, with $\counitor$ and $\coassociator$ given by the counit and comultiplication of $\Comonad$.
	More generally, every \emph{$\actor$-graded comonad} on a category $\acted$ corresponds to a trivially fibred contextad (that is, a colax action) $\acted \nepifrom{\pi_{\acted}} \acted \times \actor \nto{\Comonad_{(-)}} \acted$.

	In fact, even \emph{indexed comonads} \cite[Definition~2.12]{fujii_2-categorical_2019} are contextads: it suffices to note that an indexed comonad $\Comonad : \acted \to \Comnd(\actor, \actor)$ is equivalently a comonad on the trivial fibration $\pi_1:\acted \times \actor \to \acted$, thus captured by the contextad corresponding to $\Comonad:\acted \times \actor \to \acted \times \actor$ or, more faithfully, by hosting the contextad in $\Fib(\acted)$.
	Indeed, it shows that by replacing $\pi_1$ with a non-trivial fibration we can also talk about \emph{fibred comonads} in which, as $C \in \acted$ varies, so does the category over which $\Comonad_C$ is defined.

	When $\action$ is a comonad $\Comonad$ on a category $\acted$, then $\Ctx(\action)$ is the strict double category with tight maps given by morphisms of $\acted$, loose 1-cells by $\Comonad$-Kleisli maps and squares are just commutative squares in $\Kl(\Comonad)$.
	Similarly, when the contextad is a graded comonad $\Comonad_{(-)}$ on a category $\acted$, $\Ctx(\Comonad_{(-)})$ is a double category with tight maps from $\acted$, loose arrows $(P,f):A \looseto B$ given by maps $f:\Comonad_P A \to B$ in $\acted$, and squares filled by morphisms between grades.
	Overall this double category is completely analogous to the ones of \cref{ex:actegory} since the only difference between a graded comonad and an action of a monoidal category is in the invertibility of $\counitor$ and $\coassociator$, which doesn't change how $\Ctx(\Comonad_{(-)})$ looks.

	Finally, fibred comonads on $q:\Ea \to \acted$ give a $\Ctx$ construction whose objects are those of $\Ea$--thus equivalently pairs $\lens{E \in \Ea_C}{C \in \acted}$, maps are pairs $\lens{h^\flat}{h}:\lens{E}{C} \to \lens{E'}{C'}$ with $h:C \to C'$ in $\acted$ and $h^\flat: \Comonad_C(E) \to E'$.
	If we perform this construction in $\Fib(\acted)$, moreover, $\Ctx(\Comonad)$ is automatically fibred over $\acted$ again, by projecting the bottom part.
	For instance, if $\acted$ is cartesian monoidal, consider the indexed comonad $\acted \to \Comnd(\acted,\acted)$ given by $P \mapsto - \times P$; then its $\Ctx$ construction recovers the simple fibration \cref{eq:simp.fib} as different kind of $\Ctx$ construction.
\end{ex}

\begin{rmk}
\label{rmk:fujii.kleisli}
	We note our $\Ctx(\Comonad_{(-)})$ differs from the Fujii--Katsumata--Melli\`es Kleisli construction (originally from \cite[Definition~6]{FKM:graded.monads} for graded \emph{monads}, and spelled out in \cite[Definition~4.65]{fujii_2-categorical_2019} for graded comonads too).
	First, ours is a double category while theirs is a category; but even just considering the (bi)category of loose arrows, there are more concrete differences.
	Our $\Ctx(\Comonad_{(-)})$ has the same objects as $\acted$ and loose arrows more closely resembling the usual Kleisli arrows of a comonad;
	Fujii \emph{et al.}'s construction has objects decorated by a grade, while arrows are certain ``optics'' \cite{clarke_profunctor_2024} defined from the colax action.
	Nonetheless, both are Kleisli construction in the formal sense (ours by \cref{thm:ctx.const.on.objects} and Fujii's by \cite[Theorem~4.67]{fujii_2-categorical_2019}): the difference is explained by the difference in ambient.

	It's interesting to note that, as stated in \cite[§4]{fujii_2-categorical_2019}, Kleisli constructions for \emph{indexed} comonads are not available in the Fujii--Katsumata--Melli\`es sense, whereas as contextads, their Kleisli construction is easily given by $\Ctx$.
\end{rmk}

Let's now restate \cref{defn:colax.fibred.action} as a \emph{dependently graded comonad}, with an indexed category $\actor : \acted\op \to \Cat$ of grades taking the place of the domain fibration of $\fibcolaxaction$.

\begin{defn}
\label{defn:dep.graded.comonad}
	A \textbf{dependently graded comonad} on a category $\acted$ consists of:
	\begin{enumerate}
		\item An indexed category $\actor_{(-)} : \acted\op \to \Cat$ sending each object $C \in \acted$ to a category $\actor_C$ of \textbf{grades} which may act on $C$. Denote the compositor of this indexed category by $\gamma : g^*f^* \iso (gf)^*$.
		\item For each $C \in \acted$ and $M \in \actor_{C}$, an object $C \action M \in \acted$.
		This assignment must be functorial in both $C$ and $M$ in the sense that it should give a functor $\action : \actor := \int_{C \in \acted} \actor_C \to \acted$ on the Grothendieck construction $\actor$ of $\actor_{(-)}$.
		Explicitly, for $f : C \to C'$ and $g : M \to f^*M'$, we should have $f \action g : M \odot C \to M' \odot C'$.
		\item For each $C \in \acted$, a unit $\combineunit_C \in \actor_C$.
		This should be functorial in $C$ as a section of the projection $\actor \to \acted$ from the Grothendieck construction of $\actor_{(-)}$, and should moreover be a cartesian functor; explicitly, for $f : C \to C'$ we should have an isomorphism $\combineunitcart : \combineunit_C \iso f^*\combineunit_{C'}$.
		\item For every $N \in \actor_{C \action M}$, a combination $M \combine N \in \actor_C$, functorial in all three variables in the following sense: for $f : C \to C'$, $g : M \to f^*M'$, and $h : N \to (f \action g)^* N'$, we should have $g \combine_f h : M \combine N \to f^*(M' \combine N')$. Furthermore, $\combine$ should be a cartesian functor: if $g$ and $h$ are isomorphisms, then $g \combine_f h$ should be an isomorphism. In particular, by taking $g$ and $h$ to be identities, we must have an isomorphism $\combinecart : f^*M' \combine (f \action f^* M')^*N' \iso f^*(M' \combine N')$.
		\item A colax unitor $\counitor_C : \combineunit_C \action C \to C$, natural in $C$.
		\item For every $N \in \actor_{C \action M}$, a colax associator $\coassociator_{C,M,N} : C \action (M \combine N) \to (C \action M) \action N$, natural in all three variables.
		\item For every $M \in \actor_C$, natural unit isomorphisms $\leftunitlaw : \combineunit_C \combine \counitor^* M \iso M$ and $\rightunitlaw : M \combine \combineunit_{C \action M} \iso M$.
		\item For every $M \in \actor_C$, $N \in \actor_{C \action M}$ and $p \in \actor_{(C \action M) \action N}$, an associativity isomorphism $\associativitylaw : M \combine (N \combine p) \iso (M \combine N) \combine \coassociator^* p$.
		\item For every $M \in \actor_C$, the following diagrams commute:
		      \begin{equation}
		\label{eqn:dep.graded.left.counit.law}
			      \begin{tikzcd}
				      {C \action M} & {C \action (\combineunit_C \combine \counitor^*M)} \\
				      & {(C \action \combineunit_C) \action \counitor^* M} \\
				      & {C \action M}
				      \arrow["{C \action \leftunitlaw}", from=1-1, to=1-2]
				      \arrow[curve={height=24pt}, Rightarrow, no head, from=1-1, to=3-2]
				      \arrow["\coassociator", from=1-2, to=2-2]
				      \arrow["{\counitor \action M}", from=2-2, to=3-2]
			      \end{tikzcd}
		      \end{equation}
		      \begin{equation}
		\label{eqn:dep.graded.right.counit.law}
			      \begin{tikzcd}
				      {C \action (M \combine \combineunit_{C \action M})} \\
				      {(C \action M) \action \combineunit_{C \action M}} & {C \action M}
				      \arrow["\coassociator"', from=1-1, to=2-1]
				      \arrow["{C \action \rightunitlaw}", from=1-1, to=2-2]
				      \arrow["\counitor"', from=2-1, to=2-2]
			      \end{tikzcd}
		      \end{equation}

		\item For every $M \in \actor_C$, $N \in \actor_{C \action M}$ and $L \in \actor_{(C \action M) \action L}$, the following diagram commutes:
		      \begin{equation}
		\label{eqn:dep.graded.coassociativity.law}
			      \begin{tikzcd}
				      {C \action (M \combine (N \combine L))} & {(C \action M) \action (N \combine L)} \\
				      {C \action ((M \combine N) \combine \coassociator_{M,N}^*L)} \\
				      {(C \action (M \combine N) \action \coassociator_{M,N}^*L)} & {((C \action M) \action N) \action L}
				      \arrow["{\coassociator_{M,(N \combine L)}}", from=1-1, to=1-2]
				      \arrow["{C \action \associativitylaw}"', from=1-1, to=2-1]
				      \arrow["{\coassociator_{N,L}}", from=1-2, to=3-2]
				      \arrow["{\coassociator_{(M \combine N),\coassociator^*L}}"', from=2-1, to=3-1]
				      \arrow["{\coassociator_{M,N} \action L}"', from=3-1, to=3-2]
			      \end{tikzcd}
		      \end{equation}
	\end{enumerate}
\end{defn}

\begin{rmk}
	We could express a dependently graded comonad in a way that more resembles the ``extract-extend'' or Kleisli lift formulation. The counit remains $\counitor : C \action \combineunit_C \to C$, but we can repackage $\coassociator$ into a Kleisli lift operation taking a parameterized map $f : C \action M \to C'$ and a grade $N \in \actor_{C'}$ and yielding the extended map $(f \action N) \circ \coassociator_{M, f^*N} : C \action (M \combine f^*N) \to C' \action N$. This is evidently used in the $\Ctx$ construction. We will not work out what the laws look like in this form here, though we expect it to be a fairly straightforward translation.
\end{rmk}

The grades of a dependently graded monad will very often form a \emph{set} or more generally a \emph{gaunt} category---a category in which every isomorphism is an identity. This is the case when the grades form a set, a partially ordered set, or the single object category delooping a free monoid, to give a few examples. This will be the case for all the dependently graded monads arising as the transposes of polynomial monads. In \cref{thm:gaunt.ctx.strict.dbl}, we will prove that when the source fibration of a contextad is gaunt, the resulting $\Ctx$ construction is a \emph{strict} double category. In particular, it has a \emph{category} (and not just a bicategory) of loose arrows.

\begin{defn}
	A map $f : \Ca \to \Da$ in a 2-category is \textbf{gaunt}\footnotemark~if for all \textbf{invertible} 2-cells $\alpha : a \twoto b$ (for $a, b : \Xa \to \Ca$) such that $f\alpha$ is an identity, $\alpha$ is an identity.
	\footnotetext{Also known as \emph{amnestic} \cite[Definition~3.27]{adamekAbstractConcreteCategories2009} when $\Cosmos = \Cat$.}
	A map $f : \Ca \to \Da$ in a 2-category is \textbf{discrete} if for all 2-cells $\alpha : a \twoto b$ for $a, b: \Xa \to \Ca$, if $f \alpha$ is an identity then $\alpha$ is an identity.
\end{defn}

\begin{rmk}
	A category $\Ca$ is gaunt---every isomorphism in $\Ca$ is an identity---if and only if the terminal functor $! : \Ca \to \One$ is gaunt. A functor $f : \Ca \to \Da$ is gaunt precisely when its strict fibers $f^{-1}(d)$ (considing of all objects which $f$ maps to $d$ and all morphisms which $f$ maps to the identity of $d$) are all gaunt categories.
	Similarly, category $\Ca$ is discrete---all morphisms are identities---if and only the terminal functor is discrete and a functor is discrete when its strict fibers are discrete.
\end{rmk}

\begin{rmk}
	Gaunt categories include all discrete categories (sets), partially ordered sets, and deloopings of free monoids.
\end{rmk}

\begin{lem}
	Gaunt and discrete morphisms in a 2-category are closed under composition and preserved under (strict) pullback.
\end{lem}

\begin{thm}
\label{thm:gaunt.ctx.strict.dbl}
	Let $\fibcolaxaction$ be a contextad with $p$ a gaunt morphism in a paradise $\Cosmos$.
	Then $\Ctx(\action)$ is a \emph{strict} category object in $\Cosmos$.
\end{thm}
\begin{proof}
	The 3-cells $\leftunitlaw$, $\rightunitlaw$, $\associativitylaw$, and the witnesses $\combineunitcart$ and $\combinecart$ to cartesianness of $\combineunit$ and $\combine$ in $\FibSpan^{\Rightarrow}$ that appear in the structure of the contextad $\fibcolaxaction$ are all isomorphisms which, when whiskered with $p$, become identities (as ensured by the left half of \cref{eqn:transformation.of.right.colax.maps.of.spans}).
	If $p$ is gaunt, they must then be identities.
	Tracing through \cref{defn:colax.fibred.action}, this implies that $\Ctx(\action)$ is a strict double category.
\end{proof}

We make the following definition to summarize the property we aim to show of parametric right adjoint monads.

\begin{defn}
	A monad $T : \acted \to \acted$ is \textbf{transposable} if there is a dependently and discretely graded comonad $\action_T$ on $\acted$ together with an identity on objects equivalence between $\Kl(T)$ and the category of loose arrows of $\Ctx(\action_T)$.
\end{defn}

\subsection{Polynomial monads as dependently graded comonads}
\label{sec:poly.monad.dep.graded.comonad}
In this section, we will sketch the contextad arising as the transpose of a monad whose underlying functor is \emph{polynomial}---of the form $X \mapsto \dsum{s : S} X^{P(s)}$ for a type family $P : S \to \Type$.%
\footnote{Where $\dsum{s : S} X^{P(s)}$ denotes dependent sum over $s \in S$.}
In \cref{sec:polynomial.monad.examples} we unpack some specific examples: the writer, state, and maybe monads.

In general, suppose that $T$ is a monad whose underlying functor is polynomial: $X \mapsto \dsum{s : S} X^{P(s)}$ for some type family $P : S \to \Type$.
Note that $T \circ T$ is also polynomial, being given by $X \mapsto \dsum{s : S} \dsum{f : S^{P(s)}} X^{\dsum{p : P(s)} P(f(p))}$.
Because natural transformations between polynomial functors are given by \emph{dependent lenses} \cite[Definition~3.1]{AAG:containers} between their associated type families, the monad structure $\eta : \id \to T$ and $\mu : T \circ T \to T$ on $T$ is determined by two dependent lenses from the type families
\begin{eqalign}
	\internal{1} : 1 &\to \Type\\
	\ast &\mapsto 1
\end{eqalign}
and
\begin{eqalign}
	\dsum{s : S} S^{P(s)} &\to \Type\\
	(s,f) &\mapsto \dsum{p:P(s)}{P(f(p))}
\end{eqalign}
to $P : S \to \Type$.
Explicitly, this means that $\eta$ is given by a pair of functions
\begin{equation}
	\eta^+ : 1 \to S\ \text{ and }\ \eta^- : \dsum{x : 1}{P(\eta^+(x))} \to 1
\end{equation}
or, equivalently, by the single element $\ok : S$ picked out by $\eta^+$.
Likewise, $\mu$ corresponds to a pair of functions
\begin{equation}
	\seq : \dsum{s : S} S^{P(s)} \to S\ \text{ and }\ \splitt: \dsum{s : S} (f : S^{P(s)}) \times P(\seq(s, f)) \to \dsum{p : P(s)} P(f(p)).
\end{equation}
We record this observation in a theorem.

\begin{thm}
\label{thm:polynomial.monad}
	Let $T$ be the polynomial functor associated to a type family (container) $P : S \to \Type$. A monad structure $(\eta, \mu)$ on $T$ is equivalently given by
	\begin{enumerate}
		\item An element $\ok : S$.
		\item A function $\seq : \dsum{s : S} S^{P(s)} \to S$.
		\item A function $\splitt : \dsum{s : S} (f : S^{P(s)}) \times P(\seq(s, f)) \to \dsum{p : P(s)} P(f(p)) $.
	\end{enumerate}
	subject to the following equations:
	\begin{align}
		&\seq(s, \_ \mapsto \ok) = s,\\
		&\fst (\splitt(s, \_ \mapsto \ok, p)) = p,\\[2ex]
		&\seq(\ok, \_ \mapsto s) = s,\\
		&\snd (\splitt(\ok, \_ \mapsto s, p)) = p,\\[2ex]
		&\seq(\seq(s, f_1), f_2 \circ \splitt(s, f_1, \_)) = \seq(s, p \mapsto \seq(f_1(p), f_2(p))),\\
		&\letin{(r_1, q)}{\splitt(\seq(s, f_1), f_2 \circ \splitt(s, f_1, \_), r)}{} \notag\\
		&\qquad\qquad (\splitt(s, f_1, r_1), q) \notag\\
		&\qquad = \letin{(p_1, r_2)}{\splitt(s, p \mapsto \seq(f_1(p), f_2(p)), r)}{} \notag\\
		&\qquad\qquad (p_1, \splitt(f_1(p_1), f_2(p_1), r_2)).
	\end{align}
	In the last two equations, $f_1 : S^{P(s)}$ and $f_2 : (p : P(s)) \to S^{P(f_1(p))}$, which we uncurry where needed, and the two sides of the final equation are identified along the canonical reassociation of dependent pairs.

	Specifically, $\eta : X \to \dsum{s : S} X^{P(s)}$ is given by $x \mapsto (\ok, \_ \mapsto x)$, and the Kleisli extension of $x \mapsto (t(x), g(x)) : X \to \dsum{s : S} Y^{P(s)}$, where $t : X \to S$ and $g : (x : X) \to Y^{P(t(x))}$, is the map $\dsum{s : S} X^{P(s)} \to \dsum{s : S} Y^{P(s)}$ given by
	\begin{equation}
		(s, f) \mapsto (\seq(s, t \circ f),\ p \mapsto \letin{(p_1, p_2)}{\splitt(s, t \circ f, p)}{g(f(p_1))(p_2)}).
	\end{equation}
\end{thm}

Using this expression of a polynomial monad, we can define its transposed contextad.

\begin{defn}
	Let $(P : S \to \Type, \ok, \seq, \splitt)$ be a monad structure on a polynomial functor as in \cref{thm:polynomial.monad}.
	Its \textbf{transposed contextad} is the dependently graded comonad so defined:
	\begin{enumerate}
		\item For a type $X$, its type of grades is the type of functions $X \to S$ with its evident contravariance.
		\item For a function $s : X \to S$, the action $X \action s$ is given by $\dsum{x : X} P(s(x))$, which is evidently covariant in $X$.
		\item We define $\combineunit_C$ to be the constant function $\_ \mapsto \ok$.
		\item Given $s : X \to S$ and $t : \dsum{x : X} P(s(x)) \to S$, we define $s \combine t : X \to S$ by
		      \begin{equation}
				(s \combine t)(x) := \seq(s(x), t(x, \_)).
			  \end{equation}
		\item We define $\counitor_X : X \action \combineunit_X \to X$ to be $\fst$.
		\item We define $\coassociator_{s, t} : X \action (s \combine t) \to (X \action s) \action t$ by
		      \begin{equation}
				\coassociator_{s, t}(x, p) := \letin{(p_1, p_2)}{\splitt(s(x), t(x, \_), p)}{((x, p_1), p_2)}.
			  \end{equation}
	\end{enumerate}
	Because the grades form a set, the laws are equations which may be checked directly.
\end{defn}

\begin{thm}
\label{thm:pra.monads.transpose}
	If $T$ is a monad with $T$ a polynomial functor and $\action$ its contextad, then there is an identity-on-objects isomorphism $\Kl(T) \equi \Ctx(\action)$ given on morphisms by the canonical isomorphism:
	\begin{equation}
		X \xto{(s,f)} \dsum{s : S} Y^{P(s)}
		\quad
		\longmapsto
		\quad \begin{cases}
			s : X \to S \\
			\hat f : \dsum{x : X} P(s(x)) \to Y
		\end{cases}
	\end{equation}
\end{thm}
\begin{proof}
	A Kleisli morphism $X \to \dsum{s : S} Y^{P(s)}$ is equivalently a function $s : X \to S$ together with a dependent function $f : (x : X) \to Y^{P(s(x))}$, which we write $x \mapsto (s(x), f(x))$.
	Uncurrying $f$ gives $\hat f : \dsum{x : X} P(s(x)) \to Y$, $\hat f(x, p) := f(x)(p)$, which is a map $X \action s \to Y$; so the pair $(s, \hat f)$ is precisely a contextful arrow $X \looseto Y$.
	This correspondence is a bijection.
	Since the grades form a set, $\Ctx(\action)$ is a strict double category by \cref{thm:gaunt.ctx.strict.dbl}, so its loose arrows form a category and it remains to show that this bijection preserves identities and composition.

	The identity Kleisli morphism is the unit $\eta_X : x \mapsto (\ok, \_ \mapsto x)$.
	This transposes to the pair $(x \mapsto \ok,\ (x, p) \mapsto x) = (\combineunit_X, \fst) = (\combineunit_X, \counitor_X)$, which is the identity of $\Ctx(\action)$.

	Let $(s, f) : X \to \dsum{s : S} Y^{P(s)}$ and $(t, g) : Y \to \dsum{t : S} Z^{P(t)}$ be Kleisli morphisms.
	Their Kleisli composite sends $x$ to the Kleisli extension of $(t, g)$ applied to $(s(x), f(x))$, which by \cref{thm:polynomial.monad} is
	\begin{equation}
		x \mapsto \left(\seq(s(x), t \circ f(x)),\ r \mapsto \letin{(p_1, p_2)}{\splitt(s(x), t \circ f(x), r)}{g(f(x)(p_1))(p_2)}\right).
	\end{equation}
	On the other hand, the composite of the transposes $(s, \hat f)$ and $(t, \hat g)$ in $\Ctx(\action)$ is the contextful arrow
	\begin{equation}
		(s \combine \hat f^*t,\ \hat g \circ (\hat f \action t) \circ \coassociator_{s, \hat f^*t}),
	\end{equation}
	where $\hat f^*t = t \circ \hat f : X \action s \to S$ is the reindexing of the grade $t$ along $\hat f$, and $\hat f \action t : (X \action s) \action \hat f^*t \to Y \action t$ is $((x, p), q) \mapsto (\hat f(x, p), q)$.
	We may compute its components as follows:
	\begin{align*}
		(s \combine \hat f^*t)(x) & = \seq(s(x), \hat f^*t(x, \_))     \\
		                          & = \seq(s(x), p \mapsto t(f(x)(p))) \\
		                          & = \seq(s(x), t \circ f(x)),
	\end{align*}
	and
	\begin{align*}
		 & (\hat g \circ (\hat f \action t) \circ \coassociator_{s, \hat f^*t})(x, r)                                                     \\
		 & \qquad = \letin{(p_1, p_2)}{\splitt(s(x), \hat f^*t(x, \_), r)}{(\hat g \circ (\hat f \action t))((x, p_1), p_2)} \\
		 & \qquad = \letin{(p_1, p_2)}{\splitt(s(x), t \circ f(x), r)}{\hat g(\hat f(x, p_1), p_2)}                           \\
		 & \qquad = \letin{(p_1, p_2)}{\splitt(s(x), t \circ f(x), r)}{g(f(x)(p_1))(p_2)}.
	\end{align*}
	These are the two components of the transpose of the Kleisli composite computed above.
\end{proof}

\section{Left-fibrant spans and the Kleisli cocompletion of left-displayed spans}
\label{sec:fibs.span}

In this section, we will work out the central technical result of this work, namely that left-fibrant spans $\Ca \epifrom \Ea \to \Da$ are algebras of the arrow monad of $\Ca$ in the Kleisli arrow of the arrow monad of $\Da$.
This definition will allow, in the next section, to recognize $\FibSpan^\twoto$
as a subtricategory of the Kleisli completion of $\DispSpan$, paving the way to
our main result: \cref{thm:ctx.const.on.objects}.

\subsection{Fibrations as pseudoalgebras}
Cloven cartesian fibrations are property-like structure because they are defined in terms of adjoints (in suitable 2-categories), and adjoints to a given functor are unique up to unique isomorphism.
We may organize this observation into a theorem: cloven cartesian fibrations are algebras for a colax idempotent 2-monad \cite{kock_monads_1995}, an observation due to Street \cite{street1980fibrations}, where he defines fibrations in a bicategory as algebras of such a 2-monad.

We recall here such a definition, but specialized in the slightly stricter setting we work in.

\subsubsection{Pseudomonads}
Monads in a tricategory are naturally weakened to pseudomonads.
The definition below is expanded in more details in \cite{marmolejo_distributive_1999,lack_coherent_2000, gambino_formal_2021}.

\begin{rmk}
	We come clean immediately and warn the reader the aforementioned references define pseudomonads in $\Gray$-categories.
	Authors interested in pseudomonads prefer to work in $\Gray$-categories to save time and headache medications.
	Similarly, we work with fully weak tricategories since we never actually manipulate them directly and thus don't need to worry about their strictification.
	Since the two are triequivalent notions---that's the point!---we can use results proven with $\Gray$-categories for our tricategories.
\end{rmk}

\begin{defn}
\label{defn:psmnd}
	A \textbf{pseudomonad} on $\Ab$ in a tricategory $\Kc$ is a pseudomonoid in $\Kc(\Ab, \Ab)$.
	In particular, it's given by the data of an endomorphism $\monad:\Ab \to \Ab$, a unit $\unitmnd : \id_{\Ab} \twoto \monad$, a multiplication $\multmnd : \monad\monad \twoto \monad$ and 3-cells witnessing unitality and associativity, satisfying the usual triangular and pentagonal coherence equations.
	When the latter 3-cells are identities, we call $(\monad, \unitmnd, \multmnd)$ a \textbf{2-monad}.
\end{defn}

As for monads, it is praxis to abuse notation and refer to a pseudomonad by the name of the carrier endomorphism, or at most its triple of 2-dimensional data.

\begin{rmk}
	We will be mainly preoccupied with colax idempotent 2-monads, or `co-KZ' doctrines.
	Most of the relevant literature deals with \emph{lax idempotent} 2-monads, but it dualizes without pain since a colax idempotent 2-monad in $\Kc$ is a lax idempotent 2-monad in $\Kc\co$.
\end{rmk}

\begin{defn}
\label{defn:colax.idempotent.monad}
	A 2-monad $(\monad, \unitmnd, \multmnd)$ is \textbf{colax idempotent} iff $\unitmnd\monad \adj \multmnd$ and this adjunction has an invertible unit.
\end{defn}

\subsection{Algebras of colax idempotent 2-monads}
The purpose of presenting this material is to have it fresh in the next section, where we need to work with fibrations \emph{qua} algebraic structure in spans.
Therefore, we make some extremely specific choices in presenting the following material.
First, we only need readily available the definition of pseudoalgebras and their morphisms for colax idempotent 2-monads, which have a satisfyingly simply description---in fact, as proven in \cite{marmolejo_doctrines_1997}, admitting this description for their own pseudoalgebras characterizes colax idempotent 2-monads among 2-monads.
Second, we only consider right pseudoalgebras because this is what fibrations (and left-fibrant spans in particular) are going to be.
Third, we focus on `strict' tricategories whose hom-bicategories are in fact 2-categories: this is in fact the case for $\DispSpan$ and related tricategories.

For a broader `formal theory of pseudomonads', the reader may consult \cite{marmolejo_distributive_1999} and \cite{gambino_formal_2021}.

\bigskip

Let $\Kc$ be a tricategory and $(\monad, \unitmnd, \multmnd)$ a colax idempotent 2-monad on $\Ab \in \Kc$.
The following are \cite[Definition~2.1 and 2.4]{kock_monads_1995}, suitably dualized:

\begin{defn}
\label{defn:psalg}
	A \textbf{right $\monad$-pseudoalgebra} consists of a 1-cell $a:\Ab \to \Xb$ such that $a\unitmnd : a \twoto a\monad$ has a right adjoint (in $\Kc(\Xb, \Ab)$) $\alpha: a\monad \twoto a$ with invertible unit:
	\begin{equation}
		\begin{tikzcd}[ampersand replacement=\&]
			\& a\monad \\
			a \& a
			\arrow[""{name=0, anchor=center, inner sep=0}, "a\unitmnd", Rightarrow, from=2-1, to=1-2]
			\arrow[""{name=0p, anchor=center, inner sep=0}, phantom, from=2-1, to=1-2, start anchor=center, end anchor=center]
			\arrow["\alpha", Rightarrow, dashed, from=1-2, to=2-2]
			\arrow[""{name=1, anchor=center, inner sep=0}, Rightarrow, no head, from=2-1, to=2-2]
			\arrow[""{name=1p, anchor=center, inner sep=0}, phantom, from=2-1, to=2-2, start anchor=center, end anchor=center]
			\arrow["\unit"'{pos=0.55}, shift right=3, shorten <=4pt, shorten >=1pt, Rightarrow, "\sim" vert, from=1p, to=0p]
		\end{tikzcd}
	\end{equation}
	Similarly, a \textbf{left $\monad$-pseudoalgebra} is a right pseudoalgebra for $\monad$ in $\Kc\op$, thus carried by a 2-cell $\alpha : \monad a \twoto a$ which is likewise right adjoint right inverse to $\unitmnd a$.
\end{defn}


\begin{defn}
\label{defn:psmors}
	A \textbf{pseudomorphism of right $\monad$-pseudoalgebras} from $a:\Ab \to \Xb$ to $b:\Ab \to \Xb$ is a 2-cell $\varphi : a \twoto b$ such that the following mate is invertible:
	\begin{equation}
		\begin{tikzcd}[ampersand replacement=\&]
			{a\monad} \& {a\monad} \& {b\monad} \\
			\& a \& b \& b
			\arrow["\varphi"', Rightarrow, from=2-2, to=2-3]
			\arrow["{\varphi\monad}", Rightarrow, from=1-2, to=1-3]
			\arrow["{a\unitmnd}"', Rightarrow, from=2-2, to=1-2]
			\arrow["{b\unitmnd}", Rightarrow, from=2-3, to=1-3]
			\arrow["\counit_{a\unitmnd \adj \alpha}"{pos=0.6}, shift left=4.5, shorten <=9pt, shorten >=2pt, Rightarrow, from=2-2, to=1-2]
			\arrow["\unit_{b\unitmnd \adj \beta}"'{pos=0.3}, shift right=4.5, shorten <=2pt, shorten >=9pt, Rightarrow, from=2-3, to=1-3]
			\arrow["\alpha"', curve={height=6pt}, Rightarrow, from=1-1, to=2-2]
			\arrow["\beta", curve={height=-6pt}, Rightarrow, from=1-3, to=2-4]
			\arrow[Rightarrow, no head, from=2-3, to=2-4]
			\arrow[Rightarrow, no head, from=1-1, to=1-2]
		\end{tikzcd}
	\end{equation}
	Here we are mating the naturality square\footnotemark~of $\unitmnd$ by the adjunctions defining the algebra structure on $a$ and $b$.
	\footnotetext{More precisely: the naturality square of whiskering by $\unitmnd$.}
\end{defn}

Left and right pseudoalgebras often organize in, respectively, Eilenberg--Moore and Kleisli objects (see \cite{gambino_formal_2021} for their definition), but not every tricategory admits their construction.
Instead, one can construct them in the co/presheaves categories $[\Kc, \TwoCat]$ and $[\Kc\op, \TwoCat]$ where $\TwoCat$ is the tricategory of 2-categories (i.e.~\emph{strict} bicategories), 2-functors (i.e.~\emph{strict} bifunctors), pseudonatural transformations and their modifications (recall we assumed $\Kc$ to be `enriched in 2-categories').

Specifically, the colax idempotent 2-monad $\monad$ induces colax idempotent 2-monads in $\TwoCat$:
\begin{eqalign}
	\Kc(\monad, \Xb) &: \Kc(\Ab, \Xb) \longto \Kc(\Ab, \Xb),\\
	\Kc(\Yb, \monad) &: \Kc(\Yb, \Ab) \longto \Kc(\Yb, \Ab),
\end{eqalign}
naturally in $\Xb, \Yb \in \Kc$.
Moreover, a \emph{right} $\monad$-pseudoalgebra $a:\Ab \to \Xb$ will correspond to a \emph{left} $\Kc(\monad,\Xb)$-pseudoalgebra but still a \emph{right} $\Kc(\Yb,\monad)$-pseudoalgebra.

\medskip

The Eilenberg--Moore and Kleisli object for such pseudomonads are so defined:

\begin{defn}
\label{defn:Kleisli.2.cat}
	The \textbf{Kleisli 2-category of $\Kc(\Yb, \monad)$}, denoted as $\Kl(\Kc(\Yb, \monad),\Kc(\Yb,\Ab))$, is so defined:
	\begin{enumerate}
		\item its objects are those of $\Kc(\Yb,\Ab)$, hence 1-cells (in $\Kc$) $a:\Yb \to \Ab$,
		\item its 1-cells $\varphi:a \to b$ are 2-cells (in $\Kc$) $\varphi:a \twoto \monad b$,
		\item its 2-cells and their vertical identities and composition are those of $\Kc(\Yb,\Ab)$,
		\item horizontal identities are given by the components of the unit $\Kc(\Yb,\unitmnd):\id_{\Kc(\Yb,\Ab)} \twoto \Kc(\Yb, \monad)$ and horizontal composition is defined as
		      \begin{equation}
			      \underbrace{\begin{tikzcd}[ampersand replacement=\&]
					      a \& b \& c
					      \arrow[""{name=0, anchor=center, inner sep=0}, curve={height=-12pt}, from=1-1, to=1-2]
					      \arrow[""{name=1, anchor=center, inner sep=0}, curve={height=-12pt}, from=1-2, to=1-3]
					      \arrow[""{name=2, anchor=center, inner sep=0}, curve={height=12pt}, from=1-1, to=1-2]
					      \arrow[""{name=3, anchor=center, inner sep=0}, curve={height=12pt}, from=1-2, to=1-3]
					      \arrow["\varphi", shorten <=3pt, shorten >=3pt, Rightarrow, from=0, to=2]
					      \arrow["\psi", shorten <=3pt, shorten >=3pt, Rightarrow, from=1, to=3]
				      \end{tikzcd}}_{\in \Kl(\Kc(\Yb, \monad), \Kc(\Yb,\Ab))}
			      \ :=\ %
			      \underbrace{\begin{tikzcd}[ampersand replacement=\&]
					      a \& {\monad b} \& {\monad\monad c} \& {\monad c}
					      \arrow[""{name=0, anchor=center, inner sep=0}, curve={height=-12pt}, from=1-1, to=1-2]
					      \arrow[""{name=1, anchor=center, inner sep=0}, curve={height=-12pt}, from=1-2, to=1-3]
					      \arrow["\multmnd c", from=1-3, to=1-4]
					      \arrow[""{name=2, anchor=center, inner sep=0}, curve={height=12pt}, from=1-1, to=1-2]
					      \arrow[""{name=3, anchor=center, inner sep=0}, curve={height=12pt}, from=1-2, to=1-3]
					      \arrow["\varphi", shorten <=3pt, shorten >=3pt, Rightarrow, from=0, to=2]
					      \arrow["\monad\psi", shorten <=3pt, shorten >=3pt, Rightarrow, from=1, to=3]
				      \end{tikzcd}}_{\in \Kc(\Yb,\Ab)}
		      \end{equation}
	\end{enumerate}
\end{defn}

\begin{defn}
\label{defn:Alg.2.cat}
	The \textbf{Eilenberg--Moore 2-category of $\Kc(\monad, \Xb)$}, denoted as $\Alg(\Kc(\monad, \Xb), \Kc(\Ab,\Xb))$ is so defined:
	\begin{enumerate}
		\item its objects are right $\monad$-pseudoalgebras $\alpha : a\monad \twoto a$ for $a : \Ab \to \Xb$,
		\item its 1-cells are pseudomorphisms of right $\monad$-pseudoalgebras, and their (strict) identity and horizontal composition are those of $\Kc(\Ab, \Xb)$,
		\item its 2-cells and their vertical identities and composition are those of $\Kc(\Ab, \Xb)$.
	\end{enumerate}
\end{defn}

For brevity, and unless we want to stress a difference with other notions of algebras, we call `pseudoalgebras of a pseudomonad' simply `algebras for a monad'.
\subsection{The arrow monad}
The colax idempotent 2-monads we are interested in are the ones defined in $\DispSpan$ by the arrow objects and their domain and codomain projections:
\begin{equation}
	\begin{tikzcd}[ampersand replacement=\&,sep=scriptsize]
		\& {\Ba^\downarrow} \\
		\Ba \&\& \Ba
		\arrow["{\dom}"', from=1-2, to=2-1]
		\arrow["{\cod}", from=1-2, to=2-3]
	\end{tikzcd}
\end{equation}
Its unit and multiplication are given by composing with the map of spans respectively below left and right:
\begin{equation}
	\begin{tikzcd}[ampersand replacement=\&,sep=scriptsize]
		\Ba \& \Ba \& \Ba \& \Ba \& {\Ba^\downarrow} \& \Ba \& {\Ba^\downarrow} \& \Ba \\
		\&\&\&\&\& {\Ba^\downarrow \spancomp \Ba^\downarrow} \\
		\Ba \& {\Ba^\downarrow} \& \Ba \& \Ba \&\& {\Ba^\downarrow} \&\& \Ba
		\arrow["{{\dom}}"', from=1-5, to=1-4]
		\arrow["{{\cod}}", from=1-5, to=1-6]
		\arrow["{{\dom}}", from=3-6, to=3-4]
		\arrow["{{\dom}}"', from=1-7, to=1-6]
		\arrow["{{\cod}}"', from=3-6, to=3-8]
		\arrow["{{\cod}}", from=1-7, to=1-8]
		\arrow[Rightarrow, no head, from=3-4, to=1-4]
		\arrow[Rightarrow, no head, from=3-8, to=1-8]
		\arrow["\multarrow_{\Ba}", from=2-6, to=3-6]
		\arrow[from=2-6, to=1-5]
		\arrow[from=2-6, to=1-7]
		\arrow["\lrcorner"{anchor=center, pos=0.0025, rotate=135}, draw=none, from=2-6, to=1-6]
		\arrow[Rightarrow, no head, from=1-1, to=3-1]
		\arrow[Rightarrow, no head, from=1-3, to=3-3]
		\arrow[Rightarrow, no head, from=1-2, to=1-1]
		\arrow[Rightarrow, no head, from=1-2, to=1-3]
		\arrow["{\dom}", from=3-2, to=3-1]
		\arrow["{\cod}"', from=3-2, to=3-3]
		\arrow["\unitarrow_{\Ba}", from=1-2, to=3-2]
	\end{tikzcd}
\end{equation}
The maps $\unitarrow_{\Ba}$ and $\multarrow_{\Ba}$ are obtained by taking powers of $\Ba$ with the objects and maps of $\Delta_+$, the \emph{algebraist's simplicial category}, defined as the full subcategory of $\Cat$ spanned by the finite ordinals and equipped with ordinal sum $\oplus$ as the product.
The walking object (terminal category), the walking arrow $\downarrow$ and the walking composable pairs of arrows $\consdownarrows$ all live in $\Delta^+$ as $[0]$, $[1]$ and $[2]$ respectively.
Then $\unitarrow_{\Ba}$ is power by the terminal map $\sigma^1 : [1] \to [0]$ and $\multarrow_{\Ba}$ by inclusion of the walking composite $\delta^2_1 : [1] \to [2]$.
These maps define a comonoid structure on $[1]$ such that $\delta^2_1 \adj \sigma^1 \oplus [1]$.
Then $\Ba^{(-)}$ defines a monoidal functor from $\Delta^+$ to $\DispSpan$ that shows $\Ba^\downarrow$ is a 2-monad satisfying the condition of colax idempotency.
A more detailed construction can be found in \cite{street1980fibrations}.

We abuse notation and refer to this 2-monad as $\Ba^\downarrow$ as call it the \textbf{arrow monad on $\Ba$}.

\paragraph{Algebras of the arrow monad}
The (right) algebras of the arrow monad on $\Ba$ are spans $\Ba \nfrom{p} \Ea \nto{f} \Ca$ where $p$ is a fibration and $f$ is such that $f\cod\counitpull = \id_{f\cod}$, i.e.~$f$ kills cartesian maps.
Indeed, below is the diagram showing the adjunction defining the algebra structure:
\begin{equation}
\label{eqn:alg.arrow}
	\begin{tikzcd}[ampersand replacement=\&]
		\Ba \&\& \Ea \&\& \Ba \\
		\&\& {\Ba \comma p} \\[-4ex]
		\Ba \& \Ba \& \Ba \& \Ea \& \Ba
		\arrow["{\cod}", from=3-2, to=3-3]
		\arrow["p"', from=3-4, to=3-3]
		\arrow["f", from=3-4, to=3-5]
		\arrow["\dom"', from=2-3, to=3-2]
		\arrow["\cod", from=2-3, to=3-4]
		\arrow["\lrcorner"{anchor=center, pos=0.125, rotate=-45}, draw=none, from=2-3, to=3-3]
		\arrow["{\dom}"', from=3-2, to=3-1]
		\arrow["p"', from=1-3, to=1-1]
		\arrow["f", from=1-3, to=1-5]
		\arrow[""{name=0, anchor=center, inner sep=0}, "\pull"', curve={height=10pt}, dashed, from=2-3, to=1-3]
		\arrow[""{name=1, anchor=center, inner sep=0}, "{\internal{\id_p}}"', curve={height=10pt}, from=1-3, to=2-3]
		\arrow[Rightarrow, no head, from=1-1, to=3-1]
		\arrow[Rightarrow, no head, from=1-5, to=3-5]
		\arrow["\dashv"{anchor=center}, draw=none, from=1, to=0]
	\end{tikzcd}
\end{equation}
If we focus on the left hand side of the diagram, we see that this corresponds to the triangle appearing in \cref{defn:fib}(1), hence a fibration structure on $p$.
On the other hand, the commutativity of right hand side of the diagram implies that the counit $\counitpull$---which corresponds to the cartesian lift operation on $p$, as remarked in \cref{rmk:fib.struct}---is such that:
\begin{equation}
\label{eqn:vert.of.cart.lift}
	\begin{tikzcd}[ampersand replacement=\&, sep=scriptsize]
		{\Ba \comma p} \\
		\& \Ea \&\& \Ba \\
		{\Ba \comma p}
		\arrow["\pull"', dashed, from=3-1, to=2-2]
		\arrow["{\internal{\id_p}}"', from=2-2, to=1-1]
		\arrow["{f \cod}", curve={height=-12pt}, from=1-1, to=2-4]
		\arrow["{f \cod}"', curve={height=12pt}, from=3-1, to=2-4]
		\arrow[""{name=0, anchor=center, inner sep=0}, Rightarrow, no head, from=1-1, to=3-1]
		\arrow["\counitpull"{description}, shorten <=5pt, Rightarrow, from=0, to=2-2]
	\end{tikzcd}
\end{equation}
meaning $f\cod\counitpull = \id_{f\cod}$, as claimed.

\paragraph{Kleisli morphisms for the arrow monad}
Kleisli morphisms with respect to the 2-monad $- \spancomp \Ca^\downarrow := \DispSpan(\Ba, \Ca^\downarrow)$ induced by acting by $\Ca^\downarrow$ on the right (following \cref{defn:Kleisli.2.cat}) turn out to be right-lax maps of spans:

\begin{lem}
\label{lem:lax.strict.span.adjunction}
	The inclusion $\DispSpan^{=}(\Ba, \Ca) \to \DispSpan^{\Rightarrow}(\Ba, \Ca)$ has a right adjoint given by post-composing with
	\begin{equation}
		\begin{tikzcd}[sep=scriptsize]
			& {\Ca^\downarrow} \\
			\Ca && \Ca
			\arrow["{\dom}"', from=1-2, to=2-1]
			\arrow["{\cod}", from=1-2, to=2-3]
		\end{tikzcd}
	\end{equation}
	For this reason, these right adjoints are natural with respect to composition of spans on the left.
\end{lem}
\begin{proof}
	First, we note that for a span as below left, the composite with $\Ca^{\downarrow}$ is the comma object as below right:
	\begin{equation}
		\begin{tikzcd}[sep=scriptsize]
			& \Ea &&&& {f \comma \Ca} \\
			\Ba && \Ca && \Ba && \Ca
			\arrow["p"', from=1-2, to=2-1]
			\arrow["f", from=1-2, to=2-3]
			\arrow["{\cod}", from=1-6, to=2-7]
			\arrow["{p\dom}"', from=1-6, to=2-5]
		\end{tikzcd}
	\end{equation}
	We then need to give an isomorphism between the categories of diagrams on the left and diagrams on the right:
	\begin{equation}
	\label{diag:slice.adjunction}
		\begin{tikzcd}
			\Ba & \Ea & \Ca && \Ba & \Ea & \Ca \\
			\Ba & \Ea' & \Ca && \Ba & {f \comma \Ca} & \Ca
			\arrow["p"', from=2-2, to=2-1]
			\arrow["f", from=2-2, to=2-3]
			\arrow["{\cod}", from=2-6, to=2-7]
			\arrow["{p\dom}"', from=2-6, to=2-5]
			\arrow[from=1-2, to=1-1]
			\arrow[from=1-2, to=1-3]
			\arrow[from=1-6, to=1-5]
			\arrow[from=1-6, to=1-7]
			\arrow[""{name=0, anchor=center, inner sep=0}, "{\internal{\varphi}}"', from=1-6, to=2-6]
			\arrow[""{name=1, anchor=center, inner sep=0}, "{a}"', from=1-2, to=2-2]
			\arrow[from=1-1, to=2-1,equals]
			\arrow[""{name=2, anchor=center, inner sep=0}, from=1-3, to=2-3,equals]
			\arrow[from=1-5, to=2-5,equals]
			\arrow[""{name=3, anchor=center, inner sep=0}, from=1-7, to=2-7,equals]
			\arrow["\varphi", shorten <=6pt, shorten >=6pt, Rightarrow, from=1, to=2]
		\end{tikzcd}
	\end{equation}

	This follows directly from the universal property of the slice object.
\end{proof}

We will show that this adjunction exhibits $\DispSpan^{\Rightarrow}(\Ba, \Ca)$ as the Kleisli $2$-category for the $2$-monad $- \spancomp \Ca^\downarrow$ on $\DispSpan^{=}(\Ba, \Ca)$.

\begin{lem}
\label{lem:lax.maps.kleisli}
	There is an isomorphism
	\begin{equation}
		\DispSpan^{\Rightarrow}(\Ba, \Ca) \iso \Kl(- \spancomp \Ca^\downarrow,\, \DispSpan^{=}(\Ba, \Ca)).
	\end{equation}
\end{lem}
\begin{proof}
	Keeping in mind the adjunction of \cref{lem:lax.strict.span.adjunction}, it only remains to show that the composition laws correspond; that is, that the composite on the left corresponds to the composite on the right under the adjunction:
	\begin{equation}
		\begin{tikzcd}[ampersand replacement=\&]
			\Ba \& \Ea \& \Ca \&\& \Ba \& \Ea \& \Ca \\
			\Ba \& \Fa \& \Ca \&\& \Ba \& {f \comma \Ca} \& \Ca \\
			\Ba \& \Ga \& \Ca \&\& \Ba \& {g \comma \Ca \comma \Ca} \& \Ca \\
			\&\&\&\& \Ba \& {g \comma \Ca} \& \Ca
			\arrow["p"', from=3-2, to=3-1]
			\arrow["g", from=3-2, to=3-3]
			\arrow[from=3-6, to=3-7]
			\arrow[from=3-6, to=3-5]
			\arrow[from=2-2, to=2-1]
			\arrow["f", from=2-2, to=2-3]
			\arrow[from=2-6, to=2-5]
			\arrow[from=2-6, to=2-7]
			\arrow["{{\internal{\varphi} \comma \Ca}}"', from=2-6, to=3-6]
			\arrow[""{name=0, anchor=center, inner sep=0}, "b"', from=2-2, to=3-2]
			\arrow[Rightarrow, no head, from=2-1, to=3-1]
			\arrow[""{name=1, anchor=center, inner sep=0}, Rightarrow, no head, from=2-3, to=3-3]
			\arrow[Rightarrow, no head, from=2-5, to=3-5]
			\arrow[Rightarrow, no head, from=2-7, to=3-7]
			\arrow[from=1-2, to=1-1]
			\arrow[from=1-2, to=1-3]
			\arrow[from=1-6, to=1-5]
			\arrow[from=1-6, to=1-7]
			\arrow[""{name=2, anchor=center, inner sep=0}, "a"', from=1-2, to=2-2]
			\arrow["{{\internal{\psi}}}"', from=1-6, to=2-6]
			\arrow[from=4-6, to=4-5]
			\arrow[from=4-6, to=4-7]
			\arrow[Rightarrow, no head, from=1-1, to=2-1]
			\arrow[""{name=3, anchor=center, inner sep=0}, Rightarrow, no head, from=1-3, to=2-3]
			\arrow[Rightarrow, no head, from=1-5, to=2-5]
			\arrow[Rightarrow, no head, from=1-7, to=2-7]
			\arrow[Rightarrow, no head, from=3-5, to=4-5]
			\arrow[Rightarrow, no head, from=3-7, to=4-7]
			\arrow["{{\multarrow_{\Ca}}}", from=3-6, to=4-6]
			\arrow["\varphi", shorten <=6pt, shorten >=6pt, Rightarrow, from=0, to=1]
			\arrow["\psi", shorten <=6pt, shorten >=6pt, Rightarrow, from=2, to=3]
		\end{tikzcd}
	\end{equation}
	That is, we observe that $\internal{(\varphi a)\psi} = \multarrow_{\Ca}(\internal{\varphi} \comma \Ca) \internal{\psi}$.
	This follows by appeal to the uniqueness part of the universal property of commas.
\end{proof}

We would like to say contextads on categories are algebras of the arrow monad, but their right leg hardly kills cartesian maps.
Working with right-lax maps allows us to accomodate this fact.
In fact we have the following:

\begin{lem}
\label{lem:fspans.are.algebras}
	There is an isomorphism of 2-categories
	\begin{equation}
		\FibSpan^{\Rightarrow}(\Ba, \Ca) \iso \Alg(\Ba^{\downarrow} \spancomp -,\, \DispSpan^{\Rightarrow}(\Ba, \Ca)).
	\end{equation}
	which is the identity on the underlying spans and maps thereof.
\end{lem}
\begin{proof}
	Consider a left fibrant span $\Ba \nepifrom{p} \Ea \nto{f} \Ca$.
	The idea of the proof is to accomodate the obstructions to $f$-verticality of the $p$-cartesian maps by adding a filler to \eqref{eqn:vert.of.cart.lift}, and we achieve this by working in $\DispSpan^{\Rightarrow}(\Ba, \Ca)$.

	First, let's see how to get a 2-functor $\Alg(\Ba^{\downarrow} \spancomp -,\, \DispSpan^{\Rightarrow}(\Ba, \Ca)) \to \FibSpan^{\Rightarrow}(\Ba, \Ca)$.
	This direction is substantially trivial.
	As we have seen in \eqref{eqn:alg.arrow} for $\DispSpan$, for a span $\Ba \overset{p}{\leftarrow} \Ea \xto{f} \Ca$ to be an algebra of $\Ba^{\downarrow} \spancomp -$ in $\DispSpan^{\Rightarrow}(\Ba, \Ca)$ is for the unit $\internal{\id_p} : \Ea \to \Ba \comma p$ to have a right adjoint with invertible unit.
	In particular, this implies that $\internal{\id_p}$ has a right adjoint with invertible unit over $\Ba$, thus providing a fibration structure on $p$ (following \cref{defn:fib}).
	Similarly, the characterization of morphisms for algebras of colax idempotent 2-monads shows that any such morphism is, in particular, a cartesian functor between the left legs (see \cref{defn:psmors} and \cref{defn:cartesian.functor}).
	Finally, 2-cells between algebras are simply 2-cells between cartesian functors, so we get the desired functor.

	Now let's construct an inverse.
	Suppose $\Ba \nepifrom{p} \Ea \nto{f} \Ca$ is left-fibrant, where $\pull : \Ba \comma p \to \Ea$ is the fibration structure on $p$ with counit $\counitpull$ and unit $\unitpull$.
	We want to prove the same span and the same structure is a valid algebra structure for $\Ba^\downarrow \spancomp -$.
	This means constructing an adjoint to $\internal{\id_p}$ in $\DispSpan^{\Rightarrow}(\Ba, \Ca)$, which boils down to constructing the dashed 2-cell in the following diagram:
	\begin{equation}
		\begin{tikzcd}[ampersand replacement=\&]
			\Ba \& {\Ba \comma p} \& \Ca \\
			\Ba \& \Ea \& \Ca
			\arrow["{\dom}"', two heads, from=1-2, to=1-1]
			\arrow["{f\cod}", from=1-2, to=1-3]
			\arrow["p", two heads, from=2-2, to=2-1]
			\arrow["f"', from=2-2, to=2-3]
			\arrow[""{name=0, anchor=center, inner sep=0}, "{\pull}"', from=1-2, to=2-2]
			\arrow[""{name=0p, anchor=center, inner sep=0}, phantom, from=1-2, to=2-2, start anchor=center, end anchor=center]
			\arrow[Rightarrow, no head, from=1-1, to=2-1]
			\arrow[""{name=1, anchor=center, inner sep=0}, Rightarrow, no head, from=1-3, to=2-3]
			\arrow[""{name=1p, anchor=center, inner sep=0}, phantom, from=1-3, to=2-3, start anchor=center, end anchor=center]
			\arrow[shorten <=7pt, shorten >=7pt, Rightarrow, dashed, from=0p, to=1p]
		\end{tikzcd}
	\end{equation}
	Such a 2-cell can be defined to be $f\cod \counitpull : f\pull \Rightarrow f \cod$, which has this type since $\cod(\internal{\id_p}) = \id_{\Ea}$ and thus $f\cod(\internal{\id_p})\pull \equiv f\pull$.

	We then need to show that $\counitpull$ and $\unitpull$ give unit and counit $2$-cells in $\DispSpan^{\Rightarrow}(\Ba, \Ca)$ that exhibit $(\pull, f\cod\counitpull)$ as a right adjoint to $(\internal{\id_p}, {=})$.
	Their well-definedness corresponds to having the two diagrams below commute:
	\begin{eqalign}
	\label{eqn:unit.counit.of.arrow.algebra}
		\begin{tikzcd}[ampersand replacement=\&]
			{\Ba \comma p} \& \Ca \\
			\Ea \& \Ca \\
			{\Ba \comma p} \& \Ca
			\arrow[Rightarrow, no head, from=2-2, to=3-2]
			\arrow[Rightarrow, no head, from=1-2, to=2-2]
			\arrow["{{f\cod}}", from=1-1, to=1-2]
			\arrow["{{f\cod}}", from=3-1, to=3-2]
			\arrow["{\internal{\id_p}}"', from=2-1, to=3-1]
			\arrow["{{\pull}}"', from=1-1, to=2-1]
			\arrow["{{f\cod\counitpull}}", shorten <=5pt, shorten >=5pt, Rightarrow, from=2-1, to=1-2]
			\arrow["f", from=2-1, to=2-2]
		\end{tikzcd}
		= &
		\begin{tikzcd}[ampersand replacement=\&]
			\&[-4ex] {\Ba \comma p} \& \Ca \\
			{} \& \Ea \\
			\& {\Ba \comma p} \& \Ca
			\arrow[Rightarrow, no head, from=1-3, to=3-3]
			\arrow["{{f\cod}}", from=1-2, to=1-3]
			\arrow["{{f\cod}}", from=3-2, to=3-3]
			\arrow[curve={height=36pt}, shift right=2, Rightarrow, no head, from=1-2, to=3-2]
			\arrow["{{\pull}}"', from=1-2, to=2-2]
			\arrow["{\internal{\id_p}}"', from=2-2, to=3-2]
			\arrow["\counitpull"'{pos=0.4}, shorten <=2pt, shorten >=2pt, Rightarrow, from=2-2, to=2-1]
		\end{tikzcd}
		\quad &&&& \text{ i.e.}\quad
		\begin{tikzcd}[ampersand replacement=\&,sep=small]
			f\pull \&\& f\cod \\
			\& f\cod
			\arrow["{{f\cod\counitpull}}", Rightarrow, from=1-1, to=1-3]
			\arrow["{{f\cod\counitpull}}"', Rightarrow, from=1-1, to=2-2]
			\arrow[Rightarrow, no head, from=1-3, to=2-2]
		\end{tikzcd}\\
		\begin{tikzcd}[ampersand replacement=\&]
			\&[-4ex] \Ea \& \Ca \\
			{} \& {\Ba \comma p} \& \Ca \\
			\& \Ea \& \Ca
			\arrow[Rightarrow, no head, from=2-3, to=3-3]
			\arrow[Rightarrow, no head, from=1-3, to=2-3]
			\arrow["f", from=1-2, to=1-3]
			\arrow["f", from=3-2, to=3-3]
			\arrow["{{\pull}}"', from=2-2, to=3-2]
			\arrow["{\internal{\id_p}}"'{pos=0.55}, from=1-2, to=2-2]
			\arrow["{{f\cod}}", from=2-2, to=2-3]
			\arrow[shift right=2, curve={height=40pt}, Rightarrow, no head, from=1-2, to=3-2]
			\arrow["{{f\cod\counitpull}}", shorten <=5pt, shorten >=5pt, Rightarrow, from=3-2, to=2-3]
			\arrow["\unitpull"{pos=0.4}, Rightarrow, from=2-1, to=2-2]
		\end{tikzcd}
		= &
		\begin{tikzcd}[ampersand replacement=\&]
			\Ea \& \Ca \\[2ex]
			\\[2ex]
			\Ea \& \Ca
			\arrow[Rightarrow, no head, from=1-1, to=3-1]
			\arrow[Rightarrow, no head, from=1-2, to=3-2]
			\arrow["f", from=1-1, to=1-2]
			\arrow["f", from=3-1, to=3-2]
		\end{tikzcd}
		\quad &&&& \text{ i.e.}\quad
		\begin{tikzcd}[sep=small]
			f && f \\
			& {f\pull \internal{\id_p}}
			\arrow[Rightarrow, no head, from=1-1, to=1-3]
			\arrow["f\unitpull"', Rightarrow, from=1-1, to=2-2]
			\arrow["{f\cod \counitpull \internal{\id_p}}"', Rightarrow, from=2-2, to=1-3]
		\end{tikzcd}
	\end{eqalign}
	The first evidently does; the second does because $(\cod \counitpull \internal{\id_p})\unitpull = \id$ as $\unitpull$ is the family of cartesian lifts of identities (\cref{rmk:fib.adj}).
	Indeed, the zig-zag identities for $\counitpull$ and $\unitpull$ (as co/unit for the adjunction of maps left-fibrant spans) directly give the zig-zag identities they need to be co/unit for the adjunction in $\DispSpan^\twoto(\Ba, \Ca)$.

	Now on to 1-cells.
	Let the diagrm below be a 1-cell in $\FibSpan^\twoto(\Ba, \Ca)$.
	\begin{equation}
		\begin{tikzcd}
			\Ba & \Ea & \Ca \\
			\Ba & \Ea' & \Ca
			\arrow["p"', from=1-2, to=1-1]
			\arrow["f", from=1-2, to=1-3]
			\arrow["{p'}", from=2-2, to=2-1]
			\arrow["f'"', from=2-2, to=2-3]
			\arrow[""{name=0, anchor=center, inner sep=0}, "k"', from=1-2, to=2-2]
			\arrow[from=1-1, to=2-1, Rightarrow, no head]
			\arrow[""{name=1, anchor=center, inner sep=0}, Rightarrow, no head, from=1-3, to=2-3]
			\arrow["\varphi", shorten <=6pt, shorten >=6pt, Rightarrow, from=0, to=1]
		\end{tikzcd}
	\end{equation}
	The functor we are describing sends this 1-cell to the same one in $\Alg(\Ba^{\downarrow} \spancomp -,\, \DispSpan^{\Rightarrow}(\Ba, \Ca))$, and we must verify this definition typechecks.
	In particular, we have to show the cartesianator $\cartesianator$ (i.e.~\eqref{eqn:cartesian.fun}) below is invertible, but that is already true since $k$ is a cartesian functor by assumption.
	\begin{equation}
		\begin{tikzcd}[ampersand replacement=\&,sep=scriptsize]
			{\Ba \comma p} \&\& {\Ba \comma p'} \\
			\\
			\Ea \&\& {\Ea'}
			\arrow[""{name=0, anchor=center, inner sep=0}, "k"', from=3-1, to=3-3]
			\arrow["{\Ba \comma k}", from=1-1, to=1-3]
			\arrow["{\pull}"', from=1-1, to=3-1]
			\arrow[""{name=1, anchor=center, inner sep=0}, "{\pull'}", from=1-3, to=3-3]
			\arrow["{\cartesianator}"'{pos=0.7}, shift right=7, shorten <=16pt, shorten >=2pt, Rightarrow, from=1, to=0]
		\end{tikzcd}
	\end{equation}
	It remains to show the following prism commutes:
	\begin{equation}
		\begin{tikzcd}[ampersand replacement=\&]
			{\Ba \comma p} \&\& {\Ba \comma p'} \\
			\\
			\Ea \&\& {\Ea'} \&\& \Ca \\
			\\
			\&\&\&\& \Ca
			\arrow[""{name=0, anchor=center, inner sep=0}, "k"', from=3-1, to=3-3]
			\arrow[""{name=1, anchor=center, inner sep=0}, "{\Ba \comma k}", from=1-1, to=1-3]
			\arrow[""{name=2, anchor=center, inner sep=0}, "{\pull}"', from=1-1, to=3-1]
			\arrow[""{name=3, anchor=center, inner sep=0}, "{\pull'}", from=1-3, to=3-3]
			\arrow["{\cartesianator}"{description}, shift right=5, shorten <=8pt, shorten >=8pt, Rightarrow, from=3, to=0]
			\arrow[""{name=4, anchor=center, inner sep=0}, Rightarrow, no head, from=3-5, to=5-5]
			\arrow["{f'\cod}"{description, pos=0.3}, curve={height=-6pt}, from=1-3, to=3-5]
			\arrow["f"{pos=0.7}, curve={height=-6pt}, from=3-3, to=5-5]
			\arrow["{f'\cod\counitpull'}"{description,pos=0.6}, shift right=3, curve={height=-12pt}, shorten <=15pt, shorten >=15pt, Rightarrow, from=3, to=4]
			\arrow["{f\cod}"{description, pos=0.2}, curve={height=12pt}, from=1-1, to=3-5]
			\arrow["{f'}"{description}, curve={height=12pt}, from=3-1, to=5-5]
			\arrow["{\varphi\cod}"{description,pos=0.7}, shorten <=19pt, Rightarrow, from=1, to=3-5]
			\arrow["\varphi"{description}, shorten <=19pt, Rightarrow, from=0, to=5-5]
			\arrow["{f\cod\counitpull}"{description, pos=0.55}, curve={height=12pt}, shorten <=28pt, shorten >=28pt, Rightarrow, from=2, to=4]
		\end{tikzcd}
	\end{equation}
	Chasing an object $\beta:B \to p(E)$ of $\Ba \comma p$ we are lead to prove the solid diagram below commutes in $\Ca$:
	\begin{equation}
		\begin{tikzcd}[ampersand replacement=\&,sep=scriptsize]
			{f'(\beta^*k(E))} \&\& {f'(k(E))} \\
			{f'(k(\beta^*E))} \&\&\& {f(E)} \\
			\& {f(\beta^*E)}
			\arrow["{f'(\cartesianator_h)}\;"', "\wr", from=1-1, to=2-1]
			\arrow["{\varphi_e}", from=1-3, to=2-4]
			\arrow["{\varphi_{\beta^*E}}"', from=2-1, to=3-2]
			\arrow["{f(\lift(\beta))}"', from=3-2, to=2-4]
			\arrow["{f'(k(\lift(\beta)))}"{description}, curve={height=12pt}, dashed, from=2-1, to=1-3]
			\arrow["{f'(\lift(\beta))}"{description}, from=1-1, to=1-3]
		\end{tikzcd}
	\end{equation}
	Tracing the dashed arrow, we see that the resulting square witnesses naturality of $\varphi$, while the remaining triangle is the image under $f'$ of a triangle witnessing the universal property of $\lift(\beta)$ (recalling notation from \cref{defn:fib}) as a cartesian lift relative to $p$ (or, more directly, that we know to commute by assumption on $k$).
	This concludes the proof that $(k, \varphi)$ is a well-defined map in $\Alg(\Ba^{\downarrow} \spancomp -,\, \DispSpan^{\Rightarrow}(\Ba, \Ca))$.

	The analogous verification is trivial for 2-cells (recall \cref{defn:Alg.2.cat}), as well as verifying the correspondence just described is functorial.

	To see the two functors we defined are inverses to each other we need to show that the resulting round trips are equivalent to the identity.
	The only non-trivial condition to check is that starting with an algebra
	\begin{equation}
		\begin{tikzcd}[ampersand replacement=\&]
			\Ba \& {\Ba \comma p} \& \Ca \\
			\Ba \& \Ea \& \Ca
			\arrow["{\dom}"', from=1-2, to=1-1]
			\arrow["{f\cod}", from=1-2, to=1-3]
			\arrow["p", from=2-2, to=2-1]
			\arrow["f"', from=2-2, to=2-3]
			\arrow[""{name=0, anchor=center, inner sep=0}, "{\pull}"', from=1-2, to=2-2]
			\arrow[""{name=0p, anchor=center, inner sep=0}, phantom, from=1-2, to=2-2, start anchor=center, end anchor=center]
			\arrow[Rightarrow, no head, from=1-1, to=2-1]
			\arrow[""{name=1, anchor=center, inner sep=0}, Rightarrow, no head, from=1-3, to=2-3]
			\arrow[""{name=1p, anchor=center, inner sep=0}, phantom, from=1-3, to=2-3, start anchor=center, end anchor=center]
			\arrow["\rho", shorten <=12pt, shorten >=12pt, Rightarrow, from=0p, to=1p]
		\end{tikzcd}
	\end{equation}
	we have an equation $\rho = f\cod \counitpull$; but this follows immediately from the fact that $\counitpull$ satisfies the first of \eqref{eqn:unit.counit.of.arrow.algebra}.
\end{proof}

Putting together this and \cref{lem:lax.maps.kleisli}, we obtain the main result of this section:

\begin{thm}
\label{thm:main.iso}
	For each $\Ba, \Ca \in \Kb$, there is an isomorphism of 2-categories:
	\begin{equation}
	\label{eqn:main.iso}
		\InclusionTrifun_{\Ba,\Ca}:\FibSpan^{\Rightarrow}(\Ba, \Ca) \isoto \Alg(\Ba^{\downarrow} \spancomp -,\, \Kl(- \spancomp \Ca^\downarrow,\, \DispSpan(\Ba, \Ca))).
	\end{equation}
\end{thm}

\begin{rmk}
\label{rmk:hidden.lift}
	Beware the polymorphic notation in \eqref{eqn:main.iso}: the two $\spancomp$ denote compositions in different 2-categories, and specifically $\Ba^\downarrow \spancomp -$ is denoting the lift of the arrow monad associated to $\Ba$ from $\DispSpan(\Ba, \Ca)$ to $\Kl(- \spancomp \Ca^\downarrow,\, \DispSpan(\Ba, \Ca))$.
\end{rmk}

\subsection[The free Kleisli cocompletion of the tricategory of left-displayed spans]{The free Kleisli cocompletion of $\DispSpan$}
\label{sec:trikleisli.comp}

For anyone acquainted with Lack and Street's \emph{The formal theory of monads II} \cite{lack_formal_2002}, the right hand side of \cref{eqn:main.iso} will certainly look familiar.
That expression closely resembles the characterization of hom-categories of the cocompletion of a 2-category under Kleisli objects.
However, we cannot appeal to \emph{ibid.} \emph{tout court}, since we work one dimension higher---indeed, \cref{eqn:main.iso} concerns hom-2-categories and not categories.
Luckily, Miranda has devoted the second part of his thesis \cite{miranda_topics_2024} to Kleisli completions for tricategories.
Under the assumption of a very reasonable conjecture regarding the form of free tricocompletions, he describes the free Kleisli completion of a tricategory in Section 7.2:

\begin{defn}
\label{defn:trikleisli.cocomp}
	For a tricategory $\Kc$, its \textbf{free Kleisli completion} is a trifunctor $k: \Kc \to \KL(\Kc)$ enjoying the following universal property: for each tricategory $\Hc$ which admits all Kleisli objects of pseudomonads, the trifunctor $k^* : [\KL(\Kc), \Hc] \to [\Kc, \Hc]_{\Kl\text{-cont}}$ is an adjoint triequivalence, where $[\Kc, \Hc]_{\Kl\text{-cont}}$ is the tricategory of trifunctors preserving Kleisli objects.
\end{defn}

Conjecturally,\footnote{Miranda showed that this presentation of $\KL(\Kc)$ makes $k^*$ triessentially surjective, hence the conjectural part only concerns the fact it can be promoted to an adjoint triequivalence.} the tricategory $\KL(\Kc)$ can be presented as the tricategory whose objects are pairs $(\Ab,s)$ where $s:\Ab \to \Ab$ is a pseudomonad on $\Ab \in \Kc$ and hom-bicategories are
\begin{equation}
	\KL(\Kc)((\Ab,s), (\Bb,t)) := \Alg(\,\Kl(\Kc(s,t),\,\Kc(s, \Bb)),\ \Kl(\Kc(\Ab, t),\, \Kc(\Ab, \Bb))),
\end{equation}
where ${\Kc(-, \Ab) : \Kc\op \to \Bicat}$ is a representable trifunctor, algebras and Kleisli objects are computed in $[\Kc\op, \Bicat]$, and the monad $\Kl(\Kc(s,t),\Kc(s, \Bb))$ is the one $\Kc(s,\Bb)$ induces on the Kleisli object of $\Kc(\Ab, t)$.
Indeed, slightly abusing notation (cf. \cref{rmk:hidden.lift} above), we can write the right hand side of the latter definition as
\begin{equation}
\label{eqn:kleisli.tricomp.homs}
	\KL(\Kc)((\Ab,s), (\Bb,t)) := \Alg(- s,\, \Kl(t -,\, \Kc(\Ab, \Bb))).
\end{equation}

In unpacking this definition below, we choose a specific presentation for 2- and 3-cells which Miranda (following Lack and Street) calls \emph{reduced}: they correspond to presenting 1-cells in the Kleisli bicategory of $\Kc(\Ab, s)$ as we did in \cref{defn:Kleisli.2.cat} as opposed to considering them as morphisms of free (pseudo)algebras.

\begin{rmk}
\label{rmk:trikleisli.unpacked}
	Explicitly, a 1-cell $f:(\Ab,s) \to (\Bb, t)$ in $\KL(\Kc)$ amounts to a 1-cell $f:\Ab \to \Bb$ in $\Kc$ together with a 2-cell $w_f: fs \twoto tf$ (called \textbf{intertwiner}) and two invertible 3-cells $\mapunitor$ and $\mapmultiplicator$ witnessing the commutativity of $w_f$ with the unit and multiplication of $s$ and $t$, and this data satisfy coherence equations (details can be found in \cite[§7.2.1~and~§7.2.3]{miranda_topics_2024}).
	Notice how $w_f$ is really a $t$-Kleisli morphism equipping $f$ with an $s$-pseudoalgebra structure.
	The 3-cells $\mapunitor$ and $\mapmultiplicator$ witness exactly the lawfulness of such structure.

	To compose 1-cells $(\Ab, s) \nto{(f,w_f)} (\Bb, t) \nto{(g,w_g)} (\Comonad, u)$, one composes the underlying 1-cells in $\Kc$ and endows it with the intertwiner $gfs \nto{gw_f} gtf \nto{w_g f} ugf$.
	For details on how the data pertaining $\mapunitor$ and $\mapmultiplicator$ composes and a proof that this composition is well-defined and unital and associative, we refer to \cite{gambino_formal_2021}.

	A 2-cell $\alpha :(f',w_{f'}) \twoto (f,w_f) : (\Ab, s) \to (\Bb, t)$ is a 2-cell (in $\Kc$) $\alpha : f' \twoto tf$ plus the data and axioms of pseudomorphism of $s$-algebras in the Kleisli bicategory $\Kl(\Kc(\Ab, t),\Kc(\Ab, \Bb))$.
	These compose (vertically) in a Kleisli fashion, i.e.~$\alpha \mathbin{\underset{\KL}\cdot} \beta := \multmnd \cdot (t\alpha) \cdot \beta$ where $\multmnd$ is the multiplication of $t$.
	Horizontal composition of $\alpha :(f',w_{f'}) \twoto (f,w_f) : (\Ab, s) \to (\Bb, t)$ and $\beta :(g',w_{g'}) \twoto (g,w_g) : (\Bb, t) \to (\Comonad, u)$ is the 2-cell $\beta\alpha : (g'f', w_{g'f'}) \twoto (gf, w_{gf})$ defined as
	\begin{equation}
	\label{eqn:horizontal.composition.in.KL}
		g'f' \nto{g'\alpha} g'tf \nto{w_{g'}f} sg'f \nto{\beta} ssgf \nto{\multmnd gf} sgf
	\end{equation}
	where $\multmnd$ denotes the multiplication of $s$.

	Finally, a 3-cell $\Omega:\alpha' \threeto \alpha$ is a 3-cell in $\Kc$ between the underlying 2-cells of $\alpha'$ and $\alpha$ that moreover satisfies the equations exhibiting $\Omega$ as a 2-cell of pseudoalgebras.
\end{rmk}


Comparing \eqref{eqn:kleisli.tricomp.homs} with \cref{thm:main.iso}, we see that the latter appears to describe a fully faithful inclusion $\InclusionTrifun : \FibSpan^\twoto \into \KL(\DispSpan^=)$ exhibiting left-fibrant spans as a full sub-tricategory of the Kleisli completion of the tricategory of spans and strict maps.

In fact, \cref{thm:main.iso} tells us that left-fibrant spans are spans $\Ba \nfrom{p} \Ea \nto{f} \Ca$ equipped with an algebra structure in a Kleisli category (for, respectively, $\Ba^\downarrow$ and $\Ca^\downarrow$) which are intertwined by a map
\begin{eqalign}
\label{eqn:arrow.monads.intertwiner}
	\intertwiner_{(p,f)} : \Ba^\downarrow \spancomp \Ea \ &\longto\ \Ea \spancomp \Ca^\downarrow\\[1ex]
	\left(\begin{tikzcd}[cramped, row sep=scriptsize] c \arrow[swap]{d}{\beta}\\[1.75ex] p(e)\end{tikzcd}, e \in \Ea\right) &\longmapsto
	\left(\pull_p(\beta) \in \Ea,\begin{tikzcd}[cramped, row sep=scriptsize] f(\pull_p(\beta)) \arrow{d}{f(\lift_p(\beta))}\\[1.75ex] f(e)\end{tikzcd}\right)
\end{eqalign}

\begin{rmk}
	We emphasize here that the intertwiner $\intertwiner_{(p, \action)}$ associated to a contextad $\fibcolaxaction$ appears crucially in the definition of composition in $\Ctx(\action)$. Recall (\cref{defn:ctx.dbl.cat}) that loose composition in $\Ctx(\action)$ of loose arrows $f : A \action P \to B$ and $g : B \action Q \to C$ is defined by
	\begin{equation}
		A \action (P \combine f^*Q) \xto{\coassociator} (A \action P) \action f^*Q \xto{f \action Q} B \action Q \xto{g} C
	\end{equation}
	The map $f \action Q : (A \action P) \action f^*Q \to B \action Q$ here is the second component of $\intertwiner_{(p, \action)}(f, Q)$---specifically, $\action(\lift_p(f))$.
\end{rmk}

Moreover, the fibration structure on $p$ allows us to construct invertible 2-cells witnessing the commutativity of these diagrams:
\begin{equation}
	\begin{tikzcd}[ampersand replacement=\&]
		\& \Ea \\
		{\Ba^\downarrow \spancomp \Ea} \&\& {\Ea \spancomp \Ca^\downarrow}
		\arrow[""{name=0, anchor=center, inner sep=0}, "{\internal{\id_p}}"', from=1-2, to=2-1]
		\arrow[""{name=1, anchor=center, inner sep=0}, "{\Ea \spancomp \unitarrow_{\Ca}}", from=1-2, to=2-3]
		\arrow["\intertwiner_{(p,f)}"', from=2-1, to=2-3]
		\arrow["\mapunitor", shift right=2, shorten <=9pt, shorten >=9pt, Rightarrow, from=0, to=1]
	\end{tikzcd}
	\qquad
	\begin{tikzcd}[ampersand replacement=\&]
		{\Ba^\downarrow \spancomp \Ba^\downarrow \spancomp \Ea} \& {\Ba^\downarrow \spancomp \Ea \spancomp \Ca^\downarrow} \\
		\& {\Ea \spancomp \Ca^\downarrow \spancomp \Ca^\downarrow} \\
		{\Ba^\downarrow \spancomp \Ea} \& {\Ea \spancomp \Ca^\downarrow}
		\arrow["\intertwiner_{(p,f)}"', from=3-1, to=3-2]
		\arrow["{\Ba^\downarrow \spancomp \intertwiner_{(p,f)}}", from=1-1, to=1-2]
		\arrow["{\intertwiner_{(p,f)} \spancomp \Ca^\downarrow}", from=1-2, to=2-2]
		\arrow[""{name=0, anchor=center, inner sep=0}, "{\multarrow_{\Ba} \spancomp \Ea}"', from=1-1, to=3-1]
		\arrow["{\Ea \spancomp \multarrow_{\Ca}}", from=2-2, to=3-2]
		\arrow["\mapmultiplicator", shorten <=11pt, Rightarrow, from=0, to=2-2]
	\end{tikzcd}
\end{equation}

In this section we set out to prove that this correspondence, in fact, exists and extends to a fully faithful trifunctor $\InclusionTrifun : \FibSpan^\twoto \into \KL(\DispSpan^=)$.
In fact, the first thing we do is prove that the domain of this trifunctor is a well-defined tricategory.

Let us mention, from the outset, that since arrow monads are colax idempotent (\cref{defn:colax.idempotent.monad}) and since the maps in $\KL(\DispSpan^=)$ exhibit pseudoalgebras for these (as noted above), we get that being a left-fibrant span is a \emph{property} of a span.
This greatly simplifies calculations.

\begin{lem}
\label{lem:iota.composition}
	For each $\Ba,\Ca,\Da \in \Kb$, there is a strictly commutative diagram:
	\begin{equation}
		\begin{tikzcd}
			{\FibSpan^\twoto(\Ba, \Ca) \times \FibSpan^\twoto(\Ca, \Da)} & {\FibSpan^\twoto(\Ba, \Da)} \\
			{\KL(\DispSpan^=)(\Ba^\downarrow, \Ca^\downarrow) \times \KL(\DispSpan^=)(\Ca^\downarrow, \Da^\downarrow)} & {\KL(\DispSpan^=)(\Ba^\downarrow, \Da^\downarrow)}
			\arrow["\spancomp", from=1-1, to=1-2]
			\arrow["\then", from=2-1, to=2-2]
			\arrow["\InclusionTrifun_{\Ba, \Da}", from=1-2, to=2-2]
			\arrow["\InclusionTrifun_{\Ba,\Ca} \times \InclusionTrifun_{\Ca,\Da}"', from=1-1, to=2-1]
		\end{tikzcd}
	\end{equation}
	where $\InclusionTrifun$ is the isomorphism of 2-categories of \cref{thm:main.iso}, the top map is the composition operation of \cref{lem:fspan.composition}, and the bottom map $\then$ is composition in $\KL$ as described just above in \cref{rmk:trikleisli.unpacked}.
\end{lem}
\begin{proof}
	This proof amounts to verify that the \emph{ad hoc} composition we defined for left-fibrant spans coincides with that in $\KL(\DispSpan^=)$.

	On objects, the diagram commutes since, as noted above, being a left-fibrant span is a property-like structure and thus composition in $\KL(\DispSpan^=)$ (which, on the underlying spans, is just composition of spans) produces the same algebra structure as composition in $\FibSpan^\twoto$ which correspond to the fibration structure on the left leg given by the composition of two fibrations.

	Consider now a composable pair of 1-cells in $\FibSpan^\twoto$.
	For ease of exposition, we are going to consider the same as in \eqref{eqn:fibspan-comp} above (from which we also borrow notation):
	\begin{equation}
		\begin{tikzcd}[ampersand replacement=\&, row sep=scriptsize]
			\&\& {\Ea'} \&\&\&\& {\Fa'} \\
			\&\& \Ea \&\&\&\& \Fa \\
			\Ba \&\&\&\& \Ca \&\&\&\& \Da
			\arrow["p", two heads, from=2-3, to=3-1]
			\arrow["f"', from=2-3, to=3-5]
			\arrow["q", two heads, from=2-7, to=3-5]
			\arrow["g"', from=2-7, to=3-9]
			\arrow["{p'}"', two heads, from=1-3, to=3-1]
			\arrow[""{name=0, anchor=center, inner sep=0}, "{f'}", from=1-3, to=3-5]
			\arrow[""{name=0p, anchor=center, inner sep=0}, phantom, from=1-3, to=3-5, start anchor=center, end anchor=center]
			\arrow["k"', from=1-3, to=2-3]
			\arrow[""{name=1, anchor=center, inner sep=0}, "{g'}", from=1-7, to=3-9]
			\arrow[""{name=1p, anchor=center, inner sep=0}, phantom, from=1-7, to=3-9, start anchor=center, end anchor=center]
			\arrow["{{h}}", from=1-7, to=2-7]
			\arrow["{q'}"', two heads, from=1-7, to=3-5]
			\arrow["\varphi", shorten >=5pt, Rightarrow, from=2-3, to=0p]
			\arrow["\psi", shorten >=5pt, Rightarrow, from=2-7, to=1p]
		\end{tikzcd}
	\end{equation}
	We know from \cref{lem:lax.maps.kleisli} that these correspond to
	\begin{equation}
		\begin{tikzcd}[ampersand replacement=\&, row sep=scriptsize]
			\&\& {\Ea'} \&\&\&\& {\Fa'} \\
			\&\& {f \comma\Ea} \&\&\&\& {g \comma \Fa} \\
			\Ba \&\&\&\& \Ca \&\&\&\& \Da
			\arrow["p\dom", two heads, from=2-3, to=3-1]
			\arrow["\cod"', from=2-3, to=3-5]
			\arrow["q\dom", two heads, from=2-7, to=3-5]
			\arrow["\cod"', from=2-7, to=3-9]
			\arrow["{p'}"', two heads, from=1-3, to=3-1]
			\arrow["{f'}", from=1-3, to=3-5]
			\arrow["{{\internal{\varphi}}}"', from=1-3, to=2-3]
			\arrow["{g'}", from=1-7, to=3-9]
			\arrow["{\internal{\psi}}"', from=1-7, to=2-7]
			\arrow["{q'}"', two heads, from=1-7, to=3-5]
		\end{tikzcd}
	\end{equation}
	Despite the name, $\internal{\varphi}$ contains also the data of $k$, and likewise for $h$ and $\internal{\psi}$.

	The composite of these two maps in $\KL(\DispSpan^=)$ is a strict map of spans which on apexes is
	\begin{equation}
		\begin{tikzcd}[ampersand replacement=\&]
			{\Ea' \spancomp \Fa'} \& {\Ea \spancomp \Ca^\downarrow \spancomp \Fa'} \&[-2ex] {\Ea \spancomp \Fa' \spancomp \Da^\downarrow} \& {\Ea \spancomp \Fa' \spancomp \Da^\downarrow \spancomp \Da^\downarrow} \& {(\Ea \spancomp \Fa) \spancomp \Da^\downarrow}
			\arrow["{\internal{\varphi} \spancomp \Fa'}"', from=1-1, to=1-2]
			\arrow["{\internal{\psi} \internal{\varphi}}", curve={height=-26pt}, dashed, from=1-1, to=1-5, shift left=1.5]
			\arrow["{\Ea \spancomp \internal{\psi} \spancomp \Da}"', from=1-3, to=1-4]
			\arrow["{\Ea \spancomp \Fa \spancomp \multarrow_{\Da}}"', from=1-4, to=1-5]
			\arrow["{\Ea \spancomp \intertwiner_{(q',g')}}"', from=1-2, to=1-3]
		\end{tikzcd}
	\end{equation}
	which can be verified to amount to
	\begin{equation}
		\left( \underbrace{(kp_1',\ h\pull_{q'}(\varphi p_1'))}_{k \spancomp_\varphi h},\ \underbrace{gh\pull_{q'}(\varphi p_1') \nto{\psi_{\pull_{q'}(\varphi p_1')}} g'\pull_{q'}(\varphi p_1') \nto{g'\lift_{q'}(\varphi p_1')} g'p_2'}_{\varphi \spancomp_\varphi \psi}\right)
	\end{equation}
	which is precisely $\internal{\varphi \spancomp_\varphi \psi}$, proving strict commutativity on 1-cells.

	Analogous considerations prove functoriality on 2-cells too.
\end{proof}

In order to conclude that $\FibSpan^\twoto$ is a particular subtricategory of $\KL(\DispSpan^=)$, we should first complete the definition of a tricategory structure on $\FibSpan^\twoto$, and then show that $\InclusionTrifun$ is a fully faithful trifunctor, hence an inclusion, of $\FibSpan^\twoto$ into $\KL(\DispSpan^=)$.

However, this is unnecessarily convoluted: we would be toiling away to define a whole tricategory structure on $\FibSpan^\twoto$ only to show it is the same as one we already have, namely that coming from $\KL(\DispSpan^=)$.
Instead, we kill two birds with one stone by defining the structure of $\FibSpan^\twoto$ to be the one induced by $\InclusionTrifun$, reassured by \cref{lem:iota.composition} that the composition and identities so defined agree with those of $\FibSpan^\twoto$.

The means to prove the following theorem were suggested to us by Adrian Miranda:

\begin{thm}
\label{thm:fSpan.tricat}
	There is a tricategory $\FibSpan^\twoto$ whose objects are those of $\Kb$, hom-bicategories (which are in fact strict 2-categories) are those defined in \cref{defn:locally.span}, whose composition is that defined in \cref{lem:fspan.composition}, and identities are identity spans.
	With this structure, $\InclusionTrifun: \FibSpan^\twoto \into \KL(\DispSpan^=)$ is a fully faithful trifunctor.
\end{thm}
\begin{proof}
	Both statements are a direct application of \emph{transport of structure} \cite[Theorem~7.22]{gurski_coherence_2013} along the families of biequivalences (in fact, isomorphisms of 2-categories) defined in \cref{thm:main.iso}.
	Using \cref{lem:iota.composition} and \emph{change of composition} \cite[Theorem~7.23]{gurski_coherence_2013}, we see that we can consider $\FibSpan^\twoto$ to be endowed with the composition law expounded in \cref{lem:fspan.composition}.
\end{proof}

\subsection[The Ctx construction as a wreath product]{The $\Ctx$ construction as a wreath product.}
\label{sec:ctx.as.wreath.product}

Recall that a contextad is, by \cref{defn:colax.fibred.action}, a pseudomonad in $\FibSpan^\twoto\Paradise$.
Therefore, the objects of $\KL(\FibSpan^\twoto\Paradise)$ are contextads (and the morphisms include the contextad morphisms as those maps whose underlying span has left leg an identity).
But as we just saw in \cref{thm:fSpan.tricat}, $\FibSpan^\twoto\Paradise$ is a full sub-tricategory of $\KL(\DispSpan^=\Paradise)$ witnessed by sending $\acted \in \Cosmos$ to the arrow pseudomonad $\acted^{\downarrow}$.
Therefore, we can apply $\KL$ to this inclusion to see every contextad as an object of $\KL(\KL(\DispSpan^=\Paradise))$---that is, as a \emph{wreath}\footnotemark~around an arrow pseudomonad.
\footnotetext{Beware that these are duals to Lack and Street's `wreaths' from \cite{lack_formal_2002}, since their wreaths come from the Eilenberg--Moore completion.}
Our main observation is that the $\Ctx$ construction, and therefore also its special cases, arises as the wreath product of this wreath.

\begin{thm}
\label{thm:ctx.const.on.objects}
	There is a trifunctor
	\begin{equation}
		\begin{tikzcd}
			\KL(\FibSpan^\twoto\Paradise) \arrow[r, hook, "\KL\InclusionTrifun"] &[5ex] \KL(\KL(\DispSpan\Paradise)) \arrow[r, "\wreath"] &[-5ex] \KL(\DispSpan\Paradise)
		\end{tikzcd}
	\end{equation}
	which, on objects, sends contextads to their double category of contextful arrows.
\end{thm}

\begin{rmk}
	We will take for granted that the wreath product of pseudomonads is given by essentially the same construction as the wreath product of monads.
	Indeed, if instead of Kleisli-cocompleting the tricategory $\DispSpan$, we Kleisli completed a 2-category $\Cb$ considered as a tricategory, then we must end up with the Lack and Street's construction because pseudomonads in a 2-category are simply monads in that 2-category, and so on (that is, all the 3-cells are equalities).
\end{rmk}

We can now see that the wreath product of such a wreath will be the pseudomonad in $\DispSpan\Paradise$ whose underlying span is $\acted \leftarrow \actor \spancomp \acted^{\downarrow} \to \acted$ equipped with unit $(\combineunit, \counitor) : \acted \to \actor \spancomp \acted^{\downarrow}$ and the following multiplication:
\begin{equation}
	\left(\actor \spancomp \acted^{\downarrow}\right)^{\spancomp 2}
	\xto{\actor \spancomp \intertwiner \spancomp \acted^{\downarrow}}
	\actor^{\spancomp 2} \spancomp \left(\acted^{\downarrow}\right)^{\spancomp 2}
	\xto{\actor \spancomp \actor \spancomp \multarrow}
	\actor^{\spancomp 2} \spancomp \acted^{\downarrow}
	\xto{(\combine, \coassociator) \spancomp \acted^{\downarrow} }
	\actor \spancomp \left(\acted^{\downarrow}\right)^{\spancomp 2}
	\xto{\actor \spancomp \multarrow}
	\actor \spancomp \acted^{\downarrow}
\end{equation}
Elementwise, this is:
\begin{eqalign}
\label{eqn:para.comp}
	\big(\lens{P}{A},&\ A \action P \xto{f} B,\ \lens{Q}{B},\ B \action Q \xto{g} C\big)\\
	&\mapsto
	\left(\lens{P}{A},\ \lens{f^*Q}{A},\ \left(A \action P\right) \action f^* Q \xto{f \action Q} B \action Q,\ B \action Q \xto{g} C \right) \\
	&\mapsto
	\left(\lens{P}{A},\ \lens{f^*Q}{A},\ \left(A \action P\right) \action f^* Q \xto{f \action Q} B \action Q \xto{g} C \right) \\
	&\mapsto
	\left(\lens{P \combine f^*Q}{A},\ A \action \left(P \combine f^*Q\right) \xto{\coassociator} \left(A \action P\right) \action f^*Q,\ \left(A \action P\right) \action f^*Q \xto{f \action Q} B \action Q \xto{g} C \right)\\
	&\mapsto \left(\lens{P \combine f^*Q}{A},\ A \action \left(P \combine f^*Q\right) \xto{\coassociator} \left(A \action P\right) \action f^*Q \xto{f \action Q} B \action Q \xto{g} C \right)
\end{eqalign}

This is precisely the way morphisms in Kleisli, Span and Para constructions compose, and indeed the way we defined loose composition in the double category $\Ctx(\action)$ (\cref{defn:ctx.dbl.cat}).

This explains the action on objects of the trifunctor of \cref{thm:ctx.const.on.objects}.
But what about its action on morphisms?
Not all 1-cells in $\KL(\DispSpan)$ correspond to double functors.
This is a well-known fact (see \cite{lack_formal_2002,miranda_topics_2024}) and should be blamed on the lack of distinction between `tight' and `loose' 1-cells in the tricategory of spans we are using.
Indeed, such settings are better modelled by fibrant double categories (in our case, fibrant double 2-categories).
However, to recover the correct notion of functor between double categories it suffices to restrict to those 1-\nobreak{cells}  carried by spans whose left leg is an identity.
We follow \cite{miranda_topics_2024} in calling these spans \emph{functorial}, and we denote the locally full subtricategory they form by $\fun \subseteq \DispSpan$.

\begin{rmk}
	In unpacking definitons below, we refer back to \cref{rmk:trikleisli.unpacked}, but note that there we denote composition in algebraic order while here, since we apply $\KL$ to tricategories of spans, we use diagrammatic order composition instead.
\end{rmk}

\begin{notation}
\label{not:tight.maps}
	We adopt the notation of \cite[§2.4]{lack_formal_2002} and for $\mathsf{tight}$ a locally full subcategory\footnotemark~of $\Kc$, denote by $\KL(\Kc, \mathsf{tight})$ the wide locally full subtricategory of $\KL(\Kc)$ obtained by restricting the latter to those 1-cells $(f,\intertwiner_f)$ such that $f \in \mathsf{tight}$.
	\footnotetext{This means $\mathsf{tight}$ is a choice of 1-cells (containing identities and closed under composition) with all 2-cells and 3-cells between those being chosen too.}
\end{notation}


We record an interesting fact about $\Ctx(\action)$, consequence of the following general fact:

\begin{lem}
	Let $(t, \intertwiner_t)$ be an endomorphism of $(\dblcat A, s)$ in $\KL(\trplcat K)$.
	Then $s \wreath t$ is an $s$-$s$-bimodule.
\end{lem}
\begin{proof}
	The fact $(t, \intertwiner_t)$ is a 1-cell in $\KL(\trplcat K)$ means it is equipped with an $s$-Kleisli right $s$-algebra structure $\intertwiner_t:ts \to st$, as per \cref{rmk:trikleisli.unpacked}.
	This is equivalent to a plain right $s$-algebra structure on $st$, given by
	\begin{equation}
		r := (st)s \isoto s(ts) \xto{s\intertwiner_t} s(st) \isoto (ss)t \xto{\multmnd t} st
	\end{equation}
	On the left, we can just use the free left $s$-algebra on $t$:
	\begin{equation}
		l := s(st) \isoto (ss)t \xto{\multmnd t} st
	\end{equation}
	If we work in the strictification of the relevant hom-2-category (that is, ignoring the associators for a moment), we can see why these two structures are compatible:
	\begin{equation}
		\begin{tikzcd}[ampersand replacement=\&, row sep=scriptsize]
			\& ssts \\
			ssst \&\& sts \\
			sst \&\& sst \\
			\& st
			\arrow["{ss\intertwiner_t}"', from=1-2, to=2-1]
			\arrow["{\multmnd ts}", from=1-2, to=2-3]
			\arrow["{s\multmnd t}"', from=2-1, to=3-1]
			\arrow["{\multmnd st}"{description}, dashed, from=2-1, to=3-3]
			\arrow["{s\intertwiner_t}", from=2-3, to=3-3]
			\arrow["\sim", shorten <=35pt, shorten >=35pt, shift right=2, Rightarrow, from=3-1, to=3-3]
			\arrow["{\multmnd t}"', from=3-1, to=4-2]
			\arrow["{\multmnd t}", from=3-3, to=4-2]
		\end{tikzcd}
	\end{equation}
	The top square commutes by interchange, and the bottom square is filled by the pentagonator of $s$.
	We conclude $st$ is an $s$-$s$-bimodule.
\end{proof}

When $\trplcat K = \DispSpan\Paradise$ and $s$ is the arrow monad of $\acted$ and $t$ is a contextad $\fibcolaxaction$, the bimodule\footnotemark~structure of $\acted \from \Ctx(\action)_1 \to \acted$ is given by:
\footnotetext{Since $\fibcolaxaction$ is not just an endomorphism in $\DispSpan$ but a whole pseudomonad, the bimodule we get inherits such a structure, meaning $\Ctx(\action)$ is a pseudocategory---but we know this already.}
\begin{eqalign}
	\acted^\downarrow \spancomp \Ctx(\action)_1 \spancomp \acted^\downarrow &\longto \Ctx(\action)_1\\
	\left(A \xto{f} B,\ B \action P \xto{g} C,\ C \xto{h} D\right) &\longmapsto A \action f^*P \xto{f \action P} B \action P \xto{g} C \nto{h} D.
\end{eqalign}
Thus the span underlying every pseudocategory obtained as a $\Ctx$ construction is in fact a two-sided fibration \cite[Theorem~7.1.4]{riehl_elements_2022}, which from the point of view of the double category it describes it means all tight maps have companions:

\begin{thm}
\label{thm:ctx.has.companions}
	Let $\fibcolaxaction$ be a contextad on a category.
	Then $\Ctx(\action)$ admits all companions.
\end{thm}

Explicitly, the companion $f_\top$ of a map $f:A \to B$ in $\acted$ is $(\combineunit, f \counitor)$.
The unit and counit of the companionship are given by
\begin{equation}
	\begin{tikzcd}[ampersand replacement=\&]
		A \& B \\
		B \& B
		\arrow[""{name=0, anchor=center, inner sep=0}, "{(\combineunit, f\counitor)}", "\shortmid"{marking}, from=1-1, to=1-2]
		\arrow[""{name=0p, anchor=center, inner sep=0}, phantom, from=1-1, to=1-2, start anchor=center, end anchor=center]
		\arrow["f"', from=1-1, to=2-1]
		\arrow[Rightarrow, no head, from=1-2, to=2-2]
		\arrow[""{name=1, anchor=center, inner sep=0}, "\shortmid"{marking}, Rightarrow, no head, from=2-1, to=2-2]
		\arrow[""{name=1p, anchor=center, inner sep=0}, phantom, from=2-1, to=2-2, start anchor=center, end anchor=center]
		\arrow["{\overline\counit}", shorten <=4pt, shorten >=4pt, Rightarrow, from=0p, to=1p]
	\end{tikzcd}
	\qquad\qquad
	\begin{tikzcd}[ampersand replacement=\&]
		A \& A \\
		A \& B
		\arrow[""{name=0, anchor=center, inner sep=0}, "\shortmid"{marking}, Rightarrow, no head, from=1-1, to=1-2]
		\arrow[Rightarrow, no head, from=1-1, to=2-1]
		\arrow["f", from=1-2, to=2-2]
		\arrow[""{name=1, anchor=center, inner sep=0}, "{(\combineunit, f\counitor)}"', "\shortmid"{marking}, from=2-1, to=2-2]
		\arrow["{\overline\unit}", shorten <=4pt, shorten >=4pt, Rightarrow, from=0, to=1]
	\end{tikzcd}
\end{equation}
where $\overline\counit$ and $\overline\unit$ are the identity morphisms $\combineunit \equalto \combineunit$.

Moreover, if $f:A \to B$ is invertible then $f^{-1}_\top$ is a conjoint, as witnessed by the squares
\begin{equation}
\label{eq:conjoints}
	\begin{tikzcd}[ampersand replacement=\&]
		A \& A \\
		A \& B
		\arrow[""{name=0, anchor=center, inner sep=0}, "{(\combineunit, f^{-1}\counitor)}", "\shortmid"{marking}, from=1-1, to=1-2]
		\arrow[Rightarrow, no head, from=1-1, to=2-1]
		\arrow["f", from=1-2, to=2-2]
		\arrow[""{name=1, anchor=center, inner sep=0}, "\shortmid"{marking}, Rightarrow, no head, from=2-1, to=2-2]
		\arrow["{\overline\counit}"', shorten <=4pt, shorten >=4pt, Rightarrow, from=0, to=1]
	\end{tikzcd}
	\qquad\qquad
	\begin{tikzcd}[ampersand replacement=\&]
		A \& A \\
		A \& B
		\arrow[""{name=0, anchor=center, inner sep=0}, "\shortmid"{marking}, Rightarrow, no head, from=1-1, to=1-2]
		\arrow["f"', from=1-1, to=2-1]
		\arrow[Rightarrow, no head, from=1-2, to=2-2]
		\arrow[""{name=1, anchor=center, inner sep=0}, "{(\combineunit, f^{-1}\counitor)}"', "\shortmid"{marking}, from=2-1, to=2-2]
		\arrow["{\overline\unit}", shorten <=4pt, shorten >=4pt, Rightarrow, from=0, to=1]
	\end{tikzcd}
\end{equation}
which are also filled by identities.

\subsection{Functoriality}
The aim of this section is to define the trifunctor $\Ctx$, obtained by restricing the wreath product as described above.

\subsubsection{The tricategory of contextads}
We now give definitions of morphisms, transformations and modifications of contextads by unpacking the definition of 1-, 2- and 3-cell in $\KL(\FibSpan, \fun)\co$ but limited to those 1-cells which are carried by left-trivial spans (i.e.~whose left leg is an identity).
To understand why we are dualizing 2-cells, see \cref{rmk:pesky.co}.
For a comparison of these morphisms with those of comonad and plain actions of monoidal categories, see \cref{ex:maps.comonads,ex:maps.actegories}.

Throughout this section we fix two contextads $\fibcolaxaction$ and $\fibcolaxaction[']$ which we will abbreviate as just $\action$ and $\action'$.

\begin{defn}
\label{defn:contextad.1-cell}
	A \textbf{morphism of contextads} $\underline F:\action \to \action'$ is given by the following data:
	\begin{enumerate}
		\item A \textbf{map on contexts}
		\begin{equation}
			F:\acted \to \acted'
		\end{equation}
		\item A \textbf{map on extensions} $F^\flat:\actor \to \actor'$, forming with $F$ a cartesian square\footnotemark~with components
		\begin{equation}
			F^\flat_A : \actor_A \to \actor'_{FA},
		\end{equation}
		\footnotetext{As in: a map of fibrations over different bases.}
		\item A \textbf{lineator} $\lineator:\action'F^\flat \twoto F\action$, with components
		\begin{equation}
			\lineator_{{P \choose A}} : FA \action' F^\flat_A P \to F(A \action P).
		\end{equation}
	\end{enumerate}
	These three objects fit in the diagram below:
	\begin{equation}
	\label{eqn:morph.of.fib.colax.actions}
		\begin{tikzcd}[ampersand replacement=\&]
			\acted \&\& \actor \&\& \acted \\
			{\acted'} \&\& {\actor'} \&\& {\acted'}
			\arrow["F"', dashed, from=1-1, to=2-1]
			\arrow["p"', two heads, from=1-3, to=1-1]
			\arrow["\action", from=1-3, to=1-5]
			\arrow["{F^\flat}"', dashed, from=1-3, to=2-3]
			\arrow["F", dashed, from=1-5, to=2-5]
			\arrow["\lineator", shorten <=9pt, shorten >=14pt, Rightarrow, dashed, from=2-3, to=1-5]
			\arrow["{p'}", two heads, from=2-3, to=2-1]
			\arrow["{\action'}"', from=2-3, to=2-5]
		\end{tikzcd}
	\end{equation}
	\begin{enumerate}[resume]
		\item An invertible \textbf{unitor} $\mapunitor_A:F^\flat_A(\combineunit) \isoto \combineunit'_{FA}$ fitting in the diagram below left, which amounts to say $\mapunitor$ is $p'$-vertical and its components make the diagrams below right commute for each $A \in \acted$:
		      \begin{equation}
			      \label{eqn:unit.of.lineator}
			      \begin{tikzcd}[ampersand replacement=\&,sep=scriptsize]
				      \&\&\&\& \acted \\
				      \acted \&\&\& \actor \&\&\& \acted \\
				      \\
				      \&\&\&\& {\acted'} \\
				      \\
				      {\acted'} \&\&\& {\actor'} \&\&\& {\acted'}
				      \arrow["p"{description}, two heads, from=2-4, to=2-1]
				      \arrow["F"{description}, from=2-1, to=6-1]
				      \arrow["F"{description}, from=2-7, to=6-7]
				      \arrow["{p'}"{description}, two heads, from=6-4, to=6-1]
				      \arrow["{\action'}"{description}, from=6-4, to=6-7]
				      \arrow["{F^\flat}"{description}, from=2-4, to=6-4]
				      \arrow[""{name=0, anchor=center, inner sep=0}, "\combineunit"{description}, from=1-5, to=2-4]
				      \arrow[""{name=0p, anchor=center, inner sep=0}, phantom, from=1-5, to=2-4, start anchor=center, end anchor=center]
				      \arrow[""{name=1, anchor=center, inner sep=0}, "{\combineunit'}"{description}, from=4-5, to=6-4]
				      \arrow[""{name=1p, anchor=center, inner sep=0}, phantom, from=4-5, to=6-4, start anchor=center, end anchor=center]
				      \arrow["F"{description, pos=0.7}, from=1-5, to=4-5]
				      \arrow[Rightarrow, no head, from=1-5, to=2-1]
				      \arrow[Rightarrow, no head, from=4-5, to=6-1]
				      \arrow[""{name=2, anchor=center, inner sep=0}, Rightarrow, no head, from=4-5, to=6-7]
				      \arrow[""{name=3, anchor=center, inner sep=0}, Rightarrow, no head, from=1-5, to=2-7]
				      \arrow["\mapunitor"', shorten <=7pt, shorten >=7pt, Rightarrow, dashed, from=0p, to=1p]
				      \arrow["\action"{description}, from=2-4, to=2-7]
				      \arrow["\counitor"{description}, shorten <=13pt, shorten >=13pt, Rightarrow, from=2-4, to=3]
				      \arrow["{\counitor'}"{description}, shorten <=12pt, shorten >=6pt, Rightarrow, from=6-4, to=2]
				      \arrow["\lineator"{pos=0.7}, shift right=4, shorten <=11pt, shorten >=11pt, Rightarrow, from=6-4, to=2-7]
			      \end{tikzcd}
			      \qquad\qquad\qquad
			      \begin{tikzcd}[ampersand replacement=\&]
				      {FA \action' F^\flat \combineunit_A} \&\& {FA\action' \combineunit'_{FA}} \\
				      {F(A \action \combineunit_A)} \&\& FA
				      \arrow["{FA \action' \mapunitor}", from=1-1, to=1-3]
				      \arrow["\lineator"', from=1-1, to=2-1]
				      \arrow["{\counitor'}", from=1-3, to=2-3]
				      \arrow["{F\counitor}"', from=2-1, to=2-3]
			      \end{tikzcd}
		      \end{equation}
		\item An invertible \textbf{multiplicator} $\mapmultiplicator_{{P \choose A}, {Q \choose A \action P}} : F^\flat_A(P \combine Q) \isoto F^\flat_A P \combine' \lineator^* F^\flat_{A \action P} Q$ fitting in the diagram below left, which amounts to say $\mapmultiplicator$ is $p'$-vertical and its components make the diagrams below right commute for each pair $\lens{P}{A}, \lens{Q}{A \action P} \in \actor \spancomp \actor$:
		      \begin{equation}
			      \label{eqn:ass.of.lineator}
			      \begin{tikzcd}[ampersand replacement=\&,sep=scriptsize]
				      \&\&\&\& \actor \\
				      \acted \&\&\& {\actor \spancomp \actor} \&\&\& \acted \\
				      \\
				      \&\&\&\& {\actor'} \\
				      \\
				      {\acted'} \&\&\& {\actor' \spancomp \actor'} \&\&\& {\acted'}
				      \arrow["p"{description}, two heads, from=2-4, to=2-1]
				      \arrow["\action"{description}, from=2-4, to=2-7]
				      \arrow["F"{description}, from=2-1, to=6-1]
				      \arrow["F"{description}, from=2-7, to=6-7]
				      \arrow["{p'p'}"{description}, two heads, from=6-4, to=6-1]
				      \arrow["{\action'\action'}"{description}, from=6-4, to=6-7]
				      \arrow["{\lineator \spancomp_\lineator \lineator}"'{pos=0.7}, shift right=4, shorten <=12pt, shorten >=12pt, Rightarrow, from=6-4, to=2-7]
				      \arrow["{F^\flat \spancomp_\lineator F^\flat}"{description}, from=2-4, to=6-4]
				      \arrow[""{name=0, anchor=center, inner sep=0}, "\combine"{description}, from=2-4, to=1-5]
				      \arrow[""{name=0p, anchor=center, inner sep=0}, phantom, from=2-4, to=1-5, start anchor=center, end anchor=center]
				      \arrow[""{name=1, anchor=center, inner sep=0}, "{\combine'}"{description}, from=6-4, to=4-5]
				      \arrow[""{name=1p, anchor=center, inner sep=0}, phantom, from=6-4, to=4-5, start anchor=center, end anchor=center]
				      \arrow["{F^\flat}"{description, pos=0.7}, from=1-5, to=4-5]
				      \arrow["p"{description}, two heads, from=1-5, to=2-1]
				      \arrow["{p'}"{description}, two heads, from=4-5, to=6-1]
				      \arrow[""{name=2, anchor=center, inner sep=0}, "{\action'}"{description}, from=4-5, to=6-7]
				      \arrow[""{name=3, anchor=center, inner sep=0}, "\action"{description}, from=1-5, to=2-7]
				      \arrow["\lineator", Rightarrow, from=4-5, to=2-7]
				      \arrow["{\mapmultiplicator}"', shorten <=7pt, shorten >=7pt, Rightarrow, dashed, from=0p, to=1p]
				      \arrow["\coassociator"{description}, shorten <=11pt, shorten >=11pt, Leftarrow, from=2-4, to=3]
				      \arrow["{\coassociator'}"{description}, shorten <=14pt, shorten >=7pt, Leftarrow, from=6-4, to=2]
			      \end{tikzcd}
			      \qquad\quad
			      \begin{tikzcd}[ampersand replacement=\&, column sep=small]
				      {FA \action' F^\flat_A (P \combine Q)} \&[-5ex]\&[-5.5ex] {FA \action' (F^\flat_A P \combine' \lineator^*F^\flat_{A \action P}Q)} \\
				      {F(A \action (P \combine Q))} \&\& {(FA \action' F^\flat_A P) \action' \lineator^*F^\flat_{A \action P}Q} \\
				      \& {F((A \action P) \action Q)}
				      \arrow["{\lineator \spancomp_\lineator \lineator}", from=2-3, to=3-2]
				      \arrow["F\coassociator"', from=2-1, to=3-2]
				      \arrow["{\lineator}"', from=1-1, to=2-1]
				      \arrow["{\coassociator'}", from=1-3, to=2-3]
				      \arrow["{FA \action' \mapmultiplicator}", from=1-1, to=1-3]
			      \end{tikzcd}
		      \end{equation}
	\end{enumerate}
	Additionally, the unitor and associator satisfy triangular and pentagonal coherence laws relative to the unitors and associators of $\action$ and $\action'$, which can be found in \cite[Remark~7.2.1]{miranda_topics_2024}.
\end{defn}

\begin{rmk}
	Note the fact $(F, F^\flat)$ is a map of fibrations means $F^\flat$ has a cartesianator, thus the property-like structure of a natural family of isomorphisms indexed by maps $h:A \to A'$ in $\acted$:
	\begin{equation}
		\cartesianator^{F^\flat} : (Fh)^*(F^\flat_A P) \iso F^\flat_{A'}(h^*P)
	\end{equation}
\end{rmk}

\begin{rmk}
	This kind of morphism might be considered a categorified `1-cocycle' or `crossed homomorphism', from group cohomology.
	Specifically, $F^\flat$ is like a 1-cocycle of $\actor$ with coefficients in $\actor'$, where the cocycle equation is reified by the multiplicator $\mapmultiplicator$.
	Indeed, the multiplicator reifies an often-encountered `associativity condition' of the form $\phi(a, pq) = \phi(a, p)\phi(ap, q)$.
\end{rmk}

\begin{rmk}
	Technically speaking, \cref{defn:contextad.1-cell} describes \emph{strong} morphisms of contextads, i.e. those arising as strong morphisms of pseudoalgebras (recall \cref{rmk:trikleisli.unpacked}).
	The laxity we refer to here is that of $\mapunitor$ and $\mapmultiplicator$, the unitor and multiplicator, which in principle could be assumed to be non-invertible, yielding a notion of \textbf{(co)lax} morphism of contextads, or more commonly, be identities, yielding \textbf{strict} morphisms.

	Instead, $\lineator$ is non-invertible by default, hence, counterintuitively, \emph{strong} morphisms of contextads only \emph{laxly} preserve the action.
	We call \textbf{strongly} (resp. \textbf{strictly}) \textbf{linear} a morphism of contextads whose lineator is invertible (resp. an identity).
\end{rmk}

Next, we go to 2-morphisms.
Transformations of contextads are basically contextful natural transformations.

\begin{defn}
\label{defn:contextad.2-cell}
	Let $\underline F, \underline G : \action \to \action'$ be morphisms of contextad.
	A \textbf{transformation of contextads} $\underline \tau : \underline F \twoto \underline G$ consists of
	\begin{enumerate}
		\item a pair of 2-cells $(e, \tau)$, called the \textbf{extension} and the \textbf{transformation}, dashed in the diagram below left:
		      \begin{equation}
			      \begin{tikzcd}[ampersand replacement=\&]
				      \& \acted \\
				      {\acted'} \& {\actor'} \& {\acted'}
				      \arrow["F"', curve={height=12pt}, from=1-2, to=2-1]
				      \arrow["e"', dashed, from=1-2, to=2-2]
				      \arrow[""{name=0, anchor=center, inner sep=0}, "G", curve={height=-12pt}, from=1-2, to=2-3]
				      \arrow["{p'}", two heads, from=2-2, to=2-1]
				      \arrow["{\action'}"', from=2-2, to=2-3].
				      \arrow["{\tau}", shorten <=2pt, shorten >=4pt, Rightarrow, dashed, from=2-2, to=0]
			      \end{tikzcd}
			      \qquad\qquad
			      \begin{aligned}
					&e_A \in \actor'_{FA}\\[.5ex]
					&\tau_A : FA \action' e_A \longto GA
				  \end{aligned}
		      \end{equation}
		\item a $p'$-vertical invertible 3-cell, the \textbf{naturator}:
		\begin{equation}
			\begin{tikzcd}[ampersand replacement=\&]
				\actor \&\& {\actor' \spancomp \actor'} \\
				{\actor' \spancomp \actor'} \&\& {\actor'}
				\arrow["{(F^\flat\!,\ \pull \lineator^F)}", from=1-1, to=1-3]
				\arrow["{e \spancomp_{\tau} G^\flat}"{description}, from=1-1, to=2-1]
				\arrow["{\naturator}", shorten <=10pt, shorten >=14pt, Rightarrow, dashed, to=1-3, from=2-1]
				\arrow["{\combine'}"', from=1-3, to=2-3]
				\arrow["{\combine'}"{description}, from=2-1, to=2-3]
			\end{tikzcd}
			\qquad
			\naturator : e_A \combine' \tau^*G^\flat_A P \isoto F^\flat_AP\combine'{\lineator^F}^*e_{G(A \action P)}
		  \end{equation}
			where $\lineator^F$ denotes the lineator of $F$, which makes the following diagrams commute for each $\lens{P}{A} \in \actor$:
			\begin{equation}
			\label{eqn:coherence.of.naturator}
				\begin{tikzcd}[ampersand replacement=\&]
					{FA \action' (e_A \combine' \tau^* G^\flat_A P)} \&\& {FA \action' (F^\flat_A P \combine' {\lineator^F}^*e_{A \action P})} \\
					{(FA \action' e_A) \action' ({\tau}^* G^\flat_A P)} \&\& {(FA \action' F^\flat_A P) \action' {\lineator^F}^* e_{A \action P}} \\
					{GA \action' G^\flat_A P} \&\& {F(A \action P) \action' e_{A \action P}} \\
					\& {G(A \action P)}
					\arrow["{FA \action' \naturator}", dashed, from=1-1, to=1-3]
					\arrow["{\coassociator'}"', from=1-1, to=2-1]
					\arrow["{\coassociator'}", from=1-3, to=2-3]
					\arrow["{e_A \action' G^\flat_A P}"', from=2-1, to=3-1]
					\arrow["{\lift \lineator^F}", from=2-3, to=3-3]
					\arrow["{\lineator^G_{{P \choose A}}}"', from=3-1, to=4-2]
					\arrow["{\tau_{A \action P}}", from=3-3, to=4-2]
				\end{tikzcd}
			\end{equation}
	\end{enumerate}
	Moreover, $\underline \tau$ satisfies compatibility with unitor and associator of $F$ and $G$, in the form of the two equations appearing in \cite[Remark~7.2.5]{miranda_topics_2024}.
\end{defn}


Finally, 3-cells maps between the context extensions of the boundary transformations:

\begin{defn}
\label{defn:contextad.3-cell}
	Let $\underline \tau =(e,\tau)$, $\underline \sigma = (f, \sigma)$ be transformations of contextads $\underline F \twoto \underline G : \action \to \action'$.
	A \textbf{modification of contextads} $\omega:\underline \tau \threeto \underline \sigma$ is a $p'$-vertical 2-cell in $\Cosmos$ that fits in the commutative diagram below:
	\begin{equation}
		\begin{tikzcd}[ampersand replacement=\&]
			\&\& \acted \\
			\\
			{\acted'} \&\& {\actor'} \&\& {\acted'}
			\arrow["F"'{pos=0.6}, curve={height=24pt}, from=1-3, to=3-1]
			\arrow[""{name=0, anchor=center, inner sep=0}, "e"'{pos=0.2}, curve={height=18pt}, from=1-3, to=3-3]
			\arrow[""{name=0p, anchor=center, inner sep=0, pos=0.4}, phantom, from=1-3, to=3-3, start anchor=center, end anchor=center, curve={height=18pt}]
			\arrow[""{name=1, anchor=center, inner sep=0}, "f"{pos=0.9}, curve={height=-18pt}, from=1-3, to=3-3]
			\arrow[""{name=1p, anchor=center, inner sep=0, pos=0.7}, phantom, from=1-3, to=3-3, start anchor=center, end anchor=center, curve={height=-18pt}]
			\arrow[""{name=2, anchor=center, inner sep=0}, "G"{pos=0.6}, curve={height=-24pt}, from=1-3, to=3-5]
			\arrow[""{name=2p, anchor=center, inner sep=0,pos=0.7}, phantom, from=1-3, to=3-5, start anchor=center, end anchor=center, curve={height=-24pt}]
			\arrow["{{p'}}", two heads, from=3-3, to=3-1]
			\arrow["{{\action'}}"', from=3-3, to=3-5]
			\arrow["\omega"', curve={height=2pt}, shorten <=4pt, Rightarrow, dashed, from=0p, to=1p]
			\arrow["\tau", curve={height=-10pt}, shorten <=6pt, shorten >=7pt, Rightarrow, from=0p, to=2p]
			\arrow["\sigma"'{pos=0.4}, curve={height=6pt}, shorten <=6pt, shorten >=9pt, Rightarrow, from=1p, to=2p]
		\end{tikzcd}
		\qquad\qquad
		\omega_A : e_A \to f_A \in \actor'_{FA}
	\end{equation}
	This means, specifically, that the triangles
	\begin{equation}
		\begin{tikzcd}[ampersand replacement=\&]
			{FA \action' e_A} \&\& {FA \action' f_A} \\
			\& GA
			\arrow["{FA \action' \omega_A}", from=1-1, to=1-3]
			\arrow["\tau"', from=1-1, to=2-2]
			\arrow["\sigma", from=1-3, to=2-2]
		\end{tikzcd}
	\end{equation}
	commute.
	Additionally, it must commute with the naturators of $\underline \tau$ and $\underline \sigma$ (corresponding to the equation appearing in \cite[Remark~7.2.6]{miranda_topics_2024}):
	\begin{equation}
		\begin{tikzcd}[ampersand replacement=\&]
			{e_A \combine' \tau^* G^\flat P} \&\& {F^\flat P \combine' {\lineator^F}^*(e_{A \action P})} \\
			{f_A \combine' \sigma^* G^\flat P} \&\& {F^\flat P \combine' {\lineator^F}^*(f_{A \action P})}
			\arrow["{\omega_A \combine' \omega_A^* G^\flat P}"', from=1-1, to=2-1]
			\arrow["\naturator^\sigma"', from=2-1, to=2-3]
			\arrow["{G^\flat P \combine' {\lineator^G}^*\omega_{A \action P}}", from=1-3, to=2-3]
			\arrow["\naturator^\tau", from=1-1, to=1-3]
		\end{tikzcd}
	\end{equation}
	where $\lineator^F$ denotes the lineator of $F$ and $\naturator^\tau$/$\naturator^\sigma$ denote the respective naturators.
\end{defn}

\begin{defn}
\label{defn:tricat.of.colax.fibred.actions}
	We define the \textbf{tricategory of contextads on categories} in the paradise $\Paradise$:
	\begin{equation}
		\Cxd\Paradise := \KL(\FibSpan^\twoto\Paradise,\,\fun)\co.
	\end{equation}
	Concretely, its objects are contextads on categories (\cref{defn:colax.fibred.action}), its 1-cells are morphisms of contextads (\cref{defn:contextad.1-cell}), its 2-cells are transformations of contextads (\cref{defn:contextad.2-cell}), and its 3-cells are modifications of contextads (\cref{defn:contextad.3-cell}).
\end{defn}

\begin{rmk}
\label{rmk:pesky.co}
	The reason we take the local opposite (i.e.~we formally reverse 2-cells, the operation denoted as $(-)\co$) is to match the usual direction of natural transformations.
	We thus correct a gnarly mismatch in direction between \cite{lack_formal_2002}, \cite{miranda_topics_2024} and this work.
	Specifically, we chose to define arrow monads with left leg corresponding to domain and right leg to codomain.
	This choice leads, in turn, to consider pseudomonads in $\DispSpan\Paradise$ as internal categories oriented in the same way: their left leg displays the source, the right leg projects the target.
	However, employing this choice, the 2-cells in $\KL(\DispSpan\Paradise)$ between functorial morphisms of pseudomonads correspond to natural transformations in the opposite direction than what we expect (i.e.~a 2-cell $f \twoto g$ corresponds to a natural transformation $g \twoto f$).%
	\footnote{The reader should note that on \cite[254]{lack_formal_2002} the authors have missed this reversal. Instead, \cite{miranda_topics_2024} uses the different convention on the orientation of arrow monads which avoids this issue, though beware of the misprint in Proposition~7.3.6.}
\end{rmk}

\subsubsection{Examples of morphisms of contextads}
\begin{ex}
\label{ex:maps.comonads}
	Let $\Comonad$ be a comonad on $\acted$, and $\Comonad'$ a comonad on $\acted'$.
	A map of comonads $\underline F: \Comonad \to \Comonad'$ is a functor $F:\acted \to \acted'$ and a natural transformation $\lineator : \Comonad'F \twoto F\Comonad$ that commutes with the counits and comultiplication of the comonads:
	\begin{equation}
	\label{eqn:coh-comonad-mors}
		\begin{tikzcd}[ampersand replacement=\&,sep=scriptsize]
			{\Comonad' F} \&\& {F\Comonad} \\
			\& F
			\arrow["\lineator", from=1-1, to=1-3]
			\arrow["{\counit'}", from=2-2, to=1-1]
			\arrow["{F\counit}"', from=2-2, to=1-3]
		\end{tikzcd}
		\qquad\qquad
		\begin{tikzcd}[ampersand replacement=\&,sep=scriptsize]
			{\Comonad'\Comonad'F} \& {\Comonad'F\Comonad} \& {F\Comonad\Comonad} \\
			{\Comonad'F} \&\& {F\Comonad}
			\arrow["{\Comonad'\lineator}", from=1-1, to=1-2]
			\arrow["{\lineator\Comonad}", from=1-2, to=1-3]
			\arrow["{\coassociator'}", from=2-1, to=1-1]
			\arrow["\lineator"', from=2-1, to=2-3]
			\arrow["{F\coassociator}"', from=2-3, to=1-3]
		\end{tikzcd}
	\end{equation}
	We can see how \cref{defn:contextad.1-cell} instantiated for the contextads corresponding to $\Comonad$ and $\Comonad'$ (\cref{ex:comonad}) recovers this definition.
	First notice that the pair $(F, \lineator)$ is what remains of the data of such a morphism:
	\begin{equation}
		\begin{tikzcd}[ampersand replacement=\&]
			\acted \& \acted \& \acted \\
			{\acted'} \& {\acted'} \& {\acted'}
			\arrow["F"', dashed, from=1-1, to=2-1]
			\arrow[Rightarrow, no head, from=1-2, to=1-1]
			\arrow["\Comonad", from=1-2, to=1-3]
			\arrow["F"', dashed, from=1-2, to=2-2]
			\arrow["F", dashed, from=1-3, to=2-3]
			\arrow["\lineator", shorten <=4pt, shorten >=6pt, Rightarrow, dashed, from=2-2, to=1-3]
			\arrow[Rightarrow, no head, from=2-2, to=2-1]
			\arrow["{\Comonad'}"', from=2-2, to=2-3]
		\end{tikzcd}
	\end{equation}
	Second, the unitor and multiplicator must be trivial since they are $\id$-vertical.
	This, in turn, makes \cref{eqn:unit.of.lineator,eqn:ass.of.lineator} collapse to the ones in \eqref{eqn:coh-comonad-mors}, while the rest trivialize completely.
	Hence maps of comonads are in one-to-one correspondence with strict morphisms of contextads between the relevant contextads.
\end{ex}

\begin{ex}
\label{ex:maps.actegories}
	Let $\acted \nfrom{\pi_{\acted}} \acted \times \actor \nto{\action} \acted$ and $\acted' \nfrom{\pi_{\acted'}} \acted' \times \actor' \nto{\action'} \acted'$ be actions of monoidal categories, as in \cref{ex:actegory}.
	A \emph{lax linear functor} \cite[3.5.9]{capucci2022actegories} between them is a triple $(R,F,\lineator)$ of a monoidal functor $R : \acted \to \acted'$, a functor $F:\acted \to \acted'$ and a natural transformation $\lineator : \action'(F,R) \twoto F\action$, the latter of which satisfies compatibility axioms with the unitor and multiplicator of the actions.

	This data can be used to instantiate \cref{defn:contextad.1-cell}:
	\begin{equation}
		\begin{tikzcd}[ampersand replacement=\&]
			\acted \& {\acted \times \actor} \& \acted \\
			{\acted'} \& {\acted' \times \actor'} \& {\acted'}
			\arrow["F"', dashed, from=1-1, to=2-1]
			\arrow["{\pi_{\acted}}"', two heads, from=1-2, to=1-1]
			\arrow["\action", from=1-2, to=1-3]
			\arrow["{F \times R}"', dashed, from=1-2, to=2-2]
			\arrow["F", dashed, from=1-3, to=2-3]
			\arrow["\lineator", shorten <=6pt, shorten >=6pt, Rightarrow, dashed, from=2-2, to=1-3]
			\arrow["{\pi_{\acted'}}", two heads, from=2-2, to=2-1]
			\arrow["{\action'}"', from=2-2, to=2-3]
		\end{tikzcd}
	\end{equation}
	The unitor $\mapunitor$ and multiplicator $\mapmultiplicator$, being $\pi_{\acted'}$-vertical, are maps in $\actor'$ which can be seen to correspond to the strong monoidal structure of $R$.
	The two sets of coherence laws satisfied by these correspond to the coherence laws of a monoidal structure and those of the lineator (\cref{eqn:unit.of.lineator,eqn:ass.of.lineator}).
	Hence a lax linear morphism of actegories corresponds to a strong morphism of contextads.

	Vice versa, a morphism of contextads between $\acted \nfrom{\pi_{\acted}} \acted \times \actor \nto{\action} \acted$ and $\acted' \nfrom{\pi_{\acted'}} \acted' \times \actor' \nto{\action'} \acted'$ is more general than a morphism of actions of monoidal categories.
	The extra generality comes from the fact $F^\flat$, corresponding to $R$ above, can potentially depend on an extra argument from $\acted$, i.e.~it has type $F^\flat : \acted \times \actor \to \actor'$.
	However, since it also must be cartesian with respect to the product projections, $F^{{\flat}}$ is in many cases forced to be constant in the $\acted$ variable; for example, if $\acted$ has a terminal object.
\end{ex}

\begin{ex}
\label{ex:maps.display.maps.cats}
	We know from \cref{ex:display.map.cat} that a category with display maps $(\acted, \Display)$ has an associated contextad, whose $\Ctx$ construction is the double category of left-displayed spans in $\acted$.
	Accordingly, a functor $F: (\acted, \Display) \to (\acted', \Display')$ preserving pullbacks of display maps induces a \emph{strictly linear strict morphism} of contextads:
	\begin{equation}
		\begin{tikzcd}[ampersand replacement=\&]
			\acted \& {\acted^{\downarrow_{\Display}}} \& \acted \\
			{\acted'} \& {{\acted'}^{\downarrow_{\Display}}} \& {\acted'}
			\arrow["F"', dashed, from=1-1, to=2-1]
			\arrow["\cod"', two heads, from=1-2, to=1-1]
			\arrow["\dom", from=1-2, to=1-3]
			\arrow["{F^\downarrow}"', dashed, from=1-2, to=2-2]
			\arrow["F", dashed, from=1-3, to=2-3]
			\arrow["\cod", two heads, from=2-2, to=2-1]
			\arrow["\dom"', from=2-2, to=2-3]
		\end{tikzcd}
	\end{equation}
	The strictness comes from the fact $F$ is a functor so that the right square commutes on the nose.
\end{ex}

Since contextads generalize various kinds of structures, their morphisms are uniquely suited to express `hybrid' morphisms straddling the boundaries between those.
We will see an example in the dual setting, \cref{ex:law}.

\subsubsection[The Ctx trifunctor]{The $\Ctx$ trifunctor}
When we restrict $\KL(\DispSpan\Paradise)$ to its functorial 1-cells and reverse the 2-cells, the tricategory we get is $\PsCat\Paradise$, that of pseudocategories internal to $\Paradise$ whose source map is display.
The 1-cells are functors, the 2-cells are \emph{loose} pseudonatural transformations and the 3-cells are \emph{modifications}, i.e.~modifications of loose pseudonatural transformations (see \cite{grandis_limits_1999}).%
\footnote{Note that $\DblCat \neq \PsCat(\Cat, \{\mathsf{all}\})$ since the two differ in the choice of 2-cells---in the first they are \emph{tight} natural transformations, in the latter \emph{loose}.}
This is expounded in \cite[§7.3]{miranda_topics_2024} for $\Paradise = (\Cat, \{\mathsf{all functors}\})$, but even there it is remarked how that story is exemplar for any 2-category with pullbacks, and it is not hard to see it is actually enough to have left legs form a pullback ideal as the display maps of a paradise.

We can now give the main definition of the paper:

\begin{defn}
\label{defn:ctx.construction}
	The \textbf{$\Ctx$ construction} is the dashed trifunctor below, obtained by co/restricting the functor defined in \cref{thm:ctx.const.on.objects} to functorial spans, and dualizing 2-cells:
	\begin{equation}
		\begin{tikzcd}[ampersand replacement=\&, row sep=scriptsize]
			{\Cxd\Paradise} \&[6ex]\&[-3ex] {\PsCat\Paradise} \\[-1ex]
			{\KL(\FibSpan^\twoto\Paradise,\, \fun)\co} \& {\KL(\KL(\DispSpan\Paradise),\, \fun)\co} \& {\KL(\DispSpan\Paradise,\, \fun)\co}
			\arrow["\Ctx\Paradise", dashed, from=1-1, to=1-3]
			\arrow["{=}"{marking, allow upside down}, draw=none, from=1-1, to=2-1]
			\arrow["{=}"{marking, allow upside down}, draw=none, from=1-3, to=2-3]
			\arrow["{\KL\InclusionTrifun_\fun\co}", hook, from=2-1, to=2-2]
			\arrow["{\wreath_\fun\co}", from=2-2, to=2-3]
		\end{tikzcd}
		\hspace*{-5ex}
	\end{equation}
\end{defn}

\begin{notation}
	We basically always write $\Ctx$ in lieu of $\Ctx\Paradise$.
\end{notation}

\begin{rmk}
	Observe that $\KL(\InclusionTrifun,\, \fun)\co$ is still fully faithful, so that morphisms of contextads are as general as arbitrary morphisms of wreaths.
\end{rmk}

\begin{rmk}
	The pseudocategory $\Ctx(\action)$ is not just a left-displayed pseudocategory, it is left-fibrant again.
\end{rmk}

The functor $\Ctx$ sends a morphism of contextads $\underline F = (F, \lineator) : \action \to \action'$ to the strong double functor $\Ctx(\underline F) : \Ctx(\action) \to \Ctx(\action')$ given as
\begin{equation}
	\begin{tikzcd}[ampersand replacement=\&]
		A \& B \\
		{A'} \& {B'}
		\arrow[""{name=0, anchor=center, inner sep=0}, "{(P,p)}", "\shortmid"{marking}, from=1-1, to=1-2]
		\arrow[""{name=0p, anchor=center, inner sep=0}, phantom, from=1-1, to=1-2, start anchor=center, end anchor=center]
		\arrow["h"', from=1-1, to=2-1]
		\arrow["k", from=1-2, to=2-2]
		\arrow[""{name=1, anchor=center, inner sep=0}, "{(P',p')}"', "\shortmid"{marking}, from=2-1, to=2-2]
		\arrow[""{name=1p, anchor=center, inner sep=0}, phantom, from=2-1, to=2-2, start anchor=center, end anchor=center]
		\arrow["\varphi"', shorten <=4pt, shorten >=4pt, Rightarrow, from=0p, to=1p]
	\end{tikzcd}
	\quad \mapsto \quad
	\begin{tikzcd}[ampersand replacement=\&]
		FA \& FB \\
		{FA'} \& {FB'}
		\arrow[""{name=0, anchor=center, inner sep=0}, "{(F^\flat P, \lineator \lcomp p)}", "\shortmid"{marking}, from=1-1, to=1-2]
		\arrow[""{name=0p, anchor=center, inner sep=0}, phantom, from=1-1, to=1-2, start anchor=center, end anchor=center]
		\arrow["Fh"', from=1-1, to=2-1]
		\arrow["Fk", from=1-2, to=2-2]
		\arrow[""{name=1, anchor=center, inner sep=0}, "{(F^\flat P', \lineator \lcomp p')}"', "\shortmid"{marking}, from=2-1, to=2-2]
		\arrow[""{name=1p, anchor=center, inner sep=0}, phantom, from=2-1, to=2-2, start anchor=center, end anchor=center]
		\arrow["{F^\flat\varphi}"', shorten <=4pt, shorten >=4pt, Rightarrow, from=0p, to=1p]
	\end{tikzcd}
\end{equation}
where the latter diagram corresponds to the pasting below:
\begin{equation}
	\begin{tikzcd}[ampersand replacement=\&,row sep=scriptsize]
		{FA \action' F^\flat P} \&[2ex] {FA' \action' F^\flat P'} \\[1ex]
		{F(A \action P)} \& {F(A' \action P')} \\
		FB \& {FB'}
		\arrow["{Fh \,\action'\, F^\flat \varphi}", from=1-1, to=1-2]
		\arrow["\lineator"', from=1-1, to=2-1]
		\arrow["\lineator", from=1-2, to=2-2]
		\arrow["{F(h \action \varphi)}", from=2-1, to=2-2]
		\arrow["Fp"', from=2-1, to=3-1]
		\arrow["{Fp'}", from=2-2, to=3-2]
		\arrow["Fk"', from=3-1, to=3-2]
	\end{tikzcd}
\end{equation}
The identity and composition comparisons are built out of the unitor and multiplicator of $\underline F$.

A transformation $\underline \tau = (e, \tau) : \underline F \twoto \underline G$ would instead correspond to a loose pseudonatural transformation with components on objects given by the contextful arrows $(e_A, \tau_A) : FA \looseto GA$, and components on maps $h:A \to A'$ given by the squares
\begin{equation}
	\begin{tikzcd}[ampersand replacement=\&]
		FA \& GA \\
		{FA'} \& {GA'}
		\arrow[""{name=0, anchor=center, inner sep=0}, "{(e_A, \tau_A)}", "\shortmid"{marking}, from=1-1, to=1-2]
		\arrow[""{name=0p, anchor=center, inner sep=0}, phantom, from=1-1, to=1-2, start anchor=center, end anchor=center]
		\arrow["Fh"', from=1-1, to=2-1]
		\arrow["Gh", from=1-2, to=2-2]
		\arrow[""{name=1, anchor=center, inner sep=0}, "{(e_{A'}, \tau_{A'})}"', "\shortmid"{marking}, from=2-1, to=2-2]
		\arrow[""{name=1p, anchor=center, inner sep=0}, phantom, from=2-1, to=2-2, start anchor=center, end anchor=center]
		\arrow["{e_h}", shorten <=4pt, shorten >=4pt, Rightarrow, from=0p, to=1p]
	\end{tikzcd}
\end{equation}
where $e_h:e_A \to e_{A'}$ is the action on maps of $e:\acted \to \actor'$ (recall \cref{defn:contextad.2-cell}).
The naturator of $\underline \tau$ provides the isocells witnessing pseudonaturality of such a transformation, with \cref{eqn:coherence.of.naturator} witnessing their coherence.

Finally, a modification $\omega : (e,\tau) \threeto  (f, \sigma)$ of contextads corresponds pretty directly to a modification of double functors, with the maps $\omega_A:e_A \to f_A$ filling the squares associated to the object $A$:
\begin{equation}
	\begin{tikzcd}[ampersand replacement=\&]
		FA \& GA \\
		FA \& GA
		\arrow[""{name=0, anchor=center, inner sep=0}, "{(e_A,\tau_A)}", "\shortmid"{marking}, from=1-1, to=1-2]
		\arrow[""{name=0p, anchor=center, inner sep=0}, phantom, from=1-1, to=1-2, start anchor=center, end anchor=center]
		\arrow[Rightarrow, no head, from=1-1, to=2-1]
		\arrow[Rightarrow, no head, from=1-2, to=2-2]
		\arrow[""{name=1, anchor=center, inner sep=0}, "{(f_A,\sigma_A)}"', "\shortmid"{marking}, from=2-1, to=2-2]
		\arrow[""{name=1p, anchor=center, inner sep=0}, phantom, from=2-1, to=2-2, start anchor=center, end anchor=center]
		\arrow["{\omega_A}", shorten <=4pt, shorten >=4pt, Rightarrow, from=0p, to=1p]
	\end{tikzcd}
\end{equation}

\begin{rmk}
	Observe how $\Ctx$ produces \emph{strongly functorial} double functors out of \emph{laxly linear} but \emph{strong} morphisms of contextads.
\end{rmk}

\section[Structuring Ctx]{Structuring $\Ctx$}
\label{sec:structured.para}

In this section, we will see how $\Ctx(\action)$ can inherit algebraic structure from the contextad it comes from.
We will do this by constructing paradises of (strict) algebras for a 2-monad with colax morphisms between them. Taking the $\Ctx$ construction in such 2-categories will result in colaxly-structured double categories.

\subsection{Motivation and background}
Before we see this construction, let's revisit the original Para construction of an actegory $\action : \acted \times \actor \to \acted$ and see how it can inherit a monoidal structure from $\acted$ and $\actor$.
Suppose $(\actor, \combineunit, \combine)$ and $(\acted, \monoidalunit, \monoidal)$ are monoidal categories.
We would like to extend the monoidal product of $\acted$ to $\para(\action)$.
Given two parameterized maps $f : A \action P \to B$ and $g : C \action Q \to D$, the most obvious choice for their tensor product would be $f \monoidal g$;
however, it is of signature $(A \action P) \monoidal (C \action Q) \to B \monoidal D$ which is not in the form of a parameterized map.
If we had an interchanger $\colaxity{\action} : (A \monoidal C) \action (P \combine Q) \to (A \action P) \monoidal (C \action Q)$, we could define the tensor product of $f$ and $g$ to be the composite
\begin{equation}
\label{eqn:para.map.monoidal}
	(A \monoidal C) \action (P \combine Q) \xto{\colaxity{\action}} (A \action P) \monoidal (C \action Q) \xto{f \monoidal g} B \monoidal D.
\end{equation}
In practice, we often do have such an interchanger $\colaxity{\action}$: if $\action = \monoidal = \combine : \actor \times \actor \to \actor$ is the self-action of a braided monoidal category $\actor$, then the braiding which swaps $P$ and $C$ gives an interchanger $\colaxity{\action}$.
Then the braiding of $\actor$ comes in handy again when defining bifunctoriality of \eqref{eqn:para.map.monoidal}.

Indeed, it's folkloric that $\para(\action)$ is symmetric or braided monoidal (as a bicategory) as soon as $\action$ is, though, as far as the authors know, a proper proof has never appeared in the literature, except for actions of \emph{commutative} monoidal categories in \cite{capucci_towards_2022}, the result of Hermida and Tennent's \cite[Proposition~2.2]{hermida_monoidal_2012} regarding the 1-categorical truncation of $\para(\action)$, and the sketch of Paquet and Saville in \cite[§3.3]{paquetEffectfulSemanticsBicategories2024} for the case of the self-action of a symmetric monoidal category.

To get a full proof, one would need to manually show \eqref{eqn:para.map.monoidal} satisfies weak interchange using the braiding on $\combine$.
The only obstacle to writing a proof along these lines are the famously daunting coherences of monoidal bicategories.

One of the aims of the present paper is to provide a high-brow route for completing such proofs, not just for monoidal structure but for any kind of structure encoded by a 2-monad $T$ on the ambient paradise.
We note a great deal of simplification comes from working with double categories instead of bicategories.

\subsection{Paradises of colax morphisms}

The definition below is subsumed by \cref{defn:psalg}, but we spell it again to have it fresh and instantiated in the case we care about:

\begin{defn}
	Let $\Cosmos$ be a 2-category and $T : \Cosmos \to \Cosmos$ a 2-monad---that is, a 2-functor equipped with 2-natural transformations $\eta$ and $\mu$ which satisfy the laws of a pseudomonad (\cref{defn:psmnd}) strictly.
	By an \textbf{algebra} for $T$, we mean an algebra for this monad in the usual sense: a map $\alpha : T\Aa \to \Aa$ satisfying the following two laws
	\begin{equation}
	\label{eqn:algbera.laws}
		\begin{tikzcd}[sep=scriptsize]
			\Aa & T\Aa && {T^2\Aa} & T\Aa \\
			& \Aa && T\Aa & \Aa
			\arrow["\eta", from=1-1, to=1-2]
			\arrow[equals, from=1-1, to=2-2]
			\arrow["\alpha", from=1-2, to=2-2]
			\arrow["T\alpha", from=1-4, to=1-5]
			\arrow["\mu"', from=1-4, to=2-4]
			\arrow["\alpha", from=1-5, to=2-5]
			\arrow["\alpha"', from=2-4, to=2-5]
		\end{tikzcd}
	\end{equation}
	A \textbf{colax $T$-morphism} from $\alpha : T\Aa \to \Aa$ to $\beta : T\Ba \to \Ba$ consists of a map $f : \Aa \to \Ba$ and a 2-cell $\colaxity{f} : f\alpha \Rightarrow \beta Tf$ (which we refer to as the \textbf{colaxity} of $f$) satisfying the following laws:
	\begin{equation}
	\label{eqn.colax.map.unit}
		\begin{tikzcd}
			\Aa & \Ba &&& \Aa & \Ba \\
			& T\Aa & T\Ba & {=} &&& T\Ba \\
			& \Aa & \Ba &&& \Aa & \Ba
			\arrow["f", from=1-1, to=1-2]
			\arrow["\unit", from=1-1, to=2-2]
			\arrow[equals, curve={height=12pt}, from=1-1, to=3-2]
			\arrow["\eta", from=1-2, to=2-3]
			\arrow["f", from=1-5, to=1-6]
			\arrow[equals, curve={height=12pt}, from=1-5, to=3-6]
			\arrow["\eta", from=1-6, to=2-7]
			\arrow[equals, curve={height=12pt}, from=1-6, to=3-7]
			\arrow[""{name=0, anchor=center, inner sep=0}, "Tf", from=2-2, to=2-3]
			\arrow["\alpha"', from=2-2, to=3-2]
			\arrow["\beta", from=2-3, to=3-3]
			\arrow["\beta", from=2-7, to=3-7]
			\arrow[""{name=1, anchor=center, inner sep=0}, "f"', from=3-2, to=3-3]
			\arrow["f"', from=3-6, to=3-7]
			\arrow["{\colaxity{f}}"', shorten <=4pt, shorten >=4pt, Rightarrow, from=1, to=0]
		\end{tikzcd}
	\end{equation}
	\begin{equation}
	\label{eqn:colax.morphism.mult}
		\begin{tikzcd}
			{T^2\Aa} & {T^2\Ba} &&& {T^2\Aa} & {T^2\Ba} \\
			T\Aa & T\Aa & T\Ba & {=} & T\Aa & T\Ba & T\Ba \\
			& \Aa & \Ba &&& \Aa & \Ba
			\arrow[""{name=0, anchor=center, inner sep=0}, "{T^2f}", from=1-1, to=1-2]
			\arrow["\mu"', from=1-1, to=2-1]
			\arrow["T\alpha"', from=1-1, to=2-2]
			\arrow["T\beta", from=1-2, to=2-3]
			\arrow["{T^2f}", from=1-5, to=1-6]
			\arrow["\mu"', from=1-5, to=2-5]
			\arrow["\mu"', from=1-6, to=2-6]
			\arrow["T\beta", from=1-6, to=2-7]
			\arrow["\alpha"', from=2-1, to=3-2]
			\arrow[""{name=1, anchor=center, inner sep=0}, from=2-2, to=2-3]
			\arrow["\alpha"', from=2-2, to=3-2]
			\arrow["\beta", from=2-3, to=3-3]
			\arrow[""{name=2, anchor=center, inner sep=0}, from=2-5, to=2-6]
			\arrow[from=2-5, to=3-6]
			\arrow["\beta", from=2-6, to=3-7]
			\arrow["\beta", from=2-7, to=3-7]
			\arrow[""{name=3, anchor=center, inner sep=0}, "f"', from=3-2, to=3-3]
			\arrow[""{name=4, anchor=center, inner sep=0}, "f"', from=3-6, to=3-7]
			\arrow["{\colaxity{f}}"', shorten <=7pt, shorten >=7pt, Rightarrow, from=1, to=0]
			\arrow["{\colaxity{f}}"', shorten <=4pt, shorten >=4pt, Rightarrow, from=3, to=1]
			\arrow["{\colaxity{f}}", shorten <=7pt, shorten >=7pt, Rightarrow, from=4, to=2]
		\end{tikzcd}
	\end{equation}
	A 2-cell $\rho : (f, \colaxity{f}) \Rightarrow (g, \colaxity{g})$ is a 2-cell $\rho : f \Rightarrow g$ satisfying the following law:
	\begin{equation}
	\label{eqn:colax.morphism.2cell}
		\begin{tikzcd}[row sep = small]
			T\Aa & T\Ba && T\Aa & T\Ba \\
			&& {=} \\
			\Aa & \Ba && \Aa & \Ba
			\arrow[""{name=0, anchor=center, inner sep=0}, "Tg", curve={height=-12pt}, from=1-1, to=1-2]
			\arrow["\alpha"', from=1-1, to=3-1]
			\arrow["\beta", from=1-2, to=3-2]
			\arrow[""{name=1, anchor=center, inner sep=0}, "Tg", curve={height=-12pt}, from=1-4, to=1-5]
			\arrow[""{name=2, anchor=center, inner sep=0}, "Tf"', curve={height=12pt}, from=1-4, to=1-5]
			\arrow["\alpha"', from=1-4, to=3-4]
			\arrow["\beta", from=1-5, to=3-5]
			\arrow[""{name=3, anchor=center, inner sep=0}, "g", curve={height=-12pt}, from=3-1, to=3-2]
			\arrow[""{name=4, anchor=center, inner sep=0}, "f"', curve={height=12pt}, from=3-1, to=3-2]
			\arrow[""{name=5, anchor=center, inner sep=0}, "f"', curve={height=12pt}, from=3-4, to=3-5]
			\arrow["T\rho", shorten <=3pt, shorten >=3pt, Rightarrow, from=2, to=1]
			\arrow["{\colaxity{g}}", shorten <=9pt, shorten >=9pt, Rightarrow, from=3, to=0]
			\arrow["\rho", shorten <=3pt, shorten >=3pt, Rightarrow, from=4, to=3]
			\arrow["{\colaxity{f}}", shorten <=9pt, shorten >=9pt, Rightarrow, from=5, to=2]
		\end{tikzcd}
	\end{equation}

	A colax morphism $(f, \colaxity{f})$ is \textbf{strong} (or \textbf{pseudo}) if $\colaxity{f}$ is an isomorphism and \textbf{strict} if $\colaxity{f}$ is an identity.
	We denote by $\Alg(T)_\colax$ the 2-category of (strict) algebras and colax morphisms, and by $\Alg(T)_\strict$ the 2-category of algebras and strict morphisms.
\end{defn}

The proof of the main theorem of this section, \cref{thm:colax.paradise}, revolves around the following:

\begin{defn}
	Let $T : \Cosmos \to \Cosmos$ be a 2-monad on a 2-category.
	A colax $T$-morphism $(e, \colaxity{e}) : \Aa \to \Aa$ is a \textbf{special idempotent} if
	\begin{enumerate}
		\item $e$ is idempotent in $\Cosmos$, and
		\item $e\colaxity{e}=\id$.
	\end{enumerate}
\end{defn}

We record a few useful and straightforward facts about special idempotents.

\begin{lem}
	If $(e, \colaxity{e}) : \Aa \to \Aa$ is a special idempotent, then it is an idempotent in $\Alg(T)_{\colax}$.
\end{lem}

\begin{lem}
\label{lem:special.idempotent.whiskering}
	Let $(e, \colaxity{e}) : \Aa \to \Aa$ be a colax morphism.
	If $e = ir$ splits in $\Cosmos$, then $(e, \colaxity{e})$ is a special idempotent if and only if $r\colaxity{e}$ is an identity.
\end{lem}

We will now show that if a special idempotent splits in $\Cosmos$, this splitting can be lifted to $\Alg(T)_{\colax}$

\begin{lem}
\label{lem:strong.idempotents.split}
	Let $\alpha : T\Aa \to \Aa$ be a $T$-algebra and $(e, \colaxity{e}) : \Aa \to \Aa$ be a special idempotent colax morphism.
	Suppose that $e = \Aa \xto{r} \Sa \xto{i} \Aa$ splits in $\Cosmos$.
	Then:
	\begin{enumerate}
		\item The composite
		      \begin{equation}
			      \sigma := T\Sa \xto{Ti} T\Aa \xto{\alpha} \Aa \xto{r} \Sa
		      \end{equation}
		      endows $\Sa$ with the structure of a strict $T$-algebra,
		\item $r:(\Aa,\alpha) \to (\Sa,\sigma)$ is a strict $T$-morphism,
		\item $(i, \colaxity{i}):(\Sa,\sigma) \to (\Aa,\alpha)$ is a colax $T$-morphism, and
		\item $r$ and $(i, \colaxity{i})$ split $(e, \colaxity{e})$ in $\Alg(T)_{\colax}$.
	\end{enumerate}
\end{lem}
\begin{proof}
	First, let's check that $\sigma$ is a strict $T$-algebra by appealing to the following diagrams:
	\begin{equation}
		\begin{tikzcd}[ampersand replacement=\&,column sep=scriptsize,row sep=small]
			\Sa \&\& T\Sa \&\& {T^2\Sa} \& {T^2\Aa} \& T\Aa \& T\Sa \\
			\& \Aa \& T\Aa \&\&\&\&\& T\Aa \\
			\&\& \Aa \&\&\&\&\& \Aa \\
			\&\& \Sa \&\& T\Sa \& T\Aa \& \Aa \& \Sa
			\arrow["\eta", from=1-1, to=1-3]
			\arrow["i", from=1-1, to=2-2]
			\arrow[curve={height=6pt}, Rightarrow, no head, from=1-1, to=4-3]
			\arrow["{\,Ti}", from=1-3, to=2-3]
			\arrow["{T^2i}", from=1-5, to=1-6]
			\arrow["\mu"', from=1-5, to=4-5]
			\arrow["T\alpha", from=1-6, to=1-7]
			\arrow["\mu"', from=1-6, to=4-6]
			\arrow["Tr", from=1-7, to=1-8]
			\arrow[""{name=0, anchor=center, inner sep=0}, "Te"', from=1-7, to=2-8]
			\arrow["\alpha"', from=1-7, to=4-7]
			\arrow["{\,Ti}", from=1-8, to=2-8]
			\arrow["\eta", from=2-2, to=2-3]
			\arrow[Rightarrow, no head, from=2-2, to=3-3]
			\arrow["{\,\alpha}", from=2-3, to=3-3]
			\arrow["{\,\alpha}", from=2-8, to=3-8]
			\arrow["{\,r}", from=3-3, to=4-3]
			\arrow["{\,r}", from=3-8, to=4-8]
			\arrow["Ti"', from=4-5, to=4-6]
			\arrow["\alpha"', from=4-6, to=4-7]
			\arrow[""{name=1, anchor=center, inner sep=0}, "e"', from=4-7, to=3-8]
			\arrow["r"', from=4-7, to=4-8]
			\arrow["{\colaxity{e}}", shorten <=9pt, shorten >=9pt, Rightarrow, from=1, to=0]
		\end{tikzcd}
	\end{equation}
	On the left we have the unit law since $ri = \id_{\Sa}$.
	On the right, we have the multiplication law.
	This is filled by an identity, since $r \colaxity{e}$ is an identity by \cref{lem:special.idempotent.whiskering}.

	Next, we turn to showing that $r$ is a strict $T$-morphism and that $i$ is a colax $T$-morphism:
	\begin{equation}
		\begin{tikzcd}[row sep = small]
			T\Aa & T\Sa && T\Sa & T\Aa \\
			& T\Aa && T\Aa \\
			& \Aa && \Aa \\
			\Aa & \Sa && \Sa & \Aa
			\arrow["Tr", from=1-1, to=1-2]
			\arrow[""{name=0, anchor=center, inner sep=0}, "Te", from=1-1, to=2-2]
			\arrow["\alpha"', from=1-1, to=4-1]
			\arrow["Ti", from=1-2, to=2-2]
			\arrow["Ti", from=1-4, to=1-5]
			\arrow["Ti"', from=1-4, to=2-4]
			\arrow["\alpha", from=1-5, to=4-5]
			\arrow["\alpha", from=2-2, to=3-2]
			\arrow[""{name=1, anchor=center, inner sep=0}, "Te", from=2-4, to=1-5]
			\arrow["\alpha"', from=2-4, to=3-4]
			\arrow["r", from=3-2, to=4-2]
			\arrow["r"', from=3-4, to=4-4]
			\arrow[""{name=2, anchor=center, inner sep=0}, "e"', from=3-4, to=4-5]
			\arrow[""{name=3, anchor=center, inner sep=0}, "e"', from=4-1, to=3-2]
			\arrow["r"', from=4-1, to=4-2]
			\arrow["i"', from=4-4, to=4-5]
			\arrow["{\colaxity{e}}"', shorten <=9pt, shorten >=9pt, Rightarrow, from=2, to=1]
			\arrow["{\colaxity{e}}", shorten <=9pt, shorten >=9pt, Rightarrow, from=3, to=0]
		\end{tikzcd}
	\end{equation}
	On the left, we again find the whisker $r\colaxity{e}$, so it is actually filled by an identity.
	On the right, we find $\colaxity{i} := \colaxity{e}Ti$ endows $i$ with a colaxity, whose laws follows from those of $(e, \colaxity{e})$.
	Composing these two clearly gives $\colaxity{e}$, showing that $(e, \colaxity{e})$ splits in $\Alg(T)_\colax$.
\end{proof}

Our goal now is to show $\Alg(T)_\colax$ can be equipped with a class of display maps making it into a paradise.
First, we need to introduce terminology.

\begin{defn}
\label{defn:normal.fib}
	A fibration in a 2-category $\Cosmos$ (see \cref{defn:fib}) is \textbf{normal} if the lift of an identity is an identity.
	Equivalently, $p : \Ea \fibto \Ba$ is normal if the whisker $\counitpull \internal{\id_p} : \internal{\id_p} \pull \internal{\id_p} \twoto \internal{\id_p}$ is an identity.
\end{defn}

It is straightforward to check that:

\begin{lem}
\label{lem:normal.fibrations.closed.under.pullback}
	Normal fibrations are closed under composition and pullback.
\end{lem}

We may furthermore give two crucial examples of normal fibrations: the domain projections $\dom : \Ca^{\downarrow} \to \Ca$ and all product projections.
The first will be necessary for giving a paradise structure on $\Alg(T)_{\colax}$; the second will allow us to consider actegories $\action : \acted \times \actor \to \acted$ in $\Alg(T)_\colax$ as spans whose left leg, the product projection $\acted \leftarrow \acted \times \actor$, is a display map.

\begin{lem}
\label{lem:normal.fibration.examples}
	Let $\Cosmos$ be a 2-category.
	\begin{enumerate}
		\item If $\Ca^{\downarrow}$ is an arrow object and $\dom : \Ca^{\downarrow} \to \Ca$ is sliceable, then it is a normal fibration.
		\item If a product projection $\pi_{\Aa} : \Aa \times \Ba \to \Aa$ is sliceable, then it is a normal fibration.
	\end{enumerate}
\end{lem}
\begin{proof}
	First, suppose that $\dom : \Ca^{\downarrow} \to \Ca$ is sliceable.
	Note that the slice $\Ca \downarrow \dom$ has the universal property of the power $\Ca^{\Delta[2]}$ by the category $\Delta[2] = \{0 \xto{\overline{01}} 1 \xto{\overline{12}} 2\}$ generated by two composable arrows.
	The map $\internal{\id_{\dom}} : \Ca^{\downarrow} \to \Ca \comma \dom$ may be identified with $s_0^* : \Ca^{\Delta[1]} \to \Ca^{\Delta[2]}$ where $s_0 : \Delta[2] \to \Delta[1]$ is the degeneracy sending $\overline{01}$ to $\id_0 \in \Delta[1]$.
	Now, $s_0$ has a left adjoint $d_1 : \Delta[1] \to \Delta[2]$ defined by $d_1(\overline{01}) = \overline{12}\circ \overline{01}$, and the whisker $s_0\counit$ is an identity; this adjunction moreover takes place in the slice $\Delta[0]$ via initial objects.
	Therefore $d_1^*$ is right adjoint to $s_0^*$ fibred over the projection to $\Ca$, and $\counit s_0^*$ is an identity.

	Now, if $\pi_{\Aa} : \Aa \times \Ba \to \Aa$ is a product projection, then we have an isomorphism $\Aa \comma \pi_{\Aa} \cong \Aa^{\downarrow} \times \Ba$, and we may identify $\internal{\id_{\pi_{\Aa}}} : \Aa \times \Ba \to \Aa \comma \pi_{\Aa}$ with $\internal{\id_{\Aa}} \times \id_{\Ba} : \Aa \times \Ba \to \Aa^{\downarrow} \times \Ba$.
	Since taking the product with $\Ba$ will be 2-functorial wherever it is defined, $\internal{\id_{\Aa}} \times \id_{\Ba}$ will have a right adjoint $\dom \times \Ba : \Aa^{\downarrow} \times \Ba \to \Aa \times \Ba$, and it is straightforward to see that the whisker $\counit\internal{\id_{\Aa}}$ will be an identity for this adjunction.
\end{proof}

The nicely behaved class of maps we are eyeing are the \emph{strict displayed normal fibrations}---that is, normal fibrations in $\Alg(T)_\colax$ which are strict as $T$-morphisms (though the fibration structure itself could be colax) and whose underlying 1-cell in the paradise $\Cosmos$ is display.

To even state this definition, we need to be able to take comma objects in $\Alg(T)_\colax$ so that we can define fibrations therein.
It is proven by \cite[Proposition~4.6]{lack:limits.lax.morphisms} that $\Alg(T)_\colax$ admits comma objects $f \comma g$ when $g$ is strict, and that the forgetful functor $\Alg(T)_\colax \to \Cosmos$ preserves these.
In other words, the comma object in $\Alg(T)_\colax$ is carried by the comma object in $\Cosmos$, endowed with the algebra structure obtained by invoking the universal property of $f \comma g$ in $\Cosmos$:
\begin{equation}
	\begin{tikzcd}[ampersand replacement=\&,sep=scriptsize]
		{T(f \comma g)} \& T\Ba \& \Ba \&\& {T(f \comma g)} \& T\Ba \& \Ba \\
		T\Aa \& T\Ca \&\& {=} \& T\Aa \& {f\comma g} \\
		\Aa \&\& \Ca \&\& \Aa \&\& \Ca
		\arrow["{T\vec{\pi}_{\Ba}}", from=1-1, to=1-2]
		\arrow["{T\vec{\pi}_{\Aa}}"', from=1-1, to=2-1]
		\arrow["\beta", from=1-2, to=1-3]
		\arrow["Tp", from=1-2, to=2-2]
		\arrow["g", from=1-3, to=3-3]
		\arrow["{T\vec{\pi}_{\Ba}}", from=1-5, to=1-6]
		\arrow["{T\vec{\pi}_{\Aa}}"', from=1-5, to=2-5]
		\arrow["{\exists!\,\alpha \comma \beta}", dashed, from=1-5, to=2-6]
		\arrow["\beta", from=1-6, to=1-7]
		\arrow[""{name=0, anchor=center, inner sep=0}, "g", from=1-7, to=3-7]
		\arrow["T\generic", shorten <=6pt, shorten >=6pt, Rightarrow, from=2-1, to=1-2]
		\arrow["Tf", from=2-1, to=2-2]
		\arrow["\alpha"', from=2-1, to=3-1]
		\arrow["\gamma", from=2-2, to=3-3]
		\arrow["\alpha"', from=2-5, to=3-5]
		\arrow["{\vec{\pi}_{\Ba}}"', from=2-6, to=1-7]
		\arrow["{\vec{\pi}_{\Aa}}"', from=2-6, to=3-5]
		\arrow["{\colaxity f}", shift right, shorten <=6pt, shorten >=6pt, Rightarrow, from=3-1, to=2-2]
		\arrow["f"', from=3-1, to=3-3]
		\arrow[""{name=1, anchor=center, inner sep=0}, "f"', from=3-5, to=3-7]
		\arrow["\generic", shorten <=10pt, shorten >=10pt, Rightarrow, from=1, to=0]
	\end{tikzcd}
\end{equation}
Moreover, it is shown in \emph{ibid.} that the projections $\vec{\pi}_{\Aa}$ and $\vec{\pi}_{\Ba}$ \emph{jointly detect strictness} meaning that a colax morphism of $T$-algebras $u:\Xa \to f \comma g$ is strict if and only if both $\vec{\pi}_{\Aa}u$ and $\vec{\pi}_{\Ba}u$ are.
This implies, in particular, that the two projections are themselves strict.

\begin{ex}
\label{ex:monoidal.comma}
	Let $T$ be the free monoidal category 2-monad on $\Cat$.
	Then one can verify $f \comma g$ is equipped with `componentwise' monoidal structure: given two maps $\varphi:f(A) \to g(B)$ and $\varphi':f(A') \to g(B')$, their product is given by the dashed map below
	\begin{equation}
		\begin{tikzcd}[row sep=scriptsize]
			{f(A \monoidal A')} & {f(A) \monoidal f(A')} \\
			{g(B \monoidal B')} & {g(B) \monoidal g(B')}
			\arrow["{\colaxity{f}_{A,A'}}", from=1-1, to=1-2]
			\arrow[dashed, from=1-1, to=2-1]
			\arrow["{\varphi \monoidal \varphi'}", from=1-2, to=2-2]
			\arrow[Rightarrow, no head, from=2-1, to=2-2]
		\end{tikzcd}
	\end{equation}
	The monoidal unit is defined analogously, and in fact coincides with $\colaxity{f}_\monoidalunit$.
\end{ex}

\begin{thm}
\label{thm:colax.paradise}
	Let $\Paradise$ be a paradise, and let $T : \Cosmos \to \Cosmos$ be a 2-monad.
	Then equipping $\Alg(T)_\colax$ with the class $\mathsf{sdnf}$ of strict displayed normal fibrations endows it with the structure of a paradise.
\end{thm}
\begin{proof}
	By \cref{lem:normal.fibration.examples}, the (strict) $T$-algebra morphism $\dom : \Aa^{\downarrow} \to \Aa$ is a strict displayed normal fibration in $\Alg(T)_\colax$.
	Evidently, the class of strict displayed normal fibrations is closed under composition and contains all identities.
	We therefore only need to show it is closed under pullback.

	To that end, let $p : \Ba \fibto \Ca$ be a strict normal display fibration and let $(f, \colaxity{f}) : \Aa \to \Ca$ be any colax morphism.
	Since $p$ is a display map in $\Cosmos$ by assumption, we may form the pullback $f \spancomp p$ in $\Cosmos$; we aim to endow this with the structure of a $T$-algebra so that it becomes the pullback in $\Alg(T)_\colax$.
	We will follow the strategy which Bourke and Lack use in \cite[Proposition~4.1]{bourke_cosmoi_2023} to rectify an isocomma into an isofibration into a pullback; we will rectify a comma into a fibration into a pullback by the same method.

	First, form the comma object $f \comma p$ in $\Alg(T)_\colax$---which we can form because $p$ is assumed strict.
	Then lift the universal arrow to give a commuting square into $p$ as follows (with this diagram taking place in $\Alg(T)_\colax$):

	\begin{equation}
	\label{eqn:lift.of.comma}
		\begin{tikzcd}[row sep = small]
			{f \comma p} & \Ba && {f\comma p} & \Ba \\
			&& {=} \\[-2ex]
			\Aa & \Ca && \Aa & \Ca
			\arrow["{\vec{\pi}_{\Ba}}", from=1-1, to=1-2]
			\arrow["{\vec{\pi}_{\Aa}}"', from=1-1, to=3-1]
			\arrow["p", two heads, from=1-2, to=3-2]
			\arrow[""{name=0, anchor=center, inner sep=0}, "{\vec{\pi}_{\Ba}}", curve={height=-12pt}, from=1-4, to=1-5]
			\arrow[""{name=1, anchor=center, inner sep=0}, "{\pull \generic}"', curve={height=12pt}, dashed, from=1-4, to=1-5]
			\arrow["{\vec{\pi}_{\Aa}}"', from=1-4, to=3-4]
			\arrow["p", two heads, from=1-5, to=3-5]
			\arrow["\generic", shorten <=9pt, shorten >=9pt, Rightarrow, from=3-1, to=1-2]
			\arrow["f"', from=3-1, to=3-2]
			\arrow["f"', from=3-4, to=3-5]
			\arrow["\lift\generic"', shorten <=3pt, shorten >=3pt, Rightarrow, from=1, to=0]
		\end{tikzcd}
	\end{equation}

	By the universal property of $f \comma p$, we have a unique map $(e, \colaxity{e}) : f \comma p \to f \comma p$ which makes the following equation hold:

	\begin{equation}
		\begin{tikzcd}[ampersand replacement=\&,sep=scriptsize]
			{f\comma p} \&\&\&\& {f\comma p} \\
			\&\& \Ba \& {=} \&\& {f\comma p} \& \Ba \\
			\& \Aa \& \Ca \&\&\& \Aa \& \Ca
			\arrow["\pull\generic", curve={height=-18pt}, from=1-1, to=2-3]
			\arrow["{\vec{\pi}_{\Aa}}"', curve={height=12pt}, from=1-1, to=3-2]
			\arrow["{\exists!\,e}", dashed, from=1-5, to=2-6]
			\arrow["\pull\generic", curve={height=-18pt}, from=1-5, to=2-7]
			\arrow["{\vec{\pi}_{\Aa}}"', curve={height=12pt}, from=1-5, to=3-6]
			\arrow["p", two heads, from=2-3, to=3-3]
			\arrow["{\vec{\pi}_{\Ba}}", from=2-6, to=2-7]
			\arrow["{\vec{\pi}_{\Aa}}"', from=2-6, to=3-6]
			\arrow["p", two heads, from=2-7, to=3-7]
			\arrow["f"', from=3-2, to=3-3]
			\arrow["\generic", shorten <=5pt, shorten >=5pt, Rightarrow, from=3-6, to=2-7]
			\arrow["f"', from=3-6, to=3-7]
		\end{tikzcd}
	\end{equation}
	We will show $(e, \colaxity{e})$ is a special idempotent that splits as $f \comma p \xto{r} f \spancomp p \xto{i} f \comma p$, thus through the pullback of $p$ along $f$ in $\Cosmos$.
	\cref{lem:strong.idempotents.split}, will then endow $f \spancomp p$ with the desired algebra structure and lift this splitting to $\Alg(T)_{\colax}$.

	We start showing $(e, \colaxity{e})$ is a special idempotent.
	Take $i : f \spancomp p \to f \comma p$ to be the map induced by the pullback square defining $f \spancomp p$, depicted below left, and take $r : f \comma p \to f \spancomp p$ to be the map $( \vec{\pi}_{\Aa}, \pull\generic )$ induced by the commuting square in the right hand side of Diagram \ref{eqn:lift.of.comma}.

	\begin{equation}
	\label{eqn:splitting}
		\begin{tikzcd}[ampersand replacement=\&,sep=scriptsize]
			{f \spancomp p} \\
			\& {f \comma p} \& \Ba \\
			\& \Aa \& \Ca
			\arrow["{\exists!\,i}", dashed, from=1-1, to=2-2]
			\arrow["{\pi_{\Ba}}", curve={height=-12pt}, from=1-1, to=2-3]
			\arrow["{\pi_{\Aa}}"', curve={height=12pt}, from=1-1, to=3-2]
			\arrow["{\vec{\pi}_{\Ba}}", from=2-2, to=2-3]
			\arrow["{\vec{\pi}_{\Aa}}"', from=2-2, to=3-2]
			\arrow["p", two heads, from=2-3, to=3-3]
			\arrow["\generic", shorten <=5pt, shorten >=5pt, Rightarrow, from=3-2, to=2-3]
			\arrow["f"', from=3-2, to=3-3]
		\end{tikzcd}
		\qquad\qquad
		\begin{tikzcd}[ampersand replacement=\&,sep=scriptsize]
			{f\comma p} \\
			\& {f \spancomp p} \& \Ba \\
			\& \Aa \& \Ca
			\arrow["{\exists!\,r}", dashed, from=1-1, to=2-2]
			\arrow["\pull\generic", curve={height=-12pt}, from=1-1, to=2-3]
			\arrow["{\vec{\pi}_{\Aa}}"', curve={height=12pt}, from=1-1, to=3-2]
			\arrow["{\pi_{\Ba}}", from=2-2, to=2-3]
			\arrow["{\pi_{\Aa}}"', from=2-2, to=3-2]
			\arrow["\lrcorner"{anchor=center, pos=0.125}, draw=none, from=2-2, to=3-3]
			\arrow["p", two heads, from=2-3, to=3-3]
			\arrow["f"', from=3-2, to=3-3]
		\end{tikzcd}
	\end{equation}

	We see that $e = ir$ by universal property of the comma.
	To show that $ri = \id_{f \spancomp p}$, it suffices to show that $\pi_{\Aa} ri = \pi_{\Aa}$ and $\pi_{\Ba} ri = \pi_{\Ba}$, this time because of the universal property of the pullback.
	The first equation follows immediately from the definition of $r$ and $i$.
	For the second, we start by noticing that $\pi_{\Ba} ri = (\pull \generic)i$ from \eqref{eqn:splitting} right.
	Then consider the whisker $(\lift \generic)i : (\pull \generic)i \Rightarrow \vec{\pi}_{\Ba} i$: since $\vec{\pi}_{\Ba} i = \pi_{\Ba}$ by definition, it will suffice to show that $(\lift \generic)i$ is an identity, and because $p$ is assumed to be normal, it is in turn enough to show $p(\lift\generic)i$ is an identity.
	But $p(\lift \generic) i = \generic i$, and the latter is an identity by definition of $i$, so we conclude that indeed $ri = \id_{f \spancomp p}$.

	We now show that the whisker $r\colaxity{e}$ is an identity, making $(e, \colaxity{e})$ strong.
	\begin{equation}
		\begin{tikzcd}[ampersand replacement=\&,sep=scriptsize]
			{T(f \comma p)} \& {T(f \comma p)} \\[1ex]
			{f \comma p} \& {f \comma p} \\[-1ex]
			\&\&[-2ex] {f \spancomp p} \& \Ba \\[-1ex]
			\&\& \Aa
			\arrow[""{name=0, anchor=center, inner sep=0}, "Te", from=1-1, to=1-2]
			\arrow["{\alpha \comma \beta}"', from=1-1, to=2-1]
			\arrow["{\alpha \comma \beta}", from=1-2, to=2-2]
			\arrow[""{name=1, anchor=center, inner sep=0}, "e"', from=2-1, to=2-2]
			\arrow["r", from=2-2, to=3-3]
			\arrow["\pull\generic", curve={height=-12pt}, from=2-2, to=3-4]
			\arrow["{\vec{\pi}_{\Aa}}"', curve={height=12pt}, from=2-2, to=4-3]
			\arrow["{\pi_{\Ba}}", from=3-3, to=3-4]
			\arrow["{\pi_{\Aa}}", from=3-3, to=4-3]
			\arrow["{\colaxity e}", shorten <=4pt, shorten >=4pt, Rightarrow, from=1, to=0]
		\end{tikzcd}
	\end{equation}
	Again, since this is a 2-cell into a pullback (in $\Cosmos$), it will suffice to show that $\pi_{\Aa} r \colaxity{e}$ and $\pi_{\Ba} r \colaxity{e}$ are identities.
	For the first, since $\pi_{\Aa} r \colaxity{e} = \vec{\pi}_{\Aa} \colaxity{e}$; it suffices to show the latter is an identity.
	To this end, note that in $\Alg(T)_\colax$ we have that $\vec{\pi}_{\Aa} (e, \colaxity{e}) = \vec{\pi}_{\Aa}$; but $\vec{\pi}_{\Aa}$ is strict, so this equation of colax morphisms says exactly that $\vec{\pi}_{\Aa}\colaxity{e}$ is an idenity in $\Cosmos$.

	Next we turn to showing that $\pi_{\Ba} r \colaxity{e} = (\pull \generic)\colaxity{e}$ is an identity.
	To prove this, note if $(\lift\generic)\colaxity{e}$ is $p$-cartesian then $(\pull\generic)\colaxity{e}$ must also be $p$-cartesian;
	but $p(\pull\generic)\colaxity{e} = \generic\colaxity{e}$ which is an identity (since $\generic e$ is an identity in $\Alg(T)_\colax$).
	So, by normality we must conclude that $(\pull\generic)\colaxity{e}$ is an identity, as desired.

	It therefore remains to show that $(\lift\generic)\colaxity{e}$ is $p$-cartesian.
	We will first argue that $(e, \colaxity{e})$ is an idempotent in $\Alg(T)_\colax$ by the universal property of the comma $f \comma p$ and the normality of $p$.
	It suffices to show that $\vec{\pi_{\Aa}}e = e$ and $(\pull\generic)e = e$ in $\Alg(T)_\colax$; the first is immediate, and the second follows by normality of $p$ and the fact that $p(\lift \generic)e = \generic e$ is an identity.

	Since $e$ is an idempotent in $\Alg(T)_\colax$, we have the following equation again in $\Alg(T)_\colax$: $(\pull \generic)e = \vec{\pi}_{\Ba} ee = \vec{\pi}_{\Ba} e = \pull \generic$.
	Rendered in $\Cosmos$, this gives us the following equation:
	\begin{equation}
		\begin{tikzcd}[ampersand replacement=\&, row sep=small]
			{T(f \comma p)} \& {T(f \comma p)} \& TB \\
			\\
			{f \comma p} \& {f \comma p} \& \Ba
			\arrow[""{name=0, anchor=center, inner sep=0}, "Te", from=1-1, to=1-2]
			\arrow["{\alpha \comma \beta}"', from=1-1, to=3-1]
			\arrow[""{name=1, anchor=center, inner sep=0}, "{T\vec{\pi}_{\Ba}}", from=1-2, to=1-3]
			\arrow[""{name=1p, anchor=center, inner sep=0}, phantom, from=1-2, to=1-3, start anchor=center, end anchor=center]
			\arrow[from=1-2, to=3-2]
			\arrow["\beta", from=1-3, to=3-3]
			\arrow[""{name=2, anchor=center, inner sep=0}, "e"', from=3-1, to=3-2]
			\arrow[""{name=3, anchor=center, inner sep=0}, "{{\pull \generic}}"', from=3-2, to=3-3]
			\arrow[""{name=3p, anchor=center, inner sep=0}, phantom, from=3-2, to=3-3, start anchor=center, end anchor=center]
			\arrow["{{\colaxity{e}}}", between={0.2}{0.8}, Rightarrow, nfold, from=2, to=0]
			\arrow["{{\colaxity{\pull\generic}}}", between={0.2}{0.8}, Rightarrow, nfold, from=3p, to=1p]
		\end{tikzcd}
		\quad=\quad
		\begin{tikzcd}[ampersand replacement=\&, row sep=small]
			{T(f \comma p)} \& TB \\
			\\
			{f \comma p} \& \Ba
			\arrow[""{name=0, anchor=center, inner sep=0}, "{T\vec{\pi}_{\Ba}}", from=1-1, to=1-2]
			\arrow[""{name=0p, anchor=center, inner sep=0}, phantom, from=1-1, to=1-2, start anchor=center, end anchor=center]
			\arrow["{\alpha \comma \beta}"', from=1-1, to=3-1]
			\arrow["\beta", from=1-2, to=3-2]
			\arrow[""{name=1, anchor=center, inner sep=0}, "{\pull\generic}"', from=3-1, to=3-2]
			\arrow[""{name=1p, anchor=center, inner sep=0}, phantom, from=3-1, to=3-2, start anchor=center, end anchor=center]
			\arrow["{{\colaxity{\pull\generic}}}", between={0.2}{0.8}, Rightarrow, nfold, from=1p, to=0p]
		\end{tikzcd}
	\end{equation}
	We may therefore compute the horizontal composite $(\lift\generic)\colaxity{e}$ as follows:
	\begin{eqalign}
		\begin{tikzcd}[ampersand replacement=\&,row sep=small, column sep=scriptsize]
			{T(f\comma p)} \& {T(f\comma p)} \&\& T\Ba \&\& {T(f \comma p)} \& {T(f \comma p)} \&[2ex]\&[2ex] T\Ba \\
			\&\&\&\& {=} \\
			{f \comma p} \& {f \comma p} \&\& \Ba \&\& {f \comma p} \& {f \comma p} \&\& \Ba
			\arrow[""{name=0, anchor=center, inner sep=0}, "Te", from=1-1, to=1-2]
			\arrow[""{name=0p, anchor=center, inner sep=0}, phantom, from=1-1, to=1-2, start anchor=center, end anchor=center]
			\arrow["{\alpha \comma \beta}"', from=1-1, to=3-1]
			\arrow["{{T\vec{\pi}_{\Ba}}}", curve={height=-12pt}, from=1-2, to=1-4]
			\arrow["{\alpha \comma \beta}"', from=1-2, to=3-2]
			\arrow["\beta", from=1-4, to=3-4]
			\arrow[""{name=1, anchor=center, inner sep=0}, "Te", from=1-6, to=1-7]
			\arrow[""{name=1p, anchor=center, inner sep=0}, phantom, from=1-6, to=1-7, start anchor=center, end anchor=center]
			\arrow["{\alpha \comma \beta}"', from=1-6, to=3-6]
			\arrow[""{name=2, anchor=center, inner sep=0}, "{{T\vec{\pi}_{\Ba}}}", curve={height=-12pt}, from=1-7, to=1-9]
			\arrow[""{name=2p, anchor=center, inner sep=0}, phantom, from=1-7, to=1-9, start anchor=center, end anchor=center, curve={height=-12pt}]
			\arrow[""{name=3, anchor=center, inner sep=0}, "{{T(\pull\generic)}}"', curve={height=12pt}, from=1-7, to=1-9]
			\arrow[""{name=3p, anchor=center, inner sep=0}, phantom, from=1-7, to=1-9, start anchor=center, end anchor=center, curve={height=12pt}]
			\arrow[""{name=3p, anchor=center, inner sep=0}, phantom, from=1-7, to=1-9, start anchor=center, end anchor=center, curve={height=12pt}]
			\arrow["{\alpha \comma \beta}"', from=1-7, to=3-7]
			\arrow["\beta", from=1-9, to=3-9]
			\arrow[""{name=4, anchor=center, inner sep=0}, "e"', from=3-1, to=3-2]
			\arrow[""{name=4p, anchor=center, inner sep=0}, phantom, from=3-1, to=3-2, start anchor=center, end anchor=center]
			\arrow[""{name=5, anchor=center, inner sep=0}, "{{\vec{\pi}_{\Ba}}}", curve={height=-18pt}, from=3-2, to=3-4]
			\arrow[""{name=5p, anchor=center, inner sep=0}, phantom, from=3-2, to=3-4, start anchor=center, end anchor=center, curve={height=-18pt}]
			\arrow[""{name=6, anchor=center, inner sep=0}, "\pull\generic"', curve={height=18pt}, from=3-2, to=3-4]
			\arrow[""{name=6p, anchor=center, inner sep=0}, phantom, from=3-2, to=3-4, start anchor=center, end anchor=center, curve={height=18pt}]
			\arrow[""{name=7, anchor=center, inner sep=0}, "e"', from=3-6, to=3-7]
			\arrow[""{name=7p, anchor=center, inner sep=0}, phantom, from=3-6, to=3-7, start anchor=center, end anchor=center]
			\arrow[""{name=8, anchor=center, inner sep=0}, "{{\pull \generic}}"', curve={height=18pt}, from=3-7, to=3-9]
			\arrow[""{name=8p, anchor=center, inner sep=0}, phantom, from=3-7, to=3-9, start anchor=center, end anchor=center, curve={height=18pt}]
			\arrow["{{T(\lift\generic)}}"', shorten <=3pt, shorten >=3pt, Rightarrow, from=3p, to=2p]
			\arrow["{{\colaxity{e}}}", shorten <=9pt, shorten >=9pt, Rightarrow, from=4p, to=0p]
			\arrow["{{\lift \generic}}"', shorten <=5pt, shorten >=5pt, Rightarrow, from=6p, to=5p]
			\arrow["{{\colaxity{e}}}", shorten <=9pt, shorten >=9pt, Rightarrow, from=7p, to=1p]
			\arrow["{{\colaxity{\pull\generic}}}", shorten <=5pt, shorten >=14pt, Rightarrow, from=8p, to=3p]
		\end{tikzcd}
		\\
		\begin{tikzcd}[ampersand replacement=\&,column sep=scriptsize,row sep=small]
			\& {T(f \comma p)} \&[2ex]\&[2ex] T\Ba \&\& {T(f \comma p)} \&\& T\Ba \\
			{=} \&\&\&\& {=} \\
			\& {f \comma p} \&\& \Ba \&\& {f \comma p} \&\& \Ba
			\arrow[""{name=0, anchor=center, inner sep=0}, "{{T\vec{\pi}_{\Ba}}}", curve={height=-12pt}, from=1-2, to=1-4]
			\arrow[""{name=0p, anchor=center, inner sep=0}, phantom, from=1-2, to=1-4, start anchor=center, end anchor=center, curve={height=-12pt}]
			\arrow[""{name=1, anchor=center, inner sep=0}, "{{T(\pull\generic)}}"', curve={height=12pt}, from=1-2, to=1-4]
			\arrow[""{name=1p, anchor=center, inner sep=0}, phantom, from=1-2, to=1-4, start anchor=center, end anchor=center, curve={height=12pt}]
			\arrow[""{name=1p, anchor=center, inner sep=0}, phantom, from=1-2, to=1-4, start anchor=center, end anchor=center, curve={height=12pt}]
			\arrow["{\alpha \comma \beta}"', from=1-2, to=3-2]
			\arrow["\beta", from=1-4, to=3-4]
			\arrow["{{T\vec{\pi}_{\Ba}}}", curve={height=-12pt}, from=1-6, to=1-8]
			\arrow["{\alpha \comma \beta}"', from=1-6, to=3-6]
			\arrow["\beta", from=1-8, to=3-8]
			\arrow[""{name=2, anchor=center, inner sep=0}, "\pull\generic"', curve={height=12pt}, from=3-2, to=3-4]
			\arrow[""{name=2p, anchor=center, inner sep=0}, phantom, from=3-2, to=3-4, start anchor=center, end anchor=center, curve={height=12pt}]
			\arrow[""{name=3, anchor=center, inner sep=0}, "\pull\generic"', curve={height=12pt}, from=3-6, to=3-8]
			\arrow[""{name=4, anchor=center, inner sep=0}, "{{\vec{\pi}_{\Ba}}}", curve={height=-12pt}, from=3-6, to=3-8]
			\arrow["{{T(\lift\generic)}}"', shorten <=3pt, shorten >=3pt, Rightarrow, from=1p, to=0p]
			\arrow["{{\colaxity{\pull\generic}}}"', shorten <=4pt, shorten >=13pt, Rightarrow, from=2p, to=1p]
			\arrow["\lift\generic"', shorten <=3pt, shorten >=3pt, Rightarrow, from=3, to=4]
		\end{tikzcd}
	\end{eqalign}
	We find that $(\lift\generic)\colaxity{e} = \lift\generic$, so that in particular it is $p$-cartesian as we needed to show.

	Finally, we need to show that the induced algebra structure
	\begin{equation}
	\label{eqn:pullback.alg}
		T(f \spancomp p) \xto{Ti} T(f \comma p) \xto{\alpha \comma \beta} f \comma p \xto{r} f \spancomp p
	\end{equation}
	is indeed the pullback of $p$ along $f$ in $\Alg(T)_\colax$.
	We begin by observing that the projections $\pi_{\Aa} = \vec{\pi}_{\Aa} i$ and $\pi_{\Ba} = \vec{\pi}_{\Ba} i$ are the composite of colax $T$-morphisms and so may be endowed with the structure of colax morphisms themselves; we may therefore work entirely in $\Alg(T)_\colax$.

	First, we show the 1-categorical universal property.
	Suppose we are given $a : \Xa \to \Aa$ and $b : \Xa \to \Ba$ so that $fa = gb$ in $\Alg(T)_\colax$.
	Then we have a unique map $\internal{fa=gb} : \Xa \to f \comma p$ so that $\generic\internal{fa=gb} = \id_{fa} = \id_{gb}$.
	The map $r\internal{fa=gb} : \Xa \to f \spancomp p$ satisfies $\pi_{\Aa} r\internal{fa=gb} = \vec{\pi}_{\Aa}\internal{fa=gb} = a$ and $\pi_{\Ba}r\internal{fa=gb} = (\pull\generic)\internal{fa=gb}$, but $p(\lift\generic)\internal{fa=gb} = \generic\internal{fa=gb}$ is an identity, so by normality $(\lift\generic)\internal{fa=gb}$ is an identity.
	We conclude that $(\pull\generic)\internal{fa=gb} = \vec{\pi}_{\Ba}\internal{fa=gb} = b$.
	Now to show uniqueness of $r\internal{fa=gb}$, suppose that $z : \Xa \to f \spancomp p$ also has $\pi_{\Aa} z = a$ and $\pi_{\Ba} z = b$; then $\vec{\pi}_{\Aa} iz = \pi_{\Aa} z = a$ and $\vec{\pi}_{\Ba}i z = \pi_{\Ba} z = b$ and $\generic i z$ is an identity, so we must have $iz = \internal{fa=gb}$.
	But then $z = riz = r\internal{fa=gb}$, proving that $(a,b) := r\internal{fa=gb}$ is the unique such map into the pullback.

	Next, the 2-categorical universal property.
	Let $\varphi : a \twoto a':\Xa \to \Aa$ and $\psi : b \twoto b':\Xa \to \Ba$ so that $f\varphi = p\psi$.
	Again, we use the the universal property of $f \comma p$ in $\Alg(T)_\colax$ first to get a colax morphism $\internal{f\varphi = p\psi} : \internal{fa = pb} \twoto \internal{fa'=pb'} : \Xa \to f \comma p$:
	\begin{equation}
	\label{eqn:chi.gamma}
		\internal{f\varphi = p\psi} =
		\begin{tikzcd}[sep=small]
			fa & {fa'} \\
			pb & {pb'}
			\arrow["f\varphi", from=1-1, to=1-2]
			\arrow[Rightarrow, no head, from=1-1, to=2-1]
			\arrow[Rightarrow, no head, from=1-2, to=2-2]
			\arrow["p\psi"', from=2-1, to=2-2]
		\end{tikzcd}
	\end{equation}
	We aim to show $(\varphi,\psi) := r\internal{f\varphi = p\psi} : \Xa \to f \spancomp p$ is the desired comparison map for the pullback.
	To do so, we need first to show that $\pi_{\Aa} r \internal{f\varphi = p\psi} = \varphi$, but that follows from the fact $\pi_{\Aa}r \internal{f\varphi = p\psi} = \vec{\pi}_{\Aa}\internal{f\varphi = p\psi} = \varphi$.
	It remains to show $\pi_{\Ba}r\internal{f\varphi = p\psi} = \psi$.
	Between the two, we have a map $(\lift \generic)\internal{f\varphi = p\psi} : (\pull \generic)\internal{f\varphi = p\psi} \twoto \vec{\pi}_{\Ba}\internal{f\varphi = p\psi}$, and thus we aim to show such map is the identity.
	Now, $(\lift \generic)\internal{f\varphi = p\psi}$ is the action of $\lift$ on the square depicted in \eqref{eqn:chi.gamma}, whose sides are identities.
	It follows by naturality of $\lift$ and the fact $p$ is normal that to such square corresponds
	\begin{equation}
		\begin{tikzcd}[row sep=scriptsize]
			{\pull(\internal{fa = pb})} &[5ex] {\pull(\internal{fa' = pb'})} \\
			b & {b'}
			\arrow["{\pull(\internal{f\varphi = p\psi})}", from=1-1, to=1-2]
			\arrow["{\lift(\internal{fa = pb})}"', Rightarrow, no head, from=1-1, to=2-1]
			\arrow["{\lift(\internal{fa' = pb'})}", Rightarrow, no head, from=1-2, to=2-2]
			\arrow["\psi"', from=2-1, to=2-2]
		\end{tikzcd}
	\end{equation}
	which witnesses the required identity.

	To see $(\varphi,\psi)$ is unique, suppose $\rho:(a,b) \twoto (a',b')$ was another 2-cell satisfying $\pi_{\Aa}\rho = \varphi$ and $\pi_{\Ba}\rho = \psi$.
	Then consider $i\rho : i(a,b) \twoto i(a'b') : \Xa \to f \comma p$.
	Since $i(a,b) = \internal{fa = pb}$ and likewise for $(a',b')$, we get $i\rho = \internal{f\varphi = p\psi}$ by universal property of the comma, and thus $\rho = ri\rho = r\internal{f\varphi = p\psi} = (\varphi,\psi)$.

	Finally, knowing that $f \spancomp p$ is indeed the pullback in $\Alg(T)_\colax$ lets us conclude that $\pi_{\Aa} : f \spancomp p \to A$ is a normal fibration by \cref{lem:normal.fibrations.closed.under.pullback}.
	We have equipped $\pi_{A}$ with the colaxity of $\vec{\pi}_{\Aa} i$, which is explicitly the whisker $\vec{\pi_{\Aa}}\colaxity{e}$.
	But $\vec{\pi}_{\Aa} = \pi_{\Aa} r$, so $\vec{\pi}_{\Aa} \colaxity{e} = \pi_{\Aa} r \colaxity{e}$ is an identity; therefore, $\pi_{\Aa}$ is strict.
	%
\end{proof}

We note that a nearly identical argument to that just carried out applies for rectifying an isocomma into a strict morphism in $\Alg(T)_\pseudo$, the 2-category of strict algebras and pseudomorphisms.
We record this as a theorem as well.

\begin{thm}
\label{thm:pseudo.paradise}
	Let $\Paradise$ be a paradise and $T : \Cosmos \to \Cosmos$ be a 2-monad.
	Then the class of strict displayed normal isofibrations $\{\mathsf{sdni}\}$ equips $\Alg(T)_\pseudo$ with the structure of a paradise.
\end{thm}

\begin{ex}
\label{ex:monoidal.pullback}
	Instantiating \eqref{eqn:pullback.alg} to the case of the free monoidal category 2-monad on the paradise $(\Cat, \{\mathsf{all}\})$ of categories, in light of \cref{ex:monoidal.comma}, we get that $f \spancomp p$ is endowed with the monoidal structure so defined:
	\begin{equation}
	\label{eqn:monoidal.pullback}
		\monoidalunit := (\monoidalunit, \colaxity{f}_\monoidalunit^{\,*} \monoidalunit), \qquad
		(A,B) \monoidal (A',B') := (A \monoidal A', \colaxity{f}_{A,A'}^{\,*}(B \monoidal B')).
	\end{equation}

	Also, the projection $\pi_{\Aa}$ is a colax monoidal morphism with colaxity $\lift \colaxity{f}$.

	Moreover, suppose that $\action : \acted \times \actor \to \acted$ is a colax monoidal functor with respect to the product monoidal structure $(A, P) \monoidal (B, Q) := (A \monoidal B, P \monoidal Q)$.
	By \cref{lem:normal.fibration.examples}, the domain map $\dom : \acted^{\downarrow} \to \acted$ is a normal fibration with cartesian lift given by precomposition; it is also evidently strict and displayed (since we consider all functors to be display maps).
	Therefore, the monoidal structure \eqref{eqn:monoidal.pullback} on the pullback $\action \spancomp \dom$ (whose objects we write as displayed over $\acted \times \actor$) specializes in this case to
	\begin{eqalign}
	\label{eqn:preview.of.para.monoidal.structure}
		\monoidalunit &:= \lens{\monoidalunit \action \monoidalunit \xto{\colaxity{\action}} \monoidalunit}{\monoidalunit, \monoidalunit},\\
		\lens{A \action P \xto{f} A'}{A, P} \monoidal \lens{B \action Q \xto{g} B'}{B, Q} &:= \lens{(A \monoidal B) \action (P \monoidal Q) \xto{(f \monoidal g) \colaxity{\action}} A' \monoidal B'}{A \monoidal B, P \monoidal Q}
	\end{eqalign}
	This is exactly the monoidal structure we desired for parameterized maps in \eqref{eqn:para.map.monoidal}, and the one used for the self-action of a commutative monoidal category in \cite{capucci_towards_2022}.
\end{ex}

\subsection[Colaxly T-structured double categories]{Colaxly $T$-structured double categories}

Fix a paradise $\Paradise$ and a 2-monad $T : \Cosmos \to \Cosmos$ on it.
The whole reason for introducing the paradises $(\Alg(T)_\colax, \{\mathsf{sdnf}\})$ is to give structure to double categories of contextful morphisms.
To this end, we will introduce the notion of \emph{colaxly $T$-structured} contextads and pseudocategories. 

The general phenomenon of \emph{symmetry of internalization} \cite{arkor_2024} allows us to view a colaxly $T$-structured pseudocategory as a $T$-structure internal to pseudocategories (and lax morphisms), at least so long as $T$-structures may be \emph{sketched} --- presented as the models of an enhanced sketch in the sense of \cite{arkor_2024}. In particular

\begin{defn}
\label{defn:colaxly.T.structured.colax.fibred.action}
	A \textbf{colaxly $T$-structured contextad} is a contextad (\cref{defn:colax.fibred.action}) $\fibcolaxaction$ in $(\Alg(T)_\colax,\{\mathsf{sdnf}\})$, thus
	\begin{enumerate}
		\item $\acted$ and $\actor$ are equipped with $T$-algebra structures;
		\item $p$ is a normal display fibration and a strict $T$-morphism, $\action$ a colax $T$-morphism;
		\item the unit and multiplication have the structure of colax $T$-morphisms, and the unitor and associator isomorphisms are 2-cells between colax $T$-morphisms.
	\end{enumerate}
	A colaxly $T$-structured contextad is \textbf{strongly $T$-structured} when $\action$, the unit and multiplication maps are both strong $T$-morphisms, and \textbf{strictly $T$-structured} if they are strict $T$-morphisms.
\end{defn}

Clearly the above definition extends to morphisms, transformations and modifications of contextad, which are obtained by instantiating \cref{defn:contextad.1-cell,defn:contextad.2-cell,defn:contextad.3-cell} in $(\Alg(T)_\colax,\{\mathsf{sdnf}\})$.
Concretely, this amounts to ask to the data of such morphisms to colaxly preserve the $T$-structures around.

\begin{defn}
\label{defn:colaxly.T.structured.double.category}
	A \textbf{colaxly $T$-structured pseudocategory} $\Ca_0 \xleftarrow{s} \Ca_1 \xto{t} \Ca_0$ is a pseudomonad in $\DispSpan\Paradise$ with the following data and properties:
	\begin{enumerate}
		\item both $\Ca_0$ and $\Ca_1$ are equipped with $T$-algebra structures;
		\item with respect to these structures, $s$ and $t$ are strict $T$-morphisms;
		\item the unit and multiplication have the structure of colax $T$-morphisms, and the unitor and associator isomorphisms are 2-cells between colax $T$-morphisms.
	\end{enumerate}
	A colaxly $T$-structured pseudomonad is \textbf{strongly $T$-structured} when the unit and multiplication maps are both strong $T$-morphisms, and \textbf{strictly $T$-structured} if they are strict $T$-morphisms.
\end{defn}

Likewise, functors, transformations and modifications of colaxly $T$-structured pseudocategories coincide with those of pseudocategories with the added requirement and data of colaxly preserving the salient $T$-structure.

\begin{rmk}
	Note that a colax $T$-structure on a pseudocategory is a different structure than what one would get from a pseudomonad in $\DispSpan(\Alg(T)_\colax, \{\mathsf{sdni}\})$: for starters, both its legs are strict, and then the left leg doesn't need to be a normal fibration.

	It is more correct to say that a colax $T$-structure on a pseudocategory is a model of the enhanced sketch for pseudocategories (as in Example~5.13 of \cite{arkor_2024}), taken in $\Alg(T)_{\colax}$; indeed, colaxly $T$-structured pseudocategories appear as such in Example~4.17 of \cite{arkor_2024}. If $\Alg(T)_{\colax}$ is itself the category of models for an enhanced sketch (and by \cite{ko_2026}, any enhanced sketch whose marked cones are all tight gives rise to such a monad), then by symmetry of internalization, observe that colaxly structured $T$-categories are indeed $T$-structures internal to pseudocategories and lax morphisms.
\end{rmk}

\begin{thm}
\label{thm:para.is.colaxly.structured}
	If $\fibcolaxaction$ is a colaxly (resp. strongly) $T$-structured contextad in $\Paradise$, then $\Ctx(\action)$ is a colaxly (resp. strongly) $T$-structured pseudocategory.
\end{thm}
\begin{proof}
	We note that by virtue of being morphisms in $\Alg(T)_\colax$ (or $\Alg(T)_\pseudo$), the unit and multiplication of $\action$ are colax (or strong) $T$-morphisms, and the unitor and associator isomorphisms are $2$-cells between colax $T$-morphisms.
	All that remains to check is that the source and target maps are strict.
	\begin{equation}
	\label{eqn:para.colaxly.structured}
		\begin{tikzcd}[sep = small]
			&& {\Ctx(\action)_1} \\
			& \actor && {\acted^{\downarrow}} \\
			\acted && \acted && \acted
			\arrow["{\pi_{\actor}}"', from=1-3, to=2-2]
			\arrow["{\pi_{\acted^{\downarrow}}}", from=1-3, to=2-4]
			\arrow["\lrcorner"{anchor=center, pos=0.125, rotate=-45}, draw=none, from=1-3, to=3-3]
			\arrow["p"', from=2-2, to=3-1]
			\arrow["\action", from=2-2, to=3-3]
			\arrow["\dom"', from=2-4, to=3-3]
			\arrow["\cod", from=2-4, to=3-5]
		\end{tikzcd}
	\end{equation}
	We note that by hypothesis, $p$ is strict.
	Both $\dom$ and $\cod$ are strict.
	The projection $\pi_{\actor}$ is strict as proved in \cref{thm:colax.paradise};
	therefore, the source map $p\pi_{\actor}$ of $\Ctx(\action)$ is strict.
	Now, the projection $\pi_{\acted^{\downarrow}}$ is not strict, but we can prove that
	the composite $\cod \pi_{\acted^{\downarrow}}$ is strict.
	Since the colaxity of $\pi_{\acted^{\downarrow}}$ is defined to be the colaxity of $\vec{\pi}_{\acted^{\downarrow}}i$ (using the notation of \cref{thm:colax.paradise} from here on), which itself is $\vec{\pi}_{\acted^{\downarrow}}\colaxity{e}Ti$, it will suffice to show that the composite $\cod\vec{\pi}_{\acted^{\downarrow}}e$ in $\Alg(T)_\colax$ is strict.
	But $\cod\vec{\pi}_{\acted^{\downarrow}}e = \cod(\pull\generic)$, so it will suffice to show that $\cod(\pull\generic)$ is strict.
	Consider the following equality in $\Cosmos$:
	\begin{equation}
		\begin{tikzcd}[row sep=small, column sep=scriptsize]
			{T(\action \comma \dom)} && {\acted^{\downarrow}} & T\acted && {T(\action \comma \dom)} && {T\acted^{\downarrow}} & T\acted \\
			&&&& {=} \\
			{\action \comma \dom} && {\acted^{\downarrow}} & \acted && {\action \comma \dom} && {\acted^{\downarrow}} & \acted
			\arrow["{{T\vec{\pi}_{\acted^{\downarrow}}}}", curve={height=-12pt}, from=1-1, to=1-3]
			\arrow[from=1-1, to=3-1]
			\arrow["T\cod", from=1-3, to=1-4]
			\arrow[from=1-3, to=3-3]
			\arrow[from=1-4, to=3-4]
			\arrow[""{name=0, anchor=center, inner sep=0}, "{{T\vec{\pi}_{\acted^{\downarrow}}}}", curve={height=-12pt}, from=1-6, to=1-8]
			\arrow[""{name=0p, anchor=center, inner sep=0}, phantom, from=1-6, to=1-8, start anchor=center, end anchor=center, curve={height=-12pt}]
			\arrow[""{name=1, anchor=center, inner sep=0}, "{{T(\pull\generic)}}"', curve={height=12pt}, from=1-6, to=1-8]
			\arrow[""{name=1p, anchor=center, inner sep=0}, phantom, from=1-6, to=1-8, start anchor=center, end anchor=center, curve={height=12pt}]
			\arrow[""{name=1p, anchor=center, inner sep=0}, phantom, from=1-6, to=1-8, start anchor=center, end anchor=center, curve={height=12pt}]
			\arrow[from=1-6, to=3-6]
			\arrow["T\cod", from=1-8, to=1-9]
			\arrow[from=1-8, to=3-8]
			\arrow[from=1-9, to=3-9]
			\arrow[""{name=2, anchor=center, inner sep=0}, "{{\vec{\pi}_{\acted^{\downarrow}}}}", curve={height=-18pt}, from=3-1, to=3-3]
			\arrow[""{name=2p, anchor=center, inner sep=0}, phantom, from=3-1, to=3-3, start anchor=center, end anchor=center, curve={height=-18pt}]
			\arrow[""{name=3, anchor=center, inner sep=0}, "\pull\generic"', curve={height=18pt}, from=3-1, to=3-3]
			\arrow[""{name=3p, anchor=center, inner sep=0}, phantom, from=3-1, to=3-3, start anchor=center, end anchor=center, curve={height=18pt}]
			\arrow["\cod"', from=3-3, to=3-4]
			\arrow[""{name=4, anchor=center, inner sep=0}, "{{\pull \generic}}"', curve={height=18pt}, from=3-6, to=3-8]
			\arrow[""{name=4p, anchor=center, inner sep=0}, phantom, from=3-6, to=3-8, start anchor=center, end anchor=center, curve={height=18pt}]
			\arrow["\cod"', from=3-8, to=3-9]
			\arrow["{{T(\lift\generic)}}"', shorten <=3pt, shorten >=3pt, Rightarrow, from=1p, to=0p]
			\arrow["{{\lift \generic}}"', shorten <=5pt, shorten >=5pt, Rightarrow, from=3p, to=2p]
			\arrow["{{\colaxity{\pull\generic}}}", shorten <=5pt, shorten >=14pt, Rightarrow, from=4p, to=1p]
		\end{tikzcd}
	\end{equation}
	Now, since $\lift\generic$ is given by precomposition (see \cref{lem:normal.fibration.examples}), $\cod(\lift\generic)$ is an identity; this means that the diagram on the left is filled by an identity.
	But then $T\cod T(\lift\generic)$ is also an identity, which means that the whisker $\cod \colaxity{\pull\generic}$ must also be an identity, as desired.
\end{proof}

The latter theorem allows us to conclude there is a trifunctor $\Ctx^T_\colax$ sending colaxly $T$-structured contextads to colaxly $T$-structured pseudocategories:
\begin{equation}
	\begin{tikzcd}[ampersand replacement=\&]
		{\Cxd^T_\colax\Paradise} \&\& {\PsCat^T_\colax(\DispSpan\Paradise).}
		\arrow["{\Ctx^T_\colax\Paradise}", from=1-1, to=1-3]
	\end{tikzcd}
\end{equation}
Here $\Cxd^T_\colax\Paradise := \KL(\FibSpan^\twoto(\Alg(T)_\colax, \{\mathsf{sdni}\}),\, \fun)\co$, while $\PsCat^T_\colax(\DispSpan\Paradise)$ is the tricategory of colaxly $T$-structured pseudocategories.
Evidently, the above trifunctor restricts to $\Ctx^T_\pseudo$, going between strongly $T$-structured contextads and pseudocategories.

Notice that, since the composite left leg in \eqref{eqn:para.colaxly.structured} is a strict display normal fibration (being a composite of one and the pullback of one), then $\Ctx^T_\colax$ actually produces slightly stronger structures than colaxly $T$-structured pseudocategories, namely left-fibrant such pseudocategories.

\subsubsection{Monoidal structure}
\label{sec:monoidal.para}
Instantiating the above for $T$ the free monoidal category 2-monad on a sufficiently cocomplete cosmos $\Cosmos$ we get a monoidal $\Ctx$, and thus $\para$, construction.
In fact, in \cref{ex:higher.mons} we noticed that if $\Cosmos$ is a cosmos, then so is $\Mon(\Cosmos)$, the 2-category of pseudomonoids in $\Cosmos$.
Thus working with $\Mon(\Cosmos)$ for a generic $\Cosmos$ allows us to also cover the case of braided and symmetric pseudomonoids too (obtained as $\Mon(\Mon(\Cosmos))$ and $\Mon(\Mon(\Mon(\Cosmos)))$, respectively).

\begin{defn}
\label{defn:colax.mon.colax.fibred.action}
	A \textbf{colax monoidal contextad} is a left-displayed span $\fibcolaxaction$ with extra structure:
	\begin{enumerate}
		\item both $(\acted, \monoidalunit_{\acted}, \monoidal)$ and $(\actor, \monoidalunit_{\actor}, \monoidal)$ are monoidal categories,
		\item $\acted \nepifrom{p} \actor$ is a strict monoidal normal fibration,
		      i.e.~one whose cartesian lifts of monoidal products are monoidal products
		      of cartesian lifts, which lifts identities to identities, and that
		      preserves monoidal products up to coherent isomorphisms induced by the
		      universal property of cartesian lifts:
		      \begin{equation}
				\label{eq:colaxity-intertwiner}
				\id^*\lens{\monoidalunit_{\actor}}{\monoidalunit_{\acted}} \equalto \lens{\monoidalunit_{\actor}}{\monoidalunit_{\acted}},
				\qquad
				f^*\lens{P}{A} \monoidal g^* \lens{Q}{B} \isoto (f \monoidal g)^* \lens{P \monoidal Q}{A \monoidal B}
		      \end{equation}
		      As a consequence of $p$ being strict monoidal, we can denote the monoidal product in $\monoidal$ `componentwise', as we just did on the right-hand side of the above equation:
		      \begin{equation}
			      \monoidalunit_{\actor} \equiv \lens{\monoidalunit}{\monoidalunit_{\acted}},
			      \qquad
			      \lens{P}{A} \monoidal \lens{Q}{B} \equiv \lens{P \monoidal Q}{A \monoidal B}.
		      \end{equation}
		\item $\actor \nto{\action} \acted$ is a colax monoidal functor, thus equipped with coherent 2-cells:
		      \begin{equation}
			      \colaxity{\action} : \monoidalunit_{\acted} \action \monoidalunit_{\actor} \longto \monoidalunit_{\acted},
			      \qquad
			      \colaxity{\action} : (A \monoidal B) \action (P \monoidal R) \longto (A \action P) \monoidal (B \action R).
		      \end{equation}
		\item $\combineunit : \acted \to \actor$ is colax monoidal, thus equipped with $p$-vertical\footnotemark~coherent 2-cells:
		      \footnotetext{This is forced by the fact $p$ is strict, thus writing down the commutativity conditions for $(\combineunit, \colaxity{\combineunit})$ as a multiplication we get $p\colaxity{\combineunit}=\colaxity{p}=\id$.}
		      \begin{equation}
			      \colaxity{\combineunit} : \lens{\combineunit}{\monoidalunit_{\acted}} \to \lens{\monoidalunit}{\monoidalunit_{\acted}},
			      \qquad
			      \colaxity{\combineunit} : \lens{\combineunit}{A \monoidal B} \to \lens{\combineunit \monoidal \combineunit}{A \monoidal B}
		      \end{equation}
		\item $\combine : \actor \spancomp \actor \to \actor$ is colax monoidal.
		      Recall from \cref{ex:monoidal.pullback} the monoidal structure on $\actor \spancomp \actor$, given by
		      \begin{equation}
			      \monoidalunit := \left( \lens{\monoidalunit}{\monoidalunit_{\acted}}, \lens{\colaxity{\action}^{\,*} \monoidalunit}{\monoidalunit_{\acted} \action \combineunit} \right),
			      \qquad
			      \left(\lens{P}{A}, \lens{Q}{A \action P}\right) \monoidal \left(\lens{R}{B}, \lens{S}{B \action R}\right)
			      :=
			      \left(
			      \lens{P \monoidal Q}{A \monoidal B},\ %
			      \lens{
					      \colaxity{\action}^{\,*}(Q \monoidal S)
				      }{
					      (A \monoidal B) \action (P \monoidal R)
				      }
			      \right)
		      \end{equation}
		      Then the colaxity of $\combine$ is given by coherent $p$-vertical 2-cells:
		      \begin{equation}
			      \colaxity{\combine} : \lens{\monoidalunit \,\combine\; \colaxity{\action}^{\,*}\monoidalunit}{\monoidalunit_{\acted}} \to \lens{\monoidalunit}{\monoidalunit_{\acted}},
			      \qquad
			      \colaxity{\combine} : \lens{(P \monoidal R) \,\combine\; \colaxity{\action}^{\,*} (Q \monoidal S)}{A \monoidal B} \longto \lens{(P \combine Q) \monoidal (R \combine S)}{A \monoidal B}.
		      \end{equation}
		\item The counitor $\counitor$ is a monoidal natural transformation, and therefore satisfies:
		      \begin{equation}
			      \begin{tikzcd}[ampersand replacement=\&,sep=scriptsize]
				      {\monoidalunit_{\acted} \action \combineunit} \\[-1ex]
				      \& {\monoidalunit_{\acted}} \\[-1ex]
				      {\monoidalunit_{\acted} \action \monoidalunit_{\actor}}
				      \arrow["{\action\colaxity{\combineunit}}"', from=1-1, to=3-1]
				      \arrow["\counitor", from=1-1, to=2-2]
				      \arrow["{\colaxity{\action}}"', from=3-1, to=2-2]
			      \end{tikzcd}
			      \qquad\qquad
			      \begin{tikzcd}[ampersand replacement=\&,sep=scriptsize]
				      {(A \monoidal B) \action \combineunit} \\[-1ex]
				      \& {A \monoidal B} \\[-1ex]
				      {(A \action \combineunit) \monoidal (B \action \combineunit)}
				      \arrow["\counitor", from=1-1, to=2-2]
				      \arrow["{\action \colaxity{\combineunit}}"', from=1-1, to=3-1]
				      \arrow["{\counitor \monoidal \counitor}"', from=3-1, to=2-2]
			      \end{tikzcd}
		      \end{equation}
		\item The coassociator $\coassociator$ is a monoidal transformation, and therefore satisfies:
		      \begin{equation}
			      \begin{tikzcd}[ampersand replacement=\&,sep=scriptsize]
				      {\monoidalunit_{\acted} \action (\monoidalunit_{\actor} \combine \colaxity{\action}^{\,*}\monoidalunit_{\actor})} \&\& {(\monoidalunit_{\acted} \action \monoidalunit_{\actor}) \action (\colaxity{\action}^{\,*}\monoidalunit_{\actor})} \\
				      {\monoidalunit_{\acted} \action \monoidalunit_{\actor}} \&\& {\monoidalunit_{\acted} \action \monoidalunit_{\actor}} \\
				      \& \monoidalunit_{\acted}
				      \arrow["{\action \colaxity{\combine}}"', from=1-1, to=2-1]
				      \arrow["{\colaxity{\action}}"', from=2-1, to=3-2]
				      \arrow["\delta", from=1-1, to=1-3]
				      \arrow["{\action\colaxity{\pi}_2}", from=1-3, to=2-3]
				      \arrow["{\colaxity{\action}}", from=2-3, to=3-2]
			      \end{tikzcd}
		      \end{equation}
		      \begin{equation}
			      \begin{tikzcd}[ampersand replacement=\&, sep=scriptsize]
				      {(A \monoidal B) \action ((P \monoidal R) \combine (\colaxity{\action}^{\,*}(Q \monoidal S)))} \&\& {((A \monoidal B) \action (P \monoidal R)) \action (\colaxity{\action}^{\,*}(Q \monoidal S))} \\
				      {(A \monoidal B) \action ((P \combine Q) \monoidal (R \combine S))} \&\& {((A \action P) \monoidal (B \action R)) \action (Q \monoidal S)} \\
				      {(A \action (P \combine Q)) \monoidal (B \action (R \combine S))} \&\& {((A \action P) \action Q) \monoidal ((B \action R) \action S)}
				      \arrow["\coassociator", from=1-1, to=1-3]
				      \arrow["{\action\colaxity{\combine}}"', from=1-1, to=2-1]
				      \arrow["{\colaxity{\action}}"', from=2-1, to=3-1]
				      \arrow["{\action\colaxity{\pi}_2}", from=1-3, to=2-3]
				      \arrow["{\coassociator \monoidal \coassociator}"', from=3-1, to=3-3]
				      \arrow["{\colaxity{\action}}", from=2-3, to=3-3]
			      \end{tikzcd}
		      \end{equation}
	\end{enumerate}
\end{defn}

\begin{rmk}
	Note that \eqref{eq:colaxity-intertwiner} amounts to the colaxity of (the first component of) the intertwiner $\intertwiner$.
	The second component is a commutative diagram witnessing a coherence condition:
	\begin{equation}
		\begin{tikzcd}[ampersand replacement=\&]
			{(A \monoidal B) \action ((f \monoidal g)^*(Q \monoidal R))} \&\& {(A \monoidal B) \action ((f^*Q) \monoidal (g^*R))} \\
			{(C \monoidal D) \action (Q \monoidal R)} \&\& {(A \action (f^*Q)) \monoidal (B \action (g^*R))} \\
			\& {(C \action Q) \monoidal (D \action R)}
			\arrow["{(A \monoidal B) \action \colaxity{\intertwiner}}", from=1-1, to=1-3]
			\arrow["{(f \monoidal g) \action (Q \monoidal R)}"', from=1-1, to=2-1]
			\arrow["{\colaxity{\action}}", from=1-3, to=2-3]
			\arrow["{\colaxity{\action}}"', from=2-1, to=3-2]
			\arrow["{(f \action Q) \monoidal (g \action R)}", from=2-3, to=3-2]
		\end{tikzcd}
	\end{equation}
\end{rmk}

By \cref{thm:para.is.colaxly.structured} we may conclude that the $\Ctx$ construction of any colax monoidal contextad is a colaxly $T$-structured double category for $T$ the free monoidal category 2-monad.
This is precisely a \emph{lax monoidal double category}%
\footnote{Not to be confused with a doubly categorical version of \emph{lax monoidal categories}, which are monoidal categories whose structure morphisms are non-invertible (with a certain choice of direction, see~\cite[Definition~3.1.1]{leinsterHigherOperadsHigher2004}).}%
, to borrow terminology from \cite[Definition~4.1]{gambino2024monoidal}.%
\footnote{
	Note that in \emph{ibid.} they define a \emph{colax} monoidal double category, which is a pseudomonoid in double categories and colax double functors (whence the ``colax'').
	As they remark, this is the same thing as a pseudomonad in spans in the 2-category of monoidal categories and \emph{lax} monoidal functors with strict monoidal source and target.
	Since we are working with monoidal categories and \emph{colax} monoidal functors, we get, in their terminology, a \emph{lax} monoidal double category.
}

\begin{cor}
	Let $\fibcolaxaction$ be a colax monoidal contextad.
	Then $\Ctx(\action)$ is a lax monoidal double category.
\end{cor}

Let's unpack what this means.
First, $\Ctx(\action)$ is equipped with a monoidal product, coinciding on tight morphisms with that of $\acted$ and on loose 1-cells (contextful arrows) defined as in \cref{eqn:preview.of.para.monoidal.structure}, which we repeat here for convenience:
\begin{eqalign}
\label{eqn:para.monoidal.structure}
	\monoidalunit_{\Ctx(\action)} &:= \monoidalunit \action \monoidalunit \xto{\colaxity{\action}} \monoidalunit,\\
	(A \action P \xto{f} B) \monoidal_{\Ctx(\action)} (A' \action P' \xto{f'} B') &:= (P\monoidal P',\ (A \monoidal A') \action (P \monoidal P') \xto{(f \monoidal f') \colaxity{\action}} B \monoidal B').
\end{eqalign}
These assignments define lax double functors $\monoidalunit : 1 \to \Ctx(\action)$ and $\monoidal : \Ctx(\action) \times \Ctx(\action) \to \Ctx(\action)$---the fact these are lax is what the `lax' in `lax monoidal double category' refers to.

Concretely, this means there are coherent squares, obtained from the colaxities of $\combineunit$ and $\combine$, witnessing lax interchange of monoidal and looseward composition and units:
\begin{eqalign}
	\begin{tikzcd}[ampersand replacement=\&, row sep=scriptsize]
		{A \monoidal B} \&[1ex] {A \monoidal B} \\
		{A \monoidal B} \& {A \monoidal B}
		\arrow[""{name=0, anchor=center, inner sep=0}, "{(\combineunit_{A \monoidal B}, \counitor)}", "\shortmid"{marking}, from=1-1, to=1-2]
		\arrow[Rightarrow, no head, from=1-1, to=2-1]
		\arrow[Rightarrow, no head, from=1-2, to=2-2]
		\arrow[""{name=1, anchor=center, inner sep=0}, "{(\combineunit_A, \counitor) \monoidal (\combineunit_B, \counitor)}"', "\shortmid"{marking}, from=2-1, to=2-2]
		\arrow["{\colaxity{\combineunit}}", shorten <=4pt, shorten >=4pt, Rightarrow, from=0, to=1]
	\end{tikzcd}
	\quad&\quad
	\begin{tikzcd}[ampersand replacement=\&, row sep=scriptsize]
		\monoidalunit \& \monoidalunit \& \monoidalunit \\
		\monoidalunit \&\& \monoidalunit
		\arrow["{\colaxity{\action}}", "\shortmid"{marking}, from=1-1, to=1-2]
		\arrow["{\colaxity{\action}}", "\shortmid"{marking}, from=1-2, to=1-3]
		\arrow[Rightarrow, no head, from=2-1, to=1-1]
		\arrow[""{name=0, anchor=center, inner sep=0}, "{\colaxity{\action}}"', "\shortmid"{marking}, from=2-1, to=2-3]
		\arrow[""{name=0p, anchor=center, inner sep=0}, phantom, from=2-1, to=2-3, start anchor=center, end anchor=center]
		\arrow[Rightarrow, no head, from=2-3, to=1-3]
		\arrow[shorten >=3pt, Rightarrow, from=1-2, to=0p]
	\end{tikzcd}
	\\
	\begin{tikzcd}[ampersand replacement=\&, row sep=scriptsize]
		\monoidalunit \& \monoidalunit \\
		\monoidalunit \& \monoidalunit
		\arrow[""{name=0, anchor=center, inner sep=0}, "{(\combineunit, \counitor)}", "\shortmid"{marking}, from=1-1, to=1-2]
		\arrow[Rightarrow, no head, from=1-1, to=2-1]
		\arrow[Rightarrow, no head, from=1-2, to=2-2]
		\arrow[""{name=1, anchor=center, inner sep=0}, "{\colaxity{\action}}"', "\shortmid"{marking}, from=2-1, to=2-2]
		\arrow["{\colaxity{I}}"', shorten <=4pt, shorten >=4pt, Rightarrow, from=0, to=1]
	\end{tikzcd}
	\quad&\quad
	\begin{tikzcd}[ampersand replacement=\&,row sep=scriptsize]
		{A \monoidal A'} \&[3.5ex]\&[3.5ex] {C \monoidal C'} \\
		{A \monoidal A'} \&\& {A \monoidal A'}
		\arrow[""{name=0, anchor=center, inner sep=0}, "{((P,f) \monoidal (P',f')) \lcomp ((Q,g) \monoidal (Q',g'))}", "\shortmid"{marking}, from=1-1, to=1-3]
		\arrow[Rightarrow, no head, from=1-1, to=2-1]
		\arrow[Rightarrow, no head, from=1-3, to=2-3]
		\arrow[""{name=1, anchor=center, inner sep=0}, "{((P,f) \lcomp (Q,g)) \monoidal ((P',f') \lcomp (Q',g'))}"', "\shortmid"{marking}, from=2-1, to=2-3]
		\arrow["{\colaxity{\combine}}", shorten <=4pt, shorten >=4pt, Rightarrow, from=0, to=1]
	\end{tikzcd}
\end{eqalign}
The latter square is
\begin{eqalign}
	(P \monoidal P') \combine& ((f \monoidal f')\colaxity{\action})^*(Q \monoidal Q')\\
	&\isoto (P \monoidal P') \combine \colaxity{\action}^*(f \monoidal f')^*(Q \monoidal Q')\\
	&\nisoto{\colaxity{\intertwiner}} (P \monoidal P') \combine \colaxity{\action}^*(f^*Q \monoidal {f'}^*Q')\\
	&\nisoto{\colaxity{\combine}} (P \combine f^*Q) \monoidal (P' \combine {f'}^* Q')
\end{eqalign}
where the first is the canonical isomorphisms $((f \monoidal f')\colaxity{\action})^* \isoto \colaxity{\action}^*(f \monoidal f')^*$ and the second is \eqref{eq:colaxity-intertwiner}.

Notice how such structure does not depend on the colaxity of $\action$ itself, only on that of $\combineunit$ and $\combine$.
Therefore:

\begin{thm}
\label{thm:monoidal.ctx}
	Let $\fibcolaxaction$ be a strong monoidal contextad, meaning $p$ is a strict monoidal normal fibration and $\combineunit$ and $\combine$ are strong monoidal functors.
	Then $\Ctx(\action)$ is a monoidal double category.
\end{thm}

Let's see how the notion of \emph{colax monoidal contextad} generalizes two well-known structures, colax monoidal actegories and colax monoidal monads.

\begin{ex}
	Colax monoidal actegories appear as \cite[Definition~5.1.1]{capucci2022actegories}.
	These are actions of a braided monoidal category onto a monoidal category by means of a colax monoidal action $\action:\acted \times \actor \to \acted$.
	The colaxity of $\action$ is called \emph{mixed interchanger} and corresponds to what we described in \cref{eqn:para.map.monoidal}.

	As a contextad $\acted \nepifrom{\pi_{\acted}} \acted \times \actor \nto{\action} \acted$, a colax monoidal actegory is strong monoidal since both $\acted$ and $\acted \times \actor$ are monoidal, $\pi_{\acted}$ strictly preserves monoidal products, $\action$ is colax monoidal, and the pseudomonad structure is (strong) monoidal itself.
	The latter point is especially interesting since that's where the braided structure of $\actor$ comes into play: $\combineunit$ and $\combine$ are a second monoidal structure on top of the `background one' $(\monoidalunit, \monoidal)$ on $\actor$, and the first being monoidal with respect to the latter implies, by Eckmann--Hilton, that they actually coincide.
\end{ex}

Thus we can use \cref{thm:monoidal.ctx} to finally endow $\para(\action)$ with a monoidal structure:

\begin{cor}
\label{cor:para.monoidal}
	Let $(\acted, \action)$ be a colax monoidal $\actor$-actegory, for $\actor$ braided.
	Then $\para(\action)$ is a monoidal double category.
\end{cor}

Note that since $\para(\action)$ has all companions and all conjoints of invertible maps (cf. \eqref{eq:conjoints}), by an obvious corollary to the theorem of Shulman saying the loose bicategory of a monoidal equipment is monoidal \cite[Theorem~5.1]{shulman_constructing_2010}, the above implies that the bicategory $\para(\action)$ is monoidal too.

This result extends to braided and symmetric monoidal actegories \cite[Definition~5.4.1]{capucci2022actegories}, in fact generalizing the unproven result of \cite{capucci_towards_2022} to the colax case:

\begin{cor}
\label{cor:para.symmetric.monoidal}
	Let $(\acted, \action)$ be a colax braided (resp. symmetric) monoidal $\actor$-actegory, for $\actor$ symmetric monoidal.
	Then $\para(\action)$ is a braided (resp. symmetric) monoidal double category.
\end{cor}
For the record, a braided monoidal double category is a monoidal double category where additionally its categories of objects and arrows are braided too, with looseward unit and composition being braided.
This means in $\para(\action)$ there are squares:
\begin{equation}
	\begin{tikzcd}[ampersand replacement=\&,row sep=scriptsize]
		{A \monoidal B} \& {C \monoidal D} \\
		{B \monoidal A} \& {D \monoidal C}
		\arrow[""{name=0, anchor=center, inner sep=0}, "{(P,f) \monoidal (Q,g)}", "\shortmid"{marking}, from=1-1, to=1-2]
		\arrow[""{name=0p, anchor=center, inner sep=0}, phantom, from=1-1, to=1-2, start anchor=center, end anchor=center]
		\arrow["{\beta_{\acted}}"', from=1-1, to=2-1]
		\arrow["{\beta_{\acted}}", from=1-2, to=2-2]
		\arrow[""{name=1, anchor=center, inner sep=0}, "{(Q,g) \monoidal (P,f)}"', "\shortmid"{marking}, from=2-1, to=2-2]
		\arrow[""{name=1p, anchor=center, inner sep=0}, phantom, from=2-1, to=2-2, start anchor=center, end anchor=center]
		\arrow["{\beta_{\actor}}", shorten <=4pt, shorten >=4pt, Rightarrow, from=0p, to=1p]
	\end{tikzcd}
\end{equation}
where $\beta_{\acted}$ is braiding in $\acted$ and $\beta_{\actor}$ is braiding in $\actor$.
It's then easy to see one needs both $\acted$ and $\actor$ to be symmetric for $\para(\action)$ to be symmetric as well.

\begin{ex}
	When considering the $\Ctx$ construction for comonads $\Comonad : \acted \to \acted$, the colaxity of $\action$ corresponds to a coherent transformation $\Comonad(A \monoidal C) \xto{\colaxity{\Comonad}} \Comonad A \monoidal \Comonad B$, thus making $\Comonad$ a \emph{colax monoidal comonad}.
	Applying \cref{thm:para.is.colaxly.structured} then recovers a familiar theorem: the Kleisli category of a colax monoidal comonad inherits the monoidal structure of $\acted$.
	This theorem is most commonly known in its dual form, where a lax monoidal monad (also known as a commutative monad) has a monoidal structure on its Kleisli category, see \cite[Corollaire~7]{guitartTenseursMachines1980}.
\end{ex}

\subsubsection{Limits in Para}\label{sec:limits}

There is a nearly immediate corollary we can draw from \cref{thm:colax.paradise}: if $T$ is a free completion monad, then $\Ctx$ will have $T$-limits in the strict direction as a double category.
Free completion monads are \emph{colax idempotent} (recall \cref{defn:colax.idempotent.monad}), which are characterized by the fact that any map $f : A \to B$ between $T$-algebras extends \emph{uniquely} to a colax $T$-morphism $(f, \colaxity{f})$.

\begin{thm}
\label{thm:free.completion.structured.ctx}
	Let $T$ be a colax idempotent monad on a paradise $\Paradise$, and let $\fibcolaxaction$ be a fibred action in $\DispSpan\Paradise$ where $p$ is normal as a fibration.
	If $\acted$ and $\actor$ are both $T$-algebras and $p$ is a strict $T$-morphism, then $\Ctx(\action)$ is colaxly $T$-structured.
\end{thm}
\begin{proof}
	We will show that these assumptions are sufficient to show that $\fibcolaxaction$ is a fibred action in $(\Alg(T)_\colax, \{\mathsf{sdnf}\})$; the result will then follow by \cref{thm:para.is.colaxly.structured}.

	Since $T$ is colax idempotent, all morphisms of $\Cosmos$ are uniquely equipped with colax $T$-morphism structures and all 2-cells are $T$-2-cells.
	By hypothesis, $p$ is a strict $T$-morphism and is a normal fibration in $\Cosmos$; but this means that it is a normal fibration in $\Alg(T)_\colax$, because the right adjoint $\pull$ will be a colax $T$-morphism.
	Therefore, it is a strict displayed normal fibration.
\end{proof}

\begin{ex}
	If $\acted$ is a category with finite limits, then the fibred action $\acted \xleftarrow{\cod} \acted^{\downarrow} \xto{\dom} \acted$ is a strongly $T$-structured fibred action for $T$ the finite limits completion 2-monad.
	Therefore, by \cref{thm:free.completion.structured.ctx}, $\Ctx(\dom) = \Span(\acted)$ is a tightly complete double category---a well known fact \cite[§6.4]{grandis_limits_1999}.
\end{ex}

Recall from e.g. \cite[Definition~4.1]{lambert.patterson:cartesian.double.theories} that a \textbf{lax cartesian} (aka \textbf{precartesian}) double category is a cartesian object in $\DblCat_\lax$, and a \textbf{(strong) cartesian} one is a cartesian object in $\DblCat_\pseudo$.
We may now appeal to the characterization of lax cartesian and cartesian double categories given in \cite[Proposition~4.2]{lambert.patterson:cartesian.double.theories} to conclude that the $\Ctx$ construction on categories with finite products is lax cartesian.

\begin{cor}
	Suppose that $\fibcolaxaction$ is a contextad in $\Cat$ and that $p$ is normal as a fibration.
	If $\acted$ and $\actor$ have finite products and $p$ strictly preserves them, then $\Ctx(\action)$ is a lax cartesian double category.
\end{cor}
\begin{proof}
	By \cite[Proposition~4.2]{lambert.patterson:cartesian.double.theories}, a double category is lax cartesian if and only if its categories of objects and morphisms are cartesian and its source and target preserve products: in other words, if it is colaxly $T$-structured for $T$ the free product completion, which is a colax idempotent monad.
\end{proof}

\begin{ex}
\label{ex:lax.cartesian.ctx}
	If $\action : \acted \times \actor \to \acted$ is any colax monoidal action and $\acted$ and $\actor$ have finite products, then $\Ctx(\action)$ is lax cartesian.
	In particular, if $\action : \acted \times \actor \to \acted$ is any actegory with $\acted$ and $\actor$ both having finite products, then $\para(\action)$ is lax cartesian.

	As a consequence, if $\acted$ is a category with finite products, then for any
	comonad $\Comonad : \acted \to \acted$ (or even dependently graded comonad
	with a discrete category of grades, i.e. where $p : \acted \xleftarrow \actor$
	is a discrete functor), the Kleisli double category is lax cartesian; in fact, since the Kleisli double category is thin so that any square with source and target arrows identities is invertible, the Kleisli double category is in fact cartesian.
\end{ex}

For $\Ctx$ to be cartesian, and not just lax cartesian, it would suffice for the unit, and multiplication of the fibred action to preserve products.
This occurs, for example, when we are forming the Para construction of a cartesian category acting on itself.

\begin{ex}
	If $\acted$ is a category with finite products, then we may consider $\times : \acted \times \acted \to \acted$ as an actegory.
	It follows that $\para(\times)$ is a cartesian double category.
\end{ex}

\begin{ex}
\label{ex:simply.fibred.action}
	A useful double category is obtained from the simple fibration $\simple_{\acted} : \Simple{\acted} \to \acted$, whose total category has objects pairs $\lens{A}{B}$ and maps $\lens{A}{B} \rightrightarrows \lens{A'}{B'}$ pairs $f:B \to B'$ and $f^\flat : B \times A \to A'$ \cite[Definition~1.3.1]{jacobs_categorical_1999}.
	The projection $\simple_{\acted}$ projects the bottom component, and since $\Simple{\acted}$ has finite products given componentwise, $\simple_{\acted}$ preserves them strictly.
	Moreover, it is usually considered with the normal cleavage given by lifting a map $f: B \to \simple_{\acted}\lens{A'}{B'}$ to $\lens{\pi_{A'}}{f} : \lens{A'}{B} \rightrightarrows \lens{A'}{B'}$.
	Now $\times$ is still a well-defined action $\Simple{\acted} \to \acted$, making $\acted \nepifrom{\simple_{\acted}} \Simple{\acted} \nto{\times} \acted$ a cartesian fibred strong action, which we dub \textbf{simply fibred action}, from which a cartesian double category of contextful arrows $\Ctx(\times)$ is obtained.
\end{ex}

We conclude with an archetypal example of bimonoidal double category, arising from considering pseudomonoids in $\CartCat_\colax$.

\begin{ex}
	Consider the case of a bicartesian category $\Xa$, i.e.~a category having both finite products and coproducts.
	In this situation, there are canonical maps:
	\begin{equation}
		\varsigma : 0 \times 1 \longto 0,
		\qquad
		\varsigma : (A + B) \times (P \times Q) \longto A \times P + B \times Q
	\end{equation}
	Let now $(\acted, \monoidalunit, \monoidal) = (\Xa, 0, +)$ and $(\actor, \monoidalunit, \monoidal) = (\Xa,1,\times)$.
	Then the fibred action $\acted \nepifrom{\pi_{\acted}} \acted \times \actor \nto{\times} \acted$, with $(\combineunit,\combine) = (1,\times)$ and $\colaxity{\times} = \varsigma$, is a strong monoidal fibred action.
	The resulting double category of contextful arrows $\Ctx(\times)$ is already cartesian monoidal by \cref{ex:lax.cartesian.ctx}, but then it is also monoidal with respect to the \emph{external choice product} $(\extch)$ \cite{capucci_translating_2022}:
	\begin{equation}
		(P, A \times P \xto{f} B) \extch (P', A' \times P' \xto{f'} B') = (P \times P', (A + A') \times (P \times P') \xto{\varsigma} A \times P + A' \times P' \xto{f+f'} B+B')
	\end{equation}
	The same applies to the simply fibred action described in \cref{ex:simply.fibred.action}.
\end{ex}

\section{Conclusions and future work}

In this work we presented a new formalism for handling `context', broadly speaking, generalizing comonads, comprehension categories and actegories, and to construct (double) categories of contextful arrows generalizing the Kleisli, Para, and Span constructions.
There are several aspects of the theory of contextads we did not investigate further here but which we would like to highlight.

\paragraph{Transposing effects: duality and distributive laws}%
The first is the newfound distinction between contexts and effects we covered in \cref{sec:ctx.as.dep.graded.comonad}, recovering the original insights of Moggi and Kieburtz.
This reconceptualizes the yoga of monadic programming as a more concrete act of context manipulation, which here we presented only in half: truly, contextful arrows which \emph{observe} their context should be accompanied with contentful maps which \emph{change} their context, and these two acts cannot always be resolved to a one-sided operation.

This makes intuitive sense, though it is intriguingly hard to find examples of such situations which cannot be fixed otherwise.
For instance, in every cartesian closed category the graded monad $- \times A$ which can be used to model computations returning additional content of type $A$ can be curried to a graded comonad $(-)^A$, and thus computations reading from and writing to a context can be modelled as contextful computations for a `bigraded' state comonad $(- \times B)^A$.
Even traditional non-determinisitc effects can be reified, e.g.~probabilistic effects usually expressed via a distribution monad $\Delta(-)$ can be more realistically defined as contextful effects \emph{parameterized} by an explicit sample space $(\Omega, p)$, like we did in \cref{ex:law} but transposing the powering $\pitchfork$ back to a tensoring.
Indeed, there seems to be a general theory of duality which often allows to transpose `monadic' structures into `contextadic' ones, a first instance of which is \cref{thm:pra.monads.transpose} which we state in \cref{sec:poly.monad.dep.graded.comonad}, where we investigate such a duality for polynomial monads.

In any case, just as monads and comonads can be combined by suitable distributive laws (whose classification and definition is now a well-established industry at the intersection of the theory of programming languages, automata theory and category theory), we expect there to be an interesting theory of distributive laws between contextads and contentads, giving in turn rise to double categories of contextful-contentful maps.
Chiefly, we conjecture that polynomials (in the sense of \cite{gambino.kock:polynomial.monads}) might arise as a context-$\Cnt$ construction, just as their linear counterpart, spans, arises as a $\Ctx$ construction.

We expect such mixed distributive laws to arise from a similar structure one level higher, putting in relation the $\KL$ construction underlying contextads and their $\Ctx$ construction with the $\EM$ construction underlying contentads.

\paragraph{Contextads and factorization systems}%
As observed already towards the end of \cite{lack_formal_2002}, wreaths can be used to describe orthogonal factorization systems.
More interestingly, the converse is also true, with every double category of the form $\Ctx(\action)$ having a natural factorization system of sorts on its loose arrows: indeed, every contextful arrow $(P,f):A \looseto B$ can be written as the composite of a `contextual' part $(P, \id) : A \looseto A \action P$ and a `pure' part $(\combineunit, \counitor \lcomp f) : A \action P \looseto B$.
In $\Para$, this factorization system is called \emph{expansion/raw} in \cite{hermida_monoidal_2012}; for $\Span$ it is also obvious to observe that each span can be written as the composite of a functional and a cofunctional one; for a comonad $\Comonad$, $\Kl(\Comonad)$ also admits a factorization in iterates of the comultiplication $\coassociator$ and maps which factor through $\counitor$.
Even more interestingly, it seems reasonable to expect that from the data of a double category with such a factorization system one should be able to recover a contextad.
To even state such a reconstruction theorem would require one to, first of all, develop an adequate theory of factorization systems on double categories, of which we are not aware.

\paragraph{Microcosmic aspects}%
Finally, there is an evident `microcosmic' aspect underlying the whole story of this work.
\cref{ex:spans} shows that spans are a particular instance of wreath products and left-displayed spans are the backdrop on which our construction took place.
Therefore, we can expect that a higher-dimensional Kleisli completion can be used to define a higher-dimensional wreath product that, in turn, can produce $\DispSpan$ by a categorified version of \cref{ex:spans}---indeed, display maps and arbitrary maps clearly form an adequate triple of sorts.

By considering different wreaths at this higher level we might then obtain even more general definitions of contextad.
We have in mind a specific goal that this generalization might achieve, which is to capture (pseudo)actions in general monoidal 2-categories.
In fact contextads only generalize actions in \emph{cartesian} monoidal 2-categories---these are the ones for which the product of the actor and the actee is indeed fibred (trivially so) over the actee.
But in a general monoidal 2-category, the domain of an action $\acted \otimes \actor \to \acted$ might not admit any functor back to $\acted$, let alone a cartesian fibration.

We expect these actions to give wreaths \emph{in a different higher wreath product} than $\DispSpan(\Cosmos)$.
Specifically, one can expect every monoidal 2-category $(\Mb, I, \otimes)$ to give rise to a higher wreath $\Mb \nfrom{\pi_1} \Mb \times \Mb \nto{\otimes} \Mb$ whose product `$\Ctx(\otimes)$' can replace $\DispSpan$ as a setting in which to consider actions.
Wreaths \emph{in} $\Ctx(\otimes)$ would then correspond to (pseudo)actions in $\Mb$.

Even when it comes to contextads, one might want to encode the dependency of scalars on the actee in ways different than a fibration, e.g.~for $\Cosmos=\Cat$ indexed categories are a popular equivalent choice.
This would yield the definition of dependently graded comonad, as expounded in \cref{defn:dep.graded.comonad}.

\bibliographystyle{elsarticle-harv}
\bibliography{bibliography}

\appendix
\renewcommand{\thesection}{\Alph{section}}
\renewcommand{\theequation}{\thesubsection.\arabic{equation}}
\renewcommand{\theHequation}{\theHsubsection.\arabic{equation}}

\section{Appendix}
\label{appendix}

\subsection{Omitted proofs}
\subsubsection[Section 2]{\cref{sec:preliminaries}}
\begin{prop}
\label{prop:pb.of.rni.dbl.cats}
	Representable normal isofibrations of double categories admit strict pullbacks along arbitrary strong double functors.
\end{prop}
\begin{proof}
	We first appeal to \cref{thm:pseudo.paradise} to argue \emph{strict} normal isofibrations satisfy the required property, for $\Cosmos = \Cat^{\downdownarrows}$ and display maps all maps.

	Now consider $G$ non-strict normal isofibration.
	Using Grandis--Parè strictification (\cite[Theorem~7.5]{grandis_limits_1999}) we strictify $G$ to $\bar G$.
	Such a procedure preserves isofibrancy (and normality), since to lift a square $\bar u \isoto \bar G(\bar v)$ against $\bar G$ is to lift the square $\hat u \isoto \hat {G(v)} \isoto G(\hat v)$ against $G$.%
	\footnote{We are using notation from \emph{ibid.}: for a string $\bar u$ of composable arrows, $\hat u$ denotes its right-associated composition.}
	Thus given a cospan $\dblcat A \nto{G} \dblcat C \nfrom{F} \dblcat B$, we can find an isomorphic one $\dblcat A \nto{\bar G} \dblcat C \nfrom{F} \dblcat B$ whose pullback we know exists as required, and which is therefore a pullback for the first cospan too.
\end{proof}

\section{Familiar polynomial monads as contextads}
\label{sec:polynomial.monad.examples}

In this section, we expand on \cref{sec:poly.monad.dep.graded.comonad} by working out a few common examples.

\subsection{The writer monad}\label{sec.writer.monad.dep.graded.comonad}

The simplest monad to transpose is the \emph{writer monad} $X \mapsto X \times W$ for a fixed monoid $(W, +, 0)$. The unit of this monad is $x \mapsto (x, 0) : X \to X \times W$, and the multiplication is $((x, w_1), w_2) \mapsto (x, w_1 + w_2) : (X \times W) \times W \to X \times W$. A Kleisli morphism for the writer monad is a function $f : X \to Y \times W$. This evidently equivalent to a pair of functions $(w, f_1)$ with $w : X \to W$ and $f_1 : X \to Y$ by defining $w = \snd \circ f$ and $f_1 = \fst \circ f$. We will see $w : X \to W$ as a grade (dependent on $X$) for a dependently graded comonad transposing the writer monad, so that the pair $(w, f_1)$ becomes a contextful arrow.

\begin{defn}
\label{defn:dep.graded.writer.comonad}
	Fix a monoid $(W, +, 0)$. We define the \textbf{dependently graded writer comonad} on the category of sets as follows:
	\begin{enumerate}
		\item For each set $X$, the (discrete) category of grades over $X$ is the set of functions $X \to W$, which is evidently contravariantly functorial in $X$.
		\item For each set $X$ and grade $w : X \to W$, we define the action $X \action w := X$ to be $X$ again.
		\item We define $\combineunit_X : X \to W$ to be the constant map valued at $0 \in W$.
		\item Given $w_1 : X \to W$ and $w_2 : X \action w_1 \to W$, we define $w_1 \combine w_2 (x) = w_1(x) + w_2(x)$.
		\item We define the counitor $\counitor_X : X \action \combineunit_X \to X$ to be the identity.
		\item We define the coassociator $\coassociator_{w_1,w_2} : X \action (w_1 \combine w_2) \to (X \action w_1) \action w_2$ to be the identity.
	\end{enumerate}
	Since the categories of grades are discrete, the structure isomorphisms become equations which hold immediately by the axioms governing the monoid structure on $W$.
\end{defn}

\begin{thm}
	The writer monad is transposable: the Kleisli category of the writer monad is equivalent to the category of contextful arrows of the dependently graded writer comonad (\cref{defn:dep.graded.writer.comonad}).
\end{thm}
\begin{proof}
	All that we need to show is that the evident bijection between Kleisli maps $f : X \to Y \times W$ and contextful arrows $(w : X \to W, f_1 : X \to Y)$ is functorial. The unit $x \mapsto (x, 0)$ transposes to the contextful arrow $(x \mapsto 0, x \mapsto x) = (\combineunit_X, \counitor_X)$, which is the identity of $\Ctx(\action)$.

	If $f : X \to Y \times W$ and $g : Y \to Z \times W$ are Kleisli morphisms, their composite is $x \mapsto (g_1(f_1(x)), f_2(x) + g_2(f_1(x)))$. Meanwhile, $f$ transposes to $(f_2, f_1)$ and $g$ transposes to $(g_2, g_1)$, and these compose to
	\begin{equation}
		(f_2 \otimes f_1^*g_2,\ g_1 \circ (f_1 \action g_2) \circ \coassociator) = x \mapsto (f_2(x) + g_2(f_1 (x)),\ g_1(f_1(x))).
	\end{equation}
\end{proof}

\subsection{The state monad}
\label{sec.state.monad.dep.graded.comonad}

Fixing a set $S$ of states, the \emph{state monad} is the monad with underlying functor $X \mapsto (X \times S)^S$ induced by the tensor-hom adjunction $(-) \times S \dashv (-)^S$. Explicitly, the state monad has unit $\eta : X \to (X \times S)^S$ defined by $\eta(x)(s) = (x, s)$ and Kleisli lift taking $k : X \to (Y \times S)^S$ to the map $(X \times S)^S \to (Y \times S)^S$ given by $\varphi \mapsto (s \mapsto k(\fst \varphi(s))(\snd \varphi(s)))$. A Kleisli morphism for the state monad has the form $k : X \to (Y \times S)^S$, which is equivalent by transposing to the data of two functions
\begin{equation}
	\begin{cases}
		f : X \times S \to S & \mbox{and} \\
		k' : X \times S \to Y
	\end{cases}.
\end{equation}
We will therefore be able to see the Kleisli category of the state monad as the category of contextful arrows for a contextad---or dependently graded comonad---where $f : X \times S \to S$ is the grade and $k' : X \times S \to Y$ is the contextualized map.

\begin{defn}
\label{defn:state.dep.graded.comonad}
	Fix a set $S$ of \textbf{states}. We define the \textbf{state dependently graded comonad} $\State_S$ on the category of sets as follows:
	\begin{enumerate}
		\item For each set $X$, the (discrete) category of grades $\State_S(X) := \{X \times S \to S\}$ is the set of state machines with state set $S$ and input alphabet $X$. This is evidently contravariantly functorial in $X$.
		\item For each set $X$ and grade $f : X \times S \to S$, we define $X \action f := X \times S$.
		\item We define $\combineunit_X : X \times S \to S$ to be $\snd := (x, s) \mapsto s$ to be the second projection.
		\item We define the product $f \combine g : X \times S \to S$ of $f : X \times S \to S$ and $g : (X \times S) \times S \to S$ to be
		\begin{equation}
			  (f \combine g)(x,s) := g((x,s), f(x,s)).
		\end{equation}
		\item  We define the counitor $\counitor_X : X \times S \to X$ to be $\fst := (x, s) \mapsto x$.
		\item We define the coassociator $\coassociator_{f, g} : X \times S \to (X \times S) \times S$ to be
		\begin{equation}
			  \coassociator_{f, g}(x, s) = ((x, s), f(x, s))
		\end{equation}
	\end{enumerate}
	Since the categories of grades are discrete, the structure isomorphisms become equations which we check in \cref{sec:appendix.state.dep.comonad.well-defined}.
\end{defn}

\label{sec:appendix.state.dep.comonad.well-defined}
Here is why \cref{defn:state.dep.graded.comonad} is well-posed, i.e.~here's the check of its structure equations:
\begin{enumerate}
	\item $\leftunitlaw$ is the identity $\snd \combine \fst^*g = g$ which we check as follows:
		  \begin{align*}
			  (\snd \combine \fst^*g)(s, x) & = \fst^*g((x, s), \snd(x, s)) \\
											& = g(\fst(x, s), \snd(x,s))         \\
											& = g(x, s).
		  \end{align*}
		  $\rightunitlaw$ is the identity $f \combine \snd = f$ which we check as follows:
		  \begin{align*}
			  (f \combine \snd)(s, x) & = \snd((x, s), f(x, s)) \\
									  & = f(x, s).
		  \end{align*}
	\item $\associativitylaw$ is the identity $f \combine (g \combine h) = (f \combine g) \combine \coassociator_{f, g}^* h$ which we check as follows:
		  \begin{align*}
			  f \combine (g \combine h)(x, s) & = (g \combine h)((x, s), f(x, s))\\
											  & = h(((x, s), f(x, s)), g((x, s), f(x, s)))                   \\
											  & = h(\coassociator_{f, g}(x, s), g((x, s), f(x, s)))          \\
											  & = \coassociator_{f, g}^*h((x, s), g((x, s), f(x, s)))   \\
											  & = \coassociator_{f, g}^*h((x, s), (f \combine g)(x, s)) \\
											  & = (f \combine g) \combine \coassociator_{f,g}^*h(x, s).
		  \end{align*}
	\item For each set $X$ and grade $f : X \times S \to S$, the map $(\counitor_{X} \action f) \coassociator_{\combineunit_X,\counitor_X^*f}$ is given by
		  \begin{equation}
			  (x, s) \mapsto ((x, s), f(x, s)) \mapsto \fst((x, s), f(x, s)) = (x, s).
		  \end{equation}
		  This proves \cref{eqn:dep.graded.left.counit.law}. On the other hand, the map $\counitor_{X \action f} \circ \coassociator_{f,\combineunit_{X \action f}}$ is also given by
		  \begin{equation}
			  (x, s) \mapsto ((x, s), f(x, s)) \mapsto \fst((x, s), f(x, s)) = (x, s).
		  \end{equation}
		  This proves \cref{eqn:dep.graded.right.counit.law}.
	\item For each set $X$ and grades $f : X \times S \to S$, $g : (X \times S) \times S \to S$, and $h :((X \times S) \times S) \times S \to S$, we have the following coassociativity of $\coassociator$:
		  \begin{align*}
			  (\coassociator_{f, g} \action h)(\coassociator_{(f \combine g), \coassociator^*h}(x, s)) & =
			  (\coassociator_{f, g} \action h)((x, s), (f \combine g)(x, s))                                                                                                 \\
																											& = (((x, s), f(x, s)), (f \combine g)(x, s))                    \\
																											& = (((x, s), f(x, s)), g((x, s), f(x, s)))                      \\
																											& = \coassociator_{g,h}((x, s), f(x, s))                         \\
																											& = \coassociator_{g,h}(\coassociator_{f,(g \combine h)}(x, s)).
		  \end{align*}
		  This proves \cref{eqn:dep.graded.coassociativity.law}.
\end{enumerate}

\begin{thm}
	The category of contextful arrows for the state dependently graded comonad is isomorphic to the Kleisli category of the state monad.
\end{thm}
\begin{proof}
	As we remarked above, the data of a Kleisli morphism $k : X \to (Y \times S)^S$ is equivalent to the data of a contextful arrow $(f : X \times S \to S, k' : X \action f \to Y)$.
	It remains to show that this correspondence is functorial.
	For the purpose of this proof, write $\hat{f} : X \times S \to Y$ for the transpose of a map $f : X \to Y^S$.

	The identity Kleisli morphism is the unit $\eta : X \to (X \times S)^S$.
	This transposes to the pair $(\snd : X \times S \to S, \fst : X \times S \to X)$, which is the identity $(\combineunit, \counitor)$ of $\Ctx(\action)$ for the state dependently graded comonad.

	If $k : X \to (Y \times S)^S$ and $j : Y \to (Z \times S)^S$ are Kleisli morphisms, their composite is given by $x \mapsto (s \mapsto j(\snd k(x)(s))(\fst k(x)(s)))$.
	The contextful composite of the transposes $(\fst \hat{k}, \snd \hat{k})$ and $(\fst \hat{j}, \snd \hat{j})$ is the contextful arrow $(\fst \hat{k} \combine (\snd \hat{k})^*(\fst \hat{j}), \snd \hat{j} \circ (\snd \hat{k} \action \fst \hat{j}) \circ \coassociator_{\fst\hat{k},\fst\hat{j}})$.
	We may compute its components as follows:
	\begin{align*}
		(\fst \hat{k} \combine (\snd \hat{k})^*(\fst \hat{j}))(x, s)                                            & = (\snd \hat{k})^*(\fst \hat{j})((x, s), \fst\hat{k}(x, s))                       \\
		                                                                                                             & = \fst j(\snd k(x)(s))(\fst k(x, s))                                                   \\
		(\snd \hat{j} \circ (\snd \hat{k} \action \fst \hat{j}) \circ \coassociator_{\fst\hat{k},\fst\hat{j}})(x, s) & = (\snd \hat{j} \circ (\snd \hat{k} \action \fst \hat{j}))((x, s), \fst \hat{k}(x, s)) \\
		                                                                                                             & = \snd \hat{j}(\snd \hat{k}(x, s), \fst \hat{k}(x, s))                                 \\
		                                                                                                             & = \snd j(\snd k(x)(s), \fst k(x)(s)).
	\end{align*}
	and these are the two components of the transpose of the Kleisli composite
	\begin{equation}
		x \mapsto (s \mapsto j(\snd k(x)(s))(\fst k(x)(s))).
	\end{equation}
\end{proof}

\subsection{The Maybe monad}
\label{sec:maybe.monad.dep.graded.comonad}
As a further example, let's look at  the Maybe monad $X \mapsto X + 1$.
We will show that with dependent types we can transpose it into a dependently graded comonad.
Our construction here will work more generally for any dominance taking the place of $\Bool$ or, more generally, any $\Sigma$-closed subuniverse taking the place of decidable propositions.

For this example, we will work in Martin-L\"of type theory.
Define the type family
\begin{eqalign}
	&\BoolIf : \Bool \to \Type\\
	&\BoolIf(\True) := \Unit\\
	&\BoolIf(\False) := \Zero
\end{eqalign}
as well as the map
\begin{eqalign}
	&\supp : Y + \Unit \to \Bool\\
	&\supp(\inl(y)) := \True\\
	&\supp(\inr(\ast)) := \False
\end{eqalign}
We also need a `lemma':
\begin{eqalign}
	&e : \dsum{p : Y + \Unit} \BoolIf(\supp(p)) \to Y\\
	&e(\inl(y), \ast) := y
\end{eqalign}
No other patterns are reachable.

\vspace*{1ex}

Let $f : X \to Y + \Unit$ be a Kleisli map.
If $\supp f(x) = \True$, then $f$ is defined at $x$ and we can extract the value of $Y$ it maps to.
Explicitly, we have
\begin{eqalign}
	&\hat{f} : \dsum{x : X} \BoolIf(\supp (f (x))) \to Y\\
	&\hat{f}(x, p) := e(f(x), p)
\end{eqalign}
This gives a bijection
\begin{equation}
\label{eqn:maybe.monad.transposition}
	\{f : X \to Y + \Unit\} \iso \{ (s : X \to \Bool,\ \hat{f} : \dsum{x : X} \BoolIf(s(x)) \to Y)\}.
\end{equation}
We will interpret these pairs as contextful arrows arising from a dependently graded comonad where the type of grades associated to $X$ is $X \to \Bool$ and $X \action s := \dsum{x : X} \BoolIf(s(x))$.

To define $\combineunit$ and $\combine$, we will need the dominance structure on $\Bool$.
Specifically, we have $\True : \Bool$ and we need
\begin{eqalign}
	&\andbool : (b : \Bool) \times (\BoolIf(b) \to \Bool) \to \Bool\\
	&\andbool(\True, h) := h(\ast)\\
	&\andbool(\False, h) := \False
\end{eqalign}
We will also need the map
\begin{eqalign}
	&\splitt : (b : \Bool) \times (s : \BoolIf(b) \to \Bool) \times \BoolIf(\andbool(b, s)) \to \dsum{p : \BoolIf(b)} \BoolIf(s(p))\\
	&\splitt(\True, s, p) := (\ast, p).
\end{eqalign}

\begin{rmk}
	We note that the Maybe monad, in the language of \cref{sec:poly.monad.dep.graded.comonad} may be expressed as the polynomial functor associated to $\BoolIf : \Bool \to \Type$, since $X + 1 \simeq \dsum{b : \Bool} X^{\BoolIf(b)}$.
	We then have $\ok = \True$, $\seq = \andbool$, and $\splitt = \splitt$.
\end{rmk}

We can now define
\begin{eqalign}
	&\combineunit_X : X \to \Bool &&&&& &\_\combine\_ : (s : X \to \Bool) \times (X \action s \to \Bool) \to (X \to \Bool)\\
	&\combineunit_X(x) := \True &&&&& &(s \combine t) := x \mapsto \andbool(s(x), p \mapsto t(x, p)).
\end{eqalign}
Finally, we take $\counitor : X \action \combineunit \to X$ to be $\fst$ and $\coassociator : X \action (s \combine t) \to (X \action s) \action t$ to be
\begin{equation}
	(x, p) \mapsto \letin{(b_1, b_2)}{\splitt(p)}{((x, b_1), b_2)}.
\end{equation}

Checking the laws is tedious but straightforward.

The following, extending \cref{eqn:maybe.monad.transposition}, is now evident:

\begin{thm}
	The category of contextful arrows for the maybe dependently graded comonad is isomorphic to the Kleisli category of the Maybe monad.
\end{thm}

\begin{rmk}
	The Kleisli category for the Maybe monad is equivalent to the category of partial maps (with decidable support).
	The class of decideable subobjects in a lextensive category (having finite limits and disjoint coproducts) is a strong display map class.
	By \cref{ex:display.map.cat}, we therefore get a contextad of decideable subobjects whose double category of contexts is the double category of spans whose left leg is a decideable subobject.
	This is equivalent to the double category whose tight maps are total maps and whose loose arrows are partial maps with decideable support.
\end{rmk}

\section[Duality: contentads and the Cnt constructions]{Duality: contentads and the $\Cnt$ construction}
\label{sec:duality}
\subsection[The Cnt construction]{The $\Cnt$ construction}
The story told so far dualizes with little effort. 
The duals to contextads, that we call \emph{contentads}, are contentads on categories and `dependently' graded monads.
If contextads are all about providing extra \emph{context} to morphisms, their dual are about providing extra \emph{content}, in the form of an extra output or an `effect' (like an exception, or non-determinism).
The dual to the $\Ctx$ construction, which we call $\Cnt$, generalizes the Kleisli (for monads) and Copara constructions, yielding a double category whose loose arrows $A \to E \action B$ are maps from $A$ to $B$ which generate extra content $E$, the latter composing by pushforward and accumulation.

Let's see how this arises by repeating the construction we performed in the previous pages.
First, we dualize paradises to coparadises, which are completely analogous except its the codomain projections which are required to be display:

\begin{defn}
\label{defn:coparadise}
	A \textbf{coparadise} is a display map 2-category $\Paradise$ where:
	\begin{enumerate}
		\item $\Display$ is closed under composition and contains all isomorphisms,
		\item $\Cosmos$ has all arrow objects $\Ca^{\downarrow}$, and the codomain map $\cod : \Ca^{\downarrow} \to \Ca$ is in $\Display$.
	\end{enumerate}
\end{defn}

The fact the arrow monad $\Ca \nfrom{\dom} \Ca^\downarrow \nto{\cod} \Ca$ is now right-displayed brings our attention to the tricategory $\DispSpan_r\Paradise$ of right-displayed spans.
Similarly as before, we can prove right-displayed spans with left-colax maps between two fixed objects are the Kleisli 2-category of an arrow monad, and then algebras therein give us 2-categories of right-opfibrant spans with left-colax maps:
\begin{equation}
\label{eqn:dual.main.iso}
	\OpfibSpan^\twofrom_r\Paradise(\Ba,\Ca) \iso \Alg(- \spancomp \Ca^\downarrow, \DispSpan_r^\twofrom\Paradise(\Ba,\Ca)) \iso \Alg(- \spancomp \Ca^\downarrow, \Kl(\Ba^\downarrow \spancomp -, \DispSpan_r\Paradise(\Ba,\Ca))).
	\hspace*{-6ex}
\end{equation}
Compared to \cref{thm:main.iso}, the arrow monads now act on the opposite sides, and thus the above isomorphism witnesses that $\OpfibSpan^\twofrom_r\Paradise$ is a full subtricategory of the \emph{free Eilenberg--Moore completion} of $\DispSpan_r\Paradise$.
The free Eilenberg--Moore completion is dual to the free Kleisli cocompletion, which is related to by $\EM(-)\op = \KL(-\op)$.
It can be characterized explicitly as
\begin{equation}
\label{eqn:free.em.completion.homs}
	\EM(\Kc)((\Ab, s), (\Bb, t)) := \Alg(t-, \Kl(-s, \Kc(\Ab,\Bb))).
\end{equation}

Like $\KL$, $\EM$ is defined by a universal property which induces another \emph{wreath product}\footnote{In fact, originally, the wreath product was defined for the Eilenberg--Moore completion \cite{lack_formal_2002}.} $\wreath : \EM(\EM(\Kc)) \to \EM(\Kc)$.
Therefore, we can define a trifunctor:
\begin{equation}
\label{eqn:cnt.on.objects}
	\begin{tikzcd}[ampersand replacement=\&]
		{\EM(\OpfibSpan^\twofrom_r\Paradise)} \&[7ex] {\EM(\EM(\DispSpan_r\Paradise))} \& {\EM(\DispSpan_r\Paradise).}
		\arrow["{\EM\InclusionTrifun}", hook, from=1-1, to=1-2]
		\arrow["\wreath", from=1-2, to=1-3]
	\end{tikzcd}
	\hspace*{-4ex}
\end{equation}
The final step is restricting to functorial spans.
Notice that we are working with a dual category of spans compared to our previous treatment, hence we restrict to spans whose \emph{right leg} is an identity, aka \emph{opfunctorial spans}, which define a locally full subtricategory $\opfun \subseteq \OpfibSpan^\twofrom_r\Paradise$.
Indeed, since the right leg of a span points in the opposite way, restricting to opfunctorial spans leaves us working with the opposites of both the tricategories of contentads on categories and that of pseudocategories.

\begin{defn}
\label{defn:opfibred.lax.action}
	A \textbf{contentad} in a coparadise $\Paradise$ is a pseudomonad $(\opfiblaxaction,\ \combineunit,\ \combine)$ in $\OpfibSpan^\twofrom_r\Paradise$, the tricategory of right-opfibrant spans and left-colax maps.
	These gather in the tricategory $\Cnd\Paradise$, defined as $\EM(\OpfibSpan^\twofrom_r\Paradise,\,\opfun)\op$ (recall \cref{not:tight.maps}).
\end{defn}

Concretely, a contentad $\acted \nfrom{\action} \actor \nepito{q} \acted$ is a kind of left action on $\acted$, i.e.~the opfibration structure on $q$ allows, for any object $P \in \actor_A$ and morphism $g: A \to B$, to construct the morphism $P \action g : P \action A \to g_*P \action B$.

Given objects $P \in \actor_C$ and $Q \in \actor_{P \action C}$, one can form their \textbf{accumulation} $P \combine Q \in \actor_C$.
There is still a unit $\combineunit_A \in \actor_A$ for each $A \in \acted$, but now the structure 2-cells go in the opposite direction:
\begin{equation}
	\unitor : A \to \combineunit_A \action A, \qquad \associator : P \action (Q \action A) \to (P \combine Q) \action A.
\end{equation}
Thus contentads can be considered `lax opfibred actions', or `dependently graded monads'.

Morphisms, transformations and modifications of contentads on categories are formally dual to those of contextads, hence we don't spell out their (rather lengthy) definitions anew.
We just point out that, just like a morphism of contextads is by default \emph{lax} (with laxity given by the lineator \eqref{eqn:morph.of.fib.colax.actions}), those between contentads on categories are \emph{colax} by default, thus equipped with a \textbf{colineator}:
\begin{equation}
	\begin{tikzcd}[ampersand replacement=\&]
		\acted \& \actor \& \acted \\
		{\acted'} \& {\actor'} \& {\acted'}
		\arrow["F"', dashed, from=1-1, to=2-1]
		\arrow["\action"', from=1-2, to=1-1]
		\arrow["q", two heads, from=1-2, to=1-3]
		\arrow["\colineator"', shorten <=4pt, shorten >=6pt, Rightarrow, dashed, from=1-2, to=2-1]
		\arrow["{F^\flat}", dashed, from=1-2, to=2-2]
		\arrow["F", dashed, from=1-3, to=2-3]
		\arrow["{\action'}", from=2-2, to=2-1]
		\arrow["{q'}"', two heads, from=2-2, to=2-3]
	\end{tikzcd}
	\qquad\qquad
	\colineator_{\lens{E}{A}} : F(P \action A) \to F^\flat_A(P) \action FA
\end{equation}
Likewise, a transformation of contentads $(c,\tau) : \underline F \twoto \underline G$ is a \textbf{contentful natural transformation}:
\begin{equation}
	\begin{tikzcd}[ampersand replacement=\&]
		\& \acted \\
		{\acted'} \& {\actor'} \& {\acted'}
		\arrow["F"', curve={height=12pt}, from=1-2, to=2-1]
		\arrow["\tau"', shift left=2, shorten <=2pt, shorten >=4pt, Rightarrow, dashed, from=1-2, to=2-1]
		\arrow["c", dashed, from=1-2, to=2-2]
		\arrow["G", curve={height=-12pt}, from=1-2, to=2-3]
		\arrow["{\action'}", from=2-2, to=2-1]
		\arrow["{q'}"', two heads, from=2-2, to=2-3]
	\end{tikzcd}
	\qquad\qquad
	\begin{aligned}
		&c_A \in \actor'_{FA}\\[.5ex]
		&\tau_A : FA \longto c_A \action' GA
	\end{aligned}
\end{equation}
Notice that, unlike transformations of contextads, these point already in the right direction and that's why we didn't dualize 2-cells in \cref{defn:opfibred.lax.action}, unlike \cref{defn:tricat.of.colax.fibred.actions}.

Finally, a modification of contentads $\omega : (c,\tau) \threeto (d,\sigma)$ has the form
\begin{equation}
	\begin{tikzcd}[ampersand replacement=\&]
		\&\&[-2ex] \acted \&[-5ex]\\
		\\
		{\acted'} \&\& {\actor'} \&\& {\acted'}
		\arrow[""{name=0, anchor=center, inner sep=0}, "F"'{pos=0.6}, curve={height=18pt}, from=1-3, to=3-1]
		\arrow[""{name=1, anchor=center, inner sep=0}, "d"{pos=0.2}, curve={height=-18pt}, from=1-3, to=3-3]
		\arrow[""{name=1p, anchor=center, inner sep=0}, phantom, from=1-3, to=3-3, start anchor=center, end anchor=center, curve={height=-18pt}]
		\arrow[""{name=1p, anchor=center, inner sep=0}, phantom, from=1-3, to=3-3, start anchor=center, end anchor=center, curve={height=-18pt}]
		\arrow[""{name=2, anchor=center, inner sep=0}, "c"'{pos=0.9}, curve={height=18pt}, from=1-3, to=3-3]
		\arrow[""{name=2p, anchor=center, inner sep=0}, phantom, from=1-3, to=3-3, start anchor=center, end anchor=center, curve={height=18pt}]
		\arrow["G"{pos=0.6}, curve={height=-18pt}, from=1-3, to=3-5]
		\arrow["{{{\action'}}}", from=3-3, to=3-1]
		\arrow["{q'}"', two heads, from=3-3, to=3-5]
		\arrow["\omega"', shift right=2, shorten <=4pt, Rightarrow, dashed, from=2p, to=1p]
		\arrow["\sigma", curve={height=-12pt}, shorten <=7pt, shorten >=7pt, Rightarrow, from=0, to=1p]
		\arrow["\tau"'{pos=0.4}, curve={height=6pt}, shorten <=5pt, shorten >=8pt, Rightarrow, from=0, to=2]
	\end{tikzcd}
	\qquad\qquad
	\omega_A : c_A \to d_A
\end{equation}

\begin{defn}
\label{defn:cnt.construction}
	The \textbf{$\Cnt$ construction} is the dashed trifunctor below, obtained by co/restricting the functor defined in \cref{eqn:cnt.on.objects} to functorial spans, and dualizing 2-cells:
	\begin{equation}
		\begin{tikzcd}[ampersand replacement=\&,row sep=small]
			{\EM(\OpfibSpan^\twofrom_r\Paradise,\, \opfun)\op} \& {\Cnd\Paradise} \\
			{\EM(\EM(\DispSpan_r\Paradise),\, \opfun)\op} \\
			{\EM(\DispSpan_r\Paradise,\, \opfun)\op} \& {\PsCat\Paradise}
			\arrow[equals, nfold, from=1-1, to=1-2]
			\arrow["{\EM\InclusionTrifun\op_\opfun}"', hook, from=1-1, to=2-1]
			\arrow["{\Cnt\Paradise}", dashed, from=1-2, to=3-2]
			\arrow["{\wreath\op_\opfun}"', from=2-1, to=3-1]
			\arrow[equals, nfold, from=3-1, to=3-2]
		\end{tikzcd}
	\end{equation}
\end{defn}

A loose 1-cell in $\Cnt(\action)$ is a \textbf{contentful arrow}, thus an arrow $f:A \to P \action B$, where $P \in \actor_B$, and given $g:B \to Q \action C$ their composition is obtained by first pushing forward the content $P$ along $g$, then compose the resulting arrows, and finally combining the resulting contents using $\associator$:
\begin{equation}
\label{eqn:contentful.maps.comp}
	f \lcomp g := A \nlongto{f} P \action B \nlongto{P \action g} g_*P \action (Q \action C) \nlongto{\associator} (g_*P \combine Q) \action C.
\end{equation}
Identities are defined by the components of $\unitor$.

\begin{ex}
\label{ex:left.actegories}
	Left actions of monoidal categories are trivially opfibred strong actions ${\acted \nfrom{\action} \actor \times \acted \nepito{\pi_{\acted}} \acted}$, dualizing \cref{ex:actegory}.
	The associated $\Cnt$ construction is a doubly-categorical version of $\copara$ \cite[Remark~4]{capucci_towards_2022}: its tight cells are morphisms from $\acted$, its loose 1-cells are coparameterized maps:
	\begin{equation}
		(P,f) : A \looseto B = (P \in \acted, f:A \to P \action B).
	\end{equation}
	The 2-cells $(P,f) \twoto (P',f')$ are morphisms $\varphi:P \to P'$ commuting with $f$ and $f'$.
\end{ex}

\begin{ex}
\label{ex:monads}
	Monads $\Monad:\acted \to \acted$ are contentads on categories whose right leg is an identity, dualizing \cref{ex:comonad}.
	$\Cnt(\Monad)$ is a doubly-categorical version of the Kleisli category of $\Monad$: its loose 1-cells are Kleisli maps $f:A \to \Monad B$, with \eqref{eqn:contentful.maps.comp} specializing to the lift-compose composition typical of Kleisli categories.
	Likewise, graded monads are trivially contentads on categories $\acted \nfrom{\Monad_{(-)}} \actor \times \acted \nepito{\pi_{\acted}} \acted$, and $\Cnt(\Monad_{(-)})$ is a sort of graded Kleisli double category.
	Like noted in \cref{rmk:fujii.kleisli}, this differs from Fujii's definition of graded Kleisli category.
\end{ex}

\begin{ex}
\label{ex:codisplay.map.cat}
	Dualizing \cref{ex:display.map.cat}, consider categories $\acted$ equipped with a distinguished replete subcategory of `codisplay maps', these being maps in $\acted$ that admit and are stable under pushout by arbitrary maps.
	This makes $\dom: \acted^{\downarrow_{\Display}} \to \acted$ an opfibration, and indeed one can show (reasoning by duality) $\acted \nfrom{\cod} \acted^{\downarrow_{\Display}} \nepito{\dom} \acted$ is a contentad.
	Performing the $\Cnt$ construction we recover $\Cospan(\acted,\Display)$, the double category of cospans in $\acted$ whose right leg is codisplay.
	In particular, \eqref{eqn:contentful.maps.comp} recovers the composition rule of cospans.
\end{ex}

\begin{ex}
\label{ex:law}
	Here's an example of a morphism between contentads of different nature.
	Let $\Delta : \Set \to \Set$ be the monad of finitely supported probability distributions \cite{fritzConvexSpacesDefinition2015}.
	Consider also $\FinProb$, the category of finite (i.e.~carried by finite sets) probability spaces and measure-preserving measurable maps, along with the left action it induces on $\Set$ by powering with the underlying set:
	\begin{equation}
		(\Omega, p) \pitchfork A = A^\Omega.
	\end{equation}
	Now there are maps $\law_{A,\Omega}:A^\Omega \to \Delta A$ which send a map $f : \Omega \to A$ to its law $f_*p(a) = p(f^{-1}a)$.
	These form a natural transformation which is the colineator of a map of contentads ${\pitchfork} \to {\Delta}$:
	\begin{equation}
		\begin{tikzcd}[ampersand replacement=\&]
			\Set \& {\FinProb\op \times \Set} \& \Set \\
			\Set \& \Set \& \Set
			\arrow[Rightarrow, no head, from=1-1, to=2-1]
			\arrow["\pitchfork"', from=1-2, to=1-1]
			\arrow["{\pi_2}", from=1-2, to=1-3]
			\arrow["\law"', shorten <=7pt, shorten >=7pt, Rightarrow, from=1-2, to=2-1]
			\arrow["{\pi_2}", from=1-2, to=2-2]
			\arrow[Rightarrow, no head, from=1-3, to=2-3]
			\arrow["\Delta", from=2-2, to=2-1]
			\arrow[Rightarrow, no head, from=2-2, to=2-3]
		\end{tikzcd}
	\end{equation}
	The unitor and multiplicator of this map are strict (because the fibers of $\Set \equalto \Set$ are trivial) and thus automatically coherent.
\end{ex}

\subsubsection{Duality}
To understand the duality relating $\Ctx$ and $\Cnt$, we start by observing that if $\Paradise$ is a coparadise, then $\DualParadise$ is a paradise.
Moreover, dualizing the 2-cells in $\Cosmos$ has the effect of `reversing' the arrow objects, so that the arrow monad associated to an object $\acted$ is now $\acted \nfrom{\cod} \acted^\downarrow \nto{\dom} \acted$ (here $\cod$ and $\dom$ are, respectively, the domain and codomain 1-cells in $\Cosmos\co$ but we kept the name from $\Cosmos$).
We call this formal opposite $\acted^\uparrow$.
Observe that composing with $\acted^\uparrow$ on the left of a span is analogous to compose the reversed span with $\acted^\downarrow$ on the right:
\begin{equation}
\label{eqn:arrow.monads.duality}
	\begin{tikzcd}[ampersand replacement=\&,sep=scriptsize]
		\&\& {\Ca \comma \action} \\
		\& \Ma \&\& {\Ca^\uparrow} \\
		\Ca \&\& \Ca \&\& \Ca
		\arrow[from=1-3, to=2-2]
		\arrow[from=1-3, to=2-4]
		\arrow["\lrcorner"{anchor=center, pos=0.125, rotate=-45}, draw=none, from=1-3, to=3-3]
		\arrow["q"', two heads, from=2-2, to=3-1]
		\arrow["\action", from=2-2, to=3-3]
		\arrow["\cod"', from=2-4, to=3-3]
		\arrow["\dom", from=2-4, to=3-5]
	\end{tikzcd}
	\leftrightsquigarrow
	\begin{tikzcd}[ampersand replacement=\&,sep=scriptsize]
		\&\& {\Ca \comma \action} \\
		\& {\Ca^\downarrow} \&\& \Ma \\
		\Ca \&\& \Ca \&\& \Ca
		\arrow[from=1-3, to=2-2]
		\arrow[from=1-3, to=2-4]
		\arrow["\lrcorner"{anchor=center, pos=0.125, rotate=-45}, draw=none, from=1-3, to=3-3]
		\arrow["\dom"', from=2-2, to=3-1]
		\arrow["\cod", from=2-2, to=3-3]
		\arrow["\action"', from=2-4, to=3-3]
		\arrow["q", two heads, from=2-4, to=3-5]
	\end{tikzcd}
\end{equation}
As a consequence of this reversal, the definition of cartesian fibration (\cref{defn:fib}) also dualizes: an algebra of $\acted^\uparrow$ is a $q:\Ma \to \acted$ with a $\push : q \comma \acted \to \Ma$ left adjoint to $\internal{\id_q}$, which amounts to endowing $q$ with a cocartesian cleavage.
This explains \eqref{eqn:dual.main.iso}.

Therefore, writing $(-)\re$ for the duality reversing 3-cells, $\OpfibSpan^\twofrom_r\Paradise = \FibSpan^\twoto\DualParadise\opre$ and similarly $\DispSpan_r\Paradise = \DispSpan\DualParadise\opre$.
Applying these identities and the duality $\EM(-\opre) \iso \KL(-)\opre$, we get the duality below (note $\PsCat\DualParadise\re = \PsCat\Paradise$):
\begin{equation}
\label{eq:first.duality}
	\begin{tikzcd}[ampersand replacement=\&]
		{\Cnd\Paradise} \&\& {\PsCat\Paradise} \\
		{\Cxd\DualParadise\core} \&\& {\PsCat\Paradise\co}
		\arrow["{\Cnt\Paradise}", from=1-1, to=1-3]
		\arrow["{(-)\rev}"', "\wr", from=1-1, to=2-1]
		\arrow["{(-)\lop}", "\wr"', from=1-3, to=2-3]
		\arrow["{\Ctx\DualParadise\core}"', from=2-1, to=2-3]
	\end{tikzcd}
	\qquad\qquad
	\Cnt(\action)\lop = \Ctx(\action\rev),
\end{equation}

The duality $(-)\lop$ on pseudocategories reverses loose arrows; whereas $(-)\rev$ formally reverses a contentad to a contextad:
\begin{equation}
	\begin{pmatrix}
		\begin{tikzcd}[ampersand replacement=\&, cramped, sep=small]
			\& \actor \\
			\acted \&\& \acted
			\arrow["\action"', from=1-2, to=2-1]
			\arrow["{q\, \ \text{opfib. in $\Cosmos$}}", two heads, from=1-2, to=2-3]
		\end{tikzcd}
		\\[4ex]
		\combine:\actor \spancomp \actor \to \actor
	\end{pmatrix}
	\quad\longmapsto\quad
	\begin{pmatrix}
		\begin{tikzcd}[ampersand replacement=\&, cramped, sep=small]
			\& \actor \\
			\acted \&\& \acted
			\arrow["{\text{fib. in $\Cosmos\co$}\ \, q}"', two heads, from=1-2, to=2-1]
			\arrow["\action", from=1-2, to=2-3]
		\end{tikzcd}
		\\[4ex]
		\combine\rev := \actor \spancomp \actor \xto{\combine} \actor
	\end{pmatrix}
\end{equation}
This duality is thus a generalization of the one identifying left $\actor$-actegories and right $\actor\rev$-actegories \cite[Remark~3.1.3]{capucci2022actegories}, where $\actor\rev$ is the monoidal category obtained by reversing the order of its monoidal product.

We also have a different, if related, kind of duality, induced by a duality on the underlying paradise:

\begin{defn}
\label{defn:self.dual.paradise}
	A \textbf{self-dual paradise} is a paradise $\Paradise$ such that $\DualParadise$ is a coparadise, and equipped with a display-map-preserving equivalence $(-)\op : \Cosmos \to \Cosmos\co$.
\end{defn}

Concretely, a self-dual paradise has a class of display maps containing both projections out of arrow objects, plus a duality exchanging them.
The chief example of self-dual paradises is $\Cat$, but also $\MonCat_\pseudo$, $\PsCat(\Cat)$, etc.

For a self-dual paradise, the above duality specializes to a new relation between $\Ctx$ and $\Cnt$:
\begin{equation}
	\begin{tikzcd}[ampersand replacement=\&]
		{\Cnd\Paradise} \&\& {\PsCat\Paradise} \\
		{\Cxd\Paradise\co} \&\& {\PsCat\Paradise\core}
		\arrow["{\Cnt\Paradise}", from=1-1, to=1-3]
		\arrow["{(-)\op}"', "\wr", from=1-1, to=2-1]
		\arrow["{(-)\ltop}", "\wr"', from=1-3, to=2-3]
		\arrow["{\Ctx\Paradise\co}"', from=2-1, to=2-3]
	\end{tikzcd}
	\qquad\qquad
	\Cnt(\action)\ltop = \Ctx(\action\op),
\end{equation}
where $(-)\ltop$ reverses both loose and tight 1-cells of a pseudocategory.
Indeed, unlike the duality of \eqref{eq:first.duality} above, this one also reverses the tight maps of the resulting pseudocategories by reversing also $\acted$:
\begin{equation}
\label{eqn:real.duality}
	\begin{tikzcd}[ampersand replacement=\&, cramped, sep=small]
		\& \actor \\
		\acted \&\& \acted
		\arrow["\action"', from=1-2, to=2-1]
		\arrow["{q\, \ \text{opfib.}}", two heads, from=1-2, to=2-3]
	\end{tikzcd}
	\quad\longmapsto\quad
	\begin{tikzcd}[ampersand replacement=\&, cramped, sep=small]
		\& {\actor\op} \\
		{\acted\op} \&\& {\acted\op}
		\arrow["{q\op\, \ \text{fib.}}"', two heads, from=1-2, to=2-1]
		\arrow["{\action\op}", from=1-2, to=2-3]
	\end{tikzcd}
\end{equation}
This recovers many known dualities.
In the case where $\action$ is the action of a monoidal category (\cref{ex:left.actegories}), this duality recovers the known one between $\para$ and $\copara$ for the self-action of a monoidal category \cite[Remark~4]{capucci_towards_2022}, i.e.~$\copara(\actor)\coop = \para(\actor\op)$.%
\footnote{Notice that the duality $(-)\ltop$ induces $(-)\coop$ on the loose bicategory.}
If $\action$ is a monad, the duality reduces to the trivial statement that the Kleisli category of a monad $\Monad$ (\cref{ex:monads}) is dual to that of the induced comonad $\Monad\op$ on the opposite category: $\Kl(\Monad)\op = \Kl(\Monad\op)$.
If $\action$ is the lax opfibred action corresponding to a system of codisplay maps (\cref{ex:codisplay.map.cat}), then \eqref{eqn:real.duality} says $\Cospan(\acted,\Display)\ltop = \Span(\acted\op,\Display\op)$.

\section[Doctrines of wreaths and Ctx in general]{Doctrines of wreaths and $\Ctx$ in general}
\label{sec:docs.of.wreaths}

In this section, we will look at the wreath product of pseudomonads in $\DispSpan\Paradise$ more closely but also more generally.
This is because there are cases when we want to take wreaths not only around the pseudocategory of commuting squares but also around other pseudocategories in $\Cosmos$.

Wreaths and their product are the beating heart of the $\Ctx$ construction.
The functor $\FibSpan^\twoto \to \KL(\DispSpan^=)$ picks out wreaths of a familiar shape, but ultimately that's just restricting $\wreath$, a general operation which picks \emph{any} wreath in $\DispSpan\Paradise$ and produces a category in $\Cosmos$.
Indeed, restricting the wreath product to different classes of wreaths yields other familiar constructions.

\begin{ex}
\label{ex:spans}
	Many span constructions are an example of wreath products.
	We follow \cite[Definition~2.1]{haugseng_two-variable_2023} in defining the data for the $\Span$ construction to be an \textbf{adequate triple}, meaning a category $\Ca$ together with two wide subcategories, of \emph{forward} (denoted $\forwto$) and \emph{backward} (denoted $\backto$) morphisms, with the property that backward arrows pullback against forward arrows and the resulting projections are again forward and backward, as illustrated by the diagram below:
	\begin{equation}
		\begin{tikzcd}[ampersand replacement=\&]
			\cdot \& \cdot \\
			\cdot \& \cdot
			\arrow[two heads, from=1-1, to=2-1]
			\arrow[two heads, from=1-2, to=2-2]
			\arrow[tail, from=2-1, to=2-2]
			\arrow[tail, from=1-1, to=1-2]
			\arrow["\lrcorner"{anchor=center, pos=0.125}, draw=none, from=1-1, to=2-2]
		\end{tikzcd}
	\end{equation}
	We can then try to frame $\Span(\Ca, \forwto,\backto)$ as a $\Ctx$ construction.
	For every adequate triple $(\Ca, \forwto, \backto)$ we can form a fibration $\cod : \Ca^{\backdown}\vert_{\forwto} \to \Ca\vert_\forwto$, where $\Ca\vert_\forwto$ is the subcategory of $\Ca$ comprising only of the forward maps.
	The objects in the total category $\Ca^{\backdown}\vert_{\forwto}$ are backward maps, but the morphisms are commutative squares whose top and bottom are forward maps.
	This fibration is indeed a strict fibred action on $\Ca\vert_\forwto$, with action of an arrow $p:E \backto B$ on $B$ given by projecting out $E$:
	\begin{equation}
		\begin{tikzcd}[ampersand replacement=\&,sep=scriptsize]
			\& {\Ca^\backdown}\vert_{\forwto} \\
			{\Ca\vert_\forwto} \&\& {\Ca\vert_\forwto}
			\arrow["\cod"', two heads, from=1-2, to=2-1]
			\arrow["\dom", from=1-2, to=2-3]
		\end{tikzcd}
	\end{equation}
	The multiplication of this fibred action is composition of backward arrows, (which the pasting property of pullbacks guarantees to be a cartesian functor), the unit is inclusion of identities (which are both forward and backward by assumption).
	The resulting $\Ctx$ construction is a double category whose tight cells are forward maps, whose loose cells are spans of backward and forward arrows, and whose squares are arrows in $\Ca$ that makes the obvious squares commute:
	\begin{equation}
		\begin{tikzcd}[ampersand replacement=\&]
			A \& E \& B \\
			{A'} \& {E'} \& {B'}
			\arrow["p"', two heads, from=1-2, to=1-1]
			\arrow["f", tail, from=1-2, to=1-3]
			\arrow[tail, from=1-1, to=2-1]
			\arrow[tail, from=1-3, to=2-3]
			\arrow[tail, from=2-2, to=2-3]
			\arrow[two heads, from=2-2, to=2-1]
			\arrow[from=1-2, to=2-2]
		\end{tikzcd}
	\end{equation}
	Thus $\Ctx(\dom)$ is not what $\Span(\Ca, \forwto, \backto)$ should be, where tight cells are unconstrained, like the ones appearing inside the squares.

	To correct this, we can notice that a different wreath product would work, specifically one between the following monads in $\DispSpan(\Cat)$:
	\begin{equation}
	\label{eqn:span-monads}
		\begin{tikzcd}[ampersand replacement=\&,sep=scriptsize]
			\& {\Ca^\forwdown} \&\&\& {\Ca^\backdown} \\
			{\Ca} \&\& {\Ca} \& {\Ca} \&\& {\Ca}
			\arrow["\cod", from=1-2, to=2-3]
			\arrow["\dom"', from=1-2, to=2-1]
			\arrow["\cod"', from=1-5, to=2-4]
			\arrow["\dom", from=1-5, to=2-6]
		\end{tikzcd}
	\end{equation}
	Here, $\Ca^{\forwdown}$ and $\Ca^{\backdown}$ are the full subcategories of the arrow category $\Ca^{\downarrow}$ spanned by the forward and backward maps respectively.
	For the right one to be an algebra for the left one is sufficient that $\cod : \Ca^\backto \to \Ca$ admits cartesian lifts of forward arrows, which is true by assumption.
	Then, it's easy to see that taking the wreath $\Ca^\forwdown \wreath \Ca^\backdown$ gives a double category whose category of objects and tight maps is the whole $\Ca$, as desired.
\end{ex}

Therefore $\Span$ is a wreath product applied to a different choice of wreaths, namely those of the form \eqref{eqn:span-monads}.
We could retrace the same path we did for constructing $\Cxd$, hence considering a tricategory with objects pairs $(\Ca, \forwto)$ (submonads of the arrow monad on $\Ca$), 1-cells left-displayed spans whose left leg admits cartesian lift of forward maps and 2-cells cartesian and right-lax maps of spans.
However, pseudomonads in this tricategory are too general, since they would consist of any span $\Ca \nfrom{p} \Ma \nto{f} \Ca$, whereas we are interested in those which are again carried by a submonad of the arrow monad.

From $\FunWreaths(\DispSpan\Paradise) := \KL(\KL(\DispSpan\Paradise),\,\fun)\co$ we end up choosing the subtricategory of wreaths whose product are $\Span$ construction is the full subtricategory spanned by objects of the form $(\Ca, \Ca^\forwdown, \Ca^\backdown)$, which we can call $\AdTrpl$.

\begin{defn}
	A \textbf{doctrine of wreaths} on the paradise $\Paradise$ is a subtricategory
	\begin{equation}
		\Doc \into \FunWreaths(\DispSpan\Paradise).
	\end{equation}
	A doctrine is \textbf{full} if its inclusion is.
	A doctrine is \textbf{based} if its inclusion is of the form (of a restriction to functorial maps of) $\KL(\into)$.
\end{defn}

Thus a based doctrine of wreaths is a choice of wreaths made only on the grounds of which base pseudocategories (pseudomonads in $\KL(\DispSpan\Paradise)$) to wreath around.
For instance, the doctrine of contextads on categories is full and based since it includes all wreaths based on arrow monads, as well as all maps thereof.
Instead, as discussed above, the doctrine of adequate triples described just above is full but not based since it doesn't admit all wreaths around a given monad of forward arrows.

A doctrine of wreaths induces a restriction of the wreath product, which thus yields a specialized version of it for the wreaths picked out by the doctrine:
\begin{equation}
	\wreath_\Doc := \begin{tikzcd}[ampersand replacement=\&]
		{\Doc} \&[-2ex] {\FunWreaths(\DispSpan\Paradise)} \& {\PsCat\Paradise}
		\arrow["\wreath", from=1-2, to=1-3]
		\arrow[hook, from=1-1, to=1-2]
	\end{tikzcd}
\end{equation}

For contextads, this is the $\Ctx$ construction of \cref{defn:ctx.construction}.
For adequate triples, as shown in \cref{ex:spans}, it is the $\Span$ construction:
\begin{equation}
	\Span := \begin{tikzcd}[ampersand replacement=\&]
		{\AdTrpl} \&[-2ex] {\FunWreaths(\DispSpan\Paradise)} \& {\PsCat\Paradise}
		\arrow["\wreath", from=1-2, to=1-3]
		\arrow[hook, from=1-1, to=1-2]
	\end{tikzcd}
\end{equation}

\begin{ex}
	The \textbf{doctrine of comonads}
	\begin{equation}
		\Comnd(\Cosmos) \into \FunWreaths(\DispSpan\Paradise)
	\end{equation}
	is full but not based, since it picks out specific wreaths over the arrow monads.
	To see it's well-defined, it suffices to remember $\Comnd(\Cosmos)$ is a subtricategory of $\Cxd\Paradise$.
	Its associated wreath product is the Kleisli construction.
\end{ex}

\begin{ex}
	There are other subdoctrines of $\Cxd\Paradise$ given by considering special kinds of contextads on categories.
	Among these: the \textbf{doctrines of strong/strict contextads}, the \textbf{doctrine of actegories} (\cref{ex:actegory}, giving $\Para$ as its wreath product), the \textbf{doctrine of display map categories} (\cref{ex:display.map.cat}, giving the left-displayed $\Span$ construction), etc.
\end{ex}

\begin{ex}
	The doctrine of partial maps is a subdoctrine of the doctrine of display maps categories whose display maps are monic.
	These are called \emph{dominions} in \cite{rosoliniContinuityEffectivenessTopoi1986}, where they are define as minimal settings to talk about partial maps.
	Indeed, the wreath product of (the contextad associated to) a dominion constructs the double category of partial maps and total maps:
	\begin{eqalign}
		\Par : \Dom &\longto \PsCat(\Cat, \{\mathsf{all}\})\\
		\Ca &\mapsto \left\{
		\begin{tikzcd}[ampersand replacement=\&, sep=scriptsize]
			\cdot \& \cdot \\
			\cdot \& \cdot
			\arrow[harpoon, from=1-1, to=1-2]
			\arrow[from=1-1, to=2-1]
			\arrow[from=1-2, to=2-2]
			\arrow[harpoon, from=2-1, to=2-2]
		\end{tikzcd}
		\right\}
	\end{eqalign}
	where $\rightharpoonup$ stands for a span $\cdot \monofrom \cdot \to \cdot$.
\end{ex}

\subsection{Contextads around a double category}
Ultimately, we can contemplate the largest doctrine of wreaths there is, namely $\FunWreaths(\DispSpan\Paradise)$ itself.
It corresponds to considering wreaths around arbitrary pseudocategories in $\Paradise$.
We call such wreaths \textbf{contextads on a double category} (that is, on a pseudocategory in $\Cosmos$) (as opposed to our earlier contextads on a category---that is, an object of $\Cosmos$), since they provide a system for contexualizing arrows in a general pseudocategory in $\Cosmos$.

Turns out the intuition we have built for contextads on categories and their double categories of contextful arrows is quite helpful to understand contextads on double categories.
We have some applications in mind for these more general contextads, so we will unpack them here.

Fix a pseudocategory $(\acted \nfrom{s} \Xa \nto{t} \acted,\, \lid,\, \lcomp)$, whose loose arrows we denote as $\looseto$.
A wreath around it is, first of all, another span $\acted \nfrom{p} \actor \nto{\action} \acted$ which is a pseudoalgebra of $\Xa \spancomp -$ in the Kleisli category of $-\spancomp \Xa$.
This means there is an intertwiner $\intertwiner : \Xa \spancomp \actor \longto \actor \spancomp \Xa$, and we can think of it much in the same way as we did for contextads on categories \eqref{eqn:arrow.monads.intertwiner}, where $\intertwiner$ encoded the lifting operation afforded by the fibration $p$:
\begin{equation}
	\intertwiner : (f:A \looseto B, Q \in \actor_B) \mapsto (\tow f Q \in \actor_A, f \action Q: A \action (\tow f Q) \looseto B \action Q)
\end{equation}
Hence $\tow - - = p_1\intertwiner$ is a `synthetic' version of pulling back along a fibration, i.e.~cartesian lifting.
Unlike cartesian lifting, it does not enjoy a universal property (which came from $\acted^\downarrow \spancomp -$ being colax idempotent, recall \cref{defn:psalg}), but we still get a loose arrow $\tow f Q : A \action (\tow f Q) \looseto B \action Q$ from the second component of $\intertwiner$.
It doesn't take much to realize this is all we use of cartesian lifting when defining a contextad on a category or its double category of contextful arrows.
Hence contextads and their wreath products are not much different than what we saw in \cref{sec:wreath.product.intro}.

The reason we still call these wreaths contextads is that they generally embody a system for dealing with contexts.
One can interpret the objects of $\actor$ as \emph{context extensions}.
Being displayed over $\acted$ is very natural: each object $A$, which can be thought of as a context, has its own category $\actor_A$ of extensions.
The monad structure on $\acted \nfrom{p} \actor \nto{\action} \acted$ witnesses the fact each context has a trivial extension ($\combineunit$) and iterated extensions can be combined into a single extension ($\combine$).
The map $\action$ then perform a context extension $(A,P) \mapsto A \action P$.
The counitor and coassociator are arrows that represent the act of \emph{ignoring} a trivial context or \emph{unpacking} a combined context extension.
Finally, the intertwiner $(\tow - -, - \action -)$ represent the act of \emph{contextualizing} an arrow, not unlike a Kleisli lift for a comonad.

\begin{rmk}
	An important difference between cartesian lifting and $\tow - -$ is that the former is a structure on the right leg $p$ of $\acted \nfrom{p} \actor \nto{\action} \acted$, independent of the rest of the span, while $\tow - -$ cannot be defined from $p$ alone.
	Indeed, $f \action P$ is defined directly for the latter, unlike in the former case where it equals $f \action \lift f$.
\end{rmk}

\begin{rmk}
	As already noted above, the structure encoded by $\intertwiner = (\tow - -, - \action -)$ is that of a pseudoalgebra of $\Xa \spancomp -$ in $\Kl(- \spancomp \Xa)$.
	As opposed to $\acted^\downarrow \spancomp -$, in general $\Xa \spancomp -$ is not colax idempotent, meaning $\tow - -$ is a \emph{bona fide} structure (as opposed to property-like) on $\acted \nfrom{p} \actor \nto{\action} \acted$, and that maps thereof (the 2-cells in $\FunWreaths(\DispSpan\Paradise)$) are \emph{bona fide} pseudomorphisms of pseudoalgebras, i.e.~their structure morphisms are not unique as for maps of cartesian fibrations (recall \cref{defn:psmors}).
	Thus, when defining contextads below, remember the structure morphisms of combination ($\combineunitcart$ and $\combinecart$) are proper data now and are not automatically coherent like for contextads on categories~(\cref{defn:colax.fibred.action}).
\end{rmk}

\begin{rmk}
	Notice that, to develop the definitions of contextad in $\Cosmos$, there is no need for the extra condition paradises (\cref{defn:paradise}) have compared to display map 2-categories (\cref{defn:display.map.2.cat}), namely that arrow objects have display domain projection.
	Indeed, that last condition is only to guarantee $\acted^\downarrow$ is a well-defined pseudomonad in $\DispSpan\Paradise$.
\end{rmk}

\begin{defn}
\label{defn:contextad}
	Let $\dblcat X = (\acted \nfrom{s} \Xa \nto{t} \acted,\, \lid,\, \lcomp)$ be a pseudocategory in a display map 2-category $\Paradise$.
	A \textbf{contextad} on $\dblcat X$ is a wreath in $\DispSpan\Paradise$ around $\dblcat X$, thus consisting of the following data:
	\begin{enumerate}
		\item a \textbf{carrier} left-diplayed span
		      \begin{equation}
			      \begin{tikzcd}[ampersand replacement=\&,sep=scriptsize]
				      \& \actor \\
				      \acted \&\& \acted
				      \arrow["p"', from=1-2, to=2-1]
				      \arrow["\action", from=1-2, to=2-3]
			      \end{tikzcd}
		      \end{equation}
		\item an \textbf{intertwiner} 2-cell $\intertwiner =: (\tow - -, - \action -) : \Xa \spancomp \actor \longto \actor \spancomp \Xa$
		      \begin{equation}
			      \begin{tikzcd}[ampersand replacement=\&, cramped]
				      A \& B \\
				      {A'} \& {B'}
				      \arrow[""{name=0, anchor=center, inner sep=0}, "f", "\shortmid"{marking}, from=1-1, to=1-2]
				      \arrow[""{name=0p, anchor=center, inner sep=0}, phantom, from=1-1, to=1-2, start anchor=center, end anchor=center]
				      \arrow["h"', from=1-1, to=2-1]
				      \arrow["k", from=1-2, to=2-2]
				      \arrow[""{name=1, anchor=center, inner sep=0}, "{f'}"', "\shortmid"{marking}, from=2-1, to=2-2]
				      \arrow[""{name=1p, anchor=center, inner sep=0}, phantom, from=2-1, to=2-2, start anchor=center, end anchor=center]
				      \arrow["\beta", shorten <=4pt, shorten >=4pt, Rightarrow, from=0p, to=1p]
			      \end{tikzcd}
			      ,\quad%
			      \begin{tikzcd}[ampersand replacement=\&, cramped]
				      {\lens{P}{B}} \\
				      {\lens{P'}{B'}}
				      \arrow["{\lens{\varphi}{k}}", from=1-1, to=2-1]
			      \end{tikzcd}
			      \quad \longmapsto \quad
			      \begin{tikzcd}[ampersand replacement=\&, cramped]
				      {A \action (\tow f P)} \& {B \action P} \\
				      {A' \action (\tow{f'} P')} \& {B' \action P'}
				      \arrow[""{name=0, anchor=center, inner sep=0}, "{f \action P}", "\shortmid"{marking}, from=1-1, to=1-2]
				      \arrow[""{name=0p, anchor=center, inner sep=0}, phantom, from=1-1, to=1-2, start anchor=center, end anchor=center]
				      \arrow["{h \action (\tow \beta \varphi)}"', from=1-1, to=2-1]
				      \arrow["{k \action \varphi}", from=1-2, to=2-2]
				      \arrow[""{name=1, anchor=center, inner sep=0}, "{f' \action P'}"', "\shortmid"{marking}, from=2-1, to=2-2]
				      \arrow[""{name=1p, anchor=center, inner sep=0}, phantom, from=2-1, to=2-2, start anchor=center, end anchor=center]
				      \arrow["{\beta \action \varphi}"', shorten <=4pt, shorten >=4pt, Rightarrow, from=0p, to=1p]
			      \end{tikzcd}
		      \end{equation}
		      including coherent invertible 3-cells witnessing strong functoriality of $\tow - -$:
		      \begin{equation}
			      \begin{matrix}
				      \mapunitor : \tow{\lid_A} P \iso P \quad \text{$p$-vertical}, \\[2ex]
				      \begin{tikzcd}[ampersand replacement=\&,cramped]
					      {A \action (\tow{\lid_A} P)} \& {A \action P} \\
					      {A \action P} \& {A \action P}
					      \arrow[""{name=0, anchor=center, inner sep=0}, "{\lid_A \action P}", "\shortmid"{marking}, from=1-1, to=1-2]
					      \arrow[""{name=0p, anchor=center, inner sep=0}, phantom, from=1-1, to=1-2, start anchor=center, end anchor=center]
					      \arrow["{A \action \mapunitor}"', "\wr", from=1-1, to=2-1]
					      \arrow[Rightarrow, no head, from=1-2, to=2-2]
					      \arrow[""{name=1, anchor=center, inner sep=0}, "\shortmid"{marking}, Rightarrow, no head, from=2-1, to=2-2]
					      \arrow[""{name=1p, anchor=center, inner sep=0}, phantom, from=2-1, to=2-2, start anchor=center, end anchor=center]
					      \arrow["{\bar{\mapunitor}}"', "\wr", shorten <=4pt, shorten >=4pt, Rightarrow, from=0p, to=1p]
				      \end{tikzcd}
			      \end{matrix}
			      \qquad\qquad
			      \begin{matrix}
				      \mapmultiplicator : \tow f (\tow g P) \iso \tow{(f \lcomp g)} P \quad \text{$p$-vertical}, \\[2ex]
				      \begin{tikzcd}[ampersand replacement=\&]
					      {A \action (\tow f (\tow g P))} \& {A \action (\tow g P)} \& {A \action P} \\
					      {A \action \tow{(f \lcomp g)} P} \& {} \& {A \action P}
					      \arrow["{f \action (\tow g P)}", "\shortmid"{marking}, from=1-1, to=1-2]
					      \arrow["{A \action \mapmultiplicator}"', "\wr", from=1-1, to=2-1]
					      \arrow["{g \action P}", "\shortmid"{marking}, from=1-2, to=1-3]
					      \arrow["{\bar{\mapmultiplicator}}"'{pos=0.42}, "\wr"{pos=0.4}, shift right=1.2, shorten >=3pt, Rightarrow, from=1-2, to=2-2]
					      \arrow[Rightarrow, no head, from=1-3, to=2-3]
					      \arrow["{(f \lcomp g) \action P}"', "\shortmid"{marking}, from=2-1, to=2-3]
				      \end{tikzcd}
			      \end{matrix}
		      \end{equation}
		\item a \textbf{unit} 2-cell $(\combineunit,\, \counitor) : \acted \to \actor \spancomp \Xa$,
		      \begin{equation}
			      \label{eqn:contextad.unit.squares}
			      \begin{tikzcd}[ampersand replacement=\&, cramped]
				      A \\
				      B
				      \arrow["h"', from=1-1, to=2-1]
			      \end{tikzcd}
			      \quad \longmapsto \quad
			      \begin{tikzcd}[ampersand replacement=\&, cramped]
				      {\combineunit_A} \\
				      {\combineunit_{A'}}
				      \arrow["{\combineunit_h}"', from=1-1, to=2-1]
			      \end{tikzcd}
			      ,\quad
			      \begin{tikzcd}[ampersand replacement=\&, cramped]
				      {A \action \combineunit_A} \& A \\
				      {A' \action \combineunit_{A'}} \& {A'}
				      \arrow[""{name=0, anchor=center, inner sep=0}, "{\counitor}", "\shortmid"{marking}, from=1-1, to=1-2]
				      \arrow[""{name=0p, anchor=center, inner sep=0}, phantom, from=1-1, to=1-2, start anchor=center, end anchor=center]
				      \arrow["{h \action \combineunit_h}"', from=1-1, to=2-1]
				      \arrow["h", from=1-2, to=2-2]
				      \arrow[""{name=1, anchor=center, inner sep=0}, "{\counitor}"', "\shortmid"{marking}, from=2-1, to=2-2]
				      \arrow[""{name=1p, anchor=center, inner sep=0}, phantom, from=2-1, to=2-2, start anchor=center, end anchor=center]
				      \arrow["{\bar{\counitor}}"', shift left=2, shorten <=4pt, shorten >=4pt, Rightarrow, from=0p, to=1p]
			      \end{tikzcd}
		      \end{equation}
		      equipped with a coherent invertible 3-cell:
		      \begin{equation}
			      \begin{matrix}
				      \combineunitcart : \combineunit_A \iso \tow f \combineunit_B \quad \text{$p$-vertical}, \\[2ex]
				      \begin{tikzcd}[ampersand replacement=\&, cramped]
					      {A \action \combineunit_A} \& A \& B \\
					      {A\action (\tow f \combineunit_B)} \& {B \action \combineunit_B} \& B
					      \arrow["{\counitor}", "\shortmid"{marking}, from=1-1, to=1-2]
					      \arrow["{A \action \combineunitcart}"', "\wr", from=1-1, to=2-1]
					      \arrow["f", "\shortmid"{marking}, from=1-2, to=1-3]
					      \arrow["\barcombineunitcart"', "\wr", Rightarrow, from=1-2, to=2-2]
					      \arrow[Rightarrow, no head, from=1-3, to=2-3]
					      \arrow["{f \action \combineunit_B}"', "\shortmid"{marking}, from=2-1, to=2-2]
					      \arrow["{\counitor}"', "\shortmid"{marking}, from=2-2, to=2-3]
				      \end{tikzcd}
			      \end{matrix}
		      \end{equation}
		\item a \textbf{multiplication} 2-cell $(\combine,\, \coassociator) : \actor \spancomp \actor \to \actor \spancomp \Xa$
		      \begin{equation}
			      \label{eqn:contextad.mult.squares}
			      \begin{tikzcd}[ampersand replacement=\&, cramped]
				      {\lens{P}{A}} \\
				      {\lens{P'}{A'}}
				      \arrow["{\lens{\varphi}{h}}"', from=1-1, to=2-1]
			      \end{tikzcd}
			      ,\quad
			      \begin{tikzcd}[ampersand replacement=\&, cramped]
				      {\lens{Q}{A \action P}} \\
				      {\lens{Q'}{A' \action P'}}
				      \arrow["{\lens{\psi}{h \action \varphi}}"', from=1-1, to=2-1]
			      \end{tikzcd}
			      \quad \longmapsto \quad
			      \begin{tikzcd}[ampersand replacement=\&,cramped]
				      {A \action (P \combine Q)} \& {(A \action P) \action Q} \\
				      {A' \action (P' \combine Q')} \& {(A' \action P') \action Q'}
				      \arrow[""{name=0, anchor=center, inner sep=0}, "\coassociator", "\shortmid"{marking}, from=1-1, to=1-2]
				      \arrow[""{name=0p, anchor=center, inner sep=0}, phantom, from=1-1, to=1-2, start anchor=center, end anchor=center]
				      \arrow["{h \action (\varphi \combine \psi)}"', from=1-1, to=2-1]
				      \arrow["{(h \action \varphi) \action \psi}", from=1-2, to=2-2]
				      \arrow[""{name=1, anchor=center, inner sep=0}, "\coassociator"', "\shortmid"{marking}, from=2-1, to=2-2]
				      \arrow[""{name=1p, anchor=center, inner sep=0}, phantom, from=2-1, to=2-2, start anchor=center, end anchor=center]
				      \arrow["{\bar{\coassociator}}"', shorten <=4pt, shorten >=4pt, Rightarrow, from=0p, to=1p]
			      \end{tikzcd}
		      \end{equation}
		      equipped with a coherent invertible 3-cell:
		      \begin{equation}
			      \begin{matrix}
				      \combinecart : (\tow f P) \combine (\tow{f \action P} Q) \iso \tow f (P \combine Q) \quad \text{$p$-vertical} \\[2ex]
				      \begin{tikzcd}[ampersand replacement=\&, cramped]
					      {A \action ((\tow f P) \combine (\tow{f \action P} Q))} \&[-2ex] {(A \action (\tow f P)) \action (\tow{f \action P} Q)} \&[2ex] {(B \action P) \action Q} \\
					      {A \action (\tow f (P \combine Q))} \& {B \action (P \combine Q)} \& {(B \action P) \action Q}
					      \arrow["\shortmid"{marking}, "\coassociator", from=1-1, to=1-2]
					      \arrow["{A \action \combinecart}"', from=1-1, to=2-1]
					      \arrow["\shortmid"{marking}, "{(f \action P) \action Q}", from=1-2, to=1-3]
					      \arrow["{\barcombinecart}"', "\wr", Rightarrow, from=1-2, to=2-2]
					      \arrow[Rightarrow, no head, from=1-3, to=2-3]
					      \arrow["\shortmid"{marking}, "{f \action (P \combine Q)}"', from=2-1, to=2-2]
					      \arrow["\shortmid"{marking}, "\coassociator"', from=2-2, to=2-3]
				      \end{tikzcd}
			      \end{matrix}
		      \end{equation}
		\item as well as coherent invertible 3-cells $(\leftunitlaw,\bar\leftunitlaw)$, $(\rightunitlaw,\bar\rightunitlaw)$ and $(\associativitylaw,\bar\associativitylaw)$ witnessing the unitality and associativity of $\actor$ as a pseudomonad:
		      \begin{equation}
			      \begin{matrix}
				      \leftunitlaw : \combineunit_A \combine (\tow \counitor P) \iso P\quad \text{$p$-vertical}, \\[2ex]
				      \begin{tikzcd}[ampersand replacement=\&,cramped,column sep=scriptsize]
					      {A \action (\combineunit_A \combine \tow{\counitor} P)} \& {(A \action \combineunit_A) \action (\tow{\counitor} P)} \& {A \action P} \\
					      {A \action P} \& {} \& {A \action P}
					      \arrow["\coassociator", "\shortmid"{marking}, from=1-1, to=1-2]
					      \arrow["{A \action \leftunitlaw}"', "\wr", from=1-1, to=2-1]
					      \arrow["{\counitor \action P}", "\shortmid"{marking}, from=1-2, to=1-3]
					      \arrow["{\bar{\leftunitlaw}}"'{pos=0.4}, "\wr"{pos=0.4}, shorten <=1pt, shorten >=4pt, shift right=4.5, Rightarrow, from=1-2, to=2-2]
					      \arrow[Rightarrow, no head, from=1-3, to=2-3]
					      \arrow["\shortmid"{marking}, Rightarrow, no head, from=2-1, to=2-3]
				      \end{tikzcd}
			      \end{matrix}
			      \qquad\qquad
			      \begin{matrix}
				      \rightunitlaw : P \combine \combineunit_{A \action P} \iso P \quad \text{$p$-vertical}, \\[2ex]
				      \begin{tikzcd}[ampersand replacement=\&,cramped,column sep=scriptsize]
					      {A \action (P \combine \combineunit_{A \action P})} \& {(A \action P) \action \combineunit_{A \action P}} \& {A \action P} \\
					      {A \action P} \& {} \& {A \action P}
					      \arrow["\coassociator", "\shortmid"{marking}, from=1-1, to=1-2]
					      \arrow["{A \action \rightunitlaw}"', "\wr", from=1-1, to=2-1]
					      \arrow["{\counitor}", "\shortmid"{marking}, from=1-2, to=1-3]
					      \arrow["{\bar{\rightunitlaw}}"'{pos=0.4},"\wr"{pos=0.4}, shorten <=1pt, shorten >=4pt, shift right=3.5, Rightarrow, from=1-2, to=2-2]
					      \arrow[Rightarrow, no head, from=1-3, to=2-3]
					      \arrow["\shortmid"{marking}, Rightarrow, no head, from=2-1, to=2-3]
				      \end{tikzcd}
			      \end{matrix}
		      \end{equation}
		      \vspace*{2ex}
		      \begin{equation}
			      \begin{matrix}
				      \associativitylaw : (P \combine Q) \combine (\tow \coassociator R) \iso P \combine (Q \combine R)\quad \text{$p$-vertical}, \\[2ex]
				      \begin{tikzcd}[ampersand replacement=\&]
					      {A \action ((P \combine Q) \combine (\tow \coassociator R))} \& {(A \action (P \combine Q)) \action (\tow \coassociator R)} \& {((A \action P) \combine Q) \action R} \\
					      {A \action (P \combine (Q \combine R))} \& {(A \action P) \action (Q \combine R)} \& {((A \action P) \action Q) \action R}
					      \arrow["\coassociator", from=1-1, to=1-2]
					      \arrow["{A \action \associativitylaw}"', "\wr", from=1-1, to=2-1]
					      \arrow["{\coassociator \action R}", from=1-2, to=1-3]
					      \arrow["{\bar\associativitylaw}"', "\wr", Rightarrow, from=1-2, to=2-2]
					      \arrow[Rightarrow, no head, from=1-3, to=2-3]
					      \arrow["\coassociator"', from=2-1, to=2-2]
					      \arrow["\coassociator"', from=2-2, to=2-3]
				      \end{tikzcd}
			      \end{matrix}
		      \end{equation}
	\end{enumerate}
\end{defn}

\begin{defn}
\label{defn:general.ctx.dbl.cat}
	Let $(\actor, \combineunit, \combine, \action)$ be a contextad on the pseudocategory $\dblcat X$.
	The associated \textbf{pseudocategory of contextful arrows} $\Ctx(\action)$ is the wreath product of $\actor$ and $\dblcat X$, and is given as follows.
	\begin{enumerate}
		\item Its objects and tight 1-cells are those of $\acted$,
		\item Its loose 1-cells $A \looseto B$ are \textbf{contextful arrows}, thus pairs of a context $P \in \actor_A$ and a loose 1-cell $f:A \action P \looseto B$ in $\dblcat X$.
		      Identities in this direction are given by the pairs $(\combineunit, \counitor)$, and composition of contextful arrows is given by:
		      \begin{eqalign}
			      \left(\lens{P}{A},\ A \action P \nlooseto{f} B\right) \ \lcomp\ \left(\lens{Q}{B},\ B \action Q \nlooseto{g} C\right)
			      = \left(\lens{P \combine \tow f Q}{A},\ A \action (P \combine \tow f Q) \nlooseto{\delta} (A \action P) \action \tow f Q \nlooseto{\ f \action Q\ } B \action Q \nlooseto{g} C\right).
		      \end{eqalign}
		      This operation is associative and unital up to the invertible 2-cells defined below.
		\item Its squares are arrangements
		      \begin{equation}
			      \begin{tikzcd}[ampersand replacement=\&]
				      A \& B \\
				      {A'} \& {B'}
				      \arrow[""{name=0, anchor=center, inner sep=0}, "{{(P,f)}}", "\shortmid"{marking}, from=1-1, to=1-2]
				      \arrow[""{name=0p, anchor=center, inner sep=0}, phantom, from=1-1, to=1-2, start anchor=center, end anchor=center]
				      \arrow["h"', from=1-1, to=2-1]
				      \arrow["k", from=1-2, to=2-2]
				      \arrow[""{name=1, anchor=center, inner sep=0}, "{(P',f')}"', "\shortmid"{marking}, from=2-1, to=2-2]
				      \arrow[""{name=1p, anchor=center, inner sep=0}, phantom, from=2-1, to=2-2, start anchor=center, end anchor=center]
				      \arrow["{(\varphi,\beta)}"', shorten <=4pt, shorten >=4pt, Rightarrow, from=0p, to=1p]
			      \end{tikzcd}
		      \end{equation}
		      where $\varphi : P \to P'$ is a map in $\actor$ over $h$ and $\beta$ is a square in $\dblcat X$ as below:
		      \begin{equation}
			      \begin{tikzcd}[ampersand replacement=\&]
				      {A \action P} \& B \\
				      {A' \action P'} \& {B'}
				      \arrow[""{name=0, anchor=center, inner sep=0}, "f", "\shortmid"{marking}, from=1-1, to=1-2]
				      \arrow[""{name=0p, anchor=center, inner sep=0}, phantom, from=1-1, to=1-2, start anchor=center, end anchor=center]
				      \arrow["{h \action \varphi}"', from=1-1, to=2-1]
				      \arrow["k", from=1-2, to=2-2]
				      \arrow[""{name=1, anchor=center, inner sep=0}, "{f'}"', "\shortmid"{marking}, from=2-1, to=2-2]
				      \arrow[""{name=1p, anchor=center, inner sep=0}, phantom, from=2-1, to=2-2, start anchor=center, end anchor=center]
				      \arrow["\beta"', shift left=2, shorten <=4pt, shorten >=4pt, Rightarrow, from=0p, to=1p]
			      \end{tikzcd}
		      \end{equation}
		\item Tight composition of squares is simply given by stacking:
		      \begin{equation}
			      \begin{tikzcd}[ampersand replacement=\&]
				      A \& B \\
				      {A'} \& {B'} \\
				      {A''} \& {B''}
				      \arrow[""{name=0, anchor=center, inner sep=0}, "{(P,f)}", "\shortmid"{marking}, from=1-1, to=1-2]
				      \arrow[""{name=0p, anchor=center, inner sep=0}, phantom, from=1-1, to=1-2, start anchor=center, end anchor=center]
				      \arrow["h"', from=1-1, to=2-1]
				      \arrow["k", from=1-2, to=2-2]
				      \arrow[""{name=1, anchor=center, inner sep=0}, "\shortmid"{marking}, from=2-1, to=2-2]
				      \arrow[""{name=1p, anchor=center, inner sep=0}, phantom, from=2-1, to=2-2, start anchor=center, end anchor=center]
				      \arrow[""{name=1p, anchor=center, inner sep=0}, phantom, from=2-1, to=2-2, start anchor=center, end anchor=center]
				      \arrow["{h'}"', from=2-1, to=3-1]
				      \arrow["{k'}", from=2-2, to=3-2]
				      \arrow[""{name=2, anchor=center, inner sep=0}, "{(P'',f'')}"', "\shortmid"{marking}, from=3-1, to=3-2]
				      \arrow[""{name=2p, anchor=center, inner sep=0}, phantom, from=3-1, to=3-2, start anchor=center, end anchor=center]
				      \arrow["{(\varphi,\beta)}"', shorten <=4pt, shorten >=4pt, Rightarrow, from=0p, to=1p]
				      \arrow["{(\chi,\gamma)}"', shorten <=4pt, shorten >=4pt, Rightarrow, from=1p, to=2p]
			      \end{tikzcd}
			      \quad = \quad
			      \begin{tikzcd}[ampersand replacement=\&]
				      A \&\& B \\
				      {A''} \&\& {B''}
				      \arrow[""{name=0, anchor=center, inner sep=0}, "{(P,f)}", "\shortmid"{marking}, from=1-1, to=1-3]
				      \arrow[""{name=0p, anchor=center, inner sep=0}, phantom, from=1-1, to=1-3, start anchor=center, end anchor=center]
				      \arrow["{h'h}"', from=1-1, to=2-1]
				      \arrow["{k'k}", from=1-3, to=2-3]
				      \arrow[""{name=1, anchor=center, inner sep=0}, "{(P'',f'')}"', "\shortmid"{marking}, from=2-1, to=2-3]
				      \arrow[""{name=1p, anchor=center, inner sep=0}, phantom, from=2-1, to=2-3, start anchor=center, end anchor=center]
				      \arrow["{(\chi\varphi, \gamma\beta)}"', shorten <=4pt, shorten >=4pt, Rightarrow, from=0p, to=1p]
			      \end{tikzcd}
		      \end{equation}
		      and tight identity squares are inherited from $\dblcat X$.
		\item Loose composition of squares is defined as follows:
		      \begin{equation}
			      \begin{tikzcd}[ampersand replacement=\&]
				      A \& B \& C \& A \&\& C \\
				      {A'} \& {B'} \& {C'} \& {A'} \&\& {C'}
				      \arrow[""{name=0, anchor=center, inner sep=0}, "{(P,f)}", "\shortmid"{marking}, from=1-1, to=1-2]
				      \arrow[""{name=0p, anchor=center, inner sep=0}, phantom, from=1-1, to=1-2, start anchor=center, end anchor=center]
				      \arrow["h"', from=1-1, to=2-1]
				      \arrow[""{name=1, anchor=center, inner sep=0}, "{(Q,g)}", "\shortmid"{marking}, from=1-2, to=1-3]
				      \arrow[""{name=1p, anchor=center, inner sep=0}, phantom, from=1-2, to=1-3, start anchor=center, end anchor=center]
				      \arrow["k"{description}, from=1-2, to=2-2]
				      \arrow[""{name=2, anchor=center, inner sep=0}, "\ell", from=1-3, to=2-3]
				      \arrow[""{name=3, anchor=center, inner sep=0}, "{(P,f) \lcomp (Q,g)}", "\shortmid"{marking}, from=1-4, to=1-6]
				      \arrow[""{name=3p, anchor=center, inner sep=0}, phantom, from=1-4, to=1-6, start anchor=center, end anchor=center]
				      \arrow[""{name=4, anchor=center, inner sep=0}, "h"', from=1-4, to=2-4]
				      \arrow["\ell", from=1-6, to=2-6]
				      \arrow[""{name=5, anchor=center, inner sep=0}, "{(P',f')}"', "\shortmid"{marking}, from=2-1, to=2-2]
				      \arrow[""{name=5p, anchor=center, inner sep=0}, phantom, from=2-1, to=2-2, start anchor=center, end anchor=center]
				      \arrow[""{name=6, anchor=center, inner sep=0}, "{(Q',g')}"', "\shortmid"{marking}, from=2-2, to=2-3]
				      \arrow[""{name=6p, anchor=center, inner sep=0}, phantom, from=2-2, to=2-3, start anchor=center, end anchor=center]
				      \arrow[""{name=7, anchor=center, inner sep=0}, "{(P',f') \lcomp (Q',g')}"', "\shortmid"{marking}, from=2-4, to=2-6]
				      \arrow[""{name=7p, anchor=center, inner sep=0}, phantom, from=2-4, to=2-6, start anchor=center, end anchor=center]
				      \arrow["{(\varphi, \beta)}"', shorten <=4pt, shorten >=4pt, Rightarrow, from=0p, to=5p]
				      \arrow["{(\psi, \gamma)}"', shorten <=4pt, shorten >=4pt, Rightarrow, from=1p, to=6p]
				      \arrow["{=}"{marking, allow upside down}, draw=none, from=2, to=4]
				      \arrow["{(\varphi,\beta) \lcomp (\psi,\gamma)}"', shorten <=4pt, shorten >=4pt, Rightarrow, from=3p, to=7p]
			      \end{tikzcd}
		      \end{equation}
		      where the right hand side square denotes the following composite in $\dblcat X$:
		      \begin{equation}
			      \begin{tikzcd}[ampersand replacement=\&]
				      {A \action (P \combine \tow f Q)} \& {(A \action P) \action \tow f Q} \& {B \action Q} \& C \\
				      {A' \action (P' \combine \tow{f'} Q')} \& {(A' \action P') \action \tow{f'} Q'} \& {B' \action Q'} \& {C'}
				      \arrow[""{name=0, anchor=center, inner sep=0}, "\coassociator", "\shortmid"{marking}, from=1-1, to=1-2]
				      \arrow["{h \action (\varphi \combine \tow \beta \psi)}"', from=1-1, to=2-1]
				      \arrow[""{name=1, anchor=center, inner sep=0}, "{f \action Q}", "\shortmid"{marking}, from=1-2, to=1-3]
				      \arrow[""{name=1p, anchor=center, inner sep=0}, phantom, from=1-2, to=1-3, start anchor=center, end anchor=center]
				      \arrow["{(h \action \varphi) \action \tow \beta \psi}"{description}, from=1-2, to=2-2]
				      \arrow[""{name=2, anchor=center, inner sep=0}, "g", "\shortmid"{marking}, from=1-3, to=1-4]
				      \arrow[""{name=2p, anchor=center, inner sep=0}, phantom, from=1-3, to=1-4, start anchor=center, end anchor=center]
				      \arrow["{k \action \psi}"{description}, from=1-3, to=2-3]
				      \arrow["\ell", from=1-4, to=2-4]
				      \arrow[""{name=3, anchor=center, inner sep=0}, "\coassociator"', "\shortmid"{marking}, from=2-1, to=2-2]
				      \arrow[""{name=4, anchor=center, inner sep=0}, "{f' \action Q'}"', "\shortmid"{marking}, from=2-2, to=2-3]
				      \arrow[""{name=4p, anchor=center, inner sep=0}, phantom, from=2-2, to=2-3, start anchor=center, end anchor=center]
				      \arrow[""{name=5, anchor=center, inner sep=0}, "{g'}"', "\shortmid"{marking}, from=2-3, to=2-4]
				      \arrow[""{name=5p, anchor=center, inner sep=0}, phantom, from=2-3, to=2-4, start anchor=center, end anchor=center]
				      \arrow["{\bar{\coassociator}}"', shorten <=4pt, shorten >=4pt, Rightarrow, from=0, to=3]
				      \arrow["{\beta \action \psi}"', shift left=5, shorten <=4pt, shorten >=4pt, Rightarrow, from=1p, to=4p]
				      \arrow["\gamma"', shift left=3, shorten <=4pt, shorten >=4pt, Rightarrow, from=2p, to=5p]
			      \end{tikzcd}
		      \end{equation}
		      Identity squares are mutuated from \eqref{eqn:contextad.unit.squares}, the unit of $\actor$:
		      \begin{equation}
			      \begin{tikzcd}[ampersand replacement=\&]
				      A \& A \\
				      {A'} \& {A'}
				      \arrow[""{name=0, anchor=center, inner sep=0}, "{(\combineunit, \counitor)}", "\shortmid"{marking}, Rightarrow, no head, from=1-1, to=1-2]
				      \arrow[""{name=0p, anchor=center, inner sep=0}, phantom, from=1-1, to=1-2, start anchor=center, end anchor=center]
				      \arrow["h"', from=1-1, to=2-1]
				      \arrow["h", from=1-2, to=2-2]
				      \arrow[""{name=1, anchor=center, inner sep=0}, "{(\combineunit, \counitor)}"', "\shortmid"{marking}, Rightarrow, no head, from=2-1, to=2-2]
				      \arrow[""{name=1p, anchor=center, inner sep=0}, phantom, from=2-1, to=2-2, start anchor=center, end anchor=center]
				      \arrow["{(\combineunit,\bar{\counitor})}"', shorten <=4pt, shorten >=4pt, Rightarrow, from=0p, to=1p]
			      \end{tikzcd}
		      \end{equation}
		\item Associators and unitors for the composition of contextful arrows are built out of the pseudomonad structure on $\actor$ and the analogous 2-cells of $\dblcat X$ (denoted as $\ell$,$r$,$a$):
		      \begin{align}
			      \begin{tikzcd}[ampersand replacement=\&]
				      A \& A \& B \\
				      A \&\& B
				      \arrow["{(\combineunit, \counitor)}", "\shortmid"{marking}, Rightarrow, no head, from=1-1, to=1-2]
				      \arrow[Rightarrow, no head, from=1-1, to=2-1]
				      \arrow["{(P,f)}", "\shortmid"{marking}, from=1-2, to=1-3]
				      \arrow[Rightarrow, no head, from=1-3, to=2-3]
				      \arrow[""{name=0, anchor=center, inner sep=0}, "{(P,f)}"', "\shortmid"{marking}, from=2-1, to=2-3]
				      \arrow[""{name=0p, anchor=center, inner sep=0}, phantom, from=2-1, to=2-3, start anchor=center, end anchor=center]
				      \arrow["{\boldsymbol{\leftunitlaw}}"', "\wr", shorten >=3pt, Rightarrow, from=1-2, to=0p]
			      \end{tikzcd}
			      \quad & := \quad
			      \begin{tikzcd}[ampersand replacement=\&,sep=scriptsize]
				      {A \action (\combineunit \combine \tow f P)} \& {(A \action \combineunit) \action P} \& {A \action P} \& B \\
				      {A \action P} \& {} \& {A \action P} \& B \\
				      {A \action P} \& {} \&\& B
				      \arrow["\coassociator", "\shortmid"{marking}, from=1-1, to=1-2]
				      \arrow["{{A \action \leftunitlaw}}"', from=1-1, to=2-1]
				      \arrow["{{\counitor \action P}}", "\shortmid"{marking}, from=1-2, to=1-3]
				      \arrow["{\bar{\leftunitlaw}}"'{pos=0.4}, "\wr"{pos=0.4}, shift right=4.1, shorten >=2pt, Rightarrow, from=1-2, to=2-2]
				      \arrow[""{name=0, anchor=center, inner sep=0}, "f", "\shortmid"{marking}, from=1-3, to=1-4]
				      \arrow[Rightarrow, no head, from=1-3, to=2-3]
				      \arrow[Rightarrow, no head, from=1-4, to=2-4]
				      \arrow[Rightarrow, no head, from=2-1, to=2-3,"\shortmid"{marking}]
				      \arrow[Rightarrow, no head, from=2-1, to=3-1]
				      \arrow["\ell"'{pos=0.4}, "\wr"{pos=0.4}, shorten >=3pt, shift left=5.65, Rightarrow, from=2-2, to=3-2]
				      \arrow[""{name=1, anchor=center, inner sep=0}, "f", "\shortmid"{marking}, from=2-3, to=2-4]
				      \arrow[Rightarrow, no head, from=2-4, to=3-4]
				      \arrow["f"', "\shortmid"{marking}, from=3-1, to=3-4]
				      \arrow["{{=}}"{marking, allow upside down, pos=0.4}, draw=none, from=0, to=1]
			      \end{tikzcd}
			      \\
			      \begin{tikzcd}[ampersand replacement=\&]
				      A \& B \& B \\
				      A \&\& B
				      \arrow["{(P,f)}", "\shortmid"{marking}, from=1-1, to=1-2]
				      \arrow[Rightarrow, no head, from=1-1, to=2-1]
				      \arrow["{(\combineunit, \counitor)}", "\shortmid"{marking}, Rightarrow, no head, from=1-2, to=1-3]
				      \arrow[Rightarrow, no head, from=1-3, to=2-3]
				      \arrow[""{name=0, anchor=center, inner sep=0}, "{(P,f)}"', "\shortmid"{marking}, from=2-1, to=2-3]
				      \arrow[""{name=0p, anchor=center, inner sep=0}, phantom, from=2-1, to=2-3, start anchor=center, end anchor=center]
				      \arrow["{\boldsymbol{\rightunitlaw}}"', "\wr", shorten >=3pt, Rightarrow, from=1-2, to=0p]
			      \end{tikzcd}
			      \quad & := \quad
				  \begin{sideways}
					\begin{tikzcd}[ampersand replacement=\&]
						{A \action (P \combine \tow f \combineunit_B)} \& {(A \action P) \action \tow f \combineunit_B} \& {B \action \combineunit_B} \& B \\
						{A \action (P \combine \combineunit_{A \action P})} \& {(A \action P) \action \combineunit_{A \action P}} \& {A \action P} \& B \\
						{A \action P} \& {} \& {A \action P} \& B \\
						{A \action P} \& {} \&\& B
						\arrow[""{name=0, anchor=center, inner sep=0}, "\coassociator", "\shortmid"{marking}, from=1-1, to=1-2]
						\arrow[""{name=0p, anchor=center, inner sep=0}, phantom, from=1-1, to=1-2, start anchor=center, end anchor=center]
						\arrow["{A \action (P \combine {\combineunitcart}^{-1})}"', from=1-1, to=2-1]
						\arrow["{f \action \combineunit_B}", "\shortmid"{marking}, from=1-2, to=1-3]
						\arrow["{(A \action P) \action {\combineunitcart}^{-1}}"{description}, from=1-2, to=2-2]
						\arrow["\counitor", "\shortmid"{marking}, from=1-3, to=1-4]
						\arrow["{{{}\barcombineunitcart}^{-1}}"', shorten >=1pt, Rightarrow, from=1-3, to=2-3]
						\arrow[Rightarrow, no head, from=1-4, to=2-4]
						\arrow[""{name=1, anchor=center, inner sep=0}, "\coassociator"', "\shortmid"{marking}, from=2-1, to=2-2]
						\arrow[""{name=1p, anchor=center, inner sep=0}, phantom, from=2-1, to=2-2, start anchor=center, end anchor=center]
						\arrow["{A \action (P \combine \rightunitlaw)}"', from=2-1, to=3-1]
						\arrow["\counitor", "\shortmid"{marking}, from=2-2, to=2-3]
						\arrow["{\bar{\rightunitlaw}}"', shorten >=1pt, Rightarrow, from=2-2, to=3-2]
						\arrow[""{name=2, anchor=center, inner sep=0}, "f", "\shortmid"{marking}, from=2-3, to=2-4]
						\arrow[Rightarrow, no head, from=2-3, to=3-3]
						\arrow[Rightarrow, no head, from=2-4, to=3-4]
						\arrow["\shortmid"{marking}, Rightarrow, no head, from=3-1, to=3-3]
						\arrow[Rightarrow, no head, from=3-1, to=4-1]
						\arrow["\ell"', shorten >=1pt, Rightarrow, from=3-2, to=4-2]
						\arrow[""{name=3, anchor=center, inner sep=0}, "f"', "\shortmid"{marking}, from=3-3, to=3-4]
						\arrow[Rightarrow, no head, from=3-4, to=4-4]
						\arrow["f"', "\shortmid"{marking}, from=4-1, to=4-4]
						\arrow["{\bar{\coassociator}}"', shorten <=4pt, shorten >=4pt, Rightarrow, from=0p, to=1p]
						\arrow["{=}"{marking, allow upside down}, draw=none, from=2, to=3]
					\end{tikzcd}
				  \end{sideways}
		      \end{align}
		      The associator 2-cells
		      \begin{equation}
			      \begin{tikzcd}[ampersand replacement=\&]
				      A \&\& C \& D \\
				      A \& B \& {} \& D
				      \arrow["{(P,f) \then (Q,g)}", "\shortmid"{marking}, from=1-1, to=1-3]
				      \arrow[Rightarrow, no head, from=1-1, to=2-1]
				      \arrow["{(R,h)}", "\shortmid"{marking}, from=1-3, to=1-4]
				      \arrow["{\boldsymbol{\associativitylaw}}"', "\wr", shorten <=1pt, shorten >=2pt, shift right=5, Rightarrow, from=1-3, to=2-3]
				      \arrow[Rightarrow, no head, from=1-4, to=2-4]
				      \arrow["{(P,f)}"', "\shortmid"{marking}, from=2-1, to=2-2]
				      \arrow["{(Q,g) \then (R,h)}"', "\shortmid"{marking}, from=2-2, to=2-4]
			      \end{tikzcd}
		      \end{equation}
		      are defined as below
		      \begin{equation}
				\adjustbox{scale=.8,center}{
				\begin{sideways}
							\begin{tikzcd}[ampersand replacement=\&,row sep=10ex, column sep=scriptsize]
								{A \action ((P \combine (\tow f Q)) \combine (\tow{\coassociator \lcomp (f \action Q) \lcomp g} R))} \&[-2ex] {(A \action (P \combine (\tow f Q))) \action (\tow{\coassociator \lcomp (f \action Q) \lcomp g} R)} \&[2ex]\&[-10ex] {} \&[-1ex]\& {C \action R} \&[-1ex] D \\
								{A \action ((P \combine (\tow f Q)) \combine (\tow{\coassociator} \tow{(f \action Q)} \tow{g} R))} \& {(A \action (P \combine (\tow f Q))) \action (\tow{\coassociator} \tow{(f \action Q)} \tow{g} R)} \& {((A \action P) \action (\tow f Q)) \action \tow{(f \action Q)} \tow{g} R} \& {} \& {(B \action Q) \action \tow g R} \& {C \action R} \& D \\
								{A \action (P \combine ((\tow f Q) \combine \tow{(f \action Q)} \tow{g} R))} \& {(A \action P) \action ((\tow f Q) \combine \tow{(f \action Q)} \tow{g} R)} \& {((A \action P) \action (\tow f Q)) \action \tow{(f \action Q)} \tow{g} R} \& {} \& {(B \action Q) \action \tow g R} \\
								{A \action (P \combine \tow f(Q \combine \tow g R))} \& {(A \action P) \action \tow f(Q \combine \tow g R)} \& {} \& {B \action (Q \combine \tow g R)} \& {(B \action Q) \action \tow g R} \& {C \action R} \& D
								\arrow["{A \action (P \combine \combinecart)}"', from=3-1, to=4-1]
								\arrow[""{name=0, anchor=center, inner sep=0}, "\coassociator", "\shortmid"{marking}, from=3-1, to=3-2]
								\arrow[""{name=0p, anchor=center, inner sep=0}, phantom, from=3-1, to=3-2, start anchor=center, end anchor=center]
								\arrow["{g \action R}", "\shortmid"{marking}, from=2-5, to=2-6]
								\arrow[""{name=1, anchor=center, inner sep=0}, "h", "\shortmid"{marking}, from=2-6, to=2-7]
								\arrow["{A \action \associativitylaw}"', from=2-1, to=3-1]
								\arrow["{\coassociator \action \tow{(f \action Q)} \tow g R}", "\shortmid"{marking}, from=2-2, to=2-3]
								\arrow[""{name=2, anchor=center, inner sep=0}, "\coassociator"', "\shortmid"{marking}, from=2-1, to=2-2]
								\arrow[""{name=2p, anchor=center, inner sep=0}, phantom, from=2-1, to=2-2, start anchor=center, end anchor=center]
								\arrow["\coassociator", "\shortmid"{marking}, from=3-2, to=3-3]
								\arrow[Rightarrow, no head, from=2-3, to=3-3]
								\arrow["{(f \action Q) \action R}", "\shortmid"{marking}, from=2-3, to=2-5]
								\arrow["\coassociator"', "\shortmid"{marking}, from=4-4, to=4-5]
								\arrow["{(f \action Q) \action R}", "\shortmid"{marking}, from=3-3, to=3-5]
								\arrow[Rightarrow, no head, from=3-5, to=4-5]
								\arrow[Rightarrow, no head, from=2-5, to=3-5]
								\arrow["{g \action R}"', "\shortmid"{marking}, from=4-5, to=4-6]
								\arrow["h"', "\shortmid"{marking}, from=4-6, to=4-7]
								\arrow[Rightarrow, no head, from=2-7, to=4-7]
								\arrow["{(A \action P) \action \combinecart}", from=3-2, to=4-2]
								\arrow["{f \action (Q \combine \tow g R)}"', "\shortmid"{marking}, from=4-2, to=4-4]
								\arrow[""{name=3, anchor=center, inner sep=0}, "\coassociator"', "\shortmid"{marking}, from=4-1, to=4-2]
								\arrow[""{name=3p, anchor=center, inner sep=0}, phantom, from=4-1, to=4-2, start anchor=center, end anchor=center]
								\arrow["{\barcombinecart}"', shorten >=1pt, Rightarrow, from=3-3, to=4-3]
								\arrow["{\bar{\associativitylaw}}"', shorten >=1pt, Rightarrow, from=2-2, to=3-2]
								\arrow["{=}"{marking, allow upside down}, draw=none, from=2-6, to=4-6]
								\arrow[""{name=4, anchor=center, inner sep=0}, "\coassociator", "\shortmid"{marking}, from=1-1, to=1-2]
								\arrow[""{name=4p, anchor=center, inner sep=0}, phantom, from=1-1, to=1-2, start anchor=center, end anchor=center]
								\arrow["{A \action ((P \combine \tow f Q) \combine \mapmultiplicator)}"', from=1-1, to=2-1]
								\arrow["{(A \action (P \combine \tow f Q)) \action \mapmultiplicator}"{description}, from=1-2, to=2-2]
								\arrow["{\coassociator \lcomp (f \action Q) \lcomp g}", "\shortmid"{marking}, from=1-2, to=1-6]
								\arrow[""{name=5, anchor=center, inner sep=0}, "h", "\shortmid"{marking}, from=1-6, to=1-7]
								\arrow[Rightarrow, no head, from=1-6, to=2-6]
								\arrow[Rightarrow, no head, from=1-7, to=2-7]
								\arrow["\bar\mapmultiplicator"'{pos=0.35}, shorten <=1pt, shorten >=15pt, Rightarrow, from=1-4, to=2-4]
								\arrow["{=}"{marking, allow upside down}, draw=none, from=2-4, to=3-4]
								\arrow["{\bar{\coassociator}}"', shorten <=4pt, shorten >=4pt, Rightarrow, from=0p, to=3p]
								\arrow["\bar\coassociator"', shorten <=4pt, shorten >=4pt, Rightarrow, from=4p, to=2p]
								\arrow["{=}"{marking, allow upside down}, draw=none, from=5, to=1]
							\end{tikzcd}
						\end{sideways}
					}
		      \end{equation}
	\end{enumerate}
\end{defn}

\begin{rmk}
	Unlike for contextads on categories (\cref{thm:ctx.has.companions}), the double category of contextful maps associated to a contextads on an arbitrary double category does not always have all companions; instead, it inherits companions from the underlying double category, with $f_\top = (I, \counitor \lcomp f_\top)$.
\end{rmk}

\begin{rmk}
	Just as $\Ctx$ generalized the Kleisli construction of comonads (\cref{sec:ctx.as.dep.graded.comonad}), so $\Ctx$ generalizes the Kleisli construction of \emph{loose double comonads} (dual to the \emph{horizontal double monads} of \cite[Theorem~9.1]{gambino2024monoidal}).
	In \cite{cruttwellUnifiedFrameworkGeneralized2010}, the authors show such a Kleisli construction can be performed for \emph{tight double (co)monads} too, though the result is only a \emph{virtual} double category.

	Such an observation hints at the resolution of a potential critique to our proposed notion of contextad on a double category, namely that is more `bicategorical' than `double categorical', since its structure morphisms ($\counitor$ and $\coassociator$) are not tight but loose.
	Cruttwell--Shulman's construction of the tight Kleisli virtual double category suggests how to mantain the structural morphisms of a contextad tight while still be able to get a $\Ctx$ construction out of it.
	The trick would be to target \emph{covirtual double categories} (i.e.~where squares have arity one but arbitrary coarity), defining squares as follows (showing coarity two for clarity):
	\begin{equation}
		\begin{tikzcd}[ampersand replacement=\&,sep=scriptsize]
			A \& {} \& C \\[3ex]
			{A'} \& {B'} \& {C'}
			\arrow["{(P,f)}", "\shortmid"{marking}, from=1-1, to=1-3]
			\arrow["h"', from=1-1, to=2-1]
			\arrow["{(\varphi, \beta)}"{pos=0.6}, shorten <=1pt, shorten >=1pt, Rightarrow, from=1-2, to=2-2]
			\arrow["k", from=1-3, to=2-3]
			\arrow["{(P_1', f_1')}"', "\shortmid"{marking}, from=2-1, to=2-2]
			\arrow["{(P_2',f_2')}"', "\shortmid"{marking}, from=2-2, to=2-3]
		\end{tikzcd}
		\quad := \quad
		\begin{tikzcd}[ampersand replacement=\&,sep=scriptsize]
			{A \action P} \& {} \& C \\
			{A' \action (P_1' \combine (\tow {f_1'} P_2'))} \\
			{(A' \action P_1') \action (\tow {f_1'} P_2')} \& {B' \action P_2'} \& {C'}
			\arrow["f", "\shortmid"{marking}, from=1-1, to=1-3]
			\arrow["{h \action \varphi}"', from=1-1, to=2-1]
			\arrow["\beta", shorten <=3pt, shorten >=5pt, shift right=4.4, Rightarrow, from=1-2, to=3-2]
			\arrow["k", from=1-3, to=3-3]
			\arrow["\coassociator"', from=2-1, to=3-1]
			\arrow["{f_1' \action P_2'}"', "\shortmid"{marking}, from=3-1, to=3-2]
			\arrow["{f_2}"', "\shortmid"{marking}, from=3-2, to=3-3]
		\end{tikzcd}
	\end{equation}
	We leave investigating such a direction for future works.
\end{rmk}

\subsubsection{Functoriality}
Like for any doctrine of wreaths, the construction obtained from the wreath product is trifunctorial:
\begin{equation}
	\Ctx : \Cxd\Paradise \longto \PsCat\Paradise
\end{equation}
where $\Cxd\Paradise$ is a new name for $\FunWreaths(\DispSpan\Paradise)$.
Indeed, one can unpack the definition of 1-, 2- and 3-cells of contextads by unfolding those of the latter tricategory.
We do it only for 1-cells, leaving the interested reader to unfold the others:

\begin{defn}
\label{defn:morphism-of-contextads}
	Let $\ctxtad$ and $\ctxtad[']$ be contextads over $\dblcat X = (\acted \nfrom{s} \Xa \nto{t} \acted,\, \lid,\, \lcomp)$ and $\dblcat X' = (\acted' \nfrom{s'} \Xa' \nto{t'} \acted',\, \lid',\, \lcomp')$, respectively.
	A \textbf{morphism of contextads} $\underline F: \action \to \action'$ is given by the following data:
	\begin{enumerate}
		\item A \textbf{pseudofunctor on contexts} $F:\dblcat X \to \dblcat X'$ between the pseudocategories of contexts,
		\item A \textbf{map on extensions} $(F^\flat, \lineator):\actor \to \actor' \spancomp \Xa$, where $\lineator$ is called \textbf{lineator}:
		      \begin{equation}
			      \begin{tikzcd}[ampersand replacement=\&]
				      {\lens{P}{A} } \\
				      {\lens{P'}{A'}}
				      \arrow["{\lens{\varphi}{h}}"', from=1-1, to=2-1]
			      \end{tikzcd}
			      \quad\longmapsto\quad
			      \begin{matrix}
				      \lens{F^\flat P}{FA},
				      \\[2ex]
				      \begin{tikzcd}[ampersand replacement=\&]
					      {FA \action F^\flat P} \& {F(A \action P)} \\
					      {FA' \action F^\flat P'} \& {F(A' \action P')}
					      \arrow[""{name=0, anchor=center, inner sep=0}, "\lineator", "\shortmid"{marking}, from=1-1, to=1-2]
					      \arrow[""{name=0p, anchor=center, inner sep=0}, phantom, from=1-1, to=1-2, start anchor=center, end anchor=center]
					      \arrow["{Fh \action F^\flat \varphi}"', from=1-1, to=2-1]
					      \arrow["{F(h \action \varphi)}", from=1-2, to=2-2]
					      \arrow[""{name=1, anchor=center, inner sep=0}, "\lineator"'{yshift=-.3ex}, "\shortmid"{marking}, from=2-1, to=2-2]
					      \arrow[""{name=1p, anchor=center, inner sep=0}, phantom, from=2-1, to=2-2, start anchor=center, end anchor=center]
					      \arrow["{\bar\lineator}"', shift left=.3, shorten <=4pt, shorten >=4pt, Rightarrow, from=0p, to=1p]
				      \end{tikzcd}
			      \end{matrix}
		      \end{equation}
		      together with a coherent isomorphism (for $p:A \looseto B$, $Q \in \actor_B$)
		      \begin{equation}
			      \begin{matrix}
				      \varpi : \tow {Fp} F^\flat Q \iso F^\flat (\tow p Q)\quad \text{vertical}
				      \\[2ex]
				      \begin{tikzcd}[ampersand replacement=\&]
					      {FA \action' (\tow Fp F^\flat Q)} \& {FB \action' F^\flat Q} \& {F(B \action Q)} \\
					      {FA \action' F^\flat (\tow p Q)} \& {F(A \action (\tow p Q))} \& {F(B \action Q)}
					      \arrow["{Fp \action' F^\flat Q}", "\shortmid"{marking}, from=1-1, to=1-2]
					      \arrow["{FA \action' \varpi}"', from=1-1, to=2-1]
					      \arrow["\lineator", "\shortmid"{marking}, from=1-2, to=1-3]
					      \arrow["{\bar\varpi}"', shorten <=1pt, shorten >=1pt, Rightarrow, from=1-2, to=2-2]
					      \arrow[Rightarrow, no head, from=1-3, to=2-3]
					      \arrow["\lineator"', "\shortmid"{marking}, from=2-1, to=2-2]
					      \arrow["{p \action Q}"', "\shortmid"{marking}, from=2-2, to=2-3]
				      \end{tikzcd}
			      \end{matrix}
		      \end{equation}
		\item An invertible \textbf{unitor} (for $A \in \dblcat C$):
		      \begin{equation}
			      \begin{matrix}
				      \mapunitor : F^\flat \combineunit_A \iso \combineunit'_{FA}\quad \text{vertical}
				      \\[2ex]
				      \begin{tikzcd}[ampersand replacement=\&]
					      {FA \action' F^\flat \combineunit_A} \& {FA\action' \combineunit'_{FA}} \\
					      {F(A \action \combineunit_A)} \\
					      FA \& FA
					      \arrow["{{FA \action' \mapunitor}}", from=1-1, to=1-2]
					      \arrow["\lineator"', "\shortmid"{marking}, from=1-1, to=2-1]
					      \arrow[""{name=0, anchor=center, inner sep=0}, "{{\counitor'}}", "\shortmid"{marking}, from=1-2, to=3-2]
					      \arrow["{{F\counitor}}"', "\shortmid"{marking}, from=2-1, to=3-1]
					      \arrow[Rightarrow, no head, from=3-1, to=3-2]
					      \arrow["{\bar{\mapunitor}}", shorten <=6pt, shorten >=6pt, Rightarrow, from=2-1, to=0]
				      \end{tikzcd}
			      \end{matrix}
		      \end{equation}
		\item An invertible \textbf{multiplicator} (for each pair $\lens{P}{A}, \lens{Q}{A \action P} \in \actor \spancomp \actor$):
		      \begin{equation}
			      \begin{matrix}
				      \mapmultiplicator : F^\flat (P \combine Q) \iso F^\flat P \combine' (\tow \lineator F^\flat Q) \quad \text{vertical}
				      \\[2ex]
					  \hspace*{-9ex}
				      \begin{tikzcd}[ampersand replacement=\&, column sep=scriptsize]
					      {FA \action' F^\flat (P \combine Q)} \& {F(A \action (P \combine Q))} \& {} \& {F((A \action P) \action Q)} \\
					      {FA \action' (F^\flat P \combine' (\tow \lineator F^\flat Q))} \& {(FA \action' F^\flat P) \action' (\tow \lineator F^\flat Q)} \& {F(A \action P) \action ' F^\flat Q} \& {F((A \action P) \action Q)}
					      \arrow["{{{\lineator}}}", "\shortmid"{marking}, from=1-1, to=1-2]
					      \arrow["{{{FA \action' \mapmultiplicator}}}"', from=1-1, to=2-1]
					      \arrow["{{F\coassociator}}", "\shortmid"{marking}, from=1-2, to=1-4]
					      \arrow[""{name=0, anchor=center, inner sep=0}, draw=none, from=1-3, to=1-2]
					      \arrow[""{name=0p, anchor=center, inner sep=0}, phantom, from=1-3, to=1-2, start anchor=center, end anchor=center]
					      \arrow[Rightarrow, no head, from=1-4, to=2-4]
					      \arrow["{{{\coassociator'}}}"', "\shortmid"{marking}, from=2-1, to=2-2]
					      \arrow[""{name=1, anchor=center, inner sep=0}, "{\lineator \action' F^\flat Q}"', "\shortmid"{marking}, from=2-2, to=2-3]
					      \arrow[""{name=1p, anchor=center, inner sep=0}, phantom, from=2-2, to=2-3, start anchor=center, end anchor=center]
					      \arrow["\lineator"', "\shortmid"{marking}, from=2-3, to=2-4]
					      \arrow["{{\bar{\mapmultiplicator}}}"', shorten <=4pt, shorten >=4pt, Rightarrow, from=0p, to=1p]
				      \end{tikzcd}
			      \end{matrix}
		      \end{equation}
	\end{enumerate}
	Additionally, unitor and multiplicator satisfy triangular and pentagonal coherence laws relative to the unitors and associators of $\ctxtad$ and $\ctxtad[']$.
\end{defn}

\end{document}